\documentclass[12pt]{amsart}

\usepackage{bm,amsthm,amssymb,mathrsfs,amsfonts,verbatim,enumitem,color,leftidx}
\usepackage{units}
\usepackage{tikzsymbols}
\usepackage{etoolbox} 
\usepackage{bbm}
\usepackage[all,tips]{xy}
\usepackage{accents}
\usepackage{stmaryrd}
\usepackage{xfrac} 
\usepackage{thmtools}
\usepackage{pgfplots}
\pgfplotsset{compat=1.15}
\usepackage{mathrsfs}
\usetikzlibrary{arrows}
\usepackage{wrapfig}
\usepackage{mathtools}
\usepackage{slashbox}
\usepackage{hyperref}
\usepackage{scrextend}
\usepackage[original]{imakeidx}
\makeindex
\addtokomafont{labelinglabel}{\sffamily}

\newtheorem{thm}{Theorem}[section]
\newtheorem{lem}[thm]{Lemma}
\newtheorem{ques}[thm]{Question}
\newtheorem{cor}[thm]{Corollary}

\newtheorem{claim}[thm]{Claim}
\newtheorem{conj}[thm]{Conjecture}

\theoremstyle{definition}

\declaretheoremstyle[
  headfont=\normalfont\bfseries,
  numbered=unless unique,
  bodyfont=\normalfont,
  spaceabove=1em plus 0.75em minus 0.25em,
  spacebelow=1em plus 0.75em minus 0.25em,
  qed={$\heartsuit$},
]{hartending}

\declaretheorem[
  style=hartending,
  title=Definition,
  sibling=thm,
]{de}

\declaretheoremstyle[
  headfont=\normalfont\bfseries,
  numbered=unless unique,
  bodyfont=\normalfont,
  spaceabove=1em plus 0.75em minus 0.25em,
  spacebelow=1em plus 0.75em minus 0.25em,
  qed={$\clubsuit$},
]{clubsending}

\declaretheorem[
  style=clubsending,
  title=Observation,
  sibling=thm,
]{obs}

\declaretheoremstyle[
  headfont=\normalfont\bfseries,
  numbered=unless unique,
  bodyfont=\normalfont,
  spaceabove=1em plus 0.75em minus 0.25em,
  spacebelow=1em plus 0.75em minus 0.25em,
  qed={$\diamondsuit$},
]{diamending}

\declaretheorem[
  style=diamending,
  title=Example,
  sibling=thm,
]{ex}

\declaretheorem[
  style=diamending,
  title=Counterexample,
  sibling=thm,
]{cex}

\theoremstyle{remark}
\newtheorem{remark}[thm]{Remark}

\numberwithin{equation}{section}

\makeatletter

\newcommand{\Rmnum}[1]{\expandafter\@slowromancap\romannumeral #1@}
\makeatother

\newcommand{\brho}{{\boldsymbol{\rho}}}
\newcommand{\bpi}{{\boldsymbol{\pi}}}

\newcommand{\RR}{\mathbb{R}}
\newcommand{\NN}{\mathbb{N}}
\newcommand{\SL}{\operatorname{SL}}
\newcommand{\GL}{\operatorname{GL}}

\newcommand{\ZZ}{\mathbb{Z}}

\newcommand{\CC}{\mathbb{C}}

\DeclareMathOperator{\cov}{{\mathrm {covol}}}
\DeclareMathOperator{\vol}{{\mathrm {vol}}}

\DeclareMathOperator{\diag}{\mathrm{diag}}
\DeclareMathOperator{\spa}{\mathrm{span}}

\DeclareMathOperator{\supp}{\mathrm{supp\,}}
\DeclareMathOperator{\rk}{\mathrm{rank}}

\DeclareMathOperator{\tr}{tr}
\DeclareMathOperator{\bl}{Bl}
\newcommand{\HN}{F_{\mathrm{HN}}}

\newcommand{\bra}{\left\langle}
\newcommand{\ket}{\right\rangle}
\newcommand{\normi}{\|}

\newcommand{\conv}{{\rm conv}}

\newcommand{\fU}{\mathfrak{U}}
\newcommand{\fV}{\mathfrak{V}}

\newcommand{\cI}{\mathcal{I}}

\newcommand{\cG}{\mathcal{G}}

\newcommand{\cF}{\mathcal{F}}

\newcommand{\nin}{\notin}
\newcommand{\gr}{{{\rm Gr}(l, n)}}
\newcommand{\grl}[1]{{{\rm Gr}(#1, n)}}
\newcommand{\grln}[2]{{{\rm Gr}(#1, #2)}}
\newcommand{\Eall}{E_{\rm all}}
\newcommand{\Id}{{\rm Id}}
\DeclareMathOperator{\mult}{mult}
\newcommand{\BL}{{\mathscr B}_l^n}
\def\acts{\curvearrowright}
\newcommand{\funh}{{\text{\it\em\fontfamily{qcr}\selectfont h}}}
\newcommand{\funH}{{\text{\it\em\fontfamily{qcr}\selectfont\bf H}}}
\newcommand{\funq}{{\text{\it\em\fontfamily{qcr}\selectfont q}}}

\newcommand{\mLU}{m_{H^{-0}H}}
\newcommand{\et}{\triangleleft}
\newcommand{\eteq}{\trianglelefteq}
\newcommand{\nula}{\xrightarrow{*}}

\definecolor{bluecolor}{rgb}{0.30196078431372547,0.30196078431372547,1.}
\definecolor{lightbluecolor}{rgb}{0.6509803921568628,0.6509803921568628, 1.}
\definecolor{redcolor}{rgb}{1.,0.,0.}
\definecolor{lightredcolor}{rgb}{1.,.5,.5}
\definecolor{lightblackcolor}{rgb}{.5,.5,.5}
\definecolor{greencolor}{rgb}{0.,1.,0.}
\definecolor{lightgreencolor}{rgb}{0.5,1.,0.5}
\definecolor{pinkcolor}{rgb}{1.,0.6,0.8}

\usepackage{enumitem}
\newlist{additionalconstrians}{enumerate}{1}
\setlist[additionalconstrians,1]{label=(Add \roman*)}

\begin{document}

\title{Parametric Geometry of Numbers for a General Flow}
\author{Omri Nisan Solan}
\date{\today}
\begin{abstract}
We extend results on parametric geometry of numbers to a general diagonal flow on the space of lattices. Moreover, we compute the Hausdorff dimension of the set of trajectories with every given behavior, with respect to a nonstandard metric and  obtain bounds on the standard ones.
\end{abstract}
\thanks{I thank my advisor Barak Weiss and Eilon Solan for their comments, and acknowledge the support of ISF 2919 BSF 2016256. I wish to thank Andrei Iacob who read the paper and improved the presentation.}
\maketitle

\markright{}
\section{Introduction} 
\label{sec:introduction}
Let 
$X_n$ be the space of rank-$n$ lattices in $\RR^n$ with covolume $1$, and let
$g_t\in \SL_n(\RR)$ be a $1$-parameter diagonal flow.
There is a long history of research on the behavior of $g_t$-orbits in $X_n$, see \cite{D}, \cite{K2}, \cite{KM}. 

In 2011, Cheung \cite{C2} computed the Hausdorff dimension of the set of divergent $g_t=\exp(t\diag(-2, 1, 1))$-trajectories in $X_3$, and proved that the Hausdorff codimension is $2/3$. 
Later on, many extensions of his result were proved, e.g.,
\begin{center}
 \begin{tabular}{||c c c c c ||} 
 \hline
 Paper & space & flow eigenvalues & codimension & remarks\\ [0.5ex] 
 \hline\hline
 \cite{C2} & $X_3$ & $-2, 1,1$ & $\frac{2}{3}$ &\\ 
 \hline
 \cite{EK} & $X_3$ & $-2, 1,1$ & $\ge\frac{2}{3}$ & using entropy method\\ 
 \hline
 \cite{CC} & $X_{n+1}$ & $-n, 1,\dots,1$ & $\frac{n}{n+1}$ & \\ 
 \hline
 \cite{KKLM} & $X_{n+m}$ & $-m,\dots,-m, n,\dots,n$ & $\ge\frac{nm}{n+m}$ & \\ 
 \hline
 \cite{LSST} & $X_{3}$ & $-1, w_1, w_2;~\begin{matrix*}[l]
   w_1\le w_2\\ w_1+w_2=1
 \end{matrix*}$ & $\frac{1}{1+w_1}$ & computed on \\
 &&&&$\begin{pmatrix}
1& 0& 0\\
*& 1& 0\\
*& 0& 1\\
\end{pmatrix}$-orbits\\ 
 \hline
 \cite{GS} & general $G/\Gamma$ & General flow & $>0$ &\\
 \hline
 \cite{AGMS} & Product space & $\begin{matrix*}[l]
   \text{flow weighted} \\\text{with respect} \\\text{to product} 
 \end{matrix*}$ & exact result &\\
 \hline
\end{tabular}
\end{center}



\bigskip
In the first four results, 
all eigenvalues of the flow that are greater than $0$ are equal, and so are the eigenvalues smaller than $0$.



An independent research direction initiated by Schmidt and Summerer \cite{SS} and continued by Roy \cite{R}, though solving a seemingly different problem, classified the possible trajectories $g_t \Lambda$ for $\Lambda\in X_n$ up to equivalence. They use the action of
\begin{align}\label{eq:monoexpanding eigenvalue}
g_t:=\exp(t\diag(n,\dots,n,-m,\dots,-m))
\end{align}
on $X_{n+m}$.
We say two trajectories $g_t \Lambda_1, g_t \Lambda_2$ are equivalent if the distance between them is bounded for all $t>0$. 

Das, Fishman, Simmons, and Urbański \cite{DFSU} merged these two research directions and proved a general result computing the Hausdorff dimensions of sets of trajectories closed under equivalence in the setting of the $g_t=\exp(t\diag(n,\dots,n,-m,\dots,-m))$-action on $X_{n+m}$.

In this paper we extend the result by Das, Fishman, Simmons, and Urbański \cite{DFSU} to a set of trajectories with respect to a general one-dimensional diagonal action $g_t\acts X_n$. We compute the dimension of the set of trajectories with any specific behavior, provide exact dimension formulas with respect to a weighted metric, which depends on $g_t$, and derive lower and upper bounds on the standard Hausdorff dimension.
\subsection{Structural Exposition} 
\label{sub:structureal_exposition}
Section \ref{sec:preliminaries} is devoted to the preliminary material that will be needed to formally state the main results of this paper. Sections \ref{sec:hausdorff_dimensions}-\ref{sec:computation_of_the_dimension} are devoted to Theorem \ref{thm: general scewed dimention formula}, the main result of this paper, which determines the Hausdorff dimension of certain families of trajectories with respect to a certain (non-standard) metric. Section \ref{sec:combinatorial_computations} proves Corollary \ref{cor: general scewed dimention of divergent}, which applies Theorem \ref{thm: general scewed dimention formula} to the family of divergent trajectories. 

Theorem \ref{thm: general scewed dimention formula} transforms a dynamical question on the space of lattices, into a combinatorial question on combinatorial objects, namely $g$-templates.
In Section \ref{sec:hausdorff_dimensions} we discuss the Hausdorff dimension of sets, and present our computational tool which reduces the computation of the Hausdorff dimension to the dynamical question of computing the value of the Dimension Game introduced in \cite{DFSU}. 
In Section \ref{sec:approximation_of_templates} we discuss properties of $g$-templates. In Sections \ref{sec:geometric_theorems} and \ref{sec:Flag_perturbation_theorem} we prove geometric theorems that we will apply in the following three sections.
Finally, in Sections \ref{sec:alice_s_strategy} and \ref{sec:Bob_s_strategy} we describe strategies in the dimension game of the two players, and in Section \ref{sec:computation_of_the_dimension} we compute the game's value. 

\subsubsection{Mathematical topics} 
\label{ssub:mathematical_topics}
We use ideas from three different areas: Geometry, Dynamics, and Combinatorics. 
\begin{enumerate}[label=(\alph*)]
	\item {\bf Geometry} is used to discuss properties of the action of $\SL_n(\RR)$ and its subgroups $g_t$ and $H$ on the Grassmannian. 
	\item We discuss {\bf dynamical} properties of orbits. 
	\item We construct {\bf combinatorial} objects and then encounter combinatorial problems. 
\end{enumerate}
These ideas are not independent: the combinatorial objects are designed with a view to the geometric theorems, we iteratively apply the geometric theorems to obtain orbits with dynamical properties and the construction of each orbit follows a combinatorial object. 
\subsubsection{Dependencies} 
\label{ssub:dependencies}
Altough the proofs of the main results depend on intuition acquired in most of the chapters preceding them, in a first reading one can focus only on Sections \ref{sec:preliminaries}, \ref{sub:hausdorff_games}, and \ref{sec:approximation_of_templates} up to Subsection \ref{sub:shifting_templates}, \ref{sec:Flag_perturbation_theorem}, the introduction of Section \ref{sec:alice_s_strategy}, and Sections \ref{sec:Bob_s_strategy} -- \ref{sec:combinatorial_computations}.


\section{Preliminaries} 
\label{sec:preliminaries}
\subsection{Geometry of Trajectories} 
\label{sub:geometry_of_trajectories}
\index{Xn@$X_n$}\hypertarget{a}{} 
\index{SLnR@$\SL_n(\RR)$}\hypertarget{b}{} 
\index{SLnZ@$\SL_n(\ZZ)$}\hypertarget{c}{} 
\index{n@$n$}\hypertarget{d}{} 
Let $X_n = \SL_n(\RR) / \SL_n(\ZZ)$
denote the topological space of rank-$n$ lattices in $\RR^n$ with determinant $1$. 
Denote the standard right-invariant Riemannian metric on $\SL_n(\RR)$ by $d_{\SL_n(\RR)}$. 
\index{dXn@$d_{X_n}$}\hypertarget{e}{} 
\index{dSLn@$d_{\SL_n}$}\hypertarget{f}{} 
It induces a metric $d_{X_n}$ on $X_n$. 

By right-invariance,

\begin{align}\label{eq: close action}
\forall g\in \SL_n(\RR), \quad\Lambda\in X_n. ~~d_{X_n}(g\Lambda, \Lambda)<d_{\SL_n(\RR)}(g,\Id).
\end{align}

\index{etai@$\eta_i$}\hypertarget{g}{} 
Fix a sequence of real weights \[\eta_1\le \eta_2\le\dots\le\eta_n\text{ with }\eta_1+\eta_2+\dots+\eta_n = 0,\]
for the rest of the paper. 
\index{gt@$g_t$}\hypertarget{h}{} 
Denote $g_t := \exp(t\diag(\eta_1,\dots,\eta_n))$ be a $1$-parameter group $\{g_t:t\in \RR\}\subseteq \SL_n(\RR)$.
\index{Eall@$\Eall$}\hypertarget{i}{} 
Denote by $\Eall:=\{\eta_1, \dots, \eta_n\}$ the multiset of all weights. The multiplicity of $\eta_i$ is the number of indices $j$ such that $\eta_j = \eta_i$.

\index{Equivalent trajectories}\hypertarget{j}{} 
We say that two trajectories $\{g_t\Lambda: t\ge 0\}$ and $\{g_t\Lambda': t\ge 0\}$ are equivalent if the distance $d_{X_n}(g_t\Lambda, g_t\Lambda')$ is uniformly bounded for $t>0$.

\subsection{Auxiliary Group Actions} 
\label{sub:auxiliary_group_actions}

Set
\index{H@$H$}\hypertarget{ba}{} 
\index{H@$H^-$}\hypertarget{bb}{} 
\index{H@$H^0$}\hypertarget{bc}{} 
\index{H@$H^{-0}$}\hypertarget{bd}{} 
\begin{align*}
\mathrlap H\hphantom{H^{-0}} =& \left\{h\in\SL_n(\RR):g_{-t}hg_t\xrightarrow{t\to \infty}\Id\right\}&&\text{-- The expanding subgroup.}\\
\mathrlap {H^-}\hphantom{H^{-0}} =& \left\{h\in\SL_n(\RR):g_{t}hg_{-t}\xrightarrow{t\to \infty}\Id\right\}&&\text{-- The contracting subgroup.}\\
\mathrlap {H^0}\hphantom{H^{-0}} =& \left\{h\in\SL_n(\RR):g_{t}hg_{-t}=h\right\}&&\text{-- The stable subgroup.}\\
H^{-0} =& H^-H^0&&\text{-- The nonexpanding subgroup.}
\end{align*}
\begin{obs}\label{obs: H coordinate}
Expressing the definitions in coordinates we observe that
\begin{align*}
\mathrlap H\hphantom{H^{-0}}=& \left\{h\in\SL_n(\RR):h_{ij} = \delta_{ij}\text{ if }\eta_i \le \eta_j\right\},\\
\mathrlap{H^-}\hphantom{H^{-0}} =& \left\{h\in\SL_n(\RR):h_{ij} = \delta_{ij}\text{ if }\eta_i \ge \eta_j\right\},\\
\mathrlap{H^0}\hphantom{H^{-0}} =& \left\{h\in\SL_n(\RR):h_{ij} = 0\text{ if }\eta_i \neq \eta_j\right\},\\
H^{-0} =& \left\{h\in\SL_n(\RR):h_{ij} = 0\text{ if }\eta_i > \eta_j\right\}.
\end{align*}
If $\eta_1<\dots<\eta_n$, the group $H$ (resp. $H^-$) is the group of lower (resp. upper) triangular matrices with $1$ on the diagonal, and $H^0$ is the group of diagonal matrices.
\end{obs}
\begin{remark}\label{rem:notations}
Here we see the use of $\clubsuit$ to mark the end of an observation. Likewise, $\heartsuit$ will mark the end of a definition, $\diamondsuit$ the end of an example, and $\square$ the end of a proof.
\end{remark}
\begin{lem}\label{lem: mult map}
The image of the multiplication map
\index{mul@$\mLU$}\hypertarget{be}{} 
\begin{align*}
\mLU: H\times H^{-0}\to \SL_n(\RR),\quad(h,  q^{-0})\mapsto q^{-0}h
\end{align*}
is a dense Zariski open set that contains $\Id$. 
\end{lem}
\begin{proof}
Since $H$ and $H^{-0}$ are algebraic groups and their intersection of the complexifications $H_\CC\cap H^{-0}_\CC = \{\Id\}$, it follows that $\mLU$ is one-to-one. 
Since $\dim H + \dim H^{-0} = \dim\SL_n(\RR)$, $\mLU$ is an open map, and hence its image is open. 
Since $\SL_n(\RR)$ is irreducible as an algebraic variety, the image is dense.
\end{proof}
\begin{obs}
By Eq.~\eqref{eq: close action}, for every $q\in H^{-0}, \Lambda\in X_n$, the trajectories $g_t\Lambda$ and $g_tq \Lambda$ are equivalent.
Hence, studying the behavior of trajectories of lattices lattices in $ H^{-0}H\Lambda$ is equivalent to studying the behavior of trajectories of lattices in $H\Lambda$.  
\end{obs}
\index{dH@$d_{H}$}\hypertarget{bf}{} 
The restriction $d_{H}$ of the metric $d_{\SL_n(\RR)}$ to $H$ does not suit us, as it has bad behavior with respect to the $g_t$-action. 
We use instead the semi-metric $d_\varphi$ on $H$ defined by by 
\[d_\varphi(h_1, h_2): = \exp\left(-\min\{t\in \RR:d_H (g_th_1g_{-t}, g_th_2g_{-t})\ge 1\}\right).\]
The semi-metric $d_\varphi$ satisfies the equality 
\[d_\varphi(g_th_1g_{-t}, g_th_2g_{-t}) = \exp(t)d_\varphi(h_1,h_2),\]
while $d_H$ has no such property even approximately.
We will see that $d_\varphi$ is not a necessarily a metric, but there exists a constant $\alpha_\varphi<1$ such that $d_\varphi^{\alpha_\varphi}$ is a metric, see Subsections \ref{sub:the_expansion_metric} and \ref{sub:Case of interest}.
\index{Harder-Narasimhan filtration}\hypertarget{bg}{}
\subsection{The Harder-Narasimhan filtration} 
\label{sub:the_harder-narasimhan_filtration}
The Harder-Narasimhan filtration of a lattice was defined in \cite{HN} in the context of algebraic geometry, and adapted to lattices by Grayson \cite{G}. Its construction for standard lattices in $\RR^n$ goes as follows.

\index{cov@$\cov \Gamma$}\hypertarget{bh}{}
For every lattice $\Gamma\subset\RR^n$, denote by $\cov \Gamma$ the Euclidean volume of the torus $\spa\Gamma/\Gamma$. 
By convention, $\cov\{0\}:=1$. 

Informally, the Harder-Narasimhan filtration of a lattice $\Lambda$ is a lattice filtration $\{0\} = \Gamma_{l_0}\subset\Gamma_{l_1}\subset\dots\subset\Gamma_{l_k}=\Lambda$ of ``minimal covolume'', where $\rk \Gamma_{l_i} = l_i$. It is the set of all sublattices $\Gamma\subseteq \Lambda$ with exceptionally small covolume; if $\rk\Gamma = l$, even if we know the minimal covolume of all lattices for every possible rank $l'\neq l$, the covolume of $\Gamma$ is still particularly small.

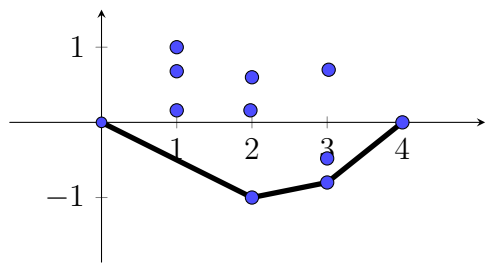
\begin{wrapfigure}{r}{0.5\textwidth}
  \begin{center}
    \begin{tikzpicture}[line cap=round,line join=round,>=triangle 45,x=1.0cm,y=1.0cm]
\begin{axis}[
x=1.0cm,y=1.0cm,
axis lines=middle,
xmin=-1.22,
xmax=5.1000000000000005,
ymin=-1.859999999999999,
ymax=1.5,
xtick={0.0,1, ...,4.0},
ytick={-1.0,0.0,1.0},]
\clip(-1.22,-1.86) rectangle (5.1,2);
\draw [line width=2.pt] (0.,0.)-- (2.,-1.);
\draw [line width=2.pt] (2.,-1.)-- (3.,-0.8);
\draw [line width=2.pt] (3.,-0.8)-- (4.,0.);
\begin{scriptsize}
\draw [fill=bluecolor] (0.,0.) circle (2.0pt);
\draw [fill=bluecolor] (4.,0.) circle (2.5pt);
\draw [fill=bluecolor] (3.,-0.8) circle (2.5pt);
\draw [fill=bluecolor] (3.,-0.48) circle (2.5pt);
\draw [fill=bluecolor] (3.02,0.7) circle (2.5pt);
\draw [fill=bluecolor] (2.,-1.) circle (2.5pt);
\draw [fill=bluecolor] (1.98,0.16) circle (2.5pt);
\draw [fill=bluecolor] (2.,0.6) circle (2.5pt);
\draw [fill=bluecolor] (1.,1.) circle (2.5pt);
\draw [fill=bluecolor] (1.,0.68) circle (2.5pt);
\draw [fill=bluecolor] (1.,0.16) circle (2.5pt);
\end{scriptsize}
\end{axis}
\end{tikzpicture}

  \end{center}
  \caption{An illustration of $S_\Lambda$. (Not based on a real lattice.)}
\end{wrapfigure}
\index{p@$p_\Gamma$}\hypertarget{bi}{}
Formally, to every lattice $\Gamma\subseteq \Lambda$ we associate the point $p_\Gamma := (\rk \Gamma, \log \cov \Gamma)\in \RR^2$.
For every lattice $\Lambda\in X_n$ define $S_\Lambda := \{p_\Gamma:\Gamma\subseteq \Lambda\}$. 
Denote the extreme points of the convex hull $\conv(S_\Lambda)$ by $\{p_l:l\in L\}$, where the index $l$ is the $x$-coordinate of $p_l$. For each $l\in L$, let $\Gamma_l \subseteq \Lambda$ satisfy $p_{\Gamma_l}=p_l$.

By \cite{G}, the lattices $\{\Gamma_l:l\in L\}$ form a filtration of $\Lambda$. Furthermore, $\Gamma_l$ is the unique subgroup that is associated to the point $p_l$ for $l\in L$. 

\index{FHN@$\HN$}\hypertarget{bj}{} 
By definition, for every $l\in L$ one has $\cov\Gamma_l < 1$. The filtration $\HN(\Lambda):=\{\Gamma_l:l\in L\}$ is called the \textit{Harder-Narasimhan filtration}.
Note that we index the elements in the filtration by their rank, and the set of ranks is $L = L(\Lambda)$.
\subsection{Various Filtration Types} 
\label{sub:various_filtration_types}
We will often use linear subspaces $V\subseteq \RR^n$. 
\index{Gr@$\gr$}\hypertarget{ca}{}
The Grassmannian $\gr$ is the space of all $l$-dimensional subspaces $V\subseteq \RR^n$. 
\index{dGr@$d_{\gr}$}\hypertarget{cb}{}
Denote by $d_{\gr}$ the path metric on $\gr$ derived from the Riemannian metric, which is invariant under multiplication by orthogonal matrices. 

The group $\{g_t\}$ acts on $\gr$. The set of $g_t$-invariant spaces is denoted $\gr^g$. 
For any $V\in \gr$ the eigenvalues of the action of $g_t$ on $V$ are $\exp (\eta t)$ for a multiset of $\eta\in \Eall$. 
\index{Il@$\cI_l$}\hypertarget{cc}{}
We let $\cI_l$ denote the set of multisets of $\Eall$ of size $l$. 

\index{Gr@$\gr$!$\gr_E^g$}\hypertarget{cd}{}
For every $E\in \cI_l$ denote by $\gr_E^g$ the space of $V\in \gr^g$ such that the eigenvalues of $g|_{V}$ are $\{\exp(t\eta):\eta\in E\}$. 
\begin{ex}\label{ex: gr inv}
Let $n=4, \Eall = \{-1, 0, 0, 1\}$, and $l=2$. Then $\cI_2 = \left\{\{-1, 0\}, \{-1, 1\},\{0, 1\}, \{0,0\}\right\}$, and
\begin{align*}
\grln{2}{4}^g_{\{-1, 0\}} &= \{\spa(e_1, v):0\neq v\in \spa(e_2, e_3)\},\\
\grln{2}{4}^g_{\{-1, 1\}} &= \{\spa(e_1, e_4)\},\\
\grln{2}{4}^g_{\{0,1\}} &= \{\spa(e_4, v):0\neq v\in \spa(e_2, e_3)\},\\
\grln{2}{4}^g_{\{0, 0\}} &= \{\spa(e_2, e_3)\}.
\end{align*}
\end{ex}
\index{Flag}\hypertarget{ce}{}
A \emph{flag} is an increasing sequence of linear spaces,
\index{Vbul@$V_\bullet$}\hypertarget{cf}{}
\[V_\bullet:~0=V_0\subset V_{l_1}\subset V_{l_2}\subset \dots\subset V_{l_k} = \RR^n; \quad\dim V_{l_i} = l_i.\]
\index{Direction filtration}\hypertarget{cg}{}
A \emph{direction filtration} is a sequence 
\index{Ebul@$E_\bullet$}\hypertarget{ch}{}
\begin{align}\label{eq: directed filtration}
E_\bullet:~\emptyset = E_0\subset E_{l_1}\subset E_{l_2}\subset \dots\subset E_{l_k} = \Eall; E_{l_i}\in \cI_{l_i}.
\end{align}
\index{L@$L(-)$}\hypertarget{ci}{}
Denote by $L(V_\bullet) = L(E_\bullet) := \{l_0,\dots,l_k\}$.

A \emph{height sequence} 
\index{abul@$a_\bullet$}\hypertarget{cj}{}
\index{Height sequence}\hypertarget{da}{}
$a_\bullet:~0=a_0, a_{1},\dots, a_n=0,$
is a convex sequence of numbers, that is, for every $l$ we have 
\[\partial^2 a_l:= a_{l-1} -2a_l + a_{l+1} \ge 0.\]

\index{Nontrivial place}\hypertarget{db}{}
\index{partial2@$\partial^2a_\bullet$}\hypertarget{dc}{}
The \emph{nontrivial places} $1\le l\le n-1$ of $a_\bullet$ are those for which $\partial^2a_l > 0$.
By convention, $l=0$ and $l = n$ are also nontrivial place. Denote the set of nontrivial places by $L(a_\bullet)$. 
Using convexity we can reconstruct $a_\bullet$ from the values $(a_l)_{l\in L(a_\bullet)}$.

A direction filtration with heights is a tuple $F = (F_{E,\bullet}, F_{H,\bullet})$ of a direction filtration and a height sequence where $L(F_{E,\bullet}) = L(F_{H,\bullet})$.
The Harder-Narasimhan filtration $\HN(\Lambda): \{\Gamma_l:l\in L\}$ of a lattice $\Lambda$ yields a flag  $\HN(\Lambda)_{V,\bullet} = \spa \Gamma_\bullet$, a height sequence defined by $\HN(\Lambda)_{H,l} = \log \cov \Gamma_l$ for every $l\in L$ and extended piecewise linearly for other $l$.
For consistency, we will write $\HN(\Lambda)_{\Gamma,l} = \Gamma_l$.

\subsection{Dynamics of Linear Spaces} 
\label{sub:dynamics_of_linear_spaces}
\index{#@$\normi\cdot\normi$}\hypertarget{dd}{}
For every $0\le l \le n$ consider the standard $L^2$ norm $\|\cdot \|$ on $\bigwedge^l \RR^n$. 
\index{Blades}\hypertarget{de}{}
\index{blades@$\BL$}\hypertarget{df}{}
\index{wedgelR@$\bigwedge^l \RR^n$}\hypertarget{dg}{}
Denote by $\BL$ the set of $l$-blades in $\bigwedge^l \RR^n$, that is, 
the set of vectors $0\neq v\in \bigwedge^l \RR^n$ that can be written as 
$v=v_1\wedge v_2\wedge\cdots\wedge v_l$ with $v_1,v_2,\dots,v_l\in \RR^n$.
Blades are useful for the following definition:
\begin{de}\label{de: blade}
Let $\Gamma\subset \RR^n$ be a lattice of rank $l$. Let $v_1,v_2,\dots,v_l$ be a basis of $\Gamma$. Define the \emph{blade of} $\Gamma$ by 
\index{BlGamma@$\bl\Gamma$}\hypertarget{dh}{}
\[\bl\,\Gamma:= v_1\wedge v_2\wedge\cdots\wedge v_l.\]
It is well-defined up to sign.
\end{de} 
We identify the real Grassmanian $\gr$, the space of $l$-dimensional subspaces of $\RR^n$, with $\BL/\RR^\times$, via \[v_1\wedge\cdots\wedge v_l\mapsto\spa(v_1,\dots,v_l).\]
\index{Rmul@$\RR^\times$}\hypertarget{di}{}
\noindent Here $\RR^\times$ denotes the multiplicative group of $\RR$.
The group $\GL_n(\RR)$ acts naturally on the spaces $\bigwedge^l \RR^n$ and $\gr$. 

Recall that the connected components of the $g_t$-invariant subspace $\gr^g$ are $\gr^g_E$ for every $E\in \cI_l$.
\index{#@$\#$}\hypertarget{dj}{}
\index{#@$\eteq$}\hypertarget{ea}{}
Define a partial order $\eteq$ on $\cI_l$ by $E_1 \eteq E_2$ if $\#\{x\in E_1:x>r\} \le \#\{x\in E_2:x>r\}$ for every $r>0$. In other words, we can obtain $E_2$ from $E_1$ by increasing elements. Here $\#$ denotes the cardinality of a set. 
We write $E_1\et E_2$ if $E_1\eteq E_2$ and $E_1\neq E_2$.
Thus we have (and will use) two partial orders on multisets: multiset inclusion $\subseteq$ and the new partial order $\eteq$.

The following result describes the dynamics of the $g_t$-action on $\gr$. Its proof is given in Subsection \ref{sub:grasmanian_theorems}.
\begin{thm}\label{thm:lin path}
For every $\varepsilon>0$, and $ 1\le l\le n-1$ there is $C>0$ such that for every $V\in \gr$ there exists a sequence $E_1\et\cdots\et E_r\in \cI_l$ and a sequence of intervals $a_1<b_1<a_2<b_2<\dots<a_r<b_r$ such that 
\index{Uepsilon@$U_\varepsilon$}\hypertarget{eb}{}
\[a_1 = -\infty,\quad b_r = \infty, \quad\sum_{k=1}^{r-1}(a_{k+1}-b_k)<C\] and for every $t\in (a_k, b_k)$, the space $g_tV$ lies in an $\varepsilon$-neighborhood $U_\varepsilon(\gr_{E_k}^g)$. 
\end{thm}
\begin{remark}\label{rem: the good ep}
Note that in Theorem \ref{thm:lin path}, if $\varepsilon$ is large enough the theorem is vacuous. 
On the other hand, if 
\begin{align*}
\varepsilon <\varepsilon_0 := \min_{1\le l\le n-1}\min_{E_1\neq E_2\in \cI_l}d_{\gr}(\gr^g_{E_1},\gr^g_{E_2})/2,
\end{align*}
then we have a uniqueness property for the output. 
That is, if $E_1\et\cdots\et E_r\in \cI_l$, $a_1<b_1<a_2<b_2<\dots<a_r<b_r$, 
and $E_1'\et\cdots\et E_{r'}'\in \cI_l$, $a_1'<b_1'<a_2'<b_2'<\dots<a_r'<b_r'$
are two valid outputs of Theorem \ref{thm:lin path}, 
and if $(a_i,b_i)\cap (a'_j,b'_j)\neq \emptyset$ for some $1\le i\le r, 1\le j\le r'$, then $E_i=E_j'$. In particular, if $b_i-a_i>C$ then $E_i$ must be one of the $E_{j}'$. 
\end{remark}
\begin{de}[Exterior eigenvalues]
\index{etaE@$\eta_E$}\hypertarget{ec}{}
For every $E\in \cI_l$ denote $\eta_{E}:=\sum_{\eta\in E}\eta$. This sum, as all operations on multisets, counts the elements $\eta\in E$ with multiplicities. The eigenvalue of $g_t$ at an $l$-blade $w\in \BL$, which corresponds to $V\in \gr^g_{E}$, is $\exp(t\eta_{E})$. 
\end{de}
Theorem \ref{thm:lin path} admits the following,
\begin{cor}\label{cor:blade behavior}
For every $w\in \BL$ there are $-\infty=t_0<\dots<t_r=\infty$, and $E_1\et\dots\et E_r\in \cI_l$ such that 
$\log \|g_tw\| = \exp(f(t)) + O(1)$, where $f(t)$ is a piecewise linear function with $f'(t) = \eta_{E_k}$ for every $t_{k-1}<t<t_{k}$. 
\end{cor}


\begin{ex}[Existence of a blade given a function]
Denote by $k$ the number of different values of $\eta$ in $\Eall$, and denote them by $\bar\eta_1<\bar\eta_2<\dots<\bar\eta_{k}$. 
Let $-\infty=t_0\le t_1\le \dots\le t_k=\infty$ and let $f:\RR \to \RR$ be a convex piecewise-linear function, such that $df/dt (t) = \eta_r$ for $t_{r-1}<t<t_{r}$. 
We will construct a vector $v\in \RR^n$ such that $\log \|g_t v\| = f(t)+O(1)$ for every $t$. 

Denote by $V_\eta$ the $\exp(\eta t)$-eigenspace of $g_t$. 
\index{V:$V_\eta$}\hypertarget{ed}{}
Fix $v_\eta\in V_{\eta}$ with $\|v_\eta\| = 1$.
Note that for some constants $(a_\eta)_{\eta \in \Eall} \subset \{-\infty\}\cup \RR$ we have $f(t) = \max_{\eta\in \Eall} (t\eta+a_\eta)$. 
Therefore, the vector $\sum_{\eta\in \Eall}\exp (a_\eta) v_\eta$ enjoys the desired property. 

If $C>0$ is large and $t_{r-1}+C\le t\le t_r-C$ then the only significant summand in 
$g_tv$ is $a_{\eta_r} g_t v_{\eta_r}$, and hence 
$g_tv\cdot \RR\in U_{\varepsilon}(\grl1^g_{\{\eta_r\}}$.

By wedging such vectors we can obtain similar results for spaces and blades with $l>1$.
\end{ex}

\subsection{A \texorpdfstring{$g$}{g}-Template} 
\label{sub:a_g_template}
In this section we describe information that is given in a trajectory $g_t \Lambda$ up to equivalence.
Though for most choices of $\Eall$ the information on the trajectory that we will provide will be equivalent to the equivalence relation on trajectories, sometimes this information will be stronger than the equivalence relation.
We will keep track of some of the information provided by the Harder-Narasimhan filtration $\HN(g_t \Lambda)$, and the direction of its linear spaces.

We wish to encode the ``behavior of $\HN(g_t \Lambda)$'' for  $t\ge 0$. 
For every $\Gamma\subseteq \Lambda$ of rank $l$ there is a (possibly empty) set $U_\Gamma\subseteq \RR$ such that $g_t\Gamma = \HN(g_t \Lambda)_{\Gamma, l}$ for all $t\in U_\Gamma$. It will follow that $U_\Gamma$ is approximately an interval. 

Inspired by Theorem \ref{thm:lin path}, we will encode the behavior of $g_t\Gamma$ for $t\in U_\Gamma$ with multisets $E\in \cI_l$. The partially ordered structure on $\cI_l$ will not suffice, as $\HN(g_t \Lambda)$ might not contain a rank-$l$ element, and $\HN(g_t \Lambda)_{V,l}$ can return to the same $U_\varepsilon(\gr_E^g)$ many times.

It will be convenient to extend $\cI_l$ to a category, though we will use only the combinatorial structure of a category and not the whole power of the theory. See Example \ref{ex:category flow tachles} for an equivalent formulation of the concepts without resorting to categories.

\begin{de}
Let $1\le l\le n-1$.
\index{#@$*$}\hypertarget{ee}{}
\index{I_star@$\cI_l^*$}\hypertarget{ef}{}
Define a category $\cI_l^*$ whose objects are $\cI_l\cup \{*\}$, where $*$ is the null object, and which has two types of arrows. 
\begin{itemize}
  \item
\index{Transition arrow}\hypertarget{eg}{}
\emph{Transition type arrows:} for $a, b\in \cI_l$ we have a transition arrow $a\xrightarrow{\eteq} b$ if $a\eteq b$. The composition of transition tow arrows $(b\xrightarrow\eteq c)\circ (a\xrightarrow\eteq b)$ is the transition arrow $a\xrightarrow\eteq c$. 
For every $a\in \cI_l$ the arrow $a\xrightarrow{\eteq} a$ is the identity arrow.
\index{Null arrow}\hypertarget{eh}{}
\item \emph{Null type arrows:} for every $a,b\in \cI_l^*$ we have a null arrow $a\nula b$. The composition of a null arrow with any arrow is a null arrow. 
\end{itemize}
By convention, for $l\in \{0,n\}$ we denote $\cI_l^* := \cI_l$, the unique arrow is the identity, the single element in $\cI_l^*$ is the null element.
\end{de}
The objects of the category $\cI_l^*$ will be used to classify the different phases of the $l$-dimensional spaces in $\HN(g_t\Lambda)$ for various $t$. 
The object $*$ indicates there is no $l$-dimensional space in the flag, and any multiset $E$ indicates that $\HN(g_t\Lambda)_{V,l}$ exists and is near $\gr^g_E$. 
Morphisms in the category $\cI_l^*$ classify which phase of the $l$-dimensional spaces in $\HN(g_t\Lambda)$ can transform to which one as $t$ increases.
\begin{de}
Extend each direction filtration $E_\bullet$ by setting $E_l=*$ whenever $E_l$ was not defined previously.

Define the category $\cI_\bullet^*$ whose objects are all the direction filtrations $E_\bullet$, and each arrow between $E_\bullet$ and $E_\bullet'$ is a collection of $n+1$ arrows $(E_l \to E_l')_{l=0,\dots,n}$.
\index{Istarbul@$\cI_\bullet^*$}\hypertarget{ei}{}
We say that the empty filtration $\emptyset\subseteq \Eall$ is the null element of $\cI_\bullet^*$, and an arrow $E_\bullet\to E_\bullet'$ in $\cI_\bullet^*$ is null if the induced arrow $E_l\to E_l'$ is null for every $l=1,...,n-1$.
\end{de}
\begin{de}
Let $C$ be a category with a null element.
\index{Category flow}\hypertarget{ej}{}
For every interval $I\subseteq \RR$, a \emph{category flow} $f: I\to C$ is a locally constant map $f:I\setminus T\to C$ for a locally finite set $T\subseteq I$ together with the choice of a non-identity arrow $f_-(t)\to f_+(t)$ for every point $t\in T$, where $f_\pm(t) = \lim_{\varepsilon\searrow 0} f(t\pm \varepsilon)$. 
\index{Null point}\hypertarget{fa}{}
A point $t\in I$ is called a \emph{null point} if 
\begin{itemize}
\item $t\in T$ and $f_-(t)\to f_+(t)$ is the null arrow, or
\item $f(t)$ is the null object.
\end{itemize}
\index{Nontrivial point}\hypertarget{fb}{}
All the other points are called \emph{nontrivial points}.

\index{irregular point}\hypertarget{fc}{}
Points in $T$ are called \emph{irregular point}, and the other are called \emph{regular points}.
\end{de}
\begin{obs}
Let $I\subset \RR$ be an interval and let $0\le l \le n$. We can restrict every category flow $f_{E,\bullet}:I_\bullet^*$ to its $l$-th entry and obtain a category flow $f_{E,l}: I\to \cI_l^*$. 
\end{obs}
\begin{ex}\label{ex:category flow tachles}
Let $I=(A,B)$ be an open interval. $A,B$ may be $\pm\infty$. A category flow $f:I\to \cI_l^*$ contains the following information:
\begin{itemize}
\item 
A collection of disjoint intervals $((a_i,b_i))_{i\in Z}$ in $I$, where $Z\subseteq \ZZ$ is a possibly infinite collection of successive integers, with $b_i\le a_{i+1}$ for every tuple of consecutive integers $i,i+1\in Z$. 
In addition, the set $\{a_i, b_i:i\in Z\}$ is required to be locally finite.
\item
A finite set $T_i = \{a_i = c_{i,0}<c_{i,1}<c_{i,2}<\cdots < c_{i,r_i}=b_i\}$ for every $i\in Z$
\item 
A sequence of multisets $E_{i,1}\et E_{i,2} \et \cdots \et E_{i,r_i}$ for every $i\in Z$.
\end{itemize}
Denote \[T := \bigcup_{i\in Z}T_i.\]
Define a category flow by 
\[f(t) = \begin{cases}
E_{i,j},&\text{if }c_{i,j-1}<t<c_{i,j},\\
*,&\text{if }t\in I\setminus \bigcup_{i\in Z}[a_i,b_i],
\end{cases}\]
and the connecting morphisms are $E_{i,j-1}\xrightarrow{\eteq} E_{i,j}$ at $t=c_{i,j}$ for every $i\in Z, 1\le j<r_i$, and the null morphism for $t=a_i,b_i$ for $i\in Z$, provided that $t\in I$ (as opposed to $t$ being an endpoint of $I$).
This is a complete description of a category flow $f:I\to \cI_l^*$.
\end{ex}
We now define our encoding for $g_t$-trajectories, which are termed $g$-templates, and play a crucial role in this paper. 
\begin{de}[$g$-template]\label{de: template}
\index{gtemplate@$g$-template}\hypertarget{fd}{}
\index{Template}\hypertarget{fe}{}
\index{f@$f$}\hypertarget{ff}{}
\index{f@$f$!$f_{E,\bullet}$}\hypertarget{fg}{}
\index{f@$f$!$f_{H,\bullet}$}\hypertarget{fh}{}
A \emph{$g$-template} $f$ over an interval $I\subseteq \RR$ is a pair of a category flow 
$f_{E,\bullet}:I\to \cI_\bullet^*$ and a continuous height sequence-valued map 
$t\mapsto f_{H,\bullet}(t)$
 such that 
\begin{enumerate}
  \item \label{prop:compatibility} 
  For every  $t\in I$ that is regular for $f_{E, \bullet}(t)$, the tuple $(f_{E, \bullet}(t), f_{H, \bullet}(t))$ forms a direction filtration with heights.
  \item \label{prop:Derivative}
  For every $1\le l < n$ and every $t\in I$ that is regular and not a null point for $f_{E, l}$,
\begin{align}\label{eq:diff eq}
  \eta_{f_{E,l}(t)} = \frac{d}{dt}f_{H, l}(t).
\end{align}
  
\end{enumerate}
Note that each $f_{E,l}:I\to \cI_l$ is a category flow.
\end{de}
\begin{de}[Nontriviality intervals]\label{de: U_l}
Let $f$ be a $g$-template on an interval $I$.
\index{Ulf@$U_l(f)$}\hypertarget{fi}{}
\index{Nontriviality interval}\hypertarget{fj}{}
Let $U_l(f)$ denote the set of nontrivial points of $f_{E, l}$, and set $U_0(f) = U_n(f) := I$. The set $U_l(f)$ is a disjoint union of open intervals (as every open subset of $\RR$), which will be referred to as the \emph{nontriviality intervals of $f$ at $l$}. 

Define the set of all nontriviality intervals by 
\index{Gf@$\cG_f$}\hypertarget{ga}{}
\[\cG_f:=\bigsqcup_{l=0}^{n} \pi_0(U_l(f)),\]
where $\pi_0(U_l(f))$ is the set of nontriviality intervals of $f$ at $l$.
\end{de}
\begin{remark}
We have already defined the concepts of nontrivial places of a height sequence, and nontrivial points of a $g$-template. They match: $l$ is a nontrivial place of $f_{H,\bullet}(t)$ if and only if $t$ is a nontrivial point of $f_{E,l}$. 
\end{remark}
\begin{ex}\label{ex: a g template}
In Figure \ref{fig: a g template} we present a $g$-template with $\Eall = \{-1.2, 0.5, 0.7\}$. 
The nontriviality intervals are depicted as solid lines and the parts outside them as dashed lines.

Figure \ref{fig: cross-sections} shows two cross-sections of $f$ at the three points $t_0, t_1$, and  $t_2$. 
Note that the value of $f_{H,2}(t_0)$ is defined by $f_{H,1}(t_0)$ and convexity. Similarly, $f_{H,1}(t_0)$ is defined by $f_{H,1}(t_0)$ and convexity.
\begin{figure}[ht]
\caption{A $g$-template.}
\label{fig: a g template}
\begin{tikzpicture}[line cap=round,line join=round,>=triangle 45,x=1.0cm,y=1.0cm]

\draw [color=black] (6.45, -1.4) node {$f_{H, 1}$};
\draw [color=redcolor] (6.45, -2.7) node {$f_{H, 2}$};
\clip(-6,-7) rectangle (6,.8);

\draw [color=black] (-5.5, -5.2) node {$f_{E, 1}:$};
\draw [color=redcolor] (-5.5, -6.2) node {$f_{E, 2}:$};

\draw [line width=1.pt, dotted] (-3.5, -.8) -- (-3.5,.2) node[above] {$t_0$};
\draw [line width=1.pt, dotted] (-1, -2.2) -- (-1,.2) node[above] {$t_1$};
\draw [line width=1.pt, dotted] (.1, -2.97) -- (.1,.2) node[above] {$t_2$};

\draw [line width=2.pt, dash pattern=on 5pt off 5pt,] (-7.0,0.03) -- (-4,0.03);
\draw [line width=2.pt, dash pattern=on 5pt off 5pt,color=redcolor] (-7.2,-.03) -- (-4,-.03);
\draw [line width=1.pt, -to] (-7,0) -- (5.5,0) node[above] {$t$};

\draw [line width=2.pt, ] (-4,0.0) -- (-2,-2.4);
\draw [line width=2.pt, ] (-2,-2.4) -- (0.05882352941176476,-1.3705882352941174);
\draw [line width=2.pt, ] (1,-1.6999999999999997) -- (2,-2.8999999999999995);
\draw [line width=2.pt, ] (2,-2.899999999999999) -- (2.8947368421052624,-2.2736842105263153);
\draw [line width=2.pt, color=redcolor,] (-3,-0.6000000000000001) -- (2,-4.1);
\draw [line width=2.pt, color=redcolor,] (2,-4.1) -- (4,-5.1);
\draw [line width=2.pt, color=redcolor,] (4,-5.099999999999999) -- (7,-1.4999999999999982);

\draw [line width=2.pt, dash pattern=on 5pt off 5pt,] (0.05882352941176476,-1.3705882352941174) -- (1,-1.6999999999999997);
\draw [line width=2.pt, dash pattern=on 5pt off 5pt,] (2.8947368421052624,-2.2736842105263153) -- (4.0,-2.5499999999999994);
\draw [line width=2.pt, dash pattern=on 5pt off 5pt,] (4.0,-2.549999999999999) -- (7.0,-0.7499999999999982);
\draw [line width=2.pt, dash pattern=on 5pt off 5pt,color=redcolor,] (-4.0,0.0) -- (-3,-0.6000000000000001);
\begin{scriptsize}

\end{scriptsize}

\draw [line width=2.pt, dash pattern=on 5pt off 5pt,] (-7,-5.5) -- (-4,-5.5);
\draw[color=black] (-4.5, -5.2) node {$*$};
\draw [line width=2.pt, ] (-4,-5.5) -- (-2,-5.5);
\draw[color=black] (-3.0, -5.2) node {$\{-1.2\}$};
\draw [line width=2.pt, ] (-2,-5.5) -- (0.05882352941176476,-5.5);
\draw[color=black] (-0.9705882352941176, -5.2) node {$\{0.5\}$};
\draw [line width=2.pt, dash pattern=on 5pt off 5pt,] (0.05882352941176476,-5.5) -- (1,-5.5);
\draw[color=black] (0.5294117647058824, -5.2) node {$*$};
\draw [line width=2.pt, ] (1,-5.5) -- (2,-5.5);
\draw[color=black] (1.4, -5.2) node {$\{-1.2\}$};
\draw [line width=2.pt, ] (2,-5.5) -- (2.8947368421052624,-5.5);
\draw[color=black] (2.5, -5.2) node {$\{0.7\}$};
\draw [line width=2.pt, dash pattern=on 5pt off 5pt,] (2.8947368421052624,-5.5) -- (7,-5.5);
\draw[color=black] (4.947368421052631, -5.2) node {$*$};
\draw [line width=2.pt, color=black] (-7,-5.6) -- (-7,-5.4); 
\draw [line width=2.pt, color=black] (-4,-5.6) -- (-4,-5.4); 
\draw [line width=2.pt, color=black] (-2,-5.6) -- (-2,-5.4); 
\draw [line width=2.pt, color=black] (0.05882352941176476,-5.6) -- (0.05882352941176476,-5.4); 
\draw [line width=2.pt, color=black] (1,-5.6) -- (1,-5.4); 
\draw [line width=2.pt, color=black] (2,-5.6) -- (2,-5.4); 
\draw [line width=2.pt, color=black] (2.8947368421052624,-5.6) -- (2.8947368421052624,-5.4); 
\draw [line width=2.pt, color=black] (7,-5.6) -- (7,-5.4); 
\draw [line width=2.pt, dash pattern=on 5pt off 5pt,color=redcolor,] (-7,-6.5) -- (-3,-6.5);
\draw[color=redcolor] (-4.5, -6.2) node {$*$};
\draw [line width=2.pt, color=redcolor,] (-3,-6.5) -- (2,-6.5);
\draw[color=redcolor] (-0.5, -6.2) node {$\{-1.2, 0.5\}$};
\draw [line width=2.pt, color=redcolor,] (2,-6.5) -- (4,-6.5);
\draw[color=redcolor] (3.0, -6.2) node {$\{-1.2, 0.7\}$};
\draw [line width=2.pt, color=redcolor,] (4,-6.5) -- (7,-6.5);
\draw[color=redcolor] (5, -6.2) node {$\{0.5, 0.7\}$};
\draw [line width=2.pt, color=redcolor] (-7,-6.6) -- (-7,-6.4); 
\draw [line width=2.pt, color=redcolor] (-3,-6.6) -- (-3,-6.4); 
\draw [line width=2.pt, color=redcolor] (2,-6.6) -- (2,-6.4); 
\draw [line width=2.pt, color=redcolor] (4,-6.6) -- (4,-6.4); 
\draw [line width=2.pt, color=redcolor] (7,-6.6) -- (7,-6.4); 


\end{tikzpicture}
\end{figure}

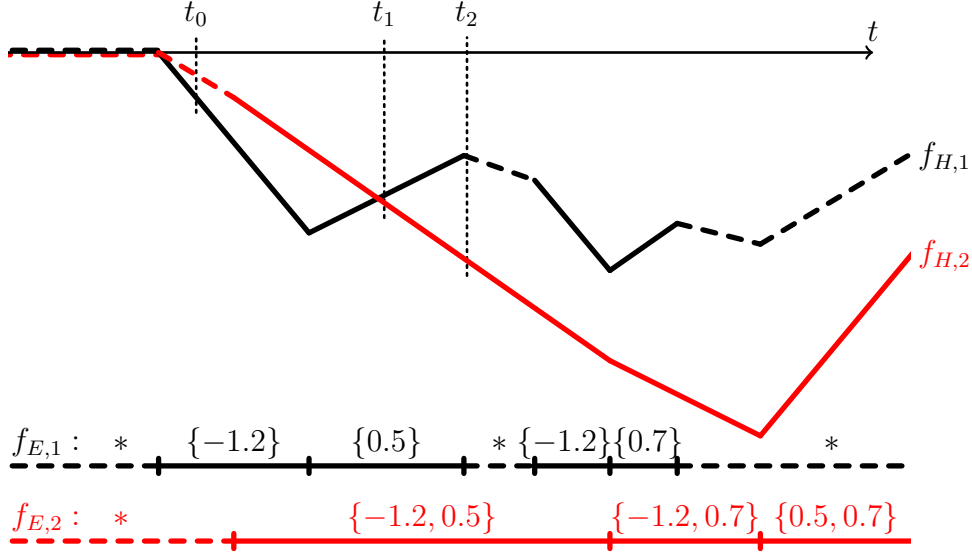
\begin{figure}[ht]
\caption{The cross-sections.}
\label{fig: cross-sections}
\begin{tikzpicture}[line cap=round,line join=round,>=triangle 45,x=1.0cm,y=1.0cm]

\clip(-6.2,-3.3) rectangle (6.2,0.6);

\draw [line width=1pt, dotted] (-6.0,0.0) -- (-5.0,-0.5999999999999996);
\draw [line width=1pt, dotted] (-5.0,-0.5999999999999996) -- (-4.0,-0.2999999999999998);
\draw [line width=1pt, dotted] (-4.0,-0.2999999999999998) -- (-3.0,0.0);

    \draw[->] (-6.0, 0.2) -- (-6.0,-0.8999999999999997);
    \draw (-5.6, -.9) node[below] {$f_{H, l}(t_0)$};
    \draw[->] (-6.2,0) -- (-2.7, 0)node[below] {$l$};\draw[shift={(-5.0, 0)}] (0pt,2pt) -- (0pt,-2pt)node[above] {$1$};
\draw[shift={(-4.0, 0)}] (0pt,2pt) -- (0pt,-2pt)node[above] {$2$};
\draw[shift={(-3.0, 0)}] (0pt,2pt) -- (0pt,-2pt)node[above] {$3$};
\draw [fill=black, color=black] (-6.0,0.0) circle (1.7pt); 
\draw [fill=black, color=black] (-5.0,-0.5999999999999996) circle (1.7pt); 
\draw [color=redcolor] (-4.0,-0.2999999999999998) circle (1.7pt); 
\draw [fill=black, color=black] (-3.0,0.0) circle (1.7pt); 
\draw [line width=1pt, dotted] (-1.5,0.0) -- (-0.5,-1.8999999999999997);
\draw [line width=1pt, dotted] (-0.5,-1.9) -- (0.5,-1.9999999999999996);
\draw [line width=1pt, dotted] (0.5,-1.9999999999999996) -- (1.5,0.0);

    \draw[->] (-1.5, 0.2) -- (-1.5,-2.3) node[right] {$f_{H, l}(t_1)$};
    \draw[->] (-1.7,0) -- (1.8, 0)node[below] {$l$};\draw[shift={(-0.5, 0)}] (0pt,2pt) -- (0pt,-2pt)node[above] {$1$};
\draw[shift={(0.5, 0)}] (0pt,2pt) -- (0pt,-2pt)node[above] {$2$};
\draw[shift={(1.5, 0)}] (0pt,2pt) -- (0pt,-2pt)node[above] {$3$};
\draw [fill=black, color=black] (-1.5,0) circle (1.7pt); 
\draw [fill=black, color=black] (-0.5,-1.9) circle (1.7pt); 
\draw [fill=redcolor, color=redcolor] (0.5,-1.9999999999999998) circle (1.7pt); 
\draw [fill=black, color=black] (1.5,0) circle (1.7pt); 
\draw [line width=1pt, dotted] (3.0,0.0) -- (4.0,-1.3849999999999998);
\draw [line width=1pt, dotted] (4.0,-1.3849999999999998) -- (5.0,-2.7699999999999996);
\draw [line width=1pt, dotted] (5.0,-2.7699999999999996) -- (6.0,0.0);

    \draw[->] (3.0, 0.2) -- (3.0,-3.0699999999999994) node[right] {$f_{H, l}(t_2)$};
    \draw[->] (2.8,0) -- (6.3, 0)node[below] {$l$};\draw[shift={(4.0, 0)}] (0pt,2pt) -- (0pt,-2pt)node[above] {$1$};
\draw[shift={(5.0, 0)}] (0pt,2pt) -- (0pt,-2pt)node[above] {$2$};
\draw[shift={(6.0, 0)}] (0pt,2pt) -- (0pt,-2pt)node[above] {$3$};
\draw [fill=black, color=black] (3.0,0.0) circle (1.7pt); 
\draw [color=black] (4.0,-1.3849999999999998) circle (1.7pt); 
\draw [fill=redcolor, color=redcolor] (5.0,-2.7699999999999996) circle (1.7pt); 
\draw [fill=black, color=black] (6.0,0.0) circle (1.7pt); 

\end{tikzpicture}
\end{figure}
\end{ex}
\begin{de}[Equivalence of $g$-templates]
We say that two $g$-templates $f,f'$ over $I\subseteq \RR$ are \emph{$C$-close} if \begin{enumerate}
    \item \label{cond: partial^2 bound}For every $0\le l\le n$ and $t\in I$ we have $|\partial^2 f_{H, l}(t)-\partial^2 f_{H, l}'(t)|<C$.
    \item For every subinterval $[a,b]\subseteq I$ of length $|a-b| \ge 2C$ such that $\partial^2 f_{H, l}(t) \ge C$ and $f_{E, l}|_{[a,b]} \equiv E\in \cI_l$, we have $f_{E, l}'|_{[a+C, b-C]} \equiv E$. We require also the symmetric claim when swapping $f$ and $f'$. 
  \end{enumerate}
The concept of $C$-closeness defines a metric on the set of $g$-templates. We say that two $g$-templates are equivalent if they are $C$-close for some $C$. 
\end{de}

\begin{remark}\label{rem:bound}
Condition \ref{cond: partial^2 bound} is equivalent to bounding the supremum $|f_{H, l}(t) - f_{H, l}'(t)|$, with a different constant. We define it as in Condition \ref{cond: partial^2 bound} for technical reasons.
\end{remark}

$g$-templates are related to $g_t$-trajectories by the following definition.
\index{fLambda@$f^\Lambda$}\hypertarget{gb}{}
\begin{de}[$g$-template matching a lattice]\label{de:f Lambda matching template}
Let $\varepsilon_0$ be as in Remark \ref{rem: the good ep}, $0< \varepsilon <\varepsilon_0$, $C>0$ and let $I\subseteq \RR$ be an interval.
We say that a lattice $\Lambda \in X_n$ \emph{$(\varepsilon, C)$-matches} a $g$-template $f$ on $I$ if the following two conditions are satisfied.
\begin{enumerate}
  \item \label{cond: template height}
For every $t\in I$ and $1\le l\le n-1$, we have
\[\left|\HN(g_t \Lambda)_{H,l} - f_{H, l}(t)\right| < C.\]
\item\label{cond: template direction} Let $1\le l\le n-1$ and $[a,b]\subseteq I$ be an interval such that for every $t\in [a,b]$ we have $f_{E,l}(t) = E\in \cI_l$ and $\partial^2f_{H, l}(t)\ge C$. Then for every $t\in [a+C,b-C]$, we have $\HN(g_t\Lambda)_{V,l} \in U_\varepsilon(\gr^g_E)$.
\end{enumerate}
Note that such a template is unique up to $g$-template equivalence. 
If such a template exists for $I=\RR$ we denote it by $f^\Lambda$.
\end{de}
\begin{obs}
By the uniqueness property observed in Remark \ref{rem: the good ep}, if $\Lambda$ $(\varepsilon,C)$-matches $f$, then for every $\varepsilon'<\varepsilon$ there is $C'>0$ such that $\Lambda$ $(\varepsilon',C')$-matches $f$. 
\end{obs}
\begin{remark}\label{rem:comparison with Schmidt}
In \cite{S1} Schmidt describes a weaker form of the $g$-template, which is similar to the definition of $g$-templates, but does not remember the multisets. It equivalent to the following definition. A \textit{$g$-system} is a height sequence-valued map $f_{H,\bullet}:\RR\to \{\text{height sequences}\}$ such 
\begin{enumerate}
  \item For every $0\le l\le n$ the function $t\mapsto f_{H,l}(t)$ is piecewise linear.
  \item For every $1\le l\le n-1, t\in \RR$ for which $\partial^2f_{H,l}(t)>0$ we have that $f_{H,l}$ is locally convex near $t$ and 
  \[\frac{d}{dt}f_{H,l}(t)\in \{\eta_E:E\in \cI_l\}.\] 
  if $\frac{d}{dt}f_{H,l}(t)$ is defined.
\end{enumerate}
\end{remark}
Here are two examples which show that the multisets conditions of $g$-templates are necessary for the existence of a lattice that $(\varepsilon, C)$-matches a template.
\begin{cex}
Let $n=4$, $\eta_1=-2, \eta_2 = -1, \eta_3 = 1, \eta_4 = 2$. Define 
Let $f_{H,\bullet}:\RR\to \{\text{height sequences}\}$ be defined by \[f_{H,\bullet}(t) = \begin{cases}
0,& \text{if }t\le 0,\\
ta_l,&\text{if }t>0.
\end{cases}\]
where $a_0=0, a_1=-1, a_2=-1, a_3=-1/2, a_4=0$. 
Then $f_{H,l}(t)$ is indeed piecewise linear, but there is no $\Lambda$ that $(\varepsilon, C)$-matches $f$ for every $\varepsilon,C>0$. Indeed, assuming the contrary there would exist $0\subset \Gamma_1\subset\Gamma_2\subset \Lambda$ of ranks $1$ and $2$, respectively, such that for every $t\ge 0$ we have $\log \cov g_t\Gamma_l =f_{H,l}(t)+O(1)$, $l=1,2$. 
Combining Theorem \ref{thm:lin path} and Corollary \ref{cor: blade behavior over time} we get that 
\[\spa g_t\Gamma_1\xrightarrow{t\to \infty}\grln{1}{4}^g_{\{-1\}}\quad\text{and}\quad\spa g_t\Gamma_2\xrightarrow{t\to \infty}\grln{2}{4}^g_{\{-2,1\}}.\]
But this contradicts the fact that no space in $\grln{2}{4}^g_{\{-2,1\}}$ contains a space in $\grln{1}{4}^g_{\{-1\}}$.
\end{cex}
\begin{cex}
Let $n=4$, $\eta_1=-2, \eta_2 = \eta_3 = -1, \eta_3 = 4$. Let $T$ be arbitrarily large. 
Let $f_{H,\bullet}:\RR\to \{\text{height sequences}\}$ be defined by \[f_{H,\bullet}(t) = \begin{cases}
0,& \text{if }|t|\ge T,\\
|T-t|a_l,&\text{if }|t|<T.
\end{cases}\]
where $a_0=0, a_1=-1, a_2=-2, a_3=-1, a_4=0$. 
Then $f_{H,l}(t)$ is indeed piecewise linear, but there is no $\Lambda$ that $(\varepsilon, C)$-matches $f$ for every $\varepsilon,C>0$, and $T$ large enough as a fuction of $C, \varepsilon$. Indeed, assuming the contrary, there would exist $0\subset \Gamma_2\subset \Lambda$ of ranks $2$ such that for every $t$ with $|t|\le T$ we have $\log \cov g_t\Gamma_2 =-2|T-t|+O(C)$. 
This contradicts Theorem \ref{thm:lin path}, as the only multisets $E\in \cI_2$ with $\eta_E = \pm2$ are $\{-1,-1\}$ and $\{-2,4\}$, and $\{-1,-1\}\not \eteq\{-2,4\}$. 
\end{cex}


We can now state the main results of this paper.
\subsection{Results} 
\label{sub:results}
\begin{thm}[Existence of Templates] \label{thm:Lattice approx}
For every $0<\varepsilon<\varepsilon_0$ there is a $C>0$ such that every $\Lambda\in X_n$ $(\varepsilon, C)$-matches a $g$-template $f$ on $\RR$. 
\end{thm}

\begin{remark}\label{rem:SchmidtSummerer extension}
Theorem \ref{thm:Lattice approx} extends the results of \cite{SS} from the flow described in Eq. \eqref{eq:monoexpanding eigenvalue} to a general diagonal flow. In addition, it sharpen \cite[Proposition 2.2]{S1}.
\end{remark}

\begin{de}\label{de: local entropy}
If $E_\bullet\in \cI_\bullet^*$ is a direction filtration,
denote
\index{detla0@$\delta(E_\bullet)$}\hypertarget{gc}{}
\[\delta(E_\bullet) := \sum_{i=1}^k\sum_{\substack{\eta\in E_{l_i} \\\eta'\in E_{l_i}\setminus E_{l_{i-1}}}}(\eta-\eta')^+,\]
\index{plus@$(-)^+$}\hypertarget{gd}{}
where $a^+ = \max(a,0)$. 

For a $g$-template $f$ on $[0,\infty)$ define 
\index{detla1@$\Delta_0(f)$}\hypertarget{ge}{}
\index{detla2@$\Delta(f)$}\hypertarget{gf}{}
\[\Delta_0(f) = \liminf_{T\to \infty} \frac1T\int_0^T \delta(f_{E,\bullet}(t))dt,\]
and $\Delta(f) = \sup_{f'\sim f}\Delta_0(f')$, where the supremum is over all $f'$ which are equivalent to $f$.
\end{de}
Let $\cF$ be a set of $g$-templates that is closed under equivalence and $\Lambda$ be a lattice. 
\index{Y@$Y_{\Lambda, \cF}$}\hypertarget{gg}{}
Denote \[Y_{\Lambda, \cF} := \{h\in H:f^{h\Lambda}\in \cF\}\subseteq H.\]
The $g$-template $f^{h \Lambda}$ is defined as in Definition \ref{de:f Lambda matching template}.
\begin{thm}\label{thm: general scewed dimention formula}
Let $\cF$ be a set of $g$-templates that is closed under equivalence. 
Then for every $\Lambda\in X_n$,
\[\dim_\funH(Y_{\Lambda,\cF}; d_\varphi) = \sup_{f\in \cF}\Delta(f) = \sup_{f\in \cF}\Delta_0(f), \]
where $d_\varphi$ is the expansion metric on $H$ with respect to the conjugation by $g_t$, see Subsection \ref{sub:Case of interest}, and $\dim_\funH$ is the Hausdorff dimension.
\end{thm}
\index{D@$D$}\hypertarget{gh}{}
By Corollary \ref{cor: H dimension}, \[D:=\dim_\funH(H;d_\varphi) = \delta(\emptyset \subseteq \Eall) = \sum_{\eta,\eta'\in \Eall}(\eta-\eta')^+.\]
\begin{cor}\label{cor: general scewed dimention of divergent}
\index{Xi@$\Xi$}\hypertarget{gi}{}
Let $n\ge 3$. Denote $\Xi := \sum_{\eta\in \Eall}\eta^+$.
Then \[\dim_\funH({\rm Sing}(H, \Lambda; g_t); d_\varphi) = D-\Xi, \]
where ${\rm Sing}(H, \Lambda; g_t)$ is the set of $h\in H$ such that $g_th\Lambda$ diverges.
\end{cor}
\begin{remark}\label{rem:diverent in average}
Corollary \ref{cor: general scewed dimention of divergent} holds for $n\ge 2$ if we replace divergence with divergence on average. The proof is similar to the proof provided here.
\end{remark}
We can pair both results with the Comparison Theorem \ref{thm:comparison}. We will construct two functions, $\overline F: [0,D]\to [0,N]$ and $\underline F:[0,N]\to [0,D]$, where $N=\dim H$ is the standard manifold dimension of $H$. 
\begin{cor}\label{cor: ineq general ineq}
Let $\cF$ be a set of $g$-templates that is closed under equivalence. Then for every $\Lambda\in X_n$,
\[\underline F^{-1}\left(\sup_{f\in \cF}\Delta_0(f)\right)\le \dim_\funH(Y_{\Lambda, \cF}; d_H) \le \overline F\left(\sup_{f\in \cF}\Delta_0(f)\right), \]
where $d_H$ is standard right-invariant metric on $H$.
In particular, if $n\ge 3$ then
\[\underline F^{-1}\left(D-\Xi\right)\le \dim_\funH({\rm Sing}(H, \Lambda; g_t); d_H) \le \overline F\left(D-\Xi\right).\]
\end{cor}

The following corollary extends the result of \cite{R} to general flows as in Eq. \eqref{eq:monoexpanding eigenvalue}. Together with Theorem \ref{thm:Lattice approx} it provides a complete classification of the possible 
functions $t,i\mapsto \HN(g_t\Lambda)_{H,i}$ up to a bounded constant for $t\ge 0$ and $0\le i\le n$, and completes the proof of the correction of \cite[Conjectures 2.3]{S1}.
The corollary follows directly from Theorem \ref{thm: general scewed dimention formula}, however it does not require the Hausdorff dimension machinery the proof of Theorem \ref{thm: general scewed dimention formula} does, and we will provide a proof that does not require this machinery at Subsection \ref{sub:existence_of_lattice_given_template}.
\begin{cor} \label{cor:template existence}
\index{Y@$Y_{\Lambda, f}$}\hypertarget{gj}{}
For every $g$-template $f$ the set 
\[Y_{\Lambda, f} := \{h\in H:f^{h\Lambda}\sim f\}\]
is non-empty.
\end{cor}

\subsection{Notations} 
\label{sub:notations}
\index{O@$O(-)$}\hypertarget{ha}{}
\index{Theta@$\Theta(-)$}\hypertarget{hb}{}
\index{Omega@$\Omega(-)$}\hypertarget{hc}{}
For every real-valued expression $F$ we will use the notation $O(F), \Theta(F)$, and $\Omega(F)$ to denote expressions that satisfy $|O(F)|<CF,~cF<\Theta(F)<CF,~\Omega(F)>cF$ for constants $C,c>0$ that depend only on $\Eall$. 

\index{#g@$\ggg$}\hypertarget{hd}{}
For two real-valued expressions $F,G$ we write $G\ggg F$ if there is an implicit function $\iota:(0,\infty)\to (0,\infty)$ such that $G>\iota(F)$. This notation will allow us to make statements of the form: ``Set $x\ggg y$. Then $x > y^3+y$''.
Hence, the function $\iota$ is determined implicitly after the expression involving $\ggg$ is stated. 
Thus, the notation $x\ggg y$ can be used as an abbreviation of an argument of the form ``Choose $x>\iota(y)$, where $\iota(y)$ will be determined later implicitly''.

\index{Bdxr@$B_d(x;r)$}\hypertarget{he}{}
For a metric space $(X,d)$, a point $x\in X$ and $r>0$ we will denote the closed ball around $x$ with radius $r$ by $B_d(x;r)$. We will use this notation even if $d$ is not a metric, but just a function $d:X\times X\to \RR$.
\section{Further Research} 
\label{sec:discussion}
There are many directions one can hope to further understand, as suggested in \cite[Section 6]{DFSU}. 
We mention here some which arise when considering general $1$-parameter flows.
\subsection{Comparison to the standard Hausdorff dimension} 
\label{sub:comparison_to_the_standard_hausdorff_dimension}
Corollary \ref{cor: ineq general ineq} bounds the standard Hausdorff dimension of trajectory sets in $H$. 
\begin{ex}\label{ex: n=3}
\cite{LSST} deals with the case $n=3$, $\eta_1=-1, \eta_2=w_1, \eta_3 = w_2$, where $w_1+w_2 \ge 0$.
Then $H = \left\{\begin{pmatrix}
  1& 0& 0\\
  *& 1& 0\\
  *& *& 1\\
\end{pmatrix}\right\}$, $H' = \left\{\begin{pmatrix}
  1& 0& 0\\
  *& 1& 0\\
  *& 0& 1\\
\end{pmatrix}\right\}$.
It is shown in \cite[Theorem 1.5]{LSST} that 
\[\dim_\funH ({\rm Sing}(H, \Lambda; g_t) \cap H';d_H) = 2-\frac{1}{1+w_1},\]
for every lattice $\Lambda$.
By \cite[Theorem 7.7, p. 104]{M} which glues the result to $H$-orbits, it follows that 
\[\dim_\funH ({\rm Sing}(H, \Lambda; g_t);d_H) \ge 3-\frac{1}{1+w_1}.\]
On the other hand, Corollary \ref{cor: ineq general ineq} implies that this lower bound is tight.
\end{ex}
Inspired by Example \ref{ex: n=3} we make the following conjecture.
\begin{conj}\label{conj: dim is upper}
Let $\cF$ be a set of $g$-templates, closed under equivalence. Then for every $\Lambda\in X_n$,
\[\dim_\funH(Y_{\Lambda, \cF}; d_H) = \overline F\left(\sup_{f\in \cF}\Delta_0(f)\right).\]
\end{conj}
A possible heuristic for the conjecture is the computation of the weighted planar Cantor-like set $C$, which is constructed by the following procedure with appropriate constants. Start with a box, and recursively choose $k$ sub-boxes to every box uniformly at random. The side-lengths of the boxes are weighted, i.e., the $l$-th side-length of a sub-box is $a_l<1$ times the $l$-th side-length of the original box. With probability $1$, the standard Hausdorff dimension of $C$ is given by the upper bound provided by the analog of Theorem \ref{thm:comparison}, depending on the Hausdorff dimension of $C$ with respect to the weighted metric.

\subsection{Dimension in smaller sets} 
\label{sub:dimension_in_smaller_sets}
It is common to consider the Hausdorff dimension of trajectory sets in the subgroups $H'\subseteq H$, as in the Dani Correspondence. A natural question is the following.
\begin{ques}
Let $\cF$ be a set of $g$-templates closed under equivalence, and let $\Lambda\in X_n$. Can one provide a combinatorial expression for
\[\dim_{\funH}(Y_{\Lambda, \cF}\cap H'; d_\varphi)?\] 
\end{ques} 
In general, one does not expect that the codimension of the intersection in $H'$ will coincide with the codimension of the original set, i.e., 
\[\dim_\funH(H'; d_\varphi)-\dim_\funH(Y_{\Lambda, \cF}\cap H'; d_\varphi) \neq \dim_\funH(H; d_\varphi)-\dim_\funH(Y_{\Lambda, \cF}; d_\varphi),\]
does not hold in general, even when it does not predict a negative dimension. Instead, we expect the dimension to be an integral involving quantities of the form $\dim_\funH(H'\cap H_{V_\bullet};d_\varphi)$ for various flags $V_\bullet$, ($H_{V_\bullet}$ is defined in Definition \ref{de: H_V_bullet}). 

\subsection{Different Lie groups} 
\label{sub:different_lie_groups}
It seems plausible that the techniques we use here can be generalized to general $1$-parameter semisimple flows in a semisimple real Lie group $G$, acting on $G/\Gamma$, where $\Gamma\subset G$ is a lattice. 
The proof of Lemma \ref{lem:single advance} stands out from the other parts of the paper by being mostly a Lie-theoretic result, and its generalization may require other tools.

Let $G$ be a complex semisimple Lie group, $B$ the Borel subgroup, and $B\subseteq P_1,...,P_r\subseteq G$ the maximal parabolic subgroups of $G$. These groups are all maximal parabolic subgroups of $G$ up to conjugation. 
Let $(g_t)_{t\in \RR}$ be a $1$-parameter semisimple flow on $G$. 
Define $(G/P)^g := \{p\in G/P:\forall t\in \RR, ~g_tp = p\}$ to be the set of fixed points in $G/P$ for every parabolic subgroup $P$.
In particular, $P$ may be the Borel subgroup $B$. 
Let $\cI_P$ denote the set of connected components of $(G/P)^g$.
Note that the quotient maps $G/B\to G/P_i$ are $g_t$-equivariant, and hence define natural maps $\pi_i:\cI_B\to \cI_{P_i}$.
For every parabolic subgroup $P$ and $E_1, E_2 \in \cI_P$, we write $E_1\eteq E_2$ if there is $p\in G/P$ such that $\lim\limits_{t\to -\infty} g_tp \in E_1$ and $\lim\limits_{t\to \infty} g_tp \in E_2$.

\begin{ques}\hfill
  \begin{itemize}
    \item Is $\eteq$ a partial order of $\cI_P$ for every parabolic $P\subseteq G$?
    \item Let $E_1, E_2\in \cI_B$. Do we have $E_1\eteq E_2$ if and only if $\pi_i(E_1)\eteq \pi_i(E_2)$ for every $1\le i\le r$?
  \end{itemize}
\end{ques}
For the relation between parabolic subgroups of $G$ and the cusps of $G/\Gamma$, see \cite{BS} for the theory and \cite{TW} for an application.






\section{Hausdorff Dimensions} 
\label{sec:hausdorff_dimensions}
\subsection{The Expansion Metric} 
\label{sub:the_expansion_metric}

\begin{de}\label{de: exp metric}
\index{boundedly expanding action}\hypertarget{hf}{}
Let $(X, d_X)$ be a metric space and $\varphi_t:X\to X$ be a group action $\RR\acts X$.
We say that $\varphi$ is a \emph{boundedly expanding action} if for some $0 < \alpha_\varphi' < \alpha_\varphi$ we have 
\[\forall t\in \RR,~\frac{d}{dt}\log d_X(\varphi_tx_1, \varphi_tx_2) \in [\alpha_\varphi, \alpha_\varphi'],\]
for every $x_1, x_2\in X$. 

From now on we assume that $\varphi$ is a boundedly expanding action.

Define $d_\varphi:X\times X\to [0,\infty)$ on $X$ by 
\[d_\varphi(x_1, x_2) := \begin{cases}\exp(-
\min \{t\in \RR:d_X(\varphi_tx_1, \varphi_tx_2) \ge 1\}), & \text{if }x_1 \neq x_2,\\
0,  & \text{if }x_1 = x_2,
\end{cases} \]
\index{expansion semi-metric}\hypertarget{hg}{}
$d_\varphi$ is called the \emph{expansion semi-metric}.
\end{de}
\begin{obs}
Note that 
\begin{align}\label{eq: d_phi equivariant}
d_\varphi(\varphi_tx_1, \varphi_tx_2) = d_t(x_1, x_2)\cdot \exp (t),
\end{align}
and the unit ball with respect to $d_\varphi$ coincides with the unit ball with respect to $d_X$.
These properties characterize $d_\varphi$.
\end{obs}
In this subsection we fix a boundedly expanding group action $\RR\acts X$, denoted $\varphi$, and denote by $0<\alpha_\varphi<\alpha_\varphi'$ the corresponding quantities in Definition \ref{de: exp metric}.
\begin{lem}\label{lem:metric}
If $0<a\le \alpha_\varphi$ then $d_\varphi$ to the power of $a$, namely, $d_\varphi^a$, is a metric. 
\end{lem}
\begin{proof}
It is clear that $d_\varphi^a$ is symmetric and that $d_\varphi^a(x,x) = 0$ for every $x\in X$. Since the action $\varphi$ is boundedly expanding, $d_\varphi^a(x_1,x_2)>0$ whenever $x_1\neq x_2$. 
It remains to show the triangle inequality for three different points $x_1, x_2, x_3\in X$. Assume $d_0 := d_\varphi(x_1, x_3)\ge d_\varphi(x_1, x_2), d_\varphi(x_2, x_3)$. By Eq. \eqref{eq: d_phi equivariant} w.l.o.g. we may assume that $d_\varphi(x_1, x_3) = 1$. Now note that since $d_\varphi(x_1, x_2) < 1$ we have $s_{12}:=-\log (d_\varphi(x_1, x_2)) \ge 0$.
The bounded expansion property yields
\begin{align*}
-\log d_X(x_1, x_2) = &\log d_X(\varphi_{s_{12}}x_1, \varphi_{s_{12}}x_2)-\log d_X(x_1, x_2) \\
=&\int_{0}^{s_{12}}\frac d{dt}\log d_X(\varphi_{t}x_1, \varphi_{t}x_2) dt \ge s_{12}\alpha_\varphi,
\end{align*}
and hence $d_X(x_1, x_2) \le d_\varphi^{\alpha_\varphi}(x_1, x_2)\le d_\varphi^{a}(x_1, x_2) $. Similarly, $d_X(x_2, x_3)\le   d_\varphi^{a}(x_2, x_3)$. 
Therefore,
\begin{align*}
d_\varphi^a(x_1,x_3)
=1
=d_X(x_1, x_3) &\le d_X(x_1, x_2) + d_X(x_2, x_3) 
\\&
\le d_\varphi^{a}(x_1, x_2)+d_\varphi^{a}(x_2, x_3),
\end{align*}
as desired.
\end{proof}
\begin{de}
\index{Hausdorff dimension}\hypertarget{hh}{}
\index{dim@$\dim_\funH$}\hypertarget{hi}{}
Recall that for a set $A\subseteq X$ the Hausdorff dimension with respect to the metric $d_\varphi^a$, is given by
\[\dim_\funH(A; d_\varphi^a) := \inf\{d\ge 0:\mathcal H^d(A;d_\varphi^a) = 0\},\] 
where 
\[\mathcal H^d(A;d_\varphi^a) = \inf\left\{\sum_ir_i^d:\,
\begin{matrix}
\text{$A$ admits a cover with balls of}\\
\text{radii $r_i>0$ in the metric $d_\varphi^a$}
\end{matrix}
\right\},\]
is the Hausdorff measure.
Since a ball of radius $r$ with respect to $d_\varphi^{a_1}$ is a ball of radius $r^{a_2/a_1}$ with respect to $d_\varphi^{a_2}$ we see that, for every $0<a_1,a_2<\alpha_\varphi$,
\begin{align*}
\mathcal H^d(A;d_\varphi^{a_1}) &= \inf\left\{\sum_ir_i^d:
\begin{matrix}
\text{$A$ admits a cover with balls of}\\
\text{radii $r_i>0$ in the metric $d^{a_1}_g$}
\end{matrix}
\right\} \\&= 
\inf\left\{\sum_ir_i^{da_1/a_2}:
\begin{matrix}
\text{$A$ admits a cover with balls of}\\
\text{radii $r_i>0$ in the metric $d^{a_2}_g$}
\end{matrix}
\right\} \\&= \mathcal H^{da_1/a_2}(A;d^{a_2}_g),
\end{align*}
and hence $\dim_\funH(A; d_\varphi^{a_1}) = \frac{a_2}{a_1}\dim_\funH(A; d_\varphi^{a_2})$.
Thus, the quantity 
\[\dim_\funH(A;d_\varphi) := a\dim_\funH(A; d_\varphi^{a}),\]
\index{Hausdorff dimension!with respect to a semi-metric}\hypertarget{hj}{}
\noindent which is called \emph{Hausdorff dimension of $A$ with respect to the semi-metric $d_\varphi$},
is independent of $a$. 
\end{de}
\begin{ex}
The simplest example is provided by $X = \RR^n$ and $\varphi_t \vec x = e^t \vec x$. Then $d_\varphi = d_X$ and the Hausdorff dimension $\dim_\funH(A;d_\varphi) = \dim_\funH(A)$ for every $A\subseteq X$. 
\end{ex}
\subsection{A Case of Interest} 
\label{sub:Case of interest}
Recall that $g_t$ is a diagonal $1$-parameter group, and $H$ is the expanding group of $g_t$, with the right-invariant Riemannian metric $d_H$.
\index{phi@$\varphi_t$}\hypertarget{ia}{}
Denote $\varphi_t(h) := g_thg_{-t}$. 
\begin{lem}
The action $\varphi_t$ is boundedly expanding with respect to the metric $d_H$.
\end{lem}
\begin{proof}
The standard Riemannian metric $d_{\SL_n(\RR)}$ is define by the bilinear form $\bra A,B\ket_{\Id} = \tr {AB^t}$ on $T_{\Id}\SL_n(\RR) = \ker \tr$, where $\tr:M_{n\times n}(\RR)\to \RR$ is the trace map, and extended to every tangent space $T_p\SL_n(\RR), p\in \SL_n(\RR)$ by right translations.

Note that the eigenspaces of the $\varphi_t$-action on $T_{\Id}(H)$ are orthogonal with respect to this bilinear form, and that the different eigenvalues are $\exp (t(\eta-\eta'))$, where $\eta>\eta'\in \Eall$.

Consequently, for every vector $\vec v\in T_{\Id}(H)$ we have 
\[\frac{d}{dt}\|(\varphi_t)_* \vec v\|/\|\vec v\|\in [\alpha_\varphi, \eta_n-\eta_1],\]
where $\alpha_\varphi$ is the minimal positive distance between elements of $\Eall$. 

For every point $p\in H$ define the right translation $R_{p^{-1}}:H\to H$ to be the multiplication from the right by $p^{-1}$.
For every $\vec v\in T_pH$ we have 
\begin{align*}
\|(\varphi_t)_*\vec v\| &= \|(R_{E_t(p)^{-1}})_*(\varphi_t)_*\vec v \|=\|(\varphi_t)_*(R_{p^{-1}})_*\vec v\| \\&
\in [\alpha_\varphi, \eta_n-\eta_1]\cdot \|(R_{p^{-1}})_*\vec v\| = [\alpha_\varphi, \eta_n-\eta_1]\cdot \|\vec v\|
\end{align*}
In particular, $\varphi_t$ satisfies the bounded expansion property when it is tested on tangent vectors instead of distances. Now the assertion of the lemma follows by the definition of the Riemannian metric.
\end{proof}
\index{dphi@$d_\varphi$}\hypertarget{ib}{}
\index{alphaphi@$\alpha_\varphi$}\hypertarget{ic}{}
Therefore, there is a constant $\alpha_\varphi$ such that 
$d_\varphi^{\alpha_\varphi}$ is a metric. 
The notations $\varphi, d_\varphi$, and $\alpha_\varphi$ will be used for the rest of the paper.
\subsection{Comparison with the Standard Hausdorff Dimension} 
\label{sub:comparison_with_standard_hausdorff_dimension}
Let $0 < \zeta_1\le\cdots\le \zeta_N$ be the differences $\{\eta_i-\eta_j > 0:1\le j\le i\le n\}$ in increasing order. 
Let $\overline F:\left[0, \sum_{i=1}^N\zeta_i\right] \to [0,N]$ be the piecewise linear function defined by the values
\index{Fbar@$\overline F$}\hypertarget{id}{}
\[\overline F\left(\sum_{i=1}^k \zeta_i\right)= k\quad 0\le k\le N,\]
and 
\index{Fubar@$\underline F$}\hypertarget{ie}{}
\[\underline F:[0,n]\to \left[0, \sum_{i=1}^N\zeta_i\right]\]
be the piecewise linear function defined by the values
\[\underline F\left(k\right)= \sum_{i=N-k+1}^N \zeta_i\quad 0\le k\le N.\]
The functions $\overline F$ and $\underline F^{-1}$ are related. Their graphs are composed of intervals of equal lengths and slopes. In $\overline F$ they are ordered from the largest slope to the smallest, and in $\underline F^{-1}$ the other way around.
\begin{thm}[Comparison Theorem]\label{thm:comparison}
Let $\varphi_t$ the flow on $H$ induced by conjugation with the diagonal flow $g_t$ and $d_\varphi$ be its expansion semi-metric.
Then for any set $X\subseteq H$
\begin{align}
\dim_\funH (X; d_H) \le \overline F(\dim_\funH (X; d_\varphi))
\end{align}
and
\begin{align}\label{eq: lower bound on dimention}
\dim_\funH (X; d_\varphi) \le \underline F(\dim_\funH (X; d_H)).
\end{align}
\end{thm}
\begin{cor}\label{cor: H dimension}
We have $\dim_\funH (H;d_\varphi) = \sum_{\eta,\eta'\in \Eall}(\eta-\eta')^+$.
\end{cor}
The proofs of \eqref{eq: lower bound on dimention} and \eqref{cor: H dimension} are similar, we will cover a ball of one semi-metric with balls of the other.

To prove the theorem we need some preparation. 
Denote by $\mu$ the two sided Haar measure on $H$.

\begin{lem}\label{lem: bound on ball intersection}
There are constants $c_1, c_2>1$ such that for every $0<r_1,r_2<1$ and $h_1,h_2\in H$ such that 
\[B_{d_H}(h_1; r_1)\cap B_{d_\varphi}(h_2; r_2)\neq \emptyset,\]
we have
\[\mu\big(B_{d_H}(h_1; c_1 r_1)\cap B_{d_\varphi}(h_1; c_2r_2)\big) =\Theta\left(\prod_{i=1}^N\min(r_1, r_2^{\zeta_i})\right).\]
\end{lem}
\begin{proof}
We will first show the result when $h_1=h_2=\Id$. By Observation \ref{obs: H coordinate}, every $h\in H$ has $n^2-N$ fixed entries and $N$ entries that can vary. 
Denote by $\kappa:H\to \RR^N$ the map that assigns to each $h\in H$ its $N$ varying coordinates. Denote by $\varphi_t':\RR^N\to \RR^N$ the diagonal action multiplying the $i$-th coordinate by $\exp (t\zeta_i)$. Note that $\kappa$ is action-equivariant, that is, the following diagram commutes:
\[\xymatrix{
H\ar[r]^\kappa \ar[d]^{\varphi_t}& \RR^N\ar[d]^{\varphi_t'}\\
H\ar[r]^\kappa & \RR^N
}\]

Since $\kappa$ is a manifold diffeomorphism, it distorts the standard metric by a bounded factor on any compact set in $H$. A direct computation shows that $\kappa$ maps the Haar measure $\mu$ on $H$ into the standard Lebesgue measure $\lambda$ on $\RR^N$. 
It follows that for some $0<\alpha<\alpha'$ we have 

\begin{align}\label{eq: kappa H ball in a box}
[-\alpha r_1,\alpha r_1]^N\subseteq \kappa(B_{d_H}(\Id;r_1)) \subseteq [-\alpha'r_1,\alpha'r_1]^N.
\end{align}

Substituting $r_1=1$ and applying $\varphi_{\log r_2}'$ to Eq. \eqref{eq: kappa H ball in a box} we obtain,
\begin{align}\label{eq: kappa phi ball in a box}
\prod_{i=1}^N[-\alpha r_2^{\zeta_i},\alpha r_2^{\zeta_i}]\subseteq \kappa(B_{d_\varphi}(\Id;r_2)) \subseteq \prod_{i=1}^N[-\alpha' r_2^{\zeta_i},\alpha' r_2^{\zeta_i}].
\end{align}
Combining \eqref{eq: kappa H ball in a box} and \eqref{eq: kappa phi ball in a box} we obtain,
\begin{align}
\mu\big(B_{d_H}(\Id;r_1)\cap B_{d_\varphi}(\Id;r_2)) = \Theta\left(\prod_{i=1}^N\min(r_1, r_2^{\zeta_i})\right).
\end{align}

Now We turn to the general case.
Let $c_1 = 2$,
and $c_2 = 2^{1/\alpha_\varphi}$, where $\alpha_\varphi$ is such that $d_\varphi^{\alpha_\varphi}$ is a metric.

Let $0<r_1,r_2<1$ and $h_1,h_2\in H$ and assume
\[B_{d_H}(h_1; r_1)\cap B_{d_\varphi}(h_2; r_2)\neq \emptyset.\]
Fix $h_3\in B_{d_H}(h_1; r_1)\cap B_{d_\varphi}(\Id; r_2)$. 
Note that 
\begin{align*}
B_{d_H}&(h_1; 2r_1)\cap B_{d_\varphi}(h_3; 2^{1/\alpha_\varphi}r_2)\supseteq B_{d_H}(h_3; r_1)\cap B_{d_\varphi}(h_3; r_2).
\end{align*}
The lower bound follows from the right-invariance of $d_H$ and $d_\varphi$.
The upper bound follows in much the same way, using the fact that
\begin{align*}
B_{d_H}&(h_1; 2r_1)\cap B_{d_\varphi}(h_3; 2^{1/\alpha_\varphi}r_2)\subseteq B_{d_H}(h_3; 4r_1)\cap B_{d_\varphi}(h_3; 4^{d_\varphi}r_2).
\end{align*}
\end{proof}

For every $d\in \RR$ denote \[\overline\chi(d) := \sum_{i=1}^N(\zeta_i-d)^+, \quad\underline\chi(d) := \sum_{i=1}^N(1-d\zeta_i)^+.\] 
\begin{lem}\label{lem: cover phi by H balls}
Let $r<1$ and $d\ge 0$. 
We can cover a $d_\varphi$-ball of radius $r$ with $O(r^{-\overline\chi(d)})$ $d_H$-ball of radii $r^d$.
\end{lem}
\begin{proof}
Let $H_1\subseteq B_{d_\varphi}(\Id;r)$ be a maximal set of points such that $d_H (h_1, h_2)\ge r^d$ for every $h_1, h_2\in H_1$. 
Then the balls $B_{d_H}(h_1;r^d/2)$ with $h_1\in H_1$ are pairwise disjoint and $\{B_{d_H}(h_1;r^d):h_1\in H_1\}$ is a cover of $B_{d_\varphi}(\Id; r)$.

Applying Lemma \ref{lem: bound on ball intersection} to $h_1\in H_1$, $h_2=\Id$, $r_1=r^d / (2c_1), r_2=r$ we obtain
\[\mu\left(B_{d_H}(h_1;r^d/2)\cap B_{d_\varphi}(\Id;rc_2)\right)=\Theta\left(\prod_{i=1}^N\min(r^d, r^{\zeta_i})\right). \]
Similarly, 
\[\mu\left(B_{d_H}(h_1;r^d)\cap B_{d_\varphi}(\Id;rc_2)\right)=\Theta\left(\prod_{i=1}^N\min(r^d, r^{\zeta_i})\right). \]
The claim follows from the fact that \[\mu\left(B_{d_\varphi}(\Id;rc_1)\right)=\Theta\left(\prod_{i=1}^Nr^{\zeta_i}\right).\]
\end{proof}
The following result and its proof are analogous to Lemma \ref{lem: cover phi by H balls}.
\begin{lem}
Let $r<1$ and $d\ge 0$.  
We can cover a $d_H$ ball of radius $r$ with $O(r^{-\overline\chi(d)})$ balls of radii $r^d$ with respect to the semi-metric $d_\varphi$.
\end{lem}
\begin{proof}[Proof of Theorem \emph{\ref{thm:comparison}}]
Let $X\subseteq H$ with $d_0=\dim_\funH (X;d_\varphi)$. Let $d>d_0$ and let $(B_{d_\varphi}(h_i;r_i))_{i\in I}$ be a cover of $X$ with $d_\varphi$-balls such that $\sum_{i\in I}r_i^d < 1$. We will cover $X$ with $d_H$-balls. Fix $d'\ge 0$. By Lemma \ref{lem: cover phi by H balls}, we can cover each $d_\varphi$-ball of radius $r_i$ with $\Theta\left(r^{-\overline\chi(d')}\right)$ $d_H$-balls of radius $r^{d'}$. 

Consequently, for $d''\ge 0$, the $d''$-Hausdorff measure of $X$ is at most
\[\sum_{i\in I}\Theta\left(r^{-\overline\chi(d')}\right)\cdot r_i^{d'd''} = \Theta\left(\sum_{i\in I}r^{-\overline\chi(d')+d'd''}\right).\]
We will derive a bound on the Hausdorff measure of $X$ only when $-\overline\chi(d')+d'd'' \ge d$.
Hence we ask ourselves what is the minimal $d''>0$ such that for some $d'\ge0$ we have $-\overline\chi(d')+d'd'' \ge d$. Inverting the role of $d$ and $d''$, we want to show that for every $d''\in [0,N]$ we have 
\begin{align}\label{eq: bar F inverse requirement}
\sup_{d' > 0}d'd''-\overline\chi(d') = \overline F^{-1}(d'').
\end{align}
Note that $\overline\chi$ vanishes on $(-\infty, \zeta_1]$ and hence the supremum in Eq. \eqref{eq: bar F inverse requirement} is obtained at $d' \ge \zeta_1 >0$ and can be taken even for all $d' \in \RR$. Hence, it is enough to prove 
\begin{align}\label{eq: bar F inverse requirement2}
\sup_{d'\in \RR}d'd''-\overline\chi(d') = \overline F^{-1}(d'').
\end{align}
In fact, $\overline F$ was defined so that $\overline F^{-1}$ is the convex conjugate of $\overline\chi$, which is stated by Eq. \eqref{eq: bar F inverse requirement2}.
Similarly, Eq. \eqref{eq: lower bound on dimention} is a consequence of the fact that $\underline F^{-1}$ is the convex conjugate of $\underline \chi$.
\end{proof}

\subsection{Dimension Games} 
\label{sub:hausdorff_games}

Schmidt \cite{S} was the first to prove Hausdorff dimension theorems via (Borel) games with a victory condition.
\index{Game - $(T, g)$}\hypertarget{if}{}
Das et al. \cite{DFSU} obtained Theorem \ref{thm: general scewed dimention formula} for simpler diagonal flows $g_t$ using a family of zero-sum games with a payoff function (and not a victory condition), which Alice tries to maximize and Bob to minimize. We here describe the game of \cite{DFSU} in our setting, which is termed the \emph{$(T, g)$-game}.

Let $T>0$ and let $X\subseteq H$ be a Borel set. 
Alice and Bob play a zero-sum alternating-move game, where Alice plays an initialization step, and afterwards Alice and Bob play alternately, first Alice and then Bob. 

In the initialization step Alice chooses $h_0\in H$ and $T_0>0$.
Along the game, Alice and Bob generate a sequence $B_0\supset B_1\supset\cdots$, where $B_m = B_{d_\varphi}(h_m; \exp(-T_m))$ and $T_m = Tm + T_0$. Recall that $d_\varphi^{\alpha_\varphi}$ is a metric.


Formally, the game is played as follows.
\begin{description}
  \item[Initialization step] Alice chooses $h_0\in H, T_0>0$. Set $B_0 := B_{d_\varphi}(h_0; \exp(-T_0))$ and $T_m := Tm + T_0$ for every $m\ge 1$.
  \item[Alice's $m$-th step] Denote $B_{m} := B_{d_\varphi}(h_{m}; \exp(-T_{m}))$. 
  Alice chooses a finite set 
  \index{Am@$A_m$}\hypertarget{ig}{}
  \index{hm@$h_m$}\hypertarget{ih}{}
  \begin{align}\label{eq: A_m location}
  A_{m+1}\subset B_{d_\varphi}(h_m; (1-\exp(-\alpha_\varphi T))^{1/\alpha_\varphi}\exp(-T_m))
  \end{align}
  such that $d_\varphi(h_{m+1}, h_{m+1}') > 3^{1/\alpha_\varphi}\exp(-T_{m})$ for every two distinct elements $h_{m+1}, h_{m+1}'\in A_{m+1}$.
  \item[Bob's $m$-th step] Bob chooses $h_{m+1}\in A_{m+1}$. 
\end{description}
\begin{remark}
The coefficient $(1-\exp(-\alpha_\varphi T))^{1/\alpha_\varphi}$ in Eq. \eqref{eq: A_m location} ensures that $B_{m+1}\subseteq B_m$.
\end{remark}
\index{hinfty@$h_\infty$}\hypertarget{ii}{}
Denote \[ h_\infty := \lim_{m\to \infty} h_m = \bigcap_{m=0}^\infty B_m\] 
and $A_\bullet = (A_m)_{m=1}^\infty$.
The payoff which Alice wants to maximize and Bob wants to minimize, is 
\begin{align}\label{eq: value definition}
\Delta(A_\bullet, h_\infty) := \begin{cases}\frac{1}{T}
\liminf_{m\to \infty}\frac1m\sum_{k=1}^m {\log \#A_k},&  \text{if }h_\infty\in X,\\
-\infty,& \text{otherwise.}
\end{cases}
\end{align}
Denote the value of this game by $D_{T, g}(X)$, i.e., Alice can guarantee any payoff less than $D_{T, g}(X)$, and Bob can deny payoffs higher than $D_{T, g}(X)$.

\index{Doubling metric}\hypertarget{ij}{}
Recall that a metric space $(X, d_X)$ is called \emph{doubling} if there exists $M>0$ such that for every $R>0$ and every point $x\in X$ there exist points $x_1,\dots,x_M \in X$ such that 
\[B_{d_X}(x; R) \subseteq \bigcup_{i=1}^M B_{d_X}(x_i;R/2).\]

The following theorem is a special case of \cite[Theorem 28.2]{DFSU}. Using \cite[Remark 28.3]{DFSU} we may apply it to any complete doubling space. $(H,d_\varphi^{\alpha_\varphi})$ is a doubling space by the compactness of $B_{d_\varphi^{\alpha_\varphi}}(1)$, the right-invariance of $d_\varphi$ and the behavior under the $\varphi$-action. Thereofre, we obtain the following result.
\begin{thm}\label{thm: hausdorff game result}
For any Borel set $Y\subseteq H$ we have $\lim_{T\to \infty}D_{Y, g}(X) = \dim_\funH(Y;d_\varphi)$.
\end{thm}
\subsection{Dimension Game Applied to Trajectory Sets} 
\label{sub:hausdorff_game_applied_to_traje}
Let $\cF$ be a set of $g_t$ templates closed under equivalence, and $\Lambda\in X_n$ be a lattice.
To prove Theorem \ref{thm: general scewed dimention formula}, we will apply Theorem \ref{thm: hausdorff game result} to the set 
$Y_{\Lambda, \cF} := Y_{\Lambda, \cF}\subseteq H$.

Defining the strategies seems complicated, as the resulting $h_\infty$ of a play is known only after the game is played, and hence $f^{h_\infty \Lambda}$ can be computed up to equivalence only at the end of the $(T, g)$-game.
Nonetheless, $f^{h_\infty\Lambda}$ can be approximated. 
\begin{obs}\label{obs: The behavior of templates}
Since $d_\varphi(h_m, h_\infty) \le \exp(-Tm)$,
\[h_t:= g_t h_mh_\infty^{-1} g_{-t}\]
is such that $d_\varphi(h_t,\Id)=O(\exp(t-Tm))$ for every $t\le Tm$. 
Using the equality $h_t g_th_{\infty}\Lambda = g_th_{m}\Lambda$ we deduce that 
$f^{h_\infty\Lambda}|_{(-\infty, Tm]}$ and $f^{h_m\Lambda}|_{(-\infty, Tm]}$ are uniformly equivalent, where the implicit constants are independent of $m$ and $T$. 
In other words, after $m$ steps of Alice and Bob we are able to approximate $f^{h_\infty\Lambda}$ up to time $Tm$. 
\end{obs}
The strategies of Alice and Bob will be constructed based on Observation \ref{obs: The behavior of templates}. In each step the two players will consider $f^{h_\infty\Lambda}$ at time $Tm$ and choose their actions to ensure a certain behavior of $f^{h_\infty\Lambda}$ at time $T(m+1)$.


\section{Templates Approximation} 
\label{sec:approximation_of_templates}
In this section we will prove two results. The first is purely combinatorial, and approximates a template with a simpler one that will be used in Alice's strategy.
The second constructs a template for every $g_t$ lattice orbit.
\begin{de}[Nontrivial places at time $t$]
\index{Lft@$L_f(t)$}\hypertarget{ja}{}
Let $f$ be a $g$-template on $I$. For $t\in I$ denote \[L_f(t) := \{0\le l \le n: t\in U_l(f)\} = L(f_{H,\bullet}(t)),\]
that is, $L_f(t)$ consists of those $l$ for which $f_{E,l}(t)$ is not null. 

\end{de}

\begin{de}[Vertices of a $g$-template]
Let $f$ be a $g$-template.
\index{Vertex}\hypertarget{jb}{}
\index{Vertex!non-null}\hypertarget{jc}{}
\index{Vertex!null}\hypertarget{jd}{}
We term the irregular points $t$ of $f_{E, \bullet}$ \emph{vertices of $f$} and the irregular points of $f_{E, l}$ \emph{vertices of $f$ at $l$}. 
A vertex $t$ is called \emph{null} (resp. \emph{non-null}) if its corresponding arrow by the category flow $f_{E, l}$ is null (resp. non-null) in the category $\cI_l^*$.

Note that the null vertices of $f$ at $l$ are the endpoints of the nontriviality intervals $J\subseteq U_l(f)$.
In addition, note that there are no vertices at $l=0,n$.
\end{de}

\noindent\textbf{Example \ref{ex: a g template}, continued.}
Figure \ref{fig: a g template with verts} depicts a $g$-template with $\Eall = \{-1.2, 0.5, 0.7\}$. 
The green vertices are non-null and the blue vertices are null.
\null\nobreak\hfill\ensuremath{\diamondsuit}

\begin{figure}[ht]
\caption{A $g$-template with its vertices.
}
\label{fig: a g template with verts}
\begin{tikzpicture}[line cap=round,line join=round,>=triangle 45,x=1.0cm,y=1.0cm]
\draw [color=black] (6.45, -1.4) node {$f_{H, 1}$};
\draw [color=redcolor] (6.45, -2.7) node {$f_{H, 2}$};

\clip(-6,-7) rectangle (6,.6);

\draw [color=black] (-5.5, -5.2) node {$f_{E, 1}:$};
\draw [color=redcolor] (-5.5, -6.2) node {$f_{E, 2}:$};

\draw [line width=2.pt, dash pattern=on 5pt off 5pt,] (-7.0,0.03) -- (-4,0.03);
\draw [line width=2.pt, dash pattern=on 5pt off 5pt,color=redcolor] (-7.2,-.03) -- (-4,-.03);
\draw [line width=1.pt, -to] (-7,0) -- (5.5,0) node[above] {$t$};

\draw [line width=2.pt, ] (-4,0.0) -- (-2,-2.4);
\draw [line width=2.pt, ] (-2,-2.4) -- (0.05882352941176476,-1.3705882352941174);
\draw [line width=2.pt, ] (1,-1.6999999999999997) -- (2,-2.8999999999999995);
\draw [line width=2.pt, ] (2,-2.899999999999999) -- (2.8947368421052624,-2.2736842105263153);
\draw [line width=2.pt, color=redcolor,] (-3,-0.6000000000000001) -- (2,-4.1);
\draw [line width=2.pt, color=redcolor,] (2,-4.1) -- (4,-5.1);
\draw [line width=2.pt, color=redcolor,] (4,-5.099999999999999) -- (7,-1.4999999999999982);

\draw [line width=2.pt, dash pattern=on 5pt off 5pt,] (0.05882352941176476,-1.3705882352941174) -- (1,-1.6999999999999997);
\draw [line width=2.pt, dash pattern=on 5pt off 5pt,] (2.8947368421052624,-2.2736842105263153) -- (4.0,-2.5499999999999994);
\draw [line width=2.pt, dash pattern=on 5pt off 5pt,] (4.0,-2.549999999999999) -- (7.0,-0.7499999999999982);
\draw [line width=2.pt, dash pattern=on 5pt off 5pt,color=redcolor,] (-4.0,0.0) -- (-3,-0.6000000000000001);

\begin{scriptsize}
\draw [fill=bluecolor] (-4,0.0) circle (2.5pt); 
\draw [fill=bluecolor] (-3,-0.6000000000000001) circle (2.5pt); 
\draw [fill=greencolor] (-2,-2.4) circle (2.5pt); 
\draw [fill=bluecolor] (0.05882352941176476,-1.3705882352941174) circle (2.5pt); 
\draw [fill=bluecolor] (1,-1.6999999999999997) circle (2.5pt); 
\draw [fill=greencolor] (2,-4.1) circle (2.5pt); 
\draw [fill=greencolor] (2,-2.899999999999999) circle (2.5pt); 
\draw [fill=bluecolor] (2.8947368421052624,-2.2736842105263153) circle (2.5pt); 
\draw [fill=greencolor] (4,-5.099999999999999) circle (2.5pt); 
\draw [fill=bluecolor] (7,-1.4999999999999982) circle (2.5pt); 
\end{scriptsize}

\draw [line width=2.pt, dash pattern=on 5pt off 5pt,] (-7,-5.5) -- (-4,-5.5);
\draw[color=black] (-4.5, -5.2) node {$*$};
\draw [line width=2.pt, ] (-4,-5.5) -- (-2,-5.5);
\draw[color=black] (-3.0, -5.2) node {$\{-1.2\}$};
\draw [line width=2.pt, ] (-2,-5.5) -- (0.05882352941176476,-5.5);
\draw[color=black] (-0.9705882352941176, -5.2) node {$\{0.5\}$};
\draw [line width=2.pt, dash pattern=on 5pt off 5pt,] (0.05882352941176476,-5.5) -- (1,-5.5);
\draw[color=black] (0.5294117647058824, -5.2) node {$*$};
\draw [line width=2.pt, ] (1,-5.5) -- (2,-5.5);
\draw[color=black] (1.4, -5.2) node {$\{-1.2\}$};
\draw [line width=2.pt, ] (2,-5.5) -- (2.8947368421052624,-5.5);
\draw[color=black] (2.5, -5.2) node {$\{0.7\}$};
\draw [line width=2.pt, dash pattern=on 5pt off 5pt,] (2.8947368421052624,-5.5) -- (7,-5.5);
\draw[color=black] (4.947368421052631, -5.2) node {$*$};
\draw [fill=bluecolor] (-7,-5.5) circle (2.5pt); 
\draw [fill=bluecolor] (-4,-5.5) circle (2.5pt); 
\draw [fill=greencolor] (-2,-5.5) circle (2.5pt) node[above] {$t_0$}; 
\draw [fill=bluecolor] (0.05882352941176476,-5.5) circle (2.5pt); 
\draw [fill=bluecolor] (1,-5.5) circle (2.5pt); 
\draw [fill=greencolor] (2,-5.5) circle (2.5pt); 
\draw [fill=bluecolor] (2.8947368421052624,-5.5) circle (2.5pt); 
\draw [fill=bluecolor] (7,-5.5) circle (2.5pt); 
\draw [line width=2.pt, dash pattern=on 5pt off 5pt,color=redcolor,] (-7,-6.5) -- (-3,-6.5);
\draw[color=redcolor] (-4.5, -6.2) node {$*$};
\draw [line width=2.pt, color=redcolor,] (-3,-6.5) -- (2,-6.5);
\draw[color=redcolor] (-0.5, -6.2) node {$\{-1.2, 0.5\}$};
\draw [line width=2.pt, color=redcolor,] (2,-6.5) -- (4,-6.5);
\draw[color=redcolor] (3.0, -6.2) node {$\{-1.2, 0.7\}$};
\draw [line width=2.pt, color=redcolor,] (4,-6.5) -- (7,-6.5);
\draw[color=redcolor] (5, -6.2) node {$\{0.5, 0.7\}$};
\draw [fill=bluecolor] (-7,-6.5) circle (2.5pt); 
\draw [fill=bluecolor] (-3,-6.5) circle (2.5pt); 
\draw [fill=greencolor] (2,-6.5) circle (2.5pt); 
\draw [fill=greencolor] (4,-6.5) circle (2.5pt) node[below] {$t_1$}; 
\draw [fill=greencolor] (7,-6.5) circle (2.5pt); 

\end{tikzpicture}
\end{figure}

\begin{de}[$C$-separated $g$-templates]
\index{gtemplate@$g$-template!separated}\hypertarget{je}{}
A $g$-template $f$ is said to be \emph{$C$-separated} for $C>0$ if for every null vertex $t$ of $f$ there is no other vertex $t'$ of $f$ with $|t-t'| < C$. Moreover, $t$ cannot be a vertex of $f$ at two values of $l$. 
\end{de}

An implication of a template being $C$-separated is that it cannot change too often, and hence it will be useful in the construction of the strategy of Alice.

\begin{obs}\label{obs: partial2 comutation}
Let $0\le l_{-1}<l_0<l_1\le n$ be three indices and $a_\bullet$ be a height sequence. If $a_\bullet$ is linear on $[l_{-1}, l_0]$ and on $[l_0, l_1]$, then \[\partial^2a_{l} = \frac{1}{l_0-l_{-1}} a_{l_{-1}} + \frac{1}{l_1-l_0} a_{l_1} - \frac{l_1-l_{-1}}{(l_1-l_0)(l_0-l_{-1})}a_{l_0}.\]
\end{obs}
\begin{cor}\label{cor: impacts is nondif}
Let $f$ be a $g$-template on an interval $I$, let $1\le l \le n-1$ and let $t\in I$ be a point such that $\partial^2 f_{H,l}$ is not differentiable at $t$. Then there exists an index $l_0$ such that $t$ is a vertex of $f$ at $l_0$. Moreover one of the following is true:
\begin{itemize}
	\item $l=l_0$.
	\item $l$ is adjacent to $l_0$ in $L_f(t)$. 
\end{itemize}
\end{cor}
Corollary \ref{cor: impacts is nondif} motivates the following definition.
\begin{de}[Vertex impact]
Let $t$ be a vertex of a $g$-template $f$ at $l_0$.
\index{Vertex!impact}\hypertarget{jf}{}
Denote by $l_{-1}<l_0<l_1$ the adjacent indices of $l_0$ in $L_f(t)$. We say that the vertex $t$ at $l_0$ \emph{impacts} the indices $l_{-1},l_0,l_1$. The set of vertices $t$ that impact $l$ contains all nondifferentiability points of $\partial^2f_{H,l}(t)$, though not necessarily all these vertices are points of nondifferentiability.
\end{de}
\noindent\textbf{Example \ref{ex: a g template}, continued.}
In Figure \ref{fig: a g template with verts} the vertex $t_0$ at $l=1$ impacts $0,1,2$, while $t_1$ at $l=2$ impacts $0, 2, 3$. \null\nobreak\hfill\ensuremath{\diamondsuit}
\begin{de}[$C$-Significant $g$-templates]
\index{gtemplate@$g$-template!significant}\hypertarget{jg}{}
We say that a $g$-template $f$ is \emph{$C$-significant} for $C>0$ if for every $1\le l\le n-1$ that is impacted by a vertex $t$ of $f$ we have $\partial^2 f_{H,l}(t) > C$ or $\partial^2 f_{H,l}(t) = 0$. 
\end{de}
\begin{remark}\label{rem:impact + L}
Let $t$ be a vertex of a $g$-template $f$ at $l_0$. If $t$ is a null vertex, then $l_0\nin L_{f}(t)$ even though this vertex itself impacts $l_0$. In this case $\partial^2 f_{H,l_0}(t) = 0$, and hence to check whether $f$ is $C$-significant or not, one can ignore this impact instance. We still regard this phenomenon as impact in view of Corollary \ref{cor: impacts is nondif}. 
\end{remark}
\begin{lem}\label{lem:separated}
For every $C>0$ there exists $C'>0$ such that for every $g$-template $f$ on the interval $I$ there exists a $C$-significant, $C$-separated $g$-template $f'$ that is $C'$-close to $f$. 
\end{lem}
The section will be divided as follows. In Subsection \ref{sub:shifting_templates} we introduce the shifting operation on $g$-templates. It is an operation that is able to perturb a $g$-template. 
In Subsection \ref{sub:proof_of_lemma_lem:significant} we introduce combinatorial tools that allow to shift a $g$-template to a significant one. 
In Subsection \ref{sub:properties_of_significant_templates} we discuss some properties of significant templates and complete the proof of Lemma \ref{lem:separated}. In Subsection \ref{sub:proof_of_lemma_thm:lattice approx} we prove Theorem \ref{thm:Lattice approx}.
\subsection{Shifting Templates} 
\label{sub:shifting_templates}
In this subsection we define the shift operation on templates.
\begin{de}[Lower convex hull]
\index{Lower convex hull}\hypertarget{jh}{}
Let $X\subset \RR$ be a finite set and let $f:X\to \RR$ be function.
The \emph{lower convex hull} $f^{\conv}:[\min X, \max X]\to \RR$ is the maximal convex function which satisfies 
$f^{\conv}(x)\le f(x)$ for every $x\in X$.

The \emph{lower convex hull of a sequence} $a_0,...,a_n$ is the lower convex hull of the function $l\mapsto a_l$, then restricted back to $\{0,...,n\}$:
\[((l\mapsto a_l)^{\conv}(l))_{l=0}^n;\]
\noindent it is alternatively referred to as \emph{convexification}.
\end{de}
\begin{de}[Shift of height sequences]
\index{Shift!of height sequence $a^\brho_\bullet$}\hypertarget{ji}{}
Let $a_\bullet$ be a height sequence and let $\brho=(\rho_l)_{l=0}^n$ be a concave sequence of numbers with $\rho_0=\rho_n=0$, i.e., the numbers $0=\rho_0, \rho_1,\dots,\rho_n=0$ satisfy $\rho_l\ge (\rho_{l-1}+\rho_{l+1})/2$.
Define the \emph{shifted height sequence} $a_\bullet^\brho$ to be the lower convex hull of the sequence $(a_l + \rho_l)_{l=0}^n$; see Example \ref{ex:shift of cross section}.
\end{de}

\begin{remark}\label{rem:flag shift is cont and support dependent}
The shift operation is continuous in $a_\bullet$ and $\brho$,
and the concavity of $\brho$ implies that the shift depends only on the values of $\brho$ in $L(a_\bullet)$. 
This yields an explicit formula for $a_\bullet^\brho$:
\[a^{\brho}_l = \min_{\substack{l_1\le l\le l_2\\ \alpha l_1+(1-\alpha) l_2 = l\\l_1, l_2\in L(a_\bullet)}}  \alpha (a_{l_1} + \rho_{l_1}) + (1-\alpha)(a_{l_2} + \rho_{l_2})\]
\end{remark}
\begin{ex}\label{ex:shift of cross section}
Figure \ref{fig: shift of cross section} shows a height sequence $a_\bullet$, a concave sequence $\brho$, and the height sequence $a_\bullet^\brho$. 
\end{ex}

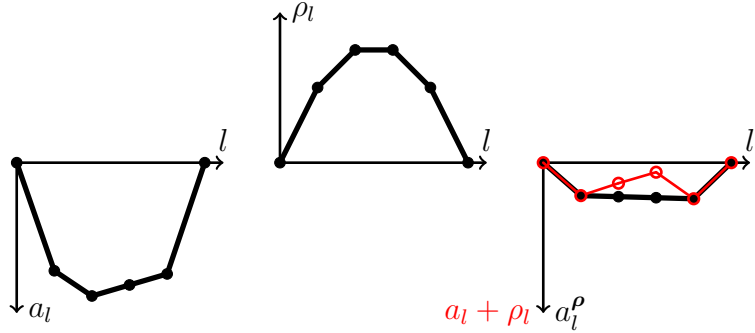
\begin{figure}
\caption{The graphs of $a_\bullet, \brho$,  $a_l+\rho_l$ and on it $a^\brho_\bullet$.}
\label{fig: shift of cross section}
\begin{tikzpicture}[line cap=round,line join=round,>=triangle 45,x=.5cm,y=.5cm]
\clip(-1,-5) rectangle (20,5);

\draw [line width=1.pt, -to] (0,0) -- (5.5,0) node[above] {$l$};
\draw [line width=1.pt, -to] (0,0) -- (0,-4) node[right] {$a_l$};

\draw [line width=1.pt, -to] (7,0) -- (12.5,0) node[above] {$l$};
\draw [line width=1.pt, -to] (7,0) -- (7,4) node[right] {$\rho_l$};

\draw [line width=1.pt, -to] (14,0) -- (19.5,0) node[above] {$l$};
\draw [line width=1.pt, -to] (14,0) -- (14,-4) node[left,color=redcolor] {$a_l+\rho_l$} node[right] {$a_l^\brho$};

\draw [line width=2.pt] (12.,0.)-- (11.,2.);
\draw [line width=2.pt] (10.,3.)-- (11.,2.);
\draw [line width=2.pt] (10.,3.)-- (9.,3.);
\draw [line width=2.pt] (8.,2.)-- (9.,3.);
\draw [line width=2.pt] (7.,0.)-- (8.,2.);
\draw [line width=2.pt] (5.,0.)-- (4.,-2.9603709831921243);
\draw [line width=2.pt] (2.,-3.5471207015622808)-- (4.,-2.9603709831921243);
\draw [line width=2.pt] (0.,0.)-- (1.,-2.8765495948535302);
\draw [line width=2.pt] (1.,-2.8765495948535302)-- (2.,-3.5471207015622808);
\draw [line width=2.pt] (14.,0.)-- (15.,-0.8765495948535302);
\draw [line width=2.pt] (19.,0.)-- (18.,-0.9603709831921243);
\draw [line width=2.pt] (15.,-0.8765495948535302)-- (18.,-0.9603709831921243);
\begin{scriptsize}
\draw [fill=black] (0.,0.) circle (2.0pt);
\draw [fill=black] (1.,-2.8765495948535302) circle (2pt);
\draw [fill=black] (2.,-3.5471207015622808) circle (2pt);
\draw [fill=black] (4.,-2.9603709831921243) circle (2pt);
\draw [fill=black] (5.,0.) circle (2.0pt);
\draw [fill=black] (7.,0.) circle (2.0pt);
\draw [fill=black] (8.,2.) circle (2.0pt);
\draw [fill=black] (9.,3.) circle (2.0pt);
\draw [fill=black] (10.,3.) circle (2.0pt);
\draw [fill=black] (11.,2.) circle (2.0pt);
\draw [fill=black] (3.,-3.2537458423772025) circle (2.0pt);
\draw [fill=black] (12.,0.) circle (2.0pt);
\draw [color=redcolor, fill=black,line width=1.pt] (14.,0.) circle (2.0pt);

\draw [color=redcolor, line width=1.pt] (14.,0.)-- (15.,-0.8765495948535302);
\draw [color=redcolor, line width=1.pt] (15.,-0.8765495948535302) -- (16.,-0.5471207015622808);
\draw [color=redcolor, line width=1.pt] (16.,-0.5471207015622808) -- (17.,-0.25374584237720255);
\draw [color=redcolor, line width=1.pt] (17.,-0.25374584237720255) -- (18.,-0.9603709831921243);
\draw [color=redcolor, line width=1.pt] (18.,-0.9603709831921243)--(19.,0.);

\draw [color=redcolor,fill=black, line width=1.pt] (15.,-0.8765495948535302) circle (2.0pt);
\draw [color=redcolor, line width=1.pt] (16.,-0.5471207015622808) circle (2.0pt);
\draw [color=redcolor, line width=1.pt] (17.,-0.25374584237720255) circle (2.0pt);
\draw [color=redcolor, fill=black, line width=1.pt] (18.,-0.9603709831921243) circle (2.0pt);
\draw [color=redcolor, fill=black, line width=1.pt] (19.,0.) circle (2.0pt);
\draw [fill=black] (17.,-0.9324305204125932) circle (2.0pt);
\draw [fill=black] (16.,-0.9044900576330619) circle (2.0pt);
\end{scriptsize}
\end{tikzpicture}
\end{figure}

\begin{de}[Shift of $g$-templates]
\index{Shift!sequence}\hypertarget{jj}{}
Let $f$ be a $g$-template on an interval $I$.
A \emph{shift sequence} is a sequence $\brho =(\rho_l)_{l=0}^n$ of $n+1$ locally constant functions $\rho_l: U_l(f) \to \RR$, such that $\rho_0\equiv \rho_n\equiv 0$, and for every $0 = l_0< \dots < l_k = n$ and $t\in U_{l_0}(f)\cap \dots\cap U_{l_k}(f)$ the piecewise linear function defined by $l_i\mapsto \rho_{l_i}(t)$ is concave. 

Let $\brho$ be a shift sequence.
For every $t\in I$, let $\bar\brho(t)$ be any extension of $\brho$ at $t$ to a concave sequence $0=\bar\rho_0(t), \bar\rho_1(t), \dots,\bar\rho_n(t) = 0$. 

Let us define the \emph{shifted $g$-template} $f^\brho$. 
\index{Shift!of $g$-template $f^\brho$}\hypertarget{baa}{}
For every $t$, set \[f_{H, \bullet}^{\brho}(t) = (f_{H, \bullet}(t))^{\bar\brho(t)};\]
this is independent of the choice of $\bar\brho$ and continuous by Remark \ref{rem:flag shift is cont and support dependent}. 

Denote by $U_l(f^\brho)$ the set of times $t$ for which $l\in \supp f^\brho_{H}(t)$ is a nontrivial index. 
This will later coincide with Definition \ref{de: U_l} of $U_l(f^\brho)$. 
Let $f^\brho_{E,l}$ be as follows:
\begin{itemize}
	\item  On $U_l(f^\brho)$ it coincides with $f_{E,l}$.
	\item Outside of $U_l(f^\brho)$ it is null.
	\item The connecting morphisms on the boundary $\partial U_l(f^\brho)$ are the null arrows.
\end{itemize}
One can verify that $f^\brho$ is a $g$-template.


For future use we note that
\begin{align}\label{eq: shift formula}
f^\brho_{H, l}:=\min_{\substack{l_1\le l\le l_2\\ \alpha l_1+(1-\alpha) l_2 = l\\l_1, l_2\in L_f(t)}}  \alpha (f_{H, l_1}(t) + \rho_{l_1}(t)) + (1-\alpha)(f_{H, l_2}(t)+\rho_{l_1}(t)).
\end{align}
\end{de}

\begin{lem}\label{lem: shift is equiv}
If $(\rho_l)_{l=0}^n$ are bounded by $C$, then $f^\brho$ is $3C$-close to $f$.
\end{lem}
\begin{proof}
Fix $t\in I$. To attain $f^\brho_H(t)$, we take $f_H(t)$, increase every value by $\rho_l(t)$ and then take the lower convex hull. It follows that $f_{H,l}(t) \le f^\brho_{H,l}(t) \le f_{H,l}(t)+C$, and hence $\left|\partial^2f_{H,l}(t) - \partial^2f^\brho_{H,l}(t)\right|\le 2C$. 
If $[a,b]\subseteq I$ is such that $\partial^2 f|_{[a,b]} \ge 3C$, then $[a,b]\subseteq U_l(f^\brho)$, and hence $f_E|_{[a,b]} \equiv f_E^\brho|_{[a,b]}$. 
\end{proof}

A useful shift sequence is the following.
\begin{de}[The independent shift sequence]
\index{Independent shift sequence}\hypertarget{bab}{}
Let $C>0$, let $f$ be a $g$-template on an interval $I$, and let $\nu_J\in [0,C]$ for $J\in \cG_f$,
where $\cG_f$ is as in Definition \ref{de: U_l}.
By convention, the unique nontriviality intervals $J$ of $f$ at $0$ and $n$ have $\nu_J=0$.

The \emph{independent shift sequence $\brho$ with parameters $C, (\nu_J)_{J\in \cG_f}$} is defined by $\rho_{l}|_J = l(n-l)C + \nu_J$ for every $J\in \pi_0(U_l(f))$.
From now on, we will use only independent shift sequences. A useful property of this sequence is that it can change independently on each interval.
\end{de}


\noindent\textbf{Example \ref{ex: a g template}, continued.}
Figures \ref{fig: example of shift} and \ref{fig: example of shift2} depict shifts of the $g$-template in Figure \ref{ex: a g template} and the original $g$-template in lighter colors behind.
We uses the independent shift sequence $\brho$ composed of the constant functions $0,1,1,0$, and $2\brho$.
Notice that the nontriviality intervals, i.e., the connected components of $U_l$ shrinks in Figure \ref{fig: example of shift}, and the interval $(t_0, t_1)$ even vanishes completely in Figure \ref{fig: example of shift2}, as the maximum of $\partial^2f_{H,1}|_{(t_0, t_1)}$ is $1.7$, which becomes negative when we add to it $\partial^2 \brho(t) = -2$. The condition $\partial^2 a_l + \partial^2\rho_l\le 0$ is sufficient for $l$ to be a trivial place of $a^\brho_\bullet$, though it is not necessary.
We see that the null (blue) vertices moved, and the non-null (green) did not. 
\null\nobreak\hfill\ensuremath{\diamondsuit}

\begin{figure}[ht]
\caption{A shift of a $g$-template.
}\label{fig:}
\label{fig: example of shift}
\begin{tikzpicture}[line cap=round,line join=round,>=triangle 45,x=1.0cm,y=1.0cm]

\draw [color=black] (6.45, -.8) node {$f_{H, 1}^\brho$};
\draw [color=lightblackcolor] (6.45, -1.3) node {$f_{H, 1}$};
\draw [color=redcolor] (6.45, -1.8) node {$f_{H, 2}^\brho$};
\draw [color=lightredcolor] (6.45, -2.7) node {$f_{H, 2}$};

\clip(-6,-7) rectangle (6,.6);


\draw [color=black] (-5.5, -5.2) node {$f_{E, 1}:$};
\draw [color=redcolor] (-5.5, -6.2) node {$f_{E, 2}:$};

\draw [line width=2.pt, color=lightblackcolor] (-4,0.0) -- (-2,-2.4);
\draw [line width=2.pt, color=lightblackcolor] (-2,-2.4) -- (0.05882352941176476,-1.3705882352941174);
\draw [line width=2.pt, color=lightblackcolor] (1,-1.6999999999999997) -- (2,-2.8999999999999995);
\draw [line width=2.pt, color=lightblackcolor] (2,-2.899999999999999) -- (2.8947368421052624,-2.2736842105263153);
\draw [line width=2.pt, color=lightredcolor,] (-3,-0.6000000000000001) -- (2,-4.1);
\draw [line width=2.pt, color=lightredcolor,] (2,-4.1) -- (4,-5.1);
\draw [line width=2.pt, color=lightredcolor,] (4,-5.099999999999999) -- (7,-1.4999999999999982);
\draw [line width=2.pt, dash pattern=on 5pt off 5pt, color=lightblackcolor] (0.05882352941176476,-1.3705882352941174) -- (1,-1.6999999999999997);
\draw [line width=2.pt, dash pattern=on 5pt off 5pt, color=lightblackcolor] (2.8947368421052624,-2.2736842105263153) -- (4.0,-2.5499999999999994);
\draw [line width=2.pt, dash pattern=on 5pt off 5pt, color=lightblackcolor] (4.0,-2.549999999999999) -- (7.0,-0.7499999999999982);
\draw [line width=2.pt, dash pattern=on 5pt off 5pt,color=lightredcolor,] (-4.0,0.0) -- (-3,-0.6000000000000001);

\draw [line width=2.pt, dash pattern=on 5pt off 5pt,color=lightblackcolor] (-7.0,0.08) -- (-4,0.08);
\draw [line width=2.pt, dash pattern=on 5pt off 5pt,color=lightredcolor] (-7.2,-.08) -- (-4,-.08);
\draw [line width=1.pt, -to] (-7,0) -- (5.5,0) node[above] {$t$};

\draw [line width=2.pt, dash pattern=on 5pt off 5pt,] (-7.2,0.03) -- (-3.1666666666666665,0.03);
\draw [line width=2.pt, dash pattern=on 5pt off 5pt,color=redcolor] (-7,-.03) -- (-3.1666666666666665,-.03);

\begin{scriptsize}
\draw [color=lightblackcolor, fill=lightbluecolor] (-3,-0.6000000000000001) circle (2.5pt); 
\draw [color=lightblackcolor, fill=lightgreencolor] (-2,-2.4) circle (2.5pt); 
\draw [color=lightblackcolor, fill=lightbluecolor] (0.05882352941176476,-1.3705882352941174) circle (2.5pt); 
\draw [color=lightblackcolor, fill=lightbluecolor] (1,-1.6999999999999997) circle (2.5pt); 
\draw [color=lightblackcolor, fill=lightgreencolor] (2,-4.1) circle (2.5pt); 
\draw [color=lightblackcolor, fill=lightgreencolor] (2,-2.899999999999999) circle (2.5pt); 
\draw [color=lightblackcolor, fill=lightbluecolor] (2.8947368421052624,-2.2736842105263153) circle (2.5pt); 
\draw [color=lightblackcolor, fill=lightgreencolor] (4,-5.099999999999999) circle (2.5pt); 
\draw [color=lightblackcolor, fill=lightbluecolor] (7,-1.4999999999999982) circle (2.5pt); 
\end{scriptsize}

\draw [line width=2.pt, ] (-3.1666666666666665,8.881784197001252e-16) -- (-2,-1.399999999999999);
\draw [line width=2.pt, ] (-2,-1.399999999999999) -- (-0.5294117647058831,-0.6647058823529406);
\draw [line width=2.pt, ] (1.5882352941176472,-1.4058823529411761) -- (2,-1.8999999999999995);
\draw [line width=2.pt, ] (2,-1.8999999999999995) -- (2.368421052631579,-1.642105263157894);
\draw [line width=2.pt, color=redcolor,] (-1.5789473684210524,-0.5947368421052626) -- (2,-3.099999999999999);
\draw [line width=2.pt, color=redcolor,] (2,-3.099999999999999) -- (4,-4.1);
\draw [line width=2.pt, color=redcolor,] (4,-4.099999999999999) -- (7,-0.4999999999999982);
\draw [line width=2.pt, dash pattern=on 5pt off 5pt,] (-0.5294117647058831,-0.6647058823529406) -- (1.5882352941176472,-1.4058823529411761);
\draw [line width=2.pt, dash pattern=on 5pt off 5pt,] (2.368421052631579,-1.642105263157894) -- (4.0,-2.0499999999999994);
\draw [line width=2.pt, dash pattern=on 5pt off 5pt,] (4.0,-2.049999999999999) -- (7.0,-0.24999999999999822);
\draw [line width=2.pt, dash pattern=on 5pt off 5pt,color=redcolor,] (-3.1666666666666665,4.440892098500626e-16) -- (-2.0,-0.6999999999999995);
\draw [line width=2.pt, dash pattern=on 5pt off 5pt,color=redcolor,] (-2.0,-0.6999999999999995) -- (-1.5789473684210524,-0.5947368421052626);
\begin{scriptsize}
\draw [color=lightblackcolor, fill=lightbluecolor] (-4,0) circle (2.5pt); 

\draw [fill=bluecolor] (-3.1666666666666665,8.881784197001252e-16) circle (2.5pt); 
\draw [fill=greencolor] (-2,-1.399999999999999) circle (2.5pt); 
\draw [fill=bluecolor] (-1.5789473684210524,-0.5947368421052626) circle (2.5pt); 
\draw [fill=bluecolor] (-0.5294117647058831,-0.6647058823529406) circle (2.5pt); 
\draw [fill=bluecolor] (1.5882352941176472,-1.4058823529411761) circle (2.5pt); 
\draw [fill=greencolor] (2,-3.099999999999999) circle (2.5pt); 
\draw [fill=greencolor] (2,-1.8999999999999995) circle (2.5pt); 
\draw [fill=bluecolor] (2.368421052631579,-1.642105263157894) circle (2.5pt); 
\draw [fill=greencolor] (4,-4.099999999999999) circle (2.5pt); 
\draw [fill=bluecolor] (7,-0.4999999999999982) circle (2.5pt); 
\end{scriptsize}

\draw [line width=2.pt, dash pattern=on 5pt off 5pt, color=lightblackcolor] (-7,-5.7) -- (-4,-5.7);
\draw [line width=2.pt, color=lightblackcolor] (-4,-5.7) -- (-2,-5.7);
\draw [line width=2.pt, color=lightblackcolor] (-2,-5.7) -- (0.05882352941176476,-5.7);
\draw [line width=2.pt, dash pattern=on 5pt off 5pt, color=lightblackcolor] (0.05882352941176476,-5.7) -- (1,-5.7);
\draw [line width=2.pt, color=lightblackcolor] (1,-5.7) -- (2,-5.7);
\draw [line width=2.pt, color=lightblackcolor] (2,-5.7) -- (2.8947368421052624,-5.7);
\draw [line width=2.pt, dash pattern=on 5pt off 5pt, color=lightblackcolor] (2.8947368421052624,-5.7) -- (7,-5.7);
\draw [line width=2.pt, color=lightblackcolor] (-7,-5.8) -- (-7,-5.6000000000000005); 
\draw [line width=2.pt, color=lightblackcolor] (-4,-5.8) -- (-4,-5.6000000000000005); 
\draw [line width=2.pt, color=lightblackcolor] (-2,-5.8) -- (-2,-5.6000000000000005); 
\draw [line width=2.pt, color=lightblackcolor] (0.05882352941176476,-5.8) -- (0.05882352941176476,-5.6000000000000005); 
\draw [line width=2.pt, color=lightblackcolor] (1,-5.8) -- (1,-5.6000000000000005); 
\draw [line width=2.pt, color=lightblackcolor] (2,-5.8) -- (2,-5.6000000000000005); 
\draw [line width=2.pt, color=lightblackcolor] (2.8947368421052624,-5.8) -- (2.8947368421052624,-5.6000000000000005); 
\draw [line width=2.pt, color=lightblackcolor] (7,-5.8) -- (7,-5.6000000000000005); 
\draw [line width=2.pt, dash pattern=on 5pt off 5pt,color=lightredcolor,] (-7,-6.7) -- (-3,-6.7);
\draw [line width=2.pt, color=lightredcolor,] (-3,-6.7) -- (2,-6.7);
\draw [line width=2.pt, color=lightredcolor,] (2,-6.7) -- (4,-6.7);
\draw [line width=2.pt, color=lightredcolor,] (4,-6.7) -- (7,-6.7);
\draw [line width=2.pt, color=lightredcolor] (-7,-6.8) -- (-7,-6.6000000000000005); 
\draw [line width=2.pt, color=lightredcolor] (-3,-6.8) -- (-3,-6.6000000000000005); 
\draw [line width=2.pt, color=lightredcolor] (2,-6.8) -- (2,-6.6000000000000005); 
\draw [line width=2.pt, color=lightredcolor] (4,-6.8) -- (4,-6.6000000000000005); 
\draw [line width=2.pt, color=lightredcolor] (7,-6.8) -- (7,-6.6000000000000005); 

\draw [line width=2.pt, dash pattern=on 5pt off 5pt,] (-7,-5.5) -- (-3.1666666666666665,-5.5);
\draw[color=black] (-4.2894736842105265, -5.2) node {$*$};
\draw [line width=2.pt, ] (-3.1666666666666665,-5.5) -- (-2,-5.5);
\draw[color=black] (-2.583333333333333, -5.2) node {$\{-1.2\}$};
\draw [line width=2.pt, ] (-2,-5.5) -- (-0.5294117647058831,-5.5);
\draw[color=black] (-1.2647058823529416, -5.2) node {$\{0.5\}$};
\draw [line width=2.pt, dash pattern=on 5pt off 5pt,] (-0.5294117647058831,-5.5) -- (1.5882352941176472,-5.5);
\draw[color=black] (0.3, -5.2) node {$*$};
\draw [line width=2.pt, ] (1.5882352941176472,-5.5) -- (2,-5.5);
\draw[color=black] (1.4, -5.2) node {$\{-1.2\}$};
\draw [line width=2.pt, ] (2,-5.5) -- (2.368421052631579,-5.5);
\draw[color=black] (2.44, -5.2) node {$\{0.7\}$};
\draw [line width=2.pt, dash pattern=on 5pt off 5pt,] (2.368421052631579,-5.5) -- (7,-5.5);
\draw[color=black] (4.684210526315789, -5.2) node {$*$};
\draw [fill=bluecolor] (-7,-5.5) circle (2.5pt); 
\draw [fill=bluecolor] (-3.1666666666666665,-5.5) circle (2.5pt); 
\draw [fill=greencolor] (-2,-5.5) circle (2.5pt); 
\draw [fill=bluecolor] (-0.5294117647058831,-5.5) circle (2.5pt); 
\draw [fill=bluecolor] (1.5882352941176472,-5.5) circle (2.5pt); 
\draw [fill=greencolor] (2,-5.5) circle (2.5pt); 
\draw [fill=bluecolor] (2.368421052631579,-5.5) circle (2.5pt); 
\draw [fill=bluecolor] (7,-5.5) circle (2.5pt); 
\draw [line width=2.pt, dash pattern=on 5pt off 5pt,color=redcolor,] (-7,-6.5) -- (-1.5789473684210524,-6.5);
\draw[color=redcolor] (-4.2894736842105265, -6.2) node {$*$};
\draw [line width=2.pt, color=redcolor,] (-1.5789473684210524,-6.5) -- (2,-6.5);
\draw[color=redcolor] (0.21052631578947378, -6.2) node {$\{-1.2, 0.5\}$};
\draw [line width=2.pt, color=redcolor,] (2,-6.5) -- (4,-6.5);
\draw[color=redcolor] (3.0, -6.2) node {$\{-1.2, 0.7\}$};
\draw [line width=2.pt, color=redcolor,] (4,-6.5) -- (7,-6.5);
\draw[color=redcolor] (5.1, -6.2) node {$\{0.5, 0.7\}$};
\draw [fill=bluecolor] (-7,-6.5) circle (2.5pt); 
\draw [fill=bluecolor] (-1.5789473684210524,-6.5) circle (2.5pt); 
\draw [fill=greencolor] (2,-6.5) circle (2.5pt); 
\draw [fill=greencolor] (4,-6.5) circle (2.5pt); 
\draw [fill=greencolor] (7,-6.5) circle (2.5pt); 

\end{tikzpicture}

\end{figure}
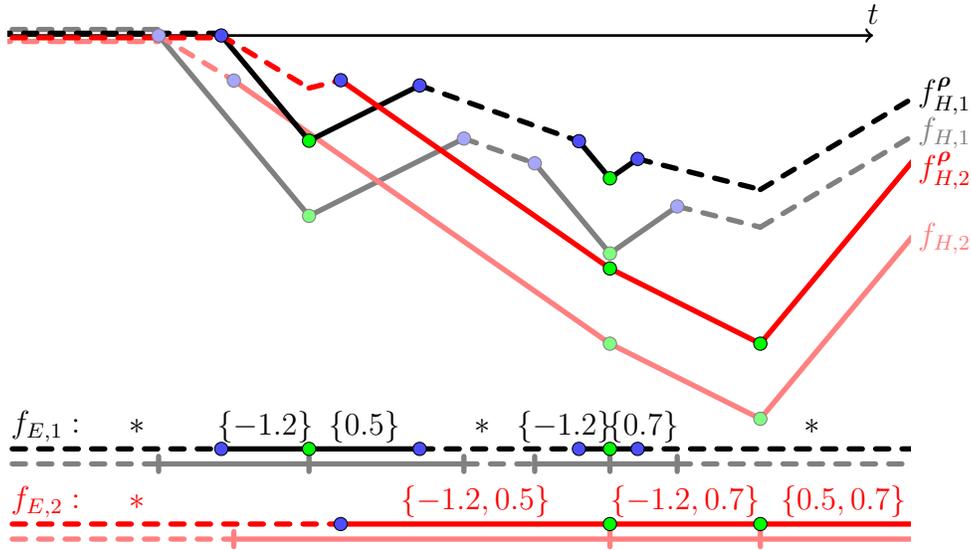

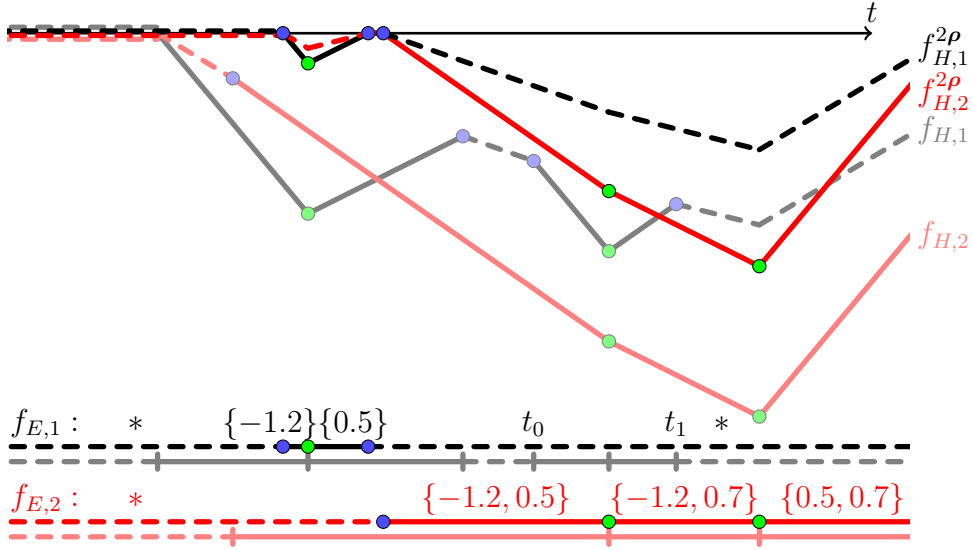
\begin{figure}[ht]
\caption{A shift of a $g$-template.
}
\label{fig: example of shift2}
\begin{tikzpicture}[line cap=round,line join=round,>=triangle 45,x=1.0cm,y=1.0cm]

\draw [color=black] (6.45, -.2) node {$f_{H, 1}^{2\brho}$};
\draw [color=lightblackcolor] (6.45, -1.3) node {$f_{H, 1}$};
\draw [color=redcolor] (6.45, -.8) node {$f_{H, 2}^{2\brho}$};
\draw [color=lightredcolor] (6.45, -2.7) node {$f_{H, 2}$};

\clip(-6,-7) rectangle (6,.6);


\draw [color=black] (-5.5, -5.2) node {$f_{E, 1}:$};
\draw [color=redcolor] (-5.5, -6.2) node {$f_{E, 2}:$};

\draw [line width=2.pt, color=lightblackcolor] (-4,0.0) -- (-2,-2.4);
\draw [line width=2.pt, color=lightblackcolor] (-2,-2.4) -- (0.05882352941176476,-1.3705882352941174);
\draw [line width=2.pt, color=lightblackcolor] (1,-1.6999999999999997) -- (2,-2.8999999999999995);
\draw [line width=2.pt, color=lightblackcolor] (2,-2.899999999999999) -- (2.8947368421052624,-2.2736842105263153);
\draw [line width=2.pt, color=lightredcolor,] (-3,-0.6000000000000001) -- (2,-4.1);
\draw [line width=2.pt, color=lightredcolor,] (2,-4.1) -- (4,-5.1);
\draw [line width=2.pt, color=lightredcolor,] (4,-5.099999999999999) -- (7,-1.4999999999999982);
\draw [line width=2.pt, dash pattern=on 5pt off 5pt, color=lightblackcolor] (0.05882352941176476,-1.3705882352941174) -- (1,-1.6999999999999997);
\draw [line width=2.pt, dash pattern=on 5pt off 5pt, color=lightblackcolor] (2.8947368421052624,-2.2736842105263153) -- (4.0,-2.5499999999999994);
\draw [line width=2.pt, dash pattern=on 5pt off 5pt, color=lightblackcolor] (4.0,-2.549999999999999) -- (7.0,-0.7499999999999982);
\draw [line width=2.pt, dash pattern=on 5pt off 5pt,color=lightredcolor,] (-4.0,0.0) -- (-3,-0.6000000000000001);

\draw [line width=2.pt, dash pattern=on 5pt off 5pt,color=lightblackcolor] (-7.0,0.08) -- (-4,0.08);
\draw [line width=2.pt, dash pattern=on 5pt off 5pt,color=lightredcolor] (-7.2,-.08) -- (-4,-.08);
\draw [line width=1.pt, -to] (-7,0) -- (5.5,0) node[above] {$t$};

\draw [line width=2.pt, dash pattern=on 5pt off 5pt,] (-7.2,0.03) -- (-2.333333333333333,0.03);
\draw [line width=2.pt, dash pattern=on 5pt off 5pt,color=redcolor] (-7,-.03) -- (-2.333333333333333,-.03);

\begin{scriptsize}
\draw [color=lightblackcolor, fill=lightbluecolor] (-3,-0.6000000000000001) circle (2.5pt); 
\draw [color=lightblackcolor, fill=lightgreencolor] (-2,-2.4) circle (2.5pt); 
\draw [color=lightblackcolor, fill=lightbluecolor] (0.05882352941176476,-1.3705882352941174) circle (2.5pt); 
\draw [color=lightblackcolor, fill=lightbluecolor] (1,-1.6999999999999997) circle (2.5pt); 
\draw [color=lightblackcolor, fill=lightgreencolor] (2,-4.1) circle (2.5pt); 
\draw [color=lightblackcolor, fill=lightgreencolor] (2,-2.899999999999999) circle (2.5pt); 
\draw [color=lightblackcolor, fill=lightbluecolor] (2.8947368421052624,-2.2736842105263153) circle (2.5pt); 
\draw [color=lightblackcolor, fill=lightgreencolor] (4,-5.099999999999999) circle (2.5pt); 
\draw [color=lightblackcolor, fill=lightbluecolor] (7,-1.4999999999999982) circle (2.5pt); 
\end{scriptsize}

\draw [line width=2.pt, ] (-2.3333333333333335,0.0) -- (-2,-0.40000000000000036);
\draw [line width=2.pt, ] (-2,-0.40000000000000036) -- (-1.1999999999999993,0.0);
\draw [line width=2.pt, color=redcolor,] (-0.9999999999999997,1.1102230246251565e-16) -- (2,-2.0999999999999996);
\draw [line width=2.pt, color=redcolor,] (2,-2.0999999999999996) -- (4,-3.0999999999999996);
\draw [line width=2.pt, color=redcolor,] (4,-3.0999999999999996) -- (6.583333333333334,8.881784197001252e-16);
\draw [line width=2.pt, dash pattern=on 5pt off 5pt,] (-1.1999999999999993,0.0) -- (-0.9999999999999997,5.551115123125783e-17);
\draw [line width=2.pt, dash pattern=on 5pt off 5pt,] (-0.9999999999999997,5.551115123125783e-17) -- (2.0,-1.0499999999999998);
\draw [line width=2.pt, dash pattern=on 5pt off 5pt,] (2.0,-1.0499999999999998) -- (4.0,-1.5499999999999998);
\draw [line width=2.pt, dash pattern=on 5pt off 5pt,] (4.0,-1.5499999999999998) -- (6.583333333333334,4.440892098500626e-16);
\draw [line width=2.pt, dash pattern=on 5pt off 5pt,color=redcolor,] (-2.3333333333333335,0.0) -- (-2.0,-0.20000000000000018);
\draw [line width=2.pt, dash pattern=on 5pt off 5pt,color=redcolor,] (-2.0,-0.20000000000000018) -- (-1.1999999999999993,0.0);
\draw [line width=2.pt, dash pattern=on 5pt off 5pt,color=redcolor,] (-1.1999999999999993,0.0) -- (-0.9999999999999997,1.1102230246251565e-16);

\begin{scriptsize}
\draw [fill=bluecolor] (-2.3333333333333335,0.0) circle (2.5pt); 
\draw [fill=greencolor] (-2,-0.40000000000000036) circle (2.5pt); 
\draw [fill=bluecolor] (-1.1999999999999993,0.0) circle (2.5pt); 
\draw [fill=bluecolor] (-0.9999999999999997,1.1102230246251565e-16) circle (2.5pt); 
\draw [fill=greencolor] (2,-2.0999999999999996) circle (2.5pt); 
\draw [fill=greencolor] (4,-3.0999999999999996) circle (2.5pt); 
\draw [fill=bluecolor] (6.583333333333334,8.881784197001252e-16) circle (2.5pt); 
\end{scriptsize}

\draw [line width=2.pt, dash pattern=on 5pt off 5pt, color=lightblackcolor] (-7,-5.7) -- (-4,-5.7);
\draw [line width=2.pt, color=lightblackcolor] (-4,-5.7) -- (-2,-5.7);
\draw [line width=2.pt, color=lightblackcolor] (-2,-5.7) -- (0.05882352941176476,-5.7);
\draw [line width=2.pt, dash pattern=on 5pt off 5pt, color=lightblackcolor] (0.05882352941176476,-5.7) -- (1,-5.7);
\draw [line width=2.pt, color=lightblackcolor] (1,-5.7) -- (2,-5.7);
\draw [line width=2.pt, color=lightblackcolor] (2,-5.7) -- (2.8947368421052624,-5.7);
\draw [line width=2.pt, dash pattern=on 5pt off 5pt, color=lightblackcolor] (2.8947368421052624,-5.7) -- (7,-5.7);
\draw [line width=2.pt, color=lightblackcolor] (-7,-5.8) -- (-7,-5.6000000000000005); 
\draw [line width=2.pt, color=lightblackcolor] (-4,-5.8) -- (-4,-5.6000000000000005); 
\draw [line width=2.pt, color=lightblackcolor] (-2,-5.8) -- (-2,-5.6000000000000005); 
\draw [line width=2.pt, color=lightblackcolor] (0.05882352941176476,-5.8) -- (0.05882352941176476,-5.6000000000000005); 
\draw [line width=2.pt, color=lightblackcolor] (1,-5.8) -- (1,-5.6000000000000005); 
\draw [line width=2.pt, color=lightblackcolor] (2,-5.8) -- (2,-5.6000000000000005); 
\draw [line width=2.pt, color=lightblackcolor] (2.8947368421052624,-5.8) -- (2.8947368421052624,-5.6000000000000005); 
\draw [line width=2.pt, color=lightblackcolor] (7,-5.8) -- (7,-5.6000000000000005); 
\draw [line width=2.pt, dash pattern=on 5pt off 5pt,color=lightredcolor,] (-7,-6.7) -- (-3,-6.7);
\draw [line width=2.pt, color=lightredcolor,] (-3,-6.7) -- (2,-6.7);
\draw [line width=2.pt, color=lightredcolor,] (2,-6.7) -- (4,-6.7);
\draw [line width=2.pt, color=lightredcolor,] (4,-6.7) -- (7,-6.7);
\draw [line width=2.pt, color=lightredcolor] (-7,-6.8) -- (-7,-6.6000000000000005); 
\draw [line width=2.pt, color=lightredcolor] (-3,-6.8) -- (-3,-6.6000000000000005); 
\draw [line width=2.pt, color=lightredcolor] (2,-6.8) -- (2,-6.6000000000000005); 
\draw [line width=2.pt, color=lightredcolor] (4,-6.8) -- (4,-6.6000000000000005); 
\draw [line width=2.pt, color=lightredcolor] (7,-6.8) -- (7,-6.6000000000000005); 

\draw [line width=2.pt, dash pattern=on 5pt off 5pt,] (-7,-5.5) -- (-2.3333333333333335,-5.5);
\draw[color=black] (-4.2894736842105265, -5.2) node {$*$};
\draw [line width=2.pt, ] (-2.3333333333333335,-5.5) -- (-2,-5.5);
\draw[color=black] (-2.5, -5.2) node {$\{-1.2\}$};
\draw [line width=2.pt, ] (-2,-5.5) -- (-1.1999999999999993,-5.5);
\draw[color=black] (-1.4, -5.2) node {$\{0.5\}$};
\draw [line width=2.pt, dash pattern=on 5pt off 5pt,] (-1.1999999999999993,-5.5) -- (7,-5.5);
\draw[color=black] (3.5, -5.2) node {$*$};
\draw [fill=bluecolor] (-7,-5.5) circle (2.5pt); 
\draw [fill=bluecolor] (-2.3333333333333335,-5.5) circle (2.5pt); 
\draw [fill=greencolor] (-2,-5.5) circle (2.5pt); 
\draw [fill=bluecolor] (-1.1999999999999993,-5.5) circle (2.5pt); 
\draw [fill=bluecolor] (7,-5.5) circle (2.5pt); 
\draw [line width=2.pt, dash pattern=on 5pt off 5pt,color=redcolor,] (-7,-6.5) -- (-0.9999999999999997,-6.5);
\draw[color=redcolor] (-4.2894736842105265, -6.2) node {$*$};
\draw [line width=2.pt, color=redcolor,] (-0.9999999999999997,-6.5) -- (2,-6.5);
\draw[color=redcolor] (0.5000000000000002, -6.2) node {$\{-1.2, 0.5\}$};
\draw [line width=2.pt, color=redcolor,] (2,-6.5) -- (4,-6.5);
\draw[color=redcolor] (3.0, -6.2) node {$\{-1.2, 0.7\}$};
\draw [line width=2.pt, color=redcolor,] (4,-6.5) -- (6.583333333333334,-6.5);
\draw[color=redcolor] (5.1, -6.2) node {$\{0.5, 0.7\}$};
\draw [line width=2.pt, dash pattern=on 5pt off 5pt,color=redcolor,] (6.583333333333334,-6.5) -- (7,-6.5);
\draw[color=redcolor] (6.791666666666667, -6.2) node {$*$};
\draw [fill=bluecolor] (-7,-6.5) circle (2.5pt); 
\draw [fill=bluecolor] (-0.9999999999999997,-6.5) circle (2.5pt); 
\draw [fill=greencolor] (2,-6.5) circle (2.5pt); 
\draw [fill=greencolor] (4,-6.5) circle (2.5pt); 
\draw [fill=bluecolor] (6.583333333333334,-6.5) circle (2.5pt); 
\draw [fill=bluecolor] (7,-6.5) circle (2.5pt); 

\draw (1,-5.5) node[above]{$t_0$};
\draw (2.8947368421052624,-5.5) node[above]{$t_1$};
\end{tikzpicture}

\end{figure}



\subsection{Constructing a Significant Approximation} 
\label{sub:proof_of_lemma_lem:significant}
In this subsection 
we will develop a technique which will enable us to solve a large system of a type of linear inequalities, and thus construct a significant approximation and a separated approximation.

The following lemma is a step in the proof of Lemma \ref{lem:separated}
\begin{lem}\label{lem:significant}
For every $C>0$ there exists $C'>0$ such that for every $g$-template $f$ on the interval $I$ there exists a $C$-significant $g$-template $f'$, which is $C'$-close to $f$. 
\end{lem}

In Subsubsection \ref{ssub:the_plan} we will go over the steps of the construction. In Subsubsection \ref{ssub:proof_of_combinatorial_lemmas} we will provide proofs for lemmas stated in Subsubsection \ref{ssub:the_plan}. Finally, in Subsubsection \ref{ssub:construction_of_enticements} we will characterize significant $g$-templates in a way that will enable us to construct such approximations.

\subsubsection{The Plan} 
\label{ssub:the_plan}
Our goal is to construct a $C$-significant template that is equivalent to $f$.
This will be obtained as a shift of $f$ by an independent shift sequence $\brho$ with parameters $C_1, (\nu_J)_{J\in \cG_f}$, where $\nu_J \in  [0,C_1]$ for every $J\in \cG_f$. 
We will see (Lemma \ref{lem: construction of enticements}) that the condition of being $C$-significant can be strengthened to a condition of the following form: 
\begin{de}
\index{Enticement}\hypertarget{bac}{}
Fix $C_1> C>0$, $\varepsilon>0$, and a ``variable set'' $\cG$.
An \emph{enticement on $\cG$}
(with parameters $C_1>C$, $\varepsilon$) is an affine function 
\begin{align*}
\varphi&:\prod_{J\in \cG}[0,C_1]\to \RR\\
\varphi&((\nu_J)_{J\in \cG}) = a+\sum_{J\in \cG'} a_J\nu_J;~~a, a_J\in \RR,
\end{align*}
where $\cG'\subseteq \cG$ is a finite subset and $|a_J|\ge \varepsilon$ for every $J\in \cG'$. 
\index{Enticement!support}\hypertarget{bad}{}
\index{Enticement!satisfaction}\hypertarget{bae}{}
We call $\cG'$ the \emph{support} of the enticement $\varphi$.
We say that a collection of values $(\nu_J)_{J\in \cG}\subseteq [0,C_1]$ \emph{satisfies} the enticement $\varphi$ if 
\[\varphi((\nu_J)_{J\in \cG})\nin [0,C].\]

\index{Enticement!solution}\hypertarget{baf}{}
Let $Z$ be an index set.
A \emph{a solution} of a collection of enticement $\{\varphi_z:z\in Z\}$ is a tuple $(\nu_J)_{J\in \cG}$ that satisfies all enticements in the collection.

\index{Enticement!finitary collection}\hypertarget{bag}{}
A collection of enticements $\{\varphi_z:z\in Z\}$ on $\cG$ is \emph{$(r,R)$-finitary} if
\begin{itemize}
\item For every $z\in Z$ the support of the enticement $\varphi_z$ has at most $r$ elements.
\item For every $\cG'\subseteq \cG$ there are at most $R$ enticements $\varphi_z$ with support $\cG'$, where $z\in Z$. 
\end{itemize}

\end{de}
We will need a collection of enticements whose index set $\cG$ has a graph structure. 
We will now introduce the desired properties of the graph.
\begin{de}[Enticement compatibility with a graph]
Let $\cG$ be a graph.
\index{Enticement!compatible with a graph}\hypertarget{bah}{}
An enticement $\varphi$ on $\cG$ is said to be \emph{compatible with the graph structure on $\cG$} if its support $\cG'$ is a clique in $\cG$.
From now on all enticements on a graph will be assumed to be compatible with it.
\end{de}
\begin{de}[$d$-constructible graph]
\index{Constructible graph}\hypertarget{bai}{}
A graph $\cG$ is said to be \emph{$d$-constructible} if there is a filtration of full subgraphs 
$\emptyset = \cG_0\subseteq \cG_1\subseteq\dots\subseteq \cG$ such that $\bigcup_m \cG_m = \cG$ and for every vertex $v\in \cG_m\setminus \cG_{m-1}$ we have $\deg_{\cG_{m}}(v)\le d$. 
\end{de}

\begin{lem}\label{lem:enticements are solvable}
Let $\{\varphi_z:z\in Z\}$ be a collection of enticements with parameters $C_1 > C, \varepsilon$ on a $d$-constructible graph $\cG$ which are $(r, R)$-finitary for some $r, R\in \ZZ_{\ge 0}$. If 
\[C_1 > \varepsilon^{-1} R\sum_{r' = 0}^{r-1} \binom{d}{r'} C,\] 
then there is a solution for the collection of enticements.
\end{lem}
We will define a collection of enticements for each $g$-template $f$. We start by defining the underlying graph $\cG_f$ as $\cG_f = \bigsqcup_{l=1}^{n-1}\pi_0(U_l(f))$ as in Definition \ref{de: U_l}.
Two vertices $J_1, J_2\in \cG_f$ are connected by an edge if the intervals intersect.
Abusing notation, we do not distinguish between a graph and its set of vertices.

\begin{lem}\label{lem:graph is constructible}
For every $g$-template $f$, the graph $\cG_f$ is $(2n-4)$-constructible.
\end{lem}
Finally we produce a collection of enticements, where the existence of a solution $(\nu_J)_{J\in \cG_f}$ guarantees that the $g$-template $f^\brho$ is $C$-significant, where $\brho$ is the independent shift sequence with parameters $C, (\nu_J)_{J\in \cG_f}$.

\begin{lem}\label{lem: construction of enticements}
There exists $\varepsilon > 0$ depending on $\Eall$ such that for every $C_1>C>0$ the following holds. For every $g$-template $f$ there exists a $(4, 49 \cdot 2^{4n}\cdot n^{12})$-finitary collection of enticements $\{\varphi_z:z\in Z\}$ on $\cG_f$ with parameters $C_1>C, \varepsilon$ such that for every solution $(\nu_J)_{J\in \cG_f}$ of $\{\varphi_z:z\in Z\}$, the $g$-template $f^\brho$ is $C$-significant, where $\brho$ is the independent shift sequence with parameters $C, (\nu_J)_{J\in \cG_f}$.
\end{lem}

Lemma \ref{lem:significant} is a direct combination of Lemmas \ref{lem: construction of enticements}, \ref{lem:graph is constructible} and \ref{lem:enticements are solvable}.

These lemmas also show that Lemma \ref{lem:significant} holds with any
\[C_1 > \varepsilon^{-1} 49 \cdot 2^{4n}\cdot n^{12}\sum_{r = 0}^{3} \binom{2n-4}{r} C.\] 

\subsubsection{Proof of Lemmas \emph{\ref{lem:enticements are solvable}} and \emph{\ref{lem:graph is constructible}}} 
\label{ssub:proof_of_combinatorial_lemmas}
\begin{proof}[Proof of Lemma \emph{\ref{lem:enticements are solvable}}]
We will choose $\nu_J$ one by one. 
For every $m$ choose a well ordering $\prec$ on $\cG_m\setminus \cG_{m-1}$ and merge the well orderings into a single well ordering on $\cG$, still denoted $\prec$, such that $\cG_1\setminus \cG_{0} \prec \cG_2\setminus \cG_{1}\prec\cdots$.
Note that by the degree condition, for every $J$ there are at most $d$ neighbors which are less or equals to $J$ according to $\prec$.

We will construct the sequence $\nu_{J}$ for every $J$ by transfinite recursion. Fix $J\in \cG$ and suppose that we already defined $(\nu_{J'})_{\{J'\in \cG:J'\prec J\}}$.
Since there are at most $d$ neighbors of $J$ which precede $J$ under $\prec$, it follows that 
there are at most $\sum_{r'=0}^{r-1} \binom {d}{r'}$ cliques of size at most $r$ that contain $J$ and are contained in $\{J'\in \cG:J'\preceq J\}$. Hence, there are at most $R\sum_{r'=0}^{r-1} \binom {d}{r'}$ enticements whose support is contained in $\{J'\in \cG:J'\preceq J\}$ and contains $J$. 

Each of the enticements imposes a linear condition on $\nu_J$ of the form $a_J\nu_J + a\nin [0,C]$ for $|a_J|\ge \varepsilon$, which fails to occur on an interval of length at most $C/\varepsilon$. 
Consequently, if $C_1 > \varepsilon^{-1}CR\sum_{r'=0}^{r-1} \binom {d}{r'}$, then there is a possible value of $\nu_J$ that satisfies these enticements.
\end{proof} 
\begin{remark}
Although we used transfinite recursion, when we use this lemma the ordinal will not exceed $n\omega$.
\end{remark}
\begin{proof}[Proof of Lemma \emph{\ref{lem:graph is constructible}}]
We will use the fact that no $n$ intervals in $\cG_f$ have a common intersection. 
Denote by $A^1$ the set of $J\in \cG_f$ that are minimal with respect to interval inclusion. 
Note that $\deg_{\cG_f}(J)\le 2n-4$ for every $J\in A^1$, because
\begin{itemize}
	\item 
	by the minimality condition,
	every neighbor must contain one of $J$'s endpoints, and 
	\item 
	by the intersection condition,	at most $n-2$ intervals can contain every endpoint. 
\end{itemize}
Let $\cG^1 := \cG_f\setminus A^1\subseteq \cG^0 := \cG$.
Let $A^2\subseteq \cG^1$ be the set of all intervals that are minimal with respect to interval inclusion in $\cG^1$. As above, for every $J\in A^2$ we have $\deg_{\cG^1}(J)\le 2n-4$. 

Continue this process and construct a decreasing sequence of subgraphs $\cG^0\supseteq \cG^1\supseteq \cG^2\supseteq \cdots$ such that for every $m$ and every $J\in \cG^m\setminus \cG^{m+1}$ we have $\deg_{\cG_m}(J)\le 2n-4$. 
We have $\cG^n=\emptyset$ because the length of the maximal increasing sequence $J_1\subset J_2\subset\dots \subset J_k$ of elements in $\cG_m$ decreases by $1$ as long as $\cG_m$ is non-empty. 
We have thus constructed the desired filtration.
\end{proof}


\subsubsection{Construction of Enticements} 
\label{ssub:construction_of_enticements}
Throughout this section we consider a tuple of variables $(\nu_J)_{J\in \cG_f}$ and the independent shift sequence $\brho$ with parameters $C,(\nu_J)_{J\in \cG_f}$.
Recall that $\rho_{l}|_J = l(n-l)C + \nu_J$. 
For every $0\le l\le n$ define a function $\tilde f_{H,l}^\brho:U_l(f)\to \RR$ by
\[\forall t\in U_l(f), ~~\tilde f_{H,l}^\brho(t) := f_{H,l}(t) + \rho_{l}(t).\]

In view of Observation \ref{obs: partial2 comutation}, we can formulate the following claim.
\begin{claim}\label{claim: threat classification}
Let $t_0$ be a vertex of $f^\brho$ that impacts $l_0$, and assume $\partial^2 f_{H, l_0}^\brho(t_0)\neq 0$.
Then there exists $l_{-1}<l_0<l_1\in 
L_{f}(t_0)$, not necessarily adjacent, such that
\begin{align}
\partial^2 f_{H, l_0}^\brho(t_0) 
\label{eq:delta_by_f_01m1}=& \frac{1}{l_0-l_{-1}} \tilde f_{H, l_{-1}}^\brho(t_0) + \frac{1}{l_1-l_{0}}\tilde f^\brho_{H, l_1}(t_0)\\\nonumber& - \frac{l_1-l_{-1}}{(l_1-l_0)(l_0-l_{-1})}\tilde f_{H, l_0}^\brho(t_0)).
\end{align} 
In addition, one of the following holds:
\begin{enumerate}[label=\emph{(\arabic*)}, ref=\arabic*]
\item \label{type:ver 1} For some $l\in\{l_0, l_{-1}, l_1\}$ the time $t_0$ is a non-null vertex of $f$ at $l$.

\item \label{type:ver 2} For some $l_{1/2} \in L_f(t_0), l_0 < l_{1/2} < l_1$ the time $t_0$ is a null vertex of $f^\brho$, at $l_{1/2}$, and moreover,
\begin{align}\label{eq:t_0 by 0, 2, 3}
0
=& \frac{1}{l_{1/2}-l_0} \tilde f_{H, l_0}^\brho(t_0) + \frac{1}{l_{1}-l_{1/2}}\tilde f^\brho_{H, l_1}(t_0) \\&\nonumber- \frac{l_1-l_0}{(l_1-l_{1/2})(l_{1/2}-l_0)}\tilde f_{H, l_{1/2}}^\brho(t_0).
\end{align}

\item \label{type:ver 3} For some $l_{-1/2} \in L_f(t_0), l_{-1} < l_{-1/2} < l_0$ the time $t_0$ is a null vertex of $f^\brho$, at $l_{-1/2}$, and moreover,
\begin{align*}
0
=& \frac{1}{l_0-l_{-1/2}} \tilde f_{H, l_0}^\brho(t_0) + \frac{1}{l_{-1/2}-l_{-1}}\tilde f^\brho_{H, l_{-1}}(t_0)\\& - \frac{l_0-l_{-1}}{(l_0-l_{-1/2})(l_{-1/2}-l_{-1})}\tilde f_{H, l_{-1/2}}^\brho(t_0)).
\end{align*}

\end{enumerate}
\end{claim}
\begin{proof}
Let $t_0$ be a non-null vertex of $f^\brho$ that impacts $l_0$. 
Choose $l_{-1}<l_0 < l_1$ to be the indices in $L_{f^\brho}(t)$ near $l_0$. 
It follows that $\tilde f_{H,l_i}^\brho(t) = f_{H,l_i}^\brho(t)$ for $i= 0, 1, 2$, and using Observation \ref{obs: partial2 comutation}, Eq. \eqref{eq:delta_by_f_01m1} holds.

Note that either $f^\brho$ has a vertex $t_0$ at $l_0$, 
or $f_{H, l_0}^\brho$ is linear near $t_0$. 
If $f^\brho$ has a vertex $t_0$ at $l_0$, then it must be a non-null vertex, as $\partial^2f_{H, l_1}^\brho(t_0)>0$, and 
Case \eqref{type:ver 1} holds. 

Otherwise, $f_{H, l_0}^\brho$ is linear near $t_0$. 
Since the vertex $t_0$ impacts $l_0$, it follows that $t_0$ is a vertex at some $l'\neq l$. 
Assume $l'>l$.
Since $l_1\in L_{f^\brho}(t)$, it follows that $l'\le l_1$. We distinguish three cases.





\textbf{Case $l'=l_1$:}
The vertex $t_0$ of $f^\brho$ at $l_1$ must be a non-null vertex as $l_1\in L_{t_0}(f^\brho)$, and Case \eqref{type:ver 1} holds.

\textbf{Case $l'=l_{1/2}<l_1$:}
Since $l_{1/2}\nin L_{t_0}(f^\brho)$, it follows that $t_0$ is a null vertex of $f^\brho$ at $l_{1/2}$.
Consequently, there is a connected component $J$ of $U_{l_{1/2}}(f^\brho)$ such that $t_0$ is an endpoint of $J$. 

Therefore, $f^\brho_{H, l_{1/2}}(t_0) = \tilde f^\brho_{H, l_{1/2}}(t_0)$, and hence Eq.
\eqref{eq:t_0 by 0, 2, 3} holds.
This yields Case \eqref{type:ver 2}.

The symmetric alternative $l' < l$ (as opposed to the assumption above $l'>l$), yields Cases \eqref{type:ver 1} and \eqref{type:ver 3}. 
\end{proof}

\begin{proof}[Proof of Lemma \emph{\ref{lem: construction of enticements}}]
We will first construct a collection of enticements to ensure each of the Cases \eqref{type:ver 1}, \eqref{type:ver 2}, \eqref{type:ver 3} of Claim \ref{claim: threat classification} fails to give a $\partial^2f^\brho_{H,l_0}(t_0) \in (0,C]$, for every vertex $t_0$ that impacts $l_0$.
We will then verify that it is $(4, 49 \cdot 2^{4n}\cdot n^{12})$-finitary.
To avoid Case \eqref{type:ver 1}, note that for every $J\in U_l(f)$ there are at most $l(n-l)\le n^2$ non-null vertices of $f$ at $l$ in an interval $J\in \pi_0(U_l)$. This is the length of the maximal sequence $E_0 \et E_1 \et\cdots \et E_k$. 

For every $l_{-1}<l_0 < l_1$, every $J_i\in U_{l_i}(f)$, and every vertex $t_0\in J_{-1}\cap J_{0}\cap J_{1}$ of $f$ at $l_{-1},l_0$ or $l_1$, the right-hand side of Eq.
\eqref{eq:delta_by_f_01m1} defines an enticement with support $\{J_{i}:i=-1, 0, 1\text{ and }l_i \neq 0,n\}$, 
\begin{align*}
\varphi((\nu_J)_{J\in \cG_f}):=&
 \frac{ f_{H, l_{-1}}(t_0) + l_{-1}(n-l_{-1})C_1 + \nu_{J_{-1}}}{l_0-l_{-1}} \\&
+\frac{ f_{H, l_1}(t_0)+ l_1(n-l_1)C_1 + \nu_{J_1}}{l_1-l_{0}} \\&
- \frac{l_1-l_{-1}}{(l_1-l_0)(l_0-l_{-1})} (f_{H, l_0}(t_0)+l_0(n-l_0)C_1 + \nu_{J_0}).
\end{align*}

In each set of three nontriviality intervals there are at most $3n^2$ vertices, hence there are at most $3n^2$ such enticements.
Since some of these intervals may be for $l=0,n$, we actually obtain at most $6n^2$ such enticements with support $\cG'\subseteq \cG_f$ with $\#\cG'\le 3$. 

Case \eqref{type:ver 2} is more tricky. 
Fix $l_{-1}<l_0<l_{1/2}<l_1$, $J_i\in U_{l_i}$ for $i=-1, 0, 1/2, 1$, and linear parts $a_it + b_i - l(n-l)C_1$ of $f_{H, l_i}(t)$, where $J_{-1}\cap J_0\cap J_{1/2}\cap J_1\neq \emptyset$.
Since in Case \eqref{type:ver 2} $t_0$ is defined by Eq. \eqref{eq:t_0 by 0, 2, 3}. 
set 
\begin{align}\label{eq:t_0 formula}
t_0 :=& -\frac{\frac{1}{l_{1/2}-l_0}(b_0 + \nu_{J_0})+\frac{1}{l_1-l_{1/2}}(b_1 + \nu_{J_1}) -\frac{l_1-l_0}{(l_1-l_{1/2})(l_{1/2}-l_0)} (b_{1/2} + \nu_{J_{1/2}})}{\frac{1}{l_{1/2}-l_0} a_0 + \frac{1}{l_1-l_{1/2}} a_1 -\frac{l_1-l_0}{(l_1-l_{1/2})(l_{1/2}-l_0)} a_{1/2}}\\
=\nonumber&-\frac{(l_1-l_{1/2})(b_0 + \nu_{J_0})+(l_{1/2}-l_0)(b_1 + \nu_{J_1}) - (l_1-l_0)(b_{1/2} + \nu_{J_{1/2}})}{(l_1-l_{1/2}) a_0 + (l_{1/2}-l_0) a_1 - (l_1-l_0)a_{1/2}},
\end{align}
which is the solution of the equation
\[\frac{a_0t_0 + b_0 + \nu_{J_0}}{l_{1/2}-l_{0}} +\frac{a_1t_0 + b_1 + \nu_{J_1}}{l_1-l_{1/2}}	 - \frac{(l_1-l_0)(a_{1/2}t_0 + b_{1/2} + \nu_{J_{1/2}})}{(l_1-l_{1/2})(l_{1/2}-l_0)} = 0.\] 
We need to consider only the case where $t_0$ is well defined, that is, the denominator in Eq. \eqref{eq:t_0 formula} does not vanish.

Substituting Eq. \eqref{eq:t_0 formula} in Eq. \eqref{eq:delta_by_f_01m1} we see that it is enough to show that $\varphi((\nu_J)_{J\in \cG_f})\nin (0,C]$, where
\begin{align*}\label{eq: the enticement}
\varphi((\nu_J)_{J\in \cG_f}):=&\frac{a_{-1}t_0 + b_{-1} + \nu_{J_{-1}}}{l_0-l_{-1}} + \frac{a_1t_0 + b_1 + \nu_{J_1}}{l_1-l_0} \\&-\frac{(l_1-l_{-1})(a_0t_0 + b_0 + \nu_{J_0})}{(l_1-l_0)(l_0-l_{-1})}
\\=&
\sum_{i=-1,0,1/2,1} c_i(b_i + \nu_{J_i}),
\end{align*}
where 
\begin{align*}
c_{-1} &:= \frac{1}{l_0-l_{-1}},\\
c_{0} &:= 
\frac{a_{1/2}l_1 - a_{-1}l_1 - a_1l_{1/2} + a_{-1}l_{1/2} + a_1l_{-1} - a_{1/2}l_{-1}}{(l_{1} - l_{1/2})a_0 + (l_{1/2}-l_0)a_{1} - (l_{1}-l_0)  a_{1/2})(l_1-l_0)},\\
c_{1/2} &:= 
\frac{-a_1l_0 + a_{-1}l_0 + a_0l_1 - a_{-1}l_1 - a_0l_{-1} + a_1l_{-1}
}{(l_{1} - l_{1/2})a_0 + (l_{1/2}-l_0)a_{1} - (l_{1}-l_0)  a_{1/2})(l_1-l_0)},\\
c_{1} &:= \frac{
\begin{pmatrix*}[l]
-2a_{1}l_0^2 + a_{1/2}l_0^2 + a_{-1}l_0^2 + 2a_0l_0l_{1} - a_{1/2}l_0l_{1} - a_{-1}l_0l_{1} - a_0l_0l_{1/2}
\\
+ 2a_{1}l_0l_{1/2} - a_{-1}l_0l_{1/2} - a_0l_{1}l_{1/2} + a_{-1}l_{1}l_{1/2} 
- a_0l_0l_{-1} + 2a_{1}l_0l_{-1}
\\
 - a_{1/2}l_0l_{-1} - a_0l_{1}l_{-1} + a_{1/2}l_{1}l_{-1} 
+ 2a_0l_{1/2}l_{-1} - 2a_{1}l_{1/2}l_{-1}
\end{pmatrix*}
}{
((l_{1} - l_{1/2})a_0 + (l_{1/2}-l_0)a_{1} - (l_{1}-l_0)  a_{1/2})  (l_{1}-l_0)  (l_0-l_{-1})}.
\end{align*}
The tricky part is that some of the $c_i$-s may vanish, and hence the support of $\varphi$ might not contain all $J_{-1}, J_0, J_{1/2}, J_1$. 
We will get that $\varphi$ is an enticement with support $\{J_i:c_i\neq 0\text{ and }l_i\neq 0,n\}$.

Note that $\varphi$ is not an enticement if $\{J_i:c_i\neq 0\text{ and }l_i\neq 0,n\}=\emptyset$.
If this happens, then $c_i(b_i + \nu_{J_i})=0$ for every $i=-1, 0, 1/2, 1$. Hence, $\varphi = 0$ and will not yield a contradiction to $f$ being $C$-significant. 

Consequently, if $\{J_i:c_i\neq 0\text{ and }l_i\nin \{0,n\}\}\neq \emptyset$ is non-empty, then $\varphi$ is an enticement on it. 
We will see that there are only finitely many possibilities for $c_i$ and hence we may choose $\varepsilon$ so that $\varepsilon<|c_i|$ whenever $c_i\ne 0$.

We now count the enticements that we constructed in Case \eqref{type:ver 2} given a support $\cG'$.

Each $a_i$ is the slope of $f_{H,{l_i}}$ at some $t\in U_{l_i}(f)$ and hence $a_i\in \{\eta_E:E\subseteq \Eall\}$, where $i=-1, 0, 1/2, 1$. Consequently, $a_i$ has at most $2^n$ possibilities.
The values of $(c_i)_{i=-1, 0, 1/2, 1}$ depend only on $(l_{i})_{i=-1, 0, 1/2, 1}$ and $(a_i)_{i=-1, 0, 1/2, 1}$, which have at most $n^4$ and $(2^n)^4$ possibilities, respectively. Consequently, the tuple $(c_i)_{i=-1, 0, 1/2, 1}$ has at most $n^4\cdot 2^{4n}$ possibilities.

For every $i=-1, 0, 1/2, 1$ the value $b_i$ has at most $n^2$ possibilities, given $J_i$. For every set $\emptyset \ne \cG'\subseteq \cG_f$ of size $r$ we need to choose the role of each $J\in \cG'$ and there at most $24$ possibilities. 
Altogether we have constructed at most $24 \cdot 2^{4n}\cdot n^{12}$ enticements with support $\cG'$. 

Case \eqref{type:ver 3} is analogous to Case \eqref{type:ver 2} and yields $24 \cdot 2^{4n}\cdot n^{12}$ enticements for every $\emptyset\ne \cG'\subseteq \cG_F$. 
We thus obtain at most $2\cdot 24 \cdot 2^{4n}\cdot n^{12} + 2n^6 \le 49 \cdot 2^{4n}\cdot n^{12}$ enticements. 

Since we can consider only enticements on sets of intersecting intervals, the enticements are compatible with the graph.
By Lemma \ref{lem:enticements are solvable}, there is a solution $(\nu_J)_{J\in \cG_f}$ to this collection of enticements.
Claim \ref{claim: threat classification} ensures that if $(\nu_J)_{J\in \cG_f}$ satisfy these enticements, then $f^\brho$ is $C$-significant.
\end{proof}


\subsection{Properties of Significant Templates} 
\label{sub:properties_of_significant_templates}
We provide useful properties of significant templates, which in particular will allow us to prove Lemma \ref{lem:separated}.
Recall that the implicit constants in the notations $O(-), \Omega(-),$ and $\Theta(-)$ depends only on $\Eall$ (see Subsection \ref{sub:notations}).
\begin{obs}
For every $g$-template $f$, and $0\le l \le n$, the function $f_{H,l}$ has at most $2^n$ possible slopes depending on $\Eall$ and hence $f_{H,l}$ is $O(1)$-Lipschitz. Consequently, so is the function $t\mapsto\partial^2f_{H,l}$.
\end{obs}
\begin{lem}\label{lem:long intervals}
For every $C$-significant $g$-template $f$ on an interval $I$ and every $0\le l \le n$, every connected component $J$ of $U_l(f)$ that has no common endpoint with $I$ is of length $\Omega(C)$.
Moreover, there exists $J'\subset J$ such that $\partial^2f_{H,l}(t)|_{J'} \ge C$ and the total length of $J\setminus J'$ is $O(C)$. 
\end{lem}
\begin{proof}
The function $t\mapsto\partial^2f_{H,l}(t)$ is $O(1)$-Lipschitz, is positive in $J$, and vanishes at the boundary of $J$.
Since $f$ is $C$-significant, the maximum in $J$ must be at least $C$, and consequently, $J$ is of length $\Omega(C)$.

Moreover, the set $J' := \{t\in J:\partial^2f_{H,l}(t)\ge C\}$ must be an interval since if it is not connected, then the minimum of $\partial^2f_{H,l}(t)$ on $(\conv J')\setminus J'$ is less than $C$ and is attained at a vertex that impacts $l$, which contradicts the $C$-significance of $f$. 

Since the derivative of $\partial^2f_{H,l}(t)$ at $J\setminus J'$ is $\Theta(1)$ we deduce the desired result.
\end{proof}
\begin{lem} \label{lem: not many vertices}
For every $C$-significant $g$-template $f$ and every interval $[a,b]\subset \RR$, there are most $2n^3 + O((b - a) / C))$ vertices of $f$ in $[a,b]$.
\end{lem}
\begin{proof}
We will prove that for every $l$ there are at most $2n^2 + O((a-b)/C)$ vertices of $f$ at $l$ in $[a,b]$. Indeed, by Lemma \ref{lem:long intervals}, the interval $I$ intersects at most $(2 + O((a-b)/C))$ connected components of $U_l(f)$ and each of them contains at most $l\cdot (n-l) + 2\le n^2$ vertices. The result follows.
\end{proof}

\begin{lem}\label{lem: construction of enticements2}
There exists $N = N(\Eall)$ that satisfies the following property.
For every $5C_1$-significant $g$-template $f$,
there exists a $(6, N)$-finitary set of enticements on $\cG_f$ such that for any solution $(\nu_J)_{J\in \cG_f}$ of it the $g$-template 
$f^\brho$ is $C$-separated, where $\brho$ is the independent shift sequence with parameters $C_1, (\nu_J)_{J\in \cG_f}$. 
\end{lem}
\begin{proof}
Consider a tuple of variables $(\nu_J)_{J\in \cG_f}$ and consider the independent shift sequence $\brho$ with parameters $C,(\nu_J)_{J\in \cG_f}$.
Each null vertex of $f^\brho$ at $l_0$ is an endpoint of an interval $J_0\in \pi_0(U_{l_0}(f^\brho))$. 
As in the proof of Lemma \ref{lem: construction of enticements}, $t_0$ can be defined by an equation 
\begin{align}
\frac{1}{l_0-l_{-1}} \tilde f_{H, l_{-1}}^{\brho} (t_0)+ \frac{1}{l_1-l_0}\tilde f_{H, l_1}^{\brho}(t_0) - \frac{l_1-l_{-1}}{(l_1-l_0)(l_0-l_{-1})}\tilde f_{H, l_0}^{\brho}(t_0) = 0.
\end{align}
As in the proof of Lemma \ref{lem:significant}, the value $f_{l_i}(t_0)$ is one of at most $n^2$ possible linear expression of the form $a_it_0 + b_0 - l_i(n-l_i)C_1$, and hence 
\[t_0 = -\frac{\frac{1}{l_0-l_{-1}}(b_{-1} + \nu_{-1}) + \frac{1}{l_1-l_0}(b_1 + \nu_1) -\frac{l_1-l_{-1}}{(l_1-l_0)(l_0-l_{-1})} (b_0 + \nu_0)}{\frac{1}{l_0-l_{-1}} a_{-1} + \frac{1}{l_1-l_0}a_1 - \frac{l_1-l_{-1}}{(l_1-l_0)(l_0-l_{-1})}a_0}.\]
If $t_0\in J_0\cap J_{-1}\cap J_1$ is a null vertex of $f^\brho$ at $l_0$, then it must lie in the closure $\overline {J_0\setminus J_0'}$, which consists of at most two intervals of length $O(C_1)$, and hence there are at most $O(n^3)$ non-null vertices of $f$ that may be close to $t_0$. For every such vertex $t'$, the condition that $t_0$ lies at distance less than $O(C)$ from $t'$ is an enticement satisfaction condition, with a support consisting of at most $3$ intervals (not exactly $3$ as $l_{-1}$ or $l_1$ may be $0$ or $n$). 

Similarly, the condition that two such null vertices lie at distance less than $C$ is an enticement satisfaction condition, with support contained in the set of the intervals involved. There may be at most $6$ intervals, three defining each null vertex. 

As in Lemma \ref{lem: construction of enticements}, we obtain a finitary set of enticements.
\end{proof}

\begin{proof}[Proof of Lemma \emph{\ref{lem:separated}}]
Let $f$ be a $g$-template. By Lemma \ref{lem:significant}, we may assume that $f$ is $5C_1$-significant.
Let $Z$ be the set of all enticements on $\cG$ constructed in Lemmas \ref{lem: construction of enticements} and \ref{lem: construction of enticements2}. 
It is $(6, N)$-finitary with $N$ and $\varepsilon$ depending only on $\Eall$. 
By Lemma \ref{lem:enticements are solvable} we can solve them with $(\nu_J)_{J\in \cG_f}\subseteq [0,C_1]$, provided that $C_1$ is large enough. 

The solution produces an independent shift sequence $\brho$ such that $f^\brho$ is $C$-significant and $C$-separated by Lemmas \ref{lem: construction of enticements} and \ref{lem: construction of enticements2}, and is $3C_1$-close to $f$ by Lemma \ref{lem: shift is equiv}. 
\end{proof}

\begin{lem}\label{lem: stability or sep+sig}
Let $f$ be a $C$-separated and $C$-significant $g$-template on an interval $I$, such that $f_{H,l}(t_0) = 0$ for every $0\le l\le n$ and every endpoint $t_0$ of $I$. Then there is and $\alpha$, $0<\alpha<1$, which depends only on $\Eall$, such that there is a one-to-one correspondence between nontriviality intervals of $f$ and nontriviality intervals of $f^\brho$, where $\brho$ is the independent shift sequence with parameters $\alpha C, (\nu_J)_{J\in \cG_f}\in [0,\alpha C]^{\cG_f}$. 

Moreover, $f^\brho$ is $C/2$-significant and $C/2$-separated in a uniform sense in $\nu_J$.
That is, for every pair $(\nu_J)_{J\in \cG_f}, (\nu_J')_{J\in \cG_f}\in [0,\alpha C]^{\cG_f}$ and let $\brho, \brho'$ be the corresponding independent shift sequences. Then for every null vertex $t_0$ of $f$ and every vertex $t_1$ of $f$, the corresponding vertices of $f^\brho$, namely $t_0^\brho, t_1^{\brho'}$, satisfy $|t_0^\brho - t_1^{\brho'}|\ge C/2$. 
\end{lem}
\begin{proof}
The first assertion follows from the $C$-significance of $f$.
The second assertion follows from the first.
\end{proof}

\subsection{Proof of the Template Existence Theorem \ref{thm:Lattice approx}} 
\label{sub:proof_of_lemma_thm:lattice approx}
Let $\Lambda$ be a lattice, and for every $t\in \RR$ consider the filtration $\HN(g_t\Lambda)$. For every $\Gamma\subseteq \Lambda$ of rank $l$ denote by
$$U_\Gamma = \{t\in \RR:g_t\Gamma\in \HN(g_t\Lambda)\}.$$
For many $\Gamma$ this set is empty, yet for some it is not.

The set $U_\Gamma$ is not necessarily an interval. 

\begin{ex}
Let $\Lambda\in X_2$ be a random rank-$2$ lattice and $\Lambda' = 2\Lambda \bigoplus \tfrac{1}{4}\ZZ\subseteq \RR^3$ be a rank-$3$ lattice. Let $g_t = \exp(t\diag(-1, 1))$ and $g_t' = \exp(t\diag(-1, 1, 0))$ be diagonal flows.
For every $t$ we have $\frac{1}{4}\ZZ e_3\subseteq g_t'\Lambda' = 2(g_t\Lambda)\bigoplus \tfrac{1}{4}\ZZ$. 
Since $g_t$ act ergodicly on $X_2$, it follows that $g_t\Lambda$ equidistributes in $X_2$ for $t\to \infty$ with probability one, and hence $\frac{1}{4}\ZZ e_3$ enters and leaves $\HN(g_t'\Lambda')$ infinitely many times as $t\to \infty$.
\end{ex}

For every $\Gamma$ denote by $\fV_\Gamma:=\pi_0(U_\Gamma)$ the set of connected components of $U_\Gamma$.
\index{Basic intervals}\hypertarget{baj}{}
Denote $\fV:=\bigsqcup_{\Gamma\subseteq \Lambda}\fV_\Gamma$, and call the intervals in $\fV$ \emph{basic intervals}. 
This union is disjoint so that if a set $U$ appears in $\fV_{\Gamma_1}$ and $\fV_{\Gamma_2}$ for two different lattices $\Gamma_1$ and $\Gamma_2$, then $U$ appears twice in $\fV$. 
For every $U\in \fV$ denote by $\Gamma_U$ the lattice $\Gamma$ for which $U\in \fV_\Gamma$. 

Let $\varepsilon>0$ and let $C_0>0$ be the corresponding constant provided by Theorem \ref{thm:lin path}. We will prove Theorem \ref{thm:Lattice approx} for a tuple $(\varepsilon, C''')$, where $C'''$ will be defined later. It will be sufficient to prove the theorem for $\varepsilon$ small enough. 
Fix $U\in \fV$.
By Theorem \ref{thm:lin path}, there exists a set of disjoint intervals 
\[\fU_U = \{(a_{i,U}, b_{i,U}):i=1,\dots,r_U\},\]
in $U$ which cover $U$ except for a set of volume $C_0$, and
a sequence of multisets
\[E_{1, U}\et E_{2, U}\et \cdots\et E_{r_U, U}\]
such that 
\[a_{1, U}<b_{1, U}<a_{2, U}<b_{2, U}<\dots<a_{r_U, U} < b_{r_U, U},\]
and $g_t\Gamma_U\in U_\varepsilon\left(\gr_{E_{i, U}}^g\right)$ for every $t\in (a_{i,U}, b_{i,U})$.

Let $U'\subseteq U$ be the interval shorter by $C$ from the right, for some $C>0$ to be determined later.
Define
\begin{align*}
\fU'_U = \{I\cap U':I\in \fU_{U},I\cap U'\neq \emptyset\}
:=\{(a'_{i,U}, b'_{i,U}):i=1,\dots,r_U'\},
\end{align*}
and denote the corresponding multisets in $\cI_l$ by
\[E_{1, U}'\et E_{2, U}'\et \cdots\et E_{r_U', U}'.\] 
Denote by $a_0'$ and $a_{r_U'+1}'$ the endpoints of $U'$.

\begin{de}
\index{Complementary interval}\hypertarget{bba}{}
For every $U\in \fV$ such that $U'\neq \emptyset$, the intervals $[b'_{i, U}, a'_{i+1, U}]$ for $0\le i\le r_{U}'$ are called \emph{complementary intervals constructed from $U$}. 
Let \[S:=\bigcup_{U\in \fV}\bigcup_{i=0}^{r_U'}[b_{i, U}', a_{i+1,U}']\subseteq \RR\] be the union of the complementary intervals. 
\end{de}
\begin{lem}\label{lem:S has long connected components}
If $C>(n-1)C_0$, then 
every connected component of $S$ is of length less than $C$, and hence does not intersect two different $U'_1,U'_2$ corresponding to basic intervals $U_1, U_2 \in \fV$ such that $\rk\Gamma_{U_1} = \rk\Gamma_{U_2}$.
\end{lem}
\begin{proof}
Let $S_0$ be a connected component of $S$. 
We will show that $S_0$ does not contain complementary intervals of two different intervals $U_1,U_2\in \fV$ with $\rk\Gamma_{U_1} = \rk\Gamma_{U_2}$.

Assume, by contradiction, that $I_1 = [a'_1, b'_1], I_2 = [a'_2, b'_2]$ with $b'_1<a'_2$ are complementary intervals contained in $S_0$ constructed from basic intervals $U_1, U_2$ respectively with lattices $\Gamma_{U_1}, \Gamma_{U_2}$ of the same rank. 
Assume that $I_1, I_2$ satisfy the following minimality property: there is no such tuple of complementary intervals both intersecting $(b'_1, a'_2)$. 

All complementary intervals that intersect $(b'_1, a'_2)$ consist of at most $n-1$ basic intervals, corresponding to lattices of different ranks. By Theorem \ref{thm:lin path}, every basic interval has complementary intervals of total length at most $C_0$. Hence, $(b'_1, a'_2)$ is covered by complementary intervals of total length at most $(n-1)C_0$. Consequently, $a'_2-b'_1\le (n-1)C_0$.
On the other hand, 
the distance between $U'_{1}, U'_{2}$ is at least $C$, and hence $a'_2-b'_1>C$. 
This contradicts the assumption $C>(n-1)C_0$.
\end{proof}

Let us define for every $U\in \fV$ an interval $U''\subseteq U'$. If the left endpoint of $U'$ is not in $S$, set $U'' := U'$. Otherwise, the left endpoint of $U'$ is in $S$; let $S_0$ be the connected component of $S$ that contains that point and set $U'' := U'\setminus S_0$. 

We next define for every $l$ a category flow $\tilde f_{E,l}:\RR\to \cI^*_l$.
\begin{claim}\label{claim:tau existence}
There exists an increasing function $\tau:\RR\to \RR$ such that 
\begin{itemize}
\item
For every $t\in \RR$ we have $\tau(t)\nin S$.
\item 
For every $t\in \RR$ we have $t-C \le \tau(t)\le t$.
\item 
If $t\in U''$, then $\tau(t)\in U''$ for every $U\in \fV$. 
\end{itemize}
\end{claim}
\begin{proof}
Write $S$ as $S=\bigsqcup_{i\in Z} [a_i, b_i]$, with $Z = [i_0, i_1]\cap \ZZ \subseteq \ZZ$ an index set of consecutive numbers such that $a_i<b_i<a_{i+1}$ for every $i_0\le i<i_1$ such that $i+1\in Z$. Note that $i_0$ may be $-\infty$ and $i_1$ may be $\infty$. 
If $i_0\neq \infty$  set $b_{i_0-1} = -\infty$ and if $i_1\neq \infty$  set $a_{i_1+1} = \infty$.

Fix $i\in Z$. By Lemma \ref{lem:S has long connected components}, $b_i - a_i < C$. For every basic interval $U$, by the definition of $U''$, if $U''\cap [a_i,b_i]\neq \emptyset$, then $a_i\in U''$. 
Consequently, we can choose $b_{i-1} < a_i' < a_i$ such that $b_i- a_i'<C$ and $a_i'\in U''$ for every $U''$ which intersects $[a_i,b_i]$. If $i_1\neq \infty$, set $a_{i_1+1}' = \infty$. 

Note that \[\bigcup_{i\in [i_0, i_1+1]} (b_{i-1}, a_{i}']\cup  \bigcup_{i\in Z} (a_{i}', b_{i}] = \RR.\]

The function 
\[\tau(t) = \begin{cases}
t, & \text{if }t\in (b_{i-1}, a_{i}'] \text{ for }i\in [i_0, i_1+1],\\
a_i', & \text{if }t\in (a_{i}', b_i] \text{ for }i\in Z,
\end{cases}\]
enjoys the desired properties.



\end{proof}
Fix a function $\tau$ as in Claim \ref{claim:tau existence}. 
Define the category flow $\tilde f_{E,l}:\RR\to \cI^*_l$ as follows:
For every $t\in \RR$, let $U$ be the unique basic interval with $\rk \Gamma_U = l$ and $\tau(t)\in (a'_{i, U}, b'_{i, U})$ and set $\tilde f'_{E,l}(t):= E_{i,U}'$, where $1\le i\le r_U'$.
Define now \[T_l=\{t\in \RR:\tilde f'_{E,l}\text{ is not locally constant at }t\},\]
and $\tilde f_{E,l}(t) = \tilde f_{E,l}'(t)$
at every point $t\nin T_l$, and undefined for $t\in T_l$.
For every $t\in T_l$ define the connecting morphisms: 
For every $U\in \fV$ with $\rk \Gamma_U = l$, if $t\in U''$, denote by $i_\pm$ the indices such that $\lim_{s\searrow 0} \tilde f_{E,l}(t\pm s)=E_{i_\pm,U}'$. 
By the monotonicity of $\tau$ we get that $i_- < i+$ and hence $E_{i_-,U}'\et E_{i_+,U}'$. 
Therefore, we may choose the connecting morphism of $\tilde f_{E,l}$ at $t$ to be $E_{i_-,U}'\xrightarrow\eteq E_{i_+,U}'$
For \[t\in T_l\setminus\bigcup_{\substack{U\in \fV\\\rk \Gamma_U = l}}U'',\]
we define the connecting morphism to be the null arrow.

By Lemma \ref{lem: E_i to V_i is functorial} below, if $\varepsilon>0$ is small enough, then 
 $\tilde f_{E,\bullet}(t)$ forms a direction filtration. 

For every $U\in \fV$, fix $t_U\in U''$, and define $\tilde f_{H,l}|_{U''}$ as follows: set $\tilde f_{H,l}(t_U) := \log \cov g_{t_U}\Gamma_U$, and extrapolate $\tilde f_{H,l}$ to $U''$ by Eq. \eqref{eq:diff eq}. 

Note that for every interval $(a'_{i, U}, b'_{i, U})$, if $a'_{i, U} - b'_{i, U} > C$, then $\tilde f_{E,l}|_{(a'_{i, U}+C, b'_{i, U})} \equiv E_{i,U}'$. 
Corollary \ref{cor:blade behavior} implies that the blade $\bl\, \Gamma$ satisfies 
\begin{align}\label{eq:tilde f small close to hn}
\left|\tilde f_{H,l}(t) - \log\|g_t\bl\, \Gamma\| \right|= O(C),
\end{align}
for every $t\in U''_\Gamma$. 

Let $C''>0$, to be determined later.
For every $t\in \RR$ define $f_{H,\bullet}(t)$ to be the lower convex hull of the function
\begin{align}\label{eq: defining f to lat}
l\mapsto &\tilde f_{H,l}(t) + l(n-l)C'', \\\nonumber&\text{ for every }0\le l\le n, \text{ for which }f_{H,l}(t)\text{ is defined.}
\end{align}

For every $U\in \fV$ with $l_0:=\rk \Gamma_U$ and every endpoint $t_0$ of $U''$ which corrsponds to an endpoint $t_1$ of $U$ we have $|t_0-t_1| \le C$. 
Since
\[\partial^2 \HN(g_t\Lambda)_{H,l_0}(t_1) = 0,\]
the Lipschitz property of $t\mapsto \HN(g_t\Lambda)_{H,l}(t)$ implies that
\[\partial^2 \HN(g_t\Lambda)_{H,l_0}(t_0) = O(C).\]
Consequently,
\[\lim_{U''\ni t\to t_0}\partial^2 \tilde f_{H,l_0}(t) = O(C).\]
Hence if $C''=O(C)$ is large enough, then 
\[\lim_{U''\ni t\to t_0}\partial^2 (\tilde f_{H,l_0}(t)-l(n-l)C'') < 0.\]
Consequently, with this choice of $C''$, near $t_0$ the index $l$ does not correspond to an extreme point of the lower convex hull of the function in Eq. \eqref{eq: defining f to lat}, hence removing it does not cause discontinuity of the lower convex hull. 

From Eq. \eqref{eq:tilde f small close to hn} we have 
$\tilde f_{H, l}(t) + l(n-l)C'' \ge \HN(g_t\Lambda)_{H,l}(t)$
 for every $t\in U_l(f)$, provided that $C''=O(C)$ is sufficiently large compared to $C$. By the definition of $f_{H,\bullet}$, we get that 
$f_{H, l}(t)\ge \HN(g_t\Lambda)_{H,l}(t)$.

By Eq. \eqref{eq:tilde f small close to hn}, for every $t\in U''$ we have 

\begin{align}\label{eq: f is close to hn}
f_{H, l}(t)\le \HN(g_t\Lambda)_{H,l}(t) + O(C'').
\end{align}
Due to the Lipschitz continuity of both functions, we get that Eq. \eqref{eq: f is close to hn} holds also for every  $t\in U$. 
Consequently, 
\begin{align}\label{}
f_{H, l}(t) \ge \HN(g_t\Lambda)_{H,l}(t)^{\brho}, 
\end{align}
where $\brho = (l(n-l)C''')_{l=0}^n$
for some $C''' = O(C'')$, and hence Eq. \eqref{eq: f is close to hn} holds for every $t\in \RR$. 

Define \[U_l(f):=\{t\in \RR: \partial^2f_{H,l}(t)>0\},\] and 
\[f_{E,l}(t):=
\begin{cases}
\tilde f_{E,l}(t),&\text{if }t\in U_l(f),\\
*,&\text{if }t\nin U_l(f).
\end{cases}
\]

The $g$-template $f$ we defined satisfies the first condition of $(\varepsilon, C''')$-matching (Definition \ref{de:f Lambda matching template}) for some constant, since 
\eqref{eq: f is close to hn} holds for every $t\in \RR$.

To show that the second condition is also satisfied, let $[a,b]$ be an interval on which $f_{E,l}(t) = E$. Then $[a,b]\subseteq U''$ for some $U \in \fV$ with $\rk \Gamma_U=l$. 
The definition of $\tilde f_{E,l}$ implies that $\tau(t)\in (a'_{i,U}, b'_{i,U})$,  where $E_{i,U}' = E$. It follows that $a'_{i,U}\le a$ and $b'_{i,U}\ge b-C'$, and hence 
$\spa_\RR(\Gamma_U)\in U_\varepsilon(\gr^g_E)$ for every $t\in [a+C', b-C']$. 


\section{Results on the Grassmanian} 
\label{sec:geometric_theorems}
In this section we will prove Theorem \ref{thm:lin path}, which describes the non-recurrent nature of the $g_t$-action on $\gr$, and Lemma \ref{lem: E_i to V_i is functorial}, which asserts a functorial behavior of the following correspondence with respect to inclusion: a space $V\in U_\varepsilon(\gr_E^g)$ corresponds to $E\in \cI_l$.
We used both results in Section \ref{sec:approximation_of_templates}, and we will use them in the following sections as well.
\subsection{Proof of Theorem \ref{thm:lin path}} 
\label{sub:grasmanian_theorems}
It follows from Theorem \ref{thm:lin path} that the $g_t$-action on $\gr$ is non-recurrent. A simpler proof for the non-recurrence asserts that $g_t\acts\gr$ is a gradient flow with respect to some function $\rho:\gr\to \RR$. We will not follow this direction, as it is harder to analyze how the flow connects the invariant points. 

Instead, we will prove a monotonicity result in a more direct way. First, in case $l=1$, for every $v\in \RR^n$ we want to quantify the fact that along the trajectory $\{g_tv:t\in \RR\}$, the coordinates with larger index increase faster. 

We will exhibit quadratic forms $Q_s:\RR^n\to \RR$ 
that separate eigenspaces of smaller eigenvalue and eigenspaces of larger eigenvalue.
For every fixed $v\in \RR^n$, the sign of $Q_s(g_tv)$ will be monotone nondecreasing in $t$.
Then for an $l$-dimensional subspace $V\subseteq \RR^n$ we can compute the signature of each $Q_s|_{g_tV}$ and control its location.

Recall that the $\eta_j$-s satisfy $\eta_1\le \eta_2\le \dots \le \eta_n$.
Denote by $k$ the number of different values of $\eta$ in $\Eall$, and denote these values by $\bar\eta_1<\bar\eta_2<\dots<\bar\eta_{k}$. Set $\bar\eta_0 = -\infty$. 
For every $0\le s\le k$ denote by $Q_s^+, Q_s^-, Q_s$ the quadratic forms on $\RR^n$:
\[Q_s^+ := \sum_{\eta_j > \bar\eta_s}x_j^2, \quad Q_s^- := \sum_{\eta_j \le \bar\eta_s}x_j^2, \quad Q_s := Q_s^+ - Q_s^-.\]

An immediate use of this definition is the following monotonicity result.

\begin{lem}\label{lem:monotonicity}
For every $s=0,1,\dots,k,$ and $ v\in \RR^n\setminus \{0\},$ if $Q_s(v) \ge 0$, then $Q_s(g_tv) > 0$ for every $t>0$. 
\end{lem}
\begin{proof}
The result is trivial for $s=0,k$. Otherwise, $Q_s^+(v)$ must be positive, otherwise $Q_s^+(v)\ge Q_s^-(v)\ge 0$ and then $v=0$. This yields
\begin{align*}
Q_s(g_tv) &= 
Q_s^+(g_tv) - Q_s^-(g_tv) 
\\&\ge
\exp(2\bar\eta_{s+1} t)Q_s^+(v) - \exp(2\bar\eta_{s} t)Q_s^-(v) 
\\&>
\exp(2\bar\eta_{s} t)Q_s^+(v) - \exp(2\bar\eta_{s} t)Q_s^-(v)
\\& = 
\exp(2\bar\eta_{s} t)Q_s(v) \ge 0.
\end{align*}

\end{proof}
We will extend the phenomenon expressed by Lemma \ref{lem:monotonicity} to $l$-dimensional subspaces.
Let $V\in\gr$. Fix $0\le s\le k$ and denote by $(\sigma^0_s(V), \sigma^+_s(V), \sigma^-_s(V))$ the signature of $Q_s|_{V}$. 

There is a subspace $V_0\subseteq V$ with $\dim V_0 = \sigma^0_s(V) + \sigma^+_s(V)$ such that $Q_s(v) \ge 0 $ for every $v\in V_0$. Lemma \ref{lem:monotonicity} implies that $Q_s(g_tv)>0$ for every $t>0$ and $v\in V_0\setminus \{0\}$, and hence $\sigma^+_s(g_tV) \ge \sigma^0_s(V) + \sigma^+_s(V)$. 
It follows that
$\sigma^+_s(g_tV)$ is monotone nondecreasing and $\sigma^-_s(g_tV)$ is monotone nonincreasing. 
Since $\sigma^+_s(g_tV)\le l$ for every $t$, it follows that $\sigma^0_s(g_tV) \neq 0$ for at most $l$ values of $t$.

Note that $\sigma^+(V)\ge r$ is an open condition on $V$, and hence $\{t\in \RR : \sigma^+_s(g_tV)\ge r\}$ is an open set for every fixed $V$. Similarly, $\{t\in \RR:\sigma^-_s(g_tV) \ge r\}$ is open.

For every $E\in \cI_l$ define
\[U_E := \left\{t\in \RR:\forall s=1,\dots,k. ~~ \substack{\sigma^+_s(g_tV) \,=\,\#\{\eta\in E:\eta>\bar\eta_s\}\\~\sigma^-_s(g_tV) \,=\,\#\{\eta\in E:\eta\le\bar\eta_s\}}\right\}. \]
This is an open set since we can replace the two equalities in its definition with $\ge$ and preserve the set. 
The union $\bigcup_{E\in \cI_l} U_E$ covers all points $t\in \RR$ exactly once, except maybe $lk$ points in which $\sigma_s^0(g_tV)\neq 0$. By the monotonicity condition, each $U_E$ is an interval.
By definition, $Q_s$ in nondegenerate on $g_tV$ for every $t\in U_E$. 

The following lemma completes the proof of Theorem \ref{thm:lin path}.
\begin{lem}
For every $\varepsilon>0$ there exists ${t_0}>0$ such that if $[t-{t_0}, t+{t_0}]\subseteq U_E$, then $g_tV$ lies in an $\varepsilon$-neighborhood $U_\varepsilon(\gr_E^g)$.
\end{lem}
\begin{proof}
Fix $\varepsilon>0$ and let ${t_0}>0$ to be determined later. 
Assume w.l.o.g. that $t=0$. 
We will construct inductively a filtration 
\begin{align}\label{eq:squence of definite -}
\{0\}=V_0^-\subseteq V_1^-\subseteq\dots \subseteq V_k^- = V
\end{align}
such that $Q_s$ is negative definite on $g_{t_0}V^-_s$ and $\dim V^-_s = \sigma^-_s(V)$. Let $s=0,...,k-1$ and assume we already defined $V^-_s$. Since $Q_s$ is negative definite on $g_{t_0}V^-_s$ and $Q_{s+1}\le Q_s$ pointwise, it follows that $Q_{s+1}$ is negative definite on $g_{t_0}V^-_s$. 
We will extend the negative definiteness of $Q_{s+1}$ from $g_{t_0}V_s^-$ to $g_{t_0} V^-_{s+1}$ of dimension $\sigma^-_s(V) = \sigma^-_s(g_{t_0}V)$. 
The orthogonal space $W$ to $g_{t_0}V^-_s$ in $g_{t_0}V$ with respect to $Q_{s+1}$ is such that $Q_{s+1}|_{W}$ has signature $(0, \sigma_{s+1}^- -\sigma_{s}^-, \sigma_{s+1}^+)$. Hence, there is $W_0\subseteq W$ of dimension $\sigma_{s+1}^- -\sigma_{s}^-$ such that $Q_{s+1}$ is negative definite on $W_0$. It follows that $Q_{s+1}$ is negative definite on $g_{t_0}V^-_s \oplus W_0$, and hence we can define $V_{s+1}^-:=V_{s}^- \oplus g_{-{t_0}}W_0$. 

Similarly we can define 
\begin{align*}
0=V_k^+\subseteq V_{k-1}^+\subseteq\dots \subseteq V_0^+ = V,
\end{align*}
with the property that $Q_s$ is positive definite on $g_{-{t_0}} V_s^-$.

It follows that $V_s^+\cap V_s^- = 0$. Denote $W_s := V_s^-\cap V_{s-1}^+$ for every $s=1,\dots,k$. One can show by induction that $\bigoplus_{s=1}^n W_s = V$. 

Fix $\delta>0$. 
Recall that $V_\eta$ denotes the $\exp(t\eta)$-eigenspace of $g_t$. 
We will show that we can choose ${t_0}$ large enough such that for every $v\in W_{s_0}$ we have $d(v, V_{\bar\eta_{s_0}})\le \delta\|v\|$. Indeed, denote by 
\[v^0\in V_{\bar\eta_{s_0}}, ~~v^+\in \bigoplus _{s=s_0+1}^kV_{\bar\eta_s}, ~~v^-\in \bigoplus _{s=1}^{s_0-1}V_{\bar\eta_s},\] 
the orthogonal projections of $v$ on the corresponding spaces. Since the three spaces are orthogonal and span $\RR^n$, we have $v=v^0+v^++v^-$. 
Since $Q_{s-1}(g_{-t_0}v) \ge 0$, 
\begin{align*}
0&\le\|g_{-{t_0}}v^+\|^2 + \|g_{-{t_0}}v^0\|^2 - \|g_{-{t_0}}v^-\|^2 \\&\le 
\exp(-2{t_0}\bar\eta_{s_0})\|v^+ + v_0\|^2 - \exp(-2{t_0}\bar\eta_{{s_0}-1})\|v^-\|^2\\&\le 
\exp(-2{t_0}\bar\eta_{s_0})\|v\|^2 - \exp(-2{t_0}\bar\eta_{{s_0}-1})\|v^-\|^2.
\end{align*}
Consequently, $\|v^-\| \le \exp(t_0(\bar\eta_{{s_0}-1} - \bar\eta_{{s_0}}))\|v\|$. 
Similarly, $\|v^+\| \le \exp(t_0(\bar\eta_{s_0} - \bar\eta_{s_0+1}))\|v\|$. 
In particular, if \[t_0 = -\left(\log \frac12\delta\right)/\min(\bar\eta_{s} - \bar\eta_{s+1}:s=1,\dots,k-1),\]
then $\|v^+\|, \|v^-\| \le \frac12\delta\|v\|,$ and so $d(v, V_{\bar\eta_{s_0}}) \le \delta\|v\|$.

Finally, for every $\varepsilon>0$ there exists $\delta>0$ such that the following holds. 
For every $E\in \cI_l$ which gives multiplicity $\alpha_\eta$ to $\eta$,
if $V'\in \gr$ has a decomposition $V' = \bigoplus_{\eta\in \Eall} W_\eta'$ 
with $\dim W_\eta' = \alpha_{\eta}$ such that for every $v\in W_\eta'$ one has 
$d(v,V_\eta)\le \delta\|v\|$,
then $V'\in U_\varepsilon(\gr^g_E)$. Indeed, the projection of $W_\eta'$ to $V_\eta$ is $\delta$ close to $W_\eta'$. Hence, the direct sum of the projections is close to $V'$. 

We deduce that $V\in U_\varepsilon(\gr^g_E)$, as desired.
\end{proof}

\subsection{Functoriality of Directions} 
 \label{sub:functoriality_of_directions} 
We end the section by proving that directions are functorial, which was used in the proof of Theorem \ref{thm:Lattice approx}, and will be used in the sequel. 
\begin{lem}\label{lem: E_i to V_i is functorial}
There exists $\varepsilon>0$ such that if 
\begin{enumerate}[label=\emph{\alph*)}, ref=(\alph*)]
\item
$l_1<l_2$,
\item
$E_i\in \cI_{l_i}\text{ and } V_i\in U_\varepsilon(\grl{l_i}^g_{E_i})\text{ for }i=1,2,$ and
\item
 $V_1\subset V_2$,
\end{enumerate}
then $E_1\subset E_2$.
\end{lem}
\begin{proof}
Assume that to the contrary, $E_1\not\subset E_2$. 
Denote 
\[X:=\{(V_1, V_2)\in \grl{l_1}\times \grl{l_2}:V_1\subset V_2\}.\]
Note that $X$ is closed.
By analyzing the eigenvalues of $g_t$, for every $V_1\in \grl{l_1}^g_{E_1}$ and $V_2\in \grl{l_2}^g_{E_2}$ we see that $V_1\not\subset V_2$ .
That is, 
\[X\cap \left(\grl{l_1}^g_{E_1}\times\grl{l_2}^g_{E_2}\right)=\emptyset.\] 
Consequently, using the metric 
\[d_\Pi((V_1,V_2), (V_1', V_2')) = \max(d_{\grl{l_1}} (V_1, V_1'), d_{\grl{l_2}} (V_2, V_2')),\]
we obtain the desired result for $\varepsilon = d_{\Pi}(X, \grl{l_1}^g_{E_1}\times\grl{l_2}^g_{E_2})$, which is positive since it is the distance between two disjoint compact sets.
\end{proof}

\section{Flag Perturbation Theorems} 
\label{sec:Flag_perturbation_theorem}
In this section we will prove three results on perturbation of flags, namely Lemmas \ref{lem:the right counting theorem}, \ref{lem:single advance}, and \ref{lem: enlarging the flag}. These results will be used to define Alice's strategy in Section \ref{sec:alice_s_strategy} and the first will be used to analyze Bob's strategy in Section \ref{sec:computation_of_the_dimension}. 

From a more conceptual viewpoint, Lemma \ref{lem:the right counting theorem} is the motivation for the definitions of $\delta(\cdot)$ and $\Delta_0(\cdot)$ (see Definition \ref{de: local entropy}), while Lemmas \ref{lem:single advance} and \ref{lem: enlarging the flag} are the motivation for the combinatorics of category flows.
\subsection{The Counting Lemma} 
\label{sub:the_counting_lemma}
\begin{de}\label{de: H_V_bullet}
\index{H@$H$!$H_V$}\hypertarget{bbb}{}
Let $0\le l\le n$, $E\in \cI_l$ and $V\in \gr^g_{E}$. 
Denote $H_V:=\{h\in H:hV = V\}$. 
Let $E_\bullet$ be a direction filtration and $V_\bullet$ be a flag such that $L(V_\bullet) = L(E_\bullet)$ and $V_l\in \gr^g_{E_l}$ for every $l\in L(V_\bullet)$.
Denote \[H_{V_\bullet}:=\bigcap_{l\in L(V_\bullet)} H_{V_l}.\]
\end{de}
Recall that for every direction filtration $E_\bullet$ we set 
\[\delta(E_\bullet) = \sum_{i=1}^k\sum_{\substack{\eta\in E_{l_i} \\\eta'\in E_{l_i}\setminus E_{l_{i-1}}}}(\eta-\eta')^+,\]
and recall the expansion metric $d_\varphi$ on $H$ as defined in Subsection \ref{sub:Case of interest} and Definition \ref{de: exp metric}.
This subsection analyzes the subgroup $H_{V_\bullet}\subseteq H$ and proves a quantitative claim analogous to the equality 
\[\delta(E_\bullet) = \dim_\funH(H_{V_\bullet}; d_\varphi),\] 
where $L(V_\bullet) = L(E_\bullet)$ and $V_l\in \gr^g_{E_l}$ for every $l\in L(V_\bullet)$.
Specifically, Lemma \ref{lem:counting} gives an upper bound on the number of points one can choose in $B_{d_\varphi}(\Id;1)\cap H_{V_\bullet}$ that  are far from one another.
\begin{lem}\label{lem:counting}
Let $V_\bullet$ be a flag, and let $E_\bullet$ be a direction filtration with $L(E_\bullet) = L(V_\bullet)$ such that 
\begin{align}\label{eq: V compatible}
\forall l\in L(V_\bullet),~V_l\in \gr^g_{E_l}.
\end{align} Let $t>0$.

There exists a set $A\subset B_{d_\varphi}(\Id; 1)$ of size $\Theta(\exp(t\delta(E_\bullet)))$ that satisfies the following conditions:
\begin{enumerate}[label=\emph{(\arabic*)}, ref=\arabic*]
  \item\label{cond: distance} $d_\varphi(h_1, h_2) \ge \exp(-t)$ for every distinct elements $h_1, h_2\in A$. 
  \item\label{cond:lies on _E_i}$hV_l\in \gr^g_{E_l}$ for every $h\in A$, $l\in L(V_\bullet)$, i.e., $A\subseteq H_{V_\bullet}$. 
\end{enumerate}
Moreover, there exists no such $A$ larger than $\Theta(\exp(t\delta(E_\bullet)))$ even if we weaken Condition \emph{\eqref{cond:lies on _E_i}} to
\begin{enumerate}[label={\emph{($\ref{cond:lies on _E_i}'$)}}, ref={$\ref{cond:lies on _E_i}'$}]
\item \label{cond: lies near _E_i} 
For every $h\in A$ there exists $h'\in H$ with $d_\varphi(h,h')<\exp(-t)$ and $h'V_l\in \gr^g_{E_l}$ for every $l\in L(V_\bullet)$. 
\end{enumerate}
\end{lem}
Fix a flag $V_\bullet$ and a direction filtration $E_\bullet$.
\index{Og@$O_g$}\hypertarget{bbc}{}
We will express Condition \eqref{eq: V compatible} in coordinates.
Denote by $O_g$ the group of orthogonal matrices that commute with $g_t$. 
Since the eigenspaces of $g_t$ are orthogonal, it follows that \[O_g = \prod_{V\text{ eigenspace of }g_t}O(V),\]
and hence $O_g$ acts transitively on the set of flags which satisfy \eqref{eq: V compatible}. 
We can thus assume that $V_l = \spa \{e_j:j\in F_l\}$, where $F_l \subseteq \{1,\dots,n\}$. 

\index{eab@$e_{a,b}$}\hypertarget{bbd}{}
\index{ei@$e_{i}$}\hypertarget{bbe}{}
Denote by $(e_i)_{i=1}^n$ the standard basis of $\RR^n$ and by $(e_{a,b})_{a,b=1}^n$ the standard basis of $M_{n\times n}(\RR)$. 
\begin{lem}\label{lem: H_V is the only preserving}
Let $E\in \cI_l$ and $V = \spa(e_j:j\in F)\in \gr^g_{E}$ for some $F\subseteq \{1,2,\dots,n\}$ of size $l$ for which $\{\eta_j:j\in F\} = E$ as multisets, that is, $V\in \gr_E^g$.
The group $H_V$ satisfies 
\begin{align*}
H_V=\,& \Id + \spa(\{e_{a,b}:\eta_a > \eta_b, b\in F\implies a\in F\})\\
     =\,& \{h\in H:hV \in \gr^g_{E}\}.
\end{align*}
\end{lem}
\begin{proof}
The first equality follows from the definition, and similarly 
\[H_V\subseteq \{h\in H:hV \in \gr^g_{E}\}.\]
Assume to the contrary that $hV\in \gr^g_{E}$, but $hV\neq V$. 
Since $V$ and $hV$ are $g_t$-invariant, $g_{-t}hg_tV = hV$. On the other hand, $g_{-t}hg_t\xrightarrow{t\to \infty} \Id$. The result follows.
\end{proof}
By Lemma \ref{lem: H_V is the only preserving}, \begin{align*}
H_{V_\bullet} =& \bigcap_{l\in L(V_\bullet)}H_{V_l}\\
       =& \Id+\spa(\{e_{a,b}:\eta_a > \eta_b,\forall l\in L(V_\bullet),~b\in F_l\implies a\in F_l\})\\
      = & \{h\in H:\forall l\in L(V_\bullet),~hV_l\in \gr\}.
\end{align*}
We want to find a large set $A\subseteq B_{d_\varphi}(\Id;1)\cap H_{V_\bullet}$ whose elements are $\exp (-t)$ far from one another with respect to the semi-metric $d_\varphi$. 
Since $d_\varphi$ is multiplied by $\exp (t)$ after applying $\varphi_t$, we will find a large set $\varphi_tA\subseteq B_{d_\varphi}(\Id; \exp (t))$
whose elements are at distance $1$ from one another.
We will use a volume argument. Let $\mu$ be the bi-invariant Haar measure on $H_{V_\bullet}$, which exists since $H_{V_\bullet}$ is nilpotent.

\begin{lem}\label{lem: C action on measure in H_V}
The $\varphi_t$-action on the measure satisfies \[(\varphi_t)_* \mu = \exp(-t\delta(E_\bullet))\mu.\]
\end{lem}
\begin{proof}
By the uniqueness of the Haar measure, it follows that $(\varphi_t)_* \mu = \alpha_t\mu$ for some $\alpha_t>0$. Since $H_{V_\bullet}$ is a Lie group, $\mu$ is the integral with respect to an invariant volume form. 
Hence, to compute the Radon-Nikodym derivative $d(\varphi_t)_* \mu/d\mu\equiv \alpha_t$, it suffices to compute the action of $(\varphi_t)_*$ locally on $\Omega^N_{\Id}H_{V_\bullet}$, where $N = \dim H_{V_\bullet}$. 
The eigenvalues of $(\varphi_t)_*$ on $T_{\Id}H_{V_\bullet}$ are
\[\exp(t(\eta_a-\eta_b))\text{ such that }\eta_a>\eta_b\text{ and } \forall l\in L(V_\bullet),~ b\in F_l\implies a\in F_l.\]
The eigenvalues of the dual action are the reciprocals of these eigenvalues, and hence the eigenvalue of the desired exterior power is the reciprocal of the product. 
A direct computation shows that the product is $\exp(t\delta(E_\bullet))$, as desired.
\end{proof}
\begin{proof}[Proof of Lemma \emph{\ref{lem:counting}}]
Let $\varphi_tA\subseteq B_{d_\varphi}(\Id; \exp(t))\cap H_{V_\bullet}$ be a maximal set such that the $d_\varphi$-distance between any two of its points is at least one. 
By maximality, every $h\in B_{d_\varphi}(\Id; \exp(t))\cap H_{V_{\bullet}}$ is at $d_\varphi$-distance at most $1$ from $A$.
Consequently, 
\begin{align*}
(\#A)\cdot \mu(B_{d_\varphi}(\Id;1))&\ge \mu(B_{d_\varphi}(\Id; \exp(t))) \\
&= \exp(t\delta(E_\bullet)) (\varphi_t)_*\mu(B_{d_\varphi}(\Id; \exp(t)))\\
&= \exp(t\delta(E_\bullet)) \mu(\varphi_{-t}B_{d_\varphi}(\Id; \exp(t)))\\
&= \exp(t\delta(E_\bullet)) \mu(B_{d_\varphi}(\Id; 1)),
\end{align*}
which yields the desired bound on $\#A$. 

A similar volume argument provides an upper bound on $\#A$ if $A$ satisfies the weaker Condition \eqref{cond: lies near _E_i} instead of (\ref{cond:lies on _E_i}). 
Recall that $d_\varphi^{\alpha_\varphi}$ is a metric (see Subsection \ref{sub:Case of interest}). 
For every point $h\in \varphi_tA$ there is a point $h'\in H_{V_\bullet}\cap B_{d_\varphi}(h;1)$.
It follows that
\[\mu(B_{d_\varphi^{\alpha_\varphi}}(h;2)) \ge \mu(B_{d_\varphi^{\alpha_\varphi}}(h';1) = \mu(B_{d_\varphi}(\Id;1)).\] 
In addition,
\[B_{d_\varphi^{\alpha_\varphi}}(h;2)\cap H_V\subseteq B_{d_\varphi^{\alpha_\varphi}}(h';3)\cap H_V\subseteq B_{d_\varphi^{\alpha_\varphi}}(\Id;\exp(\alpha_\varphi t)+3).\]
Consequently, the integral
\[\iota := \int_{B_{d_\varphi^{\alpha_\varphi}}(\Id; \exp(\alpha_\varphi t) + 3)}\#\left(A\cap B_{d_\varphi^{\alpha_\varphi}}(h_1;2)\right)d\mu(h_1)\]
counts with weight of at least $\mu(B_{d_\varphi^{\alpha_\varphi}}(\Id;1))$ every point in $A$, and hence $\iota\ge \mu(B_{d_\varphi^{\alpha_\varphi}}(\Id;1))\#A$. 
On the other hand, since $d_\varphi^{\alpha_\varphi}$ is a doubling metric, it follows that $\#\left(A\cap B_{d_\varphi^{\alpha_\varphi}}(2;h_1)\right) = O(1)$. 
Thus we get the desired upper bound.
\end{proof}

\subsection{Adjustment of Lemma \ref{lem:counting} to Almost Invariant Spaces} 
\label{sub:adjustment_to_}

\begin{de}\label{de: approx H}
\index{Vto@$V^\to$}\hypertarget{bbf}{}
Let $0\le l\le n$.
For every $V\in \gr$ denote $V^{\to} := \lim_{t\to\infty} g_tV$.
This is a non-continuous mapping $\gr\to \gr^g$.
For every $E\in \cI_l$, denote 
\index{Gr@$\gr$!toE@$\gr_{\to E}$}\hypertarget{bbg}{}
\[\gr_{\to E}:=\left\{V\in \gr:\lim_{t\to \infty}g_tV\in \gr^g_E\right\},\]
and
\index{H@$H$!HVtoE@$H_{V\to E}$}\hypertarget{bbh}{}
\[H_{V\to E}:= \{h\in H:hV\in \gr_{\to E}\}.\]
Let $E_\bullet$ be a direction filtration and 
$V_\bullet$ a flag with $L(V_\bullet)=L(E_\bullet)$ and $V_l\in \gr_{\to E_l}$ for every $l\in L(V_\bullet)$. Denote $H_{V_\bullet\to E_\bullet} := \bigcap_{l\in L(V_\bullet)} H_{V_l\to E_l}$. 
\end{de}

\noindent\textbf{Example \ref{ex: gr inv}, continued.}
To describe explicitly $\grln{2}{4}_{\to E}$ for $E\in \cI_2$, we will use a variation of the concept of Schubert cells. We will not use this representation in the rest of the paper, and instead use the more indirect description
$\grln{2}{4}_{\to E} = H^-\grln{2}{4}^g_{E}$.
For every $M\in M_{4\times 2}(\RR)$ and a subset $F\subseteq \{1,2,3,4\}$, denote by $M_{F}$ the $(\#F)\times 2$ matrix composed of the $F$ rows of $M$. We get
\begin{align*}
\grln{2}{4}_{\to\{-1, 0\}} 
&= \left\{{\rm Im}(M):M=\begin{pmatrix}
	1&0\\
	0&*\\
	0&*\\
	0&0\\
\end{pmatrix},\, \rk M =2 \right\}\\
&=\{\spa(e_1, v):0\neq v\in \spa(e_2, e_3)\},\\
\grln{2}{4}_{\to\{-1, 1\}} 
&= \left\{{\rm Im}(M):M=\begin{pmatrix}
	1&*\\
	0&*\\
	0&*\\
	0&1\\
\end{pmatrix} \right\}\\
&= \{\spa(e_1, v+e_4):v\in \spa(e_2, e_3)\},\\
\grln{2}{4}_{\to\{0,1\}} 
&= \left\{{\rm Im}(M):M=\begin{pmatrix}
	*&*\\
	*&*\\
	*&*\\
	0&1\\
\end{pmatrix}, \,\rk M_{\{2,3,4\}} = 2 \right\}\\
&= \left\{\spa(e_4+u, v + \alpha e_1):\begin{matrix*}[l]
	0\neq v\in \spa(e_2, e_3),\\
	\alpha \in \RR,\text{ and}\\
	u\in \spa(e_1, e_2, e_3)
\end{matrix*}\right\},\\
\grln{2}{4}_{\to\{0, 0\}} 
&= \left\{{\rm Im}(M):M=\begin{pmatrix}
	*&*\\
	1&0\\
	0&1\\
	0&0\\
\end{pmatrix}\right\}\\
&= \{\spa(e_2+\alpha e_1, e_3+\beta e_1):\alpha, \beta\in \RR\},
\end{align*}
where $*$ represents any real number.
Let 
$V_0 = \spa(e_1+e_2 + e_3,e_4)\in \grln{2}{4}_{\to \{0,1\}}$ and 
$V_1 = \spa(e_1 + e_2, e_1 + e_3)\in \grln{2}{4}_{\to \{0,0\}}$. 
Then 
\begin{align*}
H_{V_0\to \{0,1\}} =& \left\{\begin{pmatrix}
1&0&0&0\\
a&1&0&0\\
b&0&1&0\\
c&d&e&1
\end{pmatrix}:a,b,c,d,e\in \RR, (a,b)\neq (-1,-1)\right\},\\
H_{V_1\to \{0,0\}} =& \left\{\begin{pmatrix}
1&0&0&0\\
a&1&0&0\\
b&0&1&0\\
c&-c&-c&1
\end{pmatrix}:a,b,c\in \RR, a+b+1\neq 0\right\}.
\end{align*}
\null\nobreak\hfill\ensuremath{\diamondsuit}
\begin{remark}\label{rem: not group}
The set $H_{V_\bullet\to E_\bullet}$ is not a group. It has the property that, for every $h_0\in H_{V_\bullet\to E_\bullet}$,
\[H_{V_\bullet\to E_\bullet} = H_{h_0V_\bullet\to E_\bullet}h_0.\]
\end{remark}
Our next goal is to prove Lemma \ref{lem: understanding V to E}, which establishes a local bijection between $H_{V\to E}$ and $H_{V^{\to}}$.
\begin{lem}\label{lem: making q0}
Let $E_\bullet$ be a direction filtration and let
$V_\bullet$ be a flag with $L(V_\bullet)=L(E_\bullet)$ and $V_l\in \gr_{\to E_l}$ for every $l\in L(V_\bullet)$. There exists a matrix $q_0\in H^-$ such that $q_0\cdot V_l^\to = V_l$ for all $l\in L(V_\bullet)$. 
Moreover, 
for $E_\bullet$ fixed we can choose the mapping $V_\bullet\mapsto q_0$ so that it is differentiable as a map from a manifold of flags to $H^-$.
\end{lem}
\begin{proof}
We will first show the claim for a single space $V\in \gr_{\to E}$.
Let $W = (V^\to)^\perp$ be the orthogonal complement of $V^{\to}$. Since $V^{\to}$ decomposes as a direct sum of its intersections with eigenspaces $V_\eta$ of $g_t$,
it follows that $W$ is the direct sum of the orthogonal complements of $V^{\to}\cap V_\eta$ in $V_\eta$, and hence is is $g_t$-invariant.

Since $g_tV\xrightarrow{t\to \infty} V^{\to}$ and $V^{\to}\cap W = \{0\}$, it follows that $(g_tV)\cap W = \{0\}$. Since $W$ is $g_t$-invariant we conclude that $V\cap W = \{0\}$. Therefore for every $v\in V^{\to}$ there exists a unique $q_0(v)\in V$ such that $q(v)-v\in W$. 
Extend $q_0$ to a linear map $q_0:\RR^n\to \RR^n$, by $q_0(w) = w$ for every $w\in W$. 
As a linear map $q_0:\RR^n\to \RR^n$, we think of $q_0$ as a matrix.

We claim that $q_t:= g_tq_0g_{-t}\xrightarrow{t\to \infty}\Id$. 
Indeed, $q_t$ is the identify map on $W$. 
We will show that $q_t$ converges to the identify map on $V^\to$. 
For every $v\in V^{\to}$, the element $q_tv = g_tqg_{-t}v$ lies in $g_tV$ and is the unique element such that $q_tv-v\in W$. 
Since $g_tV\xrightarrow{t\to \infty} V^{\to}$, it follows that $g_tq_0g_{-t}v\xrightarrow{t\to \infty} v$. 
Therefore, $g_tq_0g_{-t}\xrightarrow{t\to \infty} \Id$, and hence $q_0\in H^-$. 
Note that for fixed $E$, the mapping $V\mapsto q$ is differentiable as a map from $\gr_{\to E}$ to $H^-$. 

To prove the general case, denote $L(V_\bullet) = \{0=l_0<l_1<\cdots<l_k=n\}$.
Let $q_0'$ be the matrix in $H^-$ such that $q_0'V_{l_{k-1}}^\to=V_{l_{k-1}}$. 
By induction on $n$, applying the lemma in dimension $l_{k-1}<n$ to the flow 
\[g'_t = \frac{g_t|_{V_{l_{k-1}}^\to}}{(\det g_t|_{V_{l_{k-1}}^\to})^{1/l_{k-1}}},\]
the flag $\{0\}=(q_0')^{-1}V_0\subset V_{l_1}\dots\subset (q_0')^{-1}V_{l_{k-1}}$, and the direction filtration $\emptyset=E_0\subset E_{l_1}\subset\dots\subset E_{k-1}$,
there exists $q_0'':V_{l_{k-1}}^\to \to V_{l_{k-1}}^{\to}$ such that $q_0''V_{l_i}^\to = q_0'^{-1}V_{l_i}$ for every $0\le i\le k-1$ and $g_t'q_0''g_{-t}'\xrightarrow{t\to \infty}\Id$. 
We can extend $q_0''$ to a map $q_0'':\RR^n\to \RR^n$ by $q_0''|_{(V_{l_{k-1}}^\to)^\perp} = \Id$. This map satisfies $g_tq_0''g_{-t}\xrightarrow{t\to \infty}\Id$, hence $q_0=q_0'q_0''$ is the desired matrix. 
The result follows.
\end{proof}
\begin{remark}\label{rem: q0 epsilon small}
For every $\varepsilon>0$, if $d_{\gr}(V_l, V_l^\to) < \varepsilon$ for every $l\in L(V_\bullet)$, then, by the differentiability of the construction of $q_0$, we get that $d_{H^-}(q_0,\Id) = O(\varepsilon)$. 
\end{remark}
\begin{lem}\label{lem: H is H}
For every $0\le l\le n, E\in \cI_l$, and $V\in \gr_E^g$, we have 
\[H_{V\to E}=H_V.\]
\end{lem}
\begin{proof}
By the definitions of the sets, $H_V\subseteq H_{V\to E}$. Let $h\in H_{V\to E}$. 
By the definition of $H_{V\to E}$, \[g_thV\xrightarrow{t\to \infty} V^\to \in \gr^g_E.\]
Next, the definition of $H$ implies that $g_{-t}hg_{-t}\xrightarrow{t\to\infty} \Id$, and thanks to the $g_t$-invariance of $V$, 
\begin{align}\label{eq: tends -infty}
g_{-t}hV = g_{-t}hg_{t}V\xrightarrow{t\to \infty} V \in \gr_E^g.
\end{align}
Consequently, $d_{\gr}(g_thV, \gr_E^g)\xrightarrow{t\to \pm\infty}0$, and hence, by Theorem \ref{thm:lin path}, we get that $g_thV\in \gr_E^g$ for every $t$ and in particular $t\mapsto g_thV$ is constant. Eq. \eqref{eq: tends -infty} implies that $hV = V$, hence $h\in H_V$.
\end{proof}
Recall that $H^{-0} = H^-H^0$ is the group of matrices $A\in \SL_n(\RR)$ such that $A_{ij}\neq 0$ only if $\eta_i \le \eta_j$. 
\begin{lem}\label{lem: understanding V to E}
Let $E_\bullet$ be a direction filtration and 
let $V_\bullet$ be a flag with $L(V_\bullet)=L(E_\bullet)$ and $V_l\in \gr_{\to E_l}$ for every $l\in L(V_\bullet)$. Let $q_0\in H^-$ be a map such that $q_0V_l^{\to} = V_l$ for every $l\in L(V_\bullet)$. 
Then there exist open sets $U_1\subseteq H_{V_\bullet\to E_\bullet}$ and $U_2\subseteq H_{V_\bullet^{\to}}$ containing $\Id$, a diffeomorphism $\check \funh:U_1\to U_2$, and a map $\check \funq:U_1\to H^{-0}$, such that 
\[hq_0 = \check \funq(h)\check \funh(h),\]
for every $h\in U_1$. 
Moreover, for every compact set $K\subseteq H$ and every $\varepsilon_0>0$ there exists $\varepsilon_K>0$ such that if $d_{H^-}(q_0,\Id) < \varepsilon_K$, then $U_1$ and $U_2$ contains $H_{V_\bullet\to E_\bullet}\cap K$ and $H_{V_\bullet^{\to}}\cap K$, respectively.
Furthermore, for every $h\in U_1\cap K$, we have $d_H(\check \funh(h), h) <\varepsilon_0$.
\end{lem}
\begin{proof}
The multiplication map
\begin{align*}
\mLU&: H\times H^{-0}\to \SL_n(\RR) 
\\ 
\mLU&:(h,  q^{-0})\mapsto q^{-0}h
\end{align*}
is one-to-one and has an open image.
Set
\begin{align*}
U_1'& := \{h\in H: hq_0\in {\rm Im}(\mLU)\},\\
U_2'& := \{h\in H: hq_0^{-1}\in {\rm Im}(\mLU)\}.
\end{align*}
For every $h\in U_1'$ denote 
\index{h@$\check \funh$}\hypertarget{bbi}{}
\index{q@$\check \funq$}\hypertarget{bbj}{}
\begin{align}\label{eq: tilde h tilde q}
(\check \funh(h), \check \funq(h)) := \mLU^{-1}(hq_0) 
\end{align}
and note that the equality 
$hq_0 = \check \funq(h)\check \funh(h)$
implies that $\check \funh$ induces an between $\check \funh:\tilde U_1' \to \tilde U_2'$. 
Note that $\Id\in U_1', U_2'$.

The sets $U_1 := U_1'\cap H_{V_\bullet \to E_\bullet}$ and 
$U_2 := U_2'\cap  H_{V_\bullet^{\to}}$ 
have the property that, for every $h\in U_1'$, we have 

\begin{align}
h\in U_1&\iff\label{eq:iff1} 
\forall l\in L(V_\bullet),~ hV_l\in \gr_{\to E_l} 
\\&\iff \label{eq:iff2} 
\forall l\in L(V_\bullet),~ hq_0V_l^{\to}\in \gr_{\to E_l}
\\&\iff\label{eq:iff3} 
\forall l\in L(V_\bullet),~ \check \funq(h) \check \funh (h)V_l^{\to}\in \gr_{\to E_l}
\\&\iff\label{eq:iff4} 
\forall l\in L(V_\bullet),~ \check \funh (h)V_l^{\to}\in \gr_{\to E_l}
\\&\iff\label{eq:iff5} 
\forall l\in L(V_\bullet),~ \check \funh (h)\in H_{V_l^{\to}},
\end{align}
where
\eqref{eq:iff1} follows from the definition of $U_1$,
\eqref{eq:iff2} follows from the definition of $q_0$, 
\eqref{eq:iff3} follows from Eq. \eqref{eq: tilde h tilde q},
\eqref{eq:iff4} follows from the fact that $H^-$ preserves $\gr_{\to E}$, and
\eqref{eq:iff5} follows from Lemma \ref{lem: H is H}. 

Therefore $U_1, U_2$ satisfy the desired condition. 

Note that the smaller $q_0$ is, the closer $hq_0, hq_0^{-1}$ are to $H\subset {\rm Im}(\mLU)$ and hence the second condition follows.
\end{proof}
Fix a direction filtration $E_\bullet$ and a flag 
$V_\bullet$ with $L(V_\bullet)=L(E_\bullet)$ and $V_l\in \gr_{\to E_l}$ for every $l\in L(V_\bullet)$. 
Let $q_0\in H^-$ be the element from Lemma \ref{lem: making q0} which satisfies $q_0 V_\bullet^{\to} = V_\bullet$. 
Lemma \ref{lem: understanding V to E} gives us a measure $\mu' = \check \funh^{-1}_*\mu$ on $H_{V_l\to E_l}$, where $\mu$ is the invariant measure on $H_{V_\bullet^\to}$. 

To extend Lemma \ref{lem:counting} to $H_{V_\bullet\to E_\bullet}$ we need to compute the volume $\mu'(B_{d_\varphi}(h;\exp(-t)))$ for $h\in H_{V_\bullet\to E_\bullet}$.

\begin{lem}\label{lem: measure comp}
There exists $\varepsilon>0$ such that if $d_\varphi(q_0, \Id)<\varepsilon$, then every $h_0\in H_{V_\bullet\to E_\bullet}$ with $d_\varphi(h_0,\Id)\le 1$ satisfies
\[\mu'(B_{d_\varphi}(h_0; \exp(-t))) = \Theta(\exp(-t\delta(E_\bullet)))\text{ as }t\to \infty.\]
\end{lem}
\begin{proof}
Let $\check \funh, \check \funq$ be as in Lemma \ref{lem: understanding V to E}. If $d_\varphi(q_0, \Id)$ is small enough, then $B_{d_\varphi}(\Id; 2^{1/\alpha_\varphi}) \cap H_{V_\bullet\to E_\bullet}$ is in the domain of $\check \funh, \check \funq$ and similarly $B_{d_\varphi}(\Id; 2^{1/\alpha_\varphi}) \cap H_{V_\bullet}$ is in the image of $\check \funh$.

We first show the result when $h_0=\Id$. 
We want to calculate 
\begin{align}\label{eq: ball around Id measure}
\mu(&\{h\in H_{V_\bullet}:\check \funh^{-1}(h) \in B_{d_\varphi}(\Id; \exp(-t))\})\\
 &=\nonumber
\mu\left(\check \funh(B_{d_\varphi}(\Id; \exp(-t))\cap H_{V_\bullet\to E_\bullet})\right).
\end{align}
Let $\check \funh', \check \funq'$ be the functions generated by Lemma \ref{lem: understanding V to E}  applied to $E_\bullet, q_1 V_\bullet$, and $q_1$, where $q_1 = g_tq_0g_{-t}$, and $q_1V_\bullet^\to$. 
These functions satisfy 
\begin{align}\label{eq: h'q' behavior}
\check \funq'(h)\check \funh'(h) = hq_1 = hg_tq_0g_{-t}.
\end{align}
Substituting $h = g_thg_{-t}$ in Eq. \eqref{eq: h'q' behavior} we get 
\[\check \funq'(g_thg_{-t})\check \funh'(g_thg_{-t}) = g_th_0g_{-t} = (g_{t}\check \funq(h)g_{-t})(g_t\check \funh(h) g_{-t}).\]
Since the multiplication map $\mLU:H^{-0}\times H\to \SL_n(\RR)$ is one-to-one, it follows that 
\[\check \funq'(g_thg_{-t}) = g_t\check \funq(h)g_{-t}\text{ and }\check \funh'(g_thg_{-t}) = g_t\check \funh(h)g_{-t}.\]
Conjugating by $g_t$ we get that the quantity in Eq. \eqref{eq: ball around Id measure} equals 
\begin{align*}
\exp(-t \delta(E_\bullet)
\mu\left(\check \funh'(B_{d_\varphi}(\Id; 1)\cap H_{V_\bullet\to E_\bullet})\right).
\end{align*}
Since $q_1$ is close to $\Id$, Lemma \ref{lem: understanding V to E} implies that the map $\check \funh'(h)|_{B_{d_\varphi}(\Id; 2^{1/\alpha_\varphi})}$ is close to the identity map, and hence
\[\mu(\check \funh'(B_{d_\varphi}(\Id; 1)\cap H_{g_tV_\bullet\to E_\bullet})) = \Theta(1),\]
and hence our prediction of the $\mu'$-measure of a small ball around $\Id$ is correct.

For $h_0\neq \Id$, write $\check \funq(h_0) = q_2q_2^0$, where $q_2\in H^-$ and $q_2^0\in H^0$. 
Since
\[\forall l\in L(V_\bullet), \:h_0q_0V_l^{\to} = \check \funq(h_0)\check \funh(h_0)V_l^\to = q_2q_2^0V_l^{\to},\]
it follows that for $V_l':= h_0 V_l$ we have $(V_l')^{\to} = q_2^0V_l^{\to}$. 
We can now construct $\mu''$ on $H_{V_\bullet'\to E_\bullet}$ as we constructed $\mu'$. In much the same way, since $d_{H^-}(q_0,\Id)$ is arbitrarily small and $h_0$ lies in a compact set, we deduce from Lemma \ref{lem: understanding V to E} that $d_{H^-}(\check \funq(h_0), \Id) = d_{H^-}(q_1,\Id)$ is arbitrarily small, and hence 
\[\mu''(B_{d_\varphi}(\Id; \exp(-t))\cap H_{V_\bullet'\to E_\bullet}) = \Theta(\exp(-t \delta(E_\bullet)).\]
By Lemma \ref{lem: understanding V to E}, if $d_{H^-}(q_0, \Id)$ is small enough, then both $H_{V_\bullet\to E_\bullet}$ and $H_{V_\bullet'\to E_\bullet}$ are manifolds in a bounded set.
Since $\mu'', \mu'$ are smooth measures with positive density, and since
 $H_{V_\bullet\to E_\bullet}h_0^{-1} = H_{V_\bullet'\to E_\bullet}$,
we obtain that the Radon-Nikodym derivative $(R_{h_0^{-1}})_* \mu'/\mu''$ is bounded on $B_{d_\varphi}(\Id; \exp(-t))\cap H_{V_\bullet'\to E_\bullet}$, where $R_{h_0^{-1}}$ is the map of multiplication from the right by $h_0^{-1}$, and $(R_{h_0^{-1}})_*$ is the pushforward of measure. Hence,
\begin{align*}
\mu'(B_{d_\varphi}(h; \exp(-t))\cap H_{V_\bullet\to E_\bullet}) =&\, (R_{h_0^{-1}})_*\mu'(B_{d_\varphi}(\Id; \exp(-t))\cap H_{V_\bullet'\to E_\bullet})\\
=&\,\Theta\left(\mu''(B_{d_\varphi}(\Id; \exp(-t))\cap H_{V_\bullet'\to E_\bullet})\right) \\=&\,
\Theta\exp(-t \delta(E_\bullet)),
\end{align*}
as desired.
\end{proof}

\begin{de}
\index{Gr@$\gr$!toE@$\gr_{\to E}^{\varepsilon, g}$}\hypertarget{bca}{}
For every $l=0,\dots,n$ and $E\in \cI_l$ we let $\gr_{\to E}^{\varepsilon, g}$ denote the set of vector spaces $V\in \gr_{\to E}$ such that $d_{\gr}(V,V^{\to})<\varepsilon$. 
\end{de}

Lemma \ref{lem: C action on measure in H_V} shows that 
\begin{align}\label{eq:volume ball H_V}
\mu(B_{d_\varphi}(h;\exp(-t)) = \Theta(\exp(-t\delta(E_\bullet))),
\end{align}
for every $h\in H_{V_\bullet}$ and $t>0$. Combined with standard covering arguments, Eq. \eqref{eq:volume ball H_V} proves Lemma \ref{lem:counting}. Replacing Eq. \eqref{eq:volume ball H_V} with Lemma \ref{lem: measure comp} in the proof of Lemma \ref{lem:counting} we can now prove the following.
\begin{lem}\label{lem:the right counting theorem}
There is $\varepsilon_0>0$ such that the following holds.
Let $V_\bullet$ be a flag and let $E_\bullet$ be a direction filtration with $L(V_\bullet) = L(E_\bullet)$, $t>0$, and $V_l\in \gr_{\to E_l}^{\varepsilon_0, g}$ for every $l\in L(V_\bullet)$.
There is a set $A\subset B_{d_\varphi}(\Id; 1)$ of size $\Theta(\exp(t\delta(E_\bullet)))$ that satisfies the following conditions.
\begin{enumerate}[label=\emph{(\arabic*)}, ref=\arabic*]
  \item $d_\varphi(h_1, h_2) \ge \exp(-t)$ for all pairs of distinct elements $h_1, h_2\in A$;
  \item\label{cond:element to E_i2} We have $h\in H_{V_\bullet\to E_\bullet}$ for every $h\in A$. 
\end{enumerate}
Moreover, there is no such set $A$ larger than $\Theta(\exp(t\delta(E_\bullet))),$ even if we weaken Condition \emph{\eqref{cond:element to E_i2}} to 
\begin{enumerate}[label={\emph{($\ref{cond:element to E_i2}'$)}}, ref={$\ref{cond:element to E_i2}'$}]
\item \label{cond: lies near _E_i2} For every $h\in A$ there exists $h'\in H_{V_\bullet\to E_\bullet}$ such that $d_\varphi(h,h')<\exp(-t)$. 
\end{enumerate}
\end{lem}
\subsection{The Changing Step Lemma} 
\label{sub:the_change-step_lemma}
In view of Observation \ref{obs: The behavior of templates}, the following lemma will construct Alice's step whenever $f^{h_\infty\Lambda}_{E,\bullet}$ changes from time $T_m$ to time $T_{m+1}$ via a non-null vertex. In these steps $\#A_{m+1} = 1$, that is, Alice's choice determines $h_{m+1}$.
Consequently, Lemma \ref{lem:single advance} below will prove the existence of an element $h\in B_{d_\varphi}(\Id;1)$ with a certain behavior.
\begin{de}\label{de:order definition}
\index{#@$\eteq$!$E_\bullet \eteq E'_\bullet$}\hypertarget{bcb}{}
Let $E_\bullet, E'_\bullet$ be direction filtrations with $L(E_\bullet) = L(E'_\bullet)$. 
We write $E_\bullet \eteq E'_\bullet$ if $E_l\eteq E_l'$ for every $l\in L(E_\bullet)$.
\end{de}
\begin{remark}\label{rem:order definition}
In Definition \ref{de:order definition}, the conditions $E_l\eteq E_l'$ are required independently for each $l$. That is, for every $l\in L(E_\bullet)$ we can increase some of the elements in $E_l$ to obtain $E_{l}'$. 
The increments are not required to be compatible for different values $l$. We will see (Definition \ref{de:elementary flip}) a precise formulation of increasing all sets in a compatible way and show (Theorem \ref{thm: decomposition to elementary}) that the definitions are equivalent.
\end{remark}
\begin{lem}\label{lem:single advance}
Let 
$E_\bullet\eteq E_\bullet'$ be direction filtrations
and let $V_\bullet$ be a flag such that $L(V_\bullet) = L(E_\bullet)$ and $V_l\in \gr_{\to E_l}$ for every $l\in L(E_\bullet)$.
There exists $h\in B_{d_\varphi}(\Id;1)$ such that 
$hV_l \in \gr_{\to E_l'}$ for every $l\in L(V_\bullet)$.
Moreover, the convergence rate of $\lim_{t\to \infty}g_thV_l \in \gr^g_{E_l'}$ 
is uniform on compact sets of possible $V_\bullet$-s.
\end{lem}
The proof of Lemma \ref{lem:single advance} occupies the rest of this subsection.
We will break it up into several claims.
\begin{claim}[Transitivity of Lemma \ref{lem:single advance}]\label{claim:transitive}
The lemma is transitive with respect to $E_\bullet \eteq E_\bullet'$. That is,
let $E_\bullet\eteq E_\bullet'\eteq E_\bullet''$, 
and assume the statement of the lemma holds for the tuples $E_\bullet \eteq E'_\bullet$ and $E_\bullet'\eteq E_\bullet''$. Then the statement also holds for $E_\bullet\eteq E''_\bullet$. 
\end{claim}
\begin{proof}
Let $V_\bullet$ be a flag such that  $L(V_\bullet) = L(E_\bullet)$, and $V_l\in \gr_{\to E_l}$ for every $l\in L(V_\bullet)$. Fix $t_0, t_1 > 0$ to be determined later.
By the statement of the lemma applied to $E_\bullet\eteq E'_\bullet$ and the flag $g_{t_0}V_\bullet$, 
there exists $h_0\in B_{d_\varphi}(\Id;1)$ such that 
$h_0g_{t_0}V_l\in \gr_{\to E_l'}$ for all $l\in L(E_\bullet)$. 
By the statement of the lemma for $E'_\bullet\eteq E''_\bullet$ and the flag $g_{t_1}h_0g_{t_0}V_\bullet$ there exists $h_1\in B_{d_\varphi}(\Id;1)$ such that $h_1g_{t_1}h_0g_{t_0}V_l\in \gr_{\to E_l''}$ for all $l\in L(E_\bullet)$. 

And hence also $g_{-t_1-t_0}h_1g_{t_1+t_0}g_{-t_0}h_0g_{t_0}V_l\gr_{\to E_l''}$. 
If $t_0, t_1$ are large enough, both $g_{-t_0}h_0g_{t_0}$ and $g_{-t_1-t_0}h_1g_{t_1+t_0}$ are close to $\Id$, and their product lies in $B_{d_\varphi}(\Id;1)$. 
\end{proof}
\begin{claim}[The meaning of uniform convergence]\label{claim: uniform convergence meaning}Let $0\le l \le n$, $E\in \cI_l$ and $K\subseteq \gr_{\to E}$. The convergence $g_tV\to V^{\to}$ is uniform for $V\in K$ if and only if the \emph(compact\emph) closure $\bar K\subseteq \gr$ is a subset of $\gr_{\to E}$. 
\end{claim}
\begin{proof}
By Lemma \ref{lem: making q0}, for every $V\in K$ there exists $q_0(V)$ such that $V = q_0(V)V^\to$ and moreover, the function $V\mapsto q_0(V)$ is continuous. Therefore, if $K\subseteq \gr_{\to E}$ is compact, then so is $q_0(K)$, and $g_tq_0(K)g_{-t}$ contracts uniformly.

To prove the other direction, by Lemma \ref{lem: making q0} once again, if $Q_0\subseteq H^-$ is a compact neighborhood of $\Id$, then $Q_0 \gr^g_{E}$ is a compact neighborhood of $\gr^g_{E}$ in $\gr_{\to E}$. Therefore, if the convergence rate $g_t V \to V^{\to}$ is uniform for $V\in K$, there exists $t_0>0$ such that $g_{t_0}K\subseteq Q_0\gr^g_{E}$, and hence $K\subseteq g_{-t_0}Q_0\gr^g_{E}$ as desired.
\end{proof}
From Claim \ref{claim: uniform convergence meaning} it easily follows that the uniform convergence satisfies transitivity as well.

\begin{claim}[Reduction to the case where $V_\bullet$ is $g$-invariant]\label{claim: enough when source is inv}
It is enough to prove Lemma \emph{\ref{lem:single advance}} for flags $V_\bullet$ satisfying $V_l\in \gr_{E_l}^g$ for every $l\in L(E_\bullet)$. 
\end{claim}
\begin{proof}
Assume that Lemma \ref{lem:single advance} holds whenever $V_l\in \gr_{E_l}^g$ for every $l\in L(E_\bullet)$. By Claim \ref{claim: uniform convergence meaning}, there exist compact sets $K_l\subseteq \gr_{\to E_l'}$ such that for every flag $V_\bullet$ with $L(V_\bullet) = L(E_\bullet)$ and $V_l\in \gr_{E_l}^g$ for every $l\in L(E_\bullet)$ there is $h\in B_{d_\varphi}(\Id;1)$ such that $hV_l\in K_l$ for every $l\in L(E_\bullet)$. 
We will show that for every compact set $K_l^{\rm src}\subseteq \gr_{\to E_l}$ there exists $K_l^{\rm tar}\subseteq \gr_{\to E_l'}$, such that for every flag $V_\bullet$ with $V_l\in K_l^{\rm src}$ for every $l\in L(E_\bullet)$ there exists $h\in B_{d_\varphi}(\Id;1)$ with $hV_l\in K_l^{\rm tar}$. This, together with Claim \ref{claim: uniform convergence meaning}, extablishes Lemma \ref{lem:single advance}.

Fix $\varepsilon>0$ and a flag $V_\bullet$ with $L(V_\bullet) = L(E_\bullet)$ such that $V_l\in K_l^{\rm src}$ for every $l\in L(V_\bullet)$.
After applying $g_t$, we may assume $K_l^{\rm src}\subseteq U_\varepsilon(\gr^g_{E_l})$. Hence, for some $q_0\in B_{d_{H^-}}(\Id;O(\varepsilon))$ we have $V_l = q_0V_l^{\to}$. 
After applying $g_{\log \varepsilon}$ to $K_l$ we may assume that if $V_l^\to \in\gr^g_{E_l}$, then there exists $h_0\in B_{d_\varphi}(\Id; \varepsilon)$ such that $h_0V_l^\to = h_0q_0^{-1}V_l\in K_l$. 
Since $h_0$ and $q_0$ are arbitrarily small, there exist $q_1\in B_{d_{H^{-}}}(\Id;1)$ and $h_1\in B_{d_\varphi}(\Id; 1)$ such that $q_1h_1  = h_0q_0^{-1}$. 
It follows that $h_1V_l\in B_{H^{-}}(\Id;1)K_l=:K_l^{\rm tar}$, as desired.
\end{proof}
To employ Claim \ref{claim:transitive} we will investigate the poset $\cI_\bullet$ of direction filtrations with dimension sequence $L = \{0=l_0<l_1<\dots<l_k=n\}$,
where the partial order is given by $E_\bullet\eteq E_\bullet'$, as defined in Definition \ref{de:order definition}.
\begin{de}[Multiplicity of elements in multisets]
\index{multxe@$\mult(x;E)$}\hypertarget{bcc}{}
Let $E$ be a multiset. For every element $x$ we denote by $\mult(x;E)$ the multiplicity of $x$ in $E$.
\end{de}
\begin{de}[Elementary flip]\label{de:elementary flip}
\index{Elementary flip}\hypertarget{bcd}{}
We say that $E_\bullet \et E_\bullet'$ is an \emph{elementary flip} if there exist $\eta < \eta' \in \Eall$ and $1\le a<b\le k$
such that 
\begin{align}\label{eq:def of flip}
E_{l_i}' = \begin{cases}
E_{l_i} - \{\eta\} + \{\eta'\}, &\text{if } a\le i<b,\\
E_{l_i},& \text{otherwise.}
\end{cases}
\end{align}
We use $+$ and $-$ instead of $\cup$ and $\setminus$, as these are multiset operations, under which multiplicities are added or subtracted. 
In this case $E_\bullet \et E_\bullet'$ is called an \emph{$(\eta, \eta', a,b)$-elementary flip}. We think of an elementary flip as an operation transforming $E_\bullet$ into $E_\bullet'$.
Note that a direction filtration $E_\bullet\in \cI_\bullet$ has an $(\eta, \eta', a,b)$-elementary flip $E_\bullet \et E_\bullet'$ if and only if the following condition of multiplicities is satisfied:
\[\mult(\eta;E_{l_{a-1}}) < \mult(\eta; E_{l_a})\text{ and }\mult(\eta'; E_{l_{b-1}}) < \mult(\eta'; E_{l_b}).\]
Indeed, otherwise the unique $E_\bullet'$ defined by Eq. \eqref{eq:def of flip} will fail to satisfy $E_{l_{a-1}}'\subset E_{l_a}'$ or $E_{l_{b-1}}'\subset E_{l_b}'$, respectively.
\end{de}
\begin{thm}[The poset $\cI_\bullet$ is spanned by elementary flips]\label{thm: decomposition to elementary}
If $E_\bullet\eteq E_\bullet'$, then there exists a sequence of direction filtration $E_\bullet = E_\bullet^0 \et E_\bullet^1 \et \cdots \et E_\bullet^r = E_\bullet'$, where each consecutive $E_\bullet^j \et E_\bullet^{j+1}$ is an elementary flip.
\end{thm}
\begin{proof}
\textbf{The plan:}
Since $\cI_\bullet$ is a finite poset, it is enough to show that for every $E_\bullet\et E'_\bullet$ there is an elementary flip $E_\bullet\et E_\bullet''$ such that $E_\bullet''\eteq E_\bullet'$. 

To use the fact $E_\bullet \et E'_\bullet$, define 
\begin{align}\label{eq: def sigma}
\sigma:&\{0,...,k\}\times ({\rm set}(\Eall)\cup \{-\infty\})\to \NN_0,\\
\nonumber\sigma(i,\eta) :&= \#\left(E_{l_i}\cap (-\infty, \eta]\right) = \sum_{\substack{\tilde\eta\in {\rm set}(\Eall)\\\tilde\eta\le \eta}}\mult(\tilde\eta,E_{l_i}).
\end{align}
\index{N0@$\NN_0$}\hypertarget{bce}{}
Here ${\rm set}(\Eall)$ stands for the set of distinct elements in $\Eall$ and $\NN_0 = \{0,1,2,...\}$ is the set of natural numbers.
Define $\sigma'$ in a similar fashion with respect to $E'_\bullet$ instead of $E_\bullet$. 
Note that $E_\bullet \eteq E_\bullet'$ if and only if $\sigma(i,\eta)\ge \sigma'(i,\eta)$ for all $i,\eta$. 

If $E_\bullet\et E_\bullet''$ is an $(\eta, \eta', a,b)$-elementary flip, then the functions $\sigma, \sigma''$ satisfy that $\sigma'' = \sigma - \mathbbm{1}_R$, where $R$ is the rectangle $\{(i,\tilde\eta):a\le i<b, \eta\le \tilde\eta<\eta' \}$. We need the function $\sigma''$ to remain greater or equal to $\sigma'$ to ensure $E''_\bullet\eteq E_\bullet$. 
To help us find a suitable rectangle, we will identify values $0<i_0<i_2$ such that $\sigma(i,\eta) - \sigma'(i,\eta)$ is zero for $0\le i<i_0$ and all $\eta\in \Eall\cup \{-\infty\}$ and nonzero for all $i_0\le i < i_2$ and some $\eta\in \Eall$. Moreover $i\mapsto \sigma(i,\eta) - \sigma'(i,\eta)$ will be weakly increasing in $i$ for $0\le i < i_2$ and every $\eta\in \Eall\cup \{\infty\}$. 

\textbf{Additional notations:}
Since we are interested in monotonicity properties of $\sigma-\sigma'$, define the  function $\tau:\{1,...,k\}\times ({\rm set}(\Eall)\cup \{-\infty\})\to \ZZ$ by
\begin{align*}
\tau(i, \eta):=\sigma(i, \eta) - \sigma'(i, \eta) - \sigma(i - 1, \eta) + \sigma'(i-1, \eta).
\end{align*}

By definition, for every $\eta\in \Eall\cup \{-\infty\}$ and $0\le i\le k$ we have 
\begin{align}\label{eq:sigma tau connection}
\sigma(i,\eta)-\sigma'(i,\eta) = \sum_{i'=0}^i\tau(i', \eta) = - \sum_{i'=i+1}^k\tau(i', \eta).
\end{align}
The last equality holds since $\sigma(k,\eta) = \sigma'(k,\eta)$, because $E_{l_k} = E_{l_k}' = \Eall$.
Denoting $D_i := E_{l_i}-E_{l_{i-1}}$ and $D_i' := E_{l_i}' - E_{l_{i-1}}'$, we have
\[\tau(i, \eta) = \#\{\tilde\eta\in D_i-D_i':\tilde \eta\ge \eta\}.\]
Note that here we regard $D_i$ and $D_i'$ as multisets with possibly negative multiplicities. 

\textbf{Choosing indeices:}
The minimal index $i_0$ for which the function $\eta\mapsto \tau(i_0, \eta)$ is nonzero is the minimal $i_0$ for which $E_{l_{i_0}}\neq E_{l_{i_0}}'$ (which exists because $E_\bullet\neq E_\bullet'$), and satisfies 
\[ \tau(i_0, \eta) = \sum_{i=0}^{i_0}\tau(i, \eta) = \sigma(i_0, \eta) - \sigma'(i_0, \eta) \ge 0\quad\text{for all }\eta\in \Eall.\]
Since $E_{i_0} \et E_{i_0}'$, it follows that $\tau(i_0, \eta) = \sigma(i_0, \eta) - \sigma'(i_0, \eta)$ is positive for some $\eta\in \Eall$. 

Similarly, looking at the right-hand side of Eq. \eqref{eq:sigma tau connection} we see that $\tau$ attains a negative value for some $i$, and we let $i_2$ be the minimal index such that $\eta\mapsto \tau(i_2, \eta)$ attains a negative value. 
It follows from the equality
\[\tau(i_2, \eta_{n}) = \#E_{l_{i_2}} - \#E_{l_{i_2}}' - \#E_{l_{i_2-1}} + \#E_{l_{i_2-1}}' = 0,\]
that there exists an index $1<s_2\le n$ such that $\tau(i_2, \eta_{s_2-1}) < 0$ but $\tau(i_2, \eta_{s}) \ge 0$ for all $s\ge s_2$. 

Since 
\[\sum_{i=1}^{i_2} \tau(i, \eta_{s_2-1}) = \sigma(i, \eta_{s_2-1}) - \sigma'(i, \eta_{s_2-1})\ge 0,\]
it follows that for some $i_1<i_2$ we have $\tau(i_1, \eta_{s_2-1}) > 0$. 
Let $s_1\ge 1$ be the minimal index such that $\tau(i, \eta_{s}) > 0$ for every $\eta_{s_1}\le \eta_s\le \eta_{s_2-1}$. 

\textbf{Obtaining the desired flip:}
We claim that we can perform an $(\eta_{s_1}, \eta_{s_2}, i_1, i_2)$-elementary flip to $E_\bullet$, and obtain a sequence $E''_\bullet$ for which $E''_\bullet \eteq E'_\bullet$. 

To change one occurrence of $\eta_{s_1}$ to $\eta_{s_2}$ between the indices $l_{i_1}$ and $l_{i_2-1}$, while preserving the requirement that the resulting sequence $E_\bullet''$ is a direction filtration (i.e, the multisets $E''_\bullet$ are nested, see Eq. \eqref{eq: directed filtration}), it is sufficient and necessary to show that 
\begin{align}\label{eq: flip is legal}
\mult(\eta_{s_1}; E_{l_{i_1-1}}) < \mult(\eta_{s_1}; E_{l_{i_1}})\text{ and }\mult(\eta_{s_2}; E_{l_{i_2-1}}) < \mult(\eta_{s_2}; E_{l_{i_2}})
\end{align}
Since 
\[\tau(i_1, \eta_{s_1}) - \tau(i_1, \eta_{s_1-1}) = \mult(\eta_{s_1}; E_{l_{i_1}})-\mult(\eta_{s_1}; E_{l_{i_1-1}})\]
and 
\[\tau(i_2, \eta_{s_2}) - \tau(i_2, \eta_{s_2-1}) = \mult(\eta_{s_2}; E_{l_{i_2}}) - \mult(\eta_{s_2}; E_{l_{i_2-1}}),\]
inequalities \eqref{eq: flip is legal} are equivalent to
\begin{align}
\tau(i_1, \eta_{s_1}) > \tau(i_1, \eta_{s_1-1})\text{ and }\tau(i_2, \eta_{s_2-1}) < \tau(i_2, \eta_{s_2}),
\end{align}
which follows from the choices of $s_1, s_2$. 
Define
\[E_{l_i}'' := \begin{cases}
E_{l_i} - \{\eta_{s_1}\} + \{\eta_{s_2}\}, &\text{if } i_1\le i<i_2,\\
E_{l_i},& \text{otherwise.}
\end{cases}\]

To show that $E_\bullet''\eteq E_\bullet'$,  define $\sigma''$ as in Eq. \eqref{eq: def sigma} with respect to $E''_\bullet$. 
By the definition of $E''_\bullet$ we obtain 

\[\sigma''(i,\eta) = 
\begin{cases}
\sigma(i, \eta)-1, &\text{if }i_1\le i<i_2\text{ and }\eta_{s_1}\le \eta <\eta_{s_2},\\
\sigma(i, \eta),   &\text{otherwise.}
\end{cases}\]
We need to show that $\sigma'' \ge \sigma'$ pointwise, which is equivalent to \[\sigma(i,\eta) > \sigma'(i,\eta)\text{ for every }i_1\le i< i_2, \eta_{s_1}\le \eta <\eta_{s_2}.\]
These inequalities follow from the fact that $\sigma(i,\eta) - \sigma'(i,\eta) = \sum_{i'\le i}\tau(i', \eta)$, which is a sum of nonnegative numbers by the definition of $i_2$ and where the summand $\tau(i_1, \eta)$ is positive by the definition of $\eta_{s_1}$. 

Altogether we have shown that there is $E_\bullet''$ such that $E_\bullet \et E''_\bullet \eteq E_\bullet'$, and the tuple $E_\bullet\et E''_\bullet$ is an elementary flip.
Iterating this procedure for $E''_\bullet \eteq E_\bullet'$ yields the desired result. 
\end{proof}

\begin{ex}
Let $n=4, \Eall = \{-2, -1, 1, 2\}$, and $l_i=i$ for $i=0,\dots,4$.
Denote 
\[E_\bullet = \{\emptyset \subset \{-2\}\subset \{-2, 1\}, \{-2, -1, 1\}, \Eall\},\]
and
\[E_\bullet' = \{\emptyset \subset \{2\}\subset \{-1, 2\}, \{-1, 1, 2\}, \Eall\}.\]
We will describe in detail the process elaborated in the proof of Theorem \ref{thm: decomposition to elementary}.
First, compute $\sigma, \sigma', \tau$.
\newcounter{nodecount}
\newcommand\tabnode[1]{\addtocounter{nodecount}{1}\tikz\node(\arabic{nodecount}){#1};}
\newcommand\mytabnode[2]{\tikz\node(#1){#2};}
\tikzstyle{every picture}+=[remember picture,baseline]

\begin{center}


\begin{tabular}{l}
$\sigma(i,\eta):$\\
\begin{tabular}{|c|| c c c c c |} 
 \hline
 \backslashbox{$i$}{$\eta$}
 & $-\infty$ & $-2$ & $-1$ & $1$ & $2$ \\
 \hline\hline
 \tabnode{0} & \tabnode{$0$} & \tabnode{$0$} & \tabnode{$0$} & \tabnode{$0$} & \tabnode{$0$} \\
 \hline
 \tabnode{$1$} & \tabnode{$0$} & \mytabnode{nodeaa}{$1$} & \mytabnode{nodeab}{$1$} & \tabnode{$1$} & \tabnode{$1$} \\ 
 \hline
 \tabnode{2} & \tabnode{$0$} & \tabnode{$1$} & \tabnode{$1$} & \tabnode{$2$} & \tabnode{$2$} \\ 
 \hline
 \tabnode{3} & \tabnode{$0$} & \tabnode{$1$} & \tabnode{$2$} & \tabnode{$3$} & \tabnode{$3$} \\ 
 \hline
 \tabnode{4} & \tabnode{$0$} & \tabnode{$1$} & \tabnode{$2$} & \tabnode{$3$} & \tabnode{$4$} \\
 \hline
\end{tabular}
\end{tabular}
\begin{tikzpicture}[overlay]
\draw [blue](nodeaa.north west) -- (nodeab.north east) -| (nodeab.south east) |- (nodeaa.south west) |- (nodeaa.north west);
\end{tikzpicture}
~~~
\begin{tabular}{l}
$\sigma'(i,\eta)$:\\
\begin{tabular}{|c|| c c c c c |} 
 \hline
 \backslashbox{$i$}{$\eta$} & $-\infty$ & $-2$ & $-1$ & $1$ & $2$ \\
 \hline\hline
 0 & $0$ & $0$ & $0$ & $0$ & $0$ \\
 \hline
 1 & $0$ & $0$ & $0$ & $0$ & $1$ \\ 
 \hline
 2 & $0$ & $0$ & $1$ & $1$ & $2$ \\ 
 \hline
 3 & $0$ & $0$ & $1$ & $2$ & $3$ \\ 
 \hline
 4 & $0$ & $1$ & $2$ & $3$ & $4$ \\
 \hline
\end{tabular}
\end{tabular}
\begin{tabular}{l}
$\tau(i,\eta)$:\\
\begin{tabular}{|c|| c c c c c |} 
 \hline
 \backslashbox{$i$}{$\eta$} & $-\infty$ & $-2$ & $-1$ & $1$ & $2$ \\
 \hline\hline
 1 & $0$ & $1$ & $1$ & $1$ & $0$ \\ 
 \hline
 2 & $0$ & $0$ & $-1$ & $0$ & $0$ \\ 
 \hline
 3 & $0$ & $0$ & $1$ & $0$ & $0$ \\ 
 \hline
 4 & $0$ & $-1$ & $-1$ & $-1$ & $0$ \\
 \hline
\end{tabular}
\end{tabular}
\end{center}
Hence, $i_0 = 1$, $i_2 = 2$, $s_2 = 3$, $\eta_{s_2} = 1$, $i_1 = 1$, $s_1 = 1, \eta_{s_1} = -2$. Consequently, we obtain the elementary flip
$E_\bullet\et E_\bullet^{(1)}$, where \[E_\bullet^{(1)} = \{\emptyset \subset \{1\}\subset \{-2, 1\}, \{-2, -1, 1\}, \Eall\}.\]
The blue rectangle in the $\sigma(i,\eta)$ table denotes the rectangle $R$ for which $\sigma^{(1)} = \sigma-\mathbbm{1}_R$. 
Continuing the process for $E_\bullet^{(1)}\eteq E_\bullet'$, we have 
\begin{center}
\begin{tabular}{l}
$\sigma^{(1)}(i,\eta):$\\
\begin{tabular}{|c|| c c c c c |} 
 \hline
 \backslashbox{$i$}{$\eta$} & $-\infty$ & $-2$ & $-1$ & $1$ & $2$ \\
 \hline\hline
 \tabnode{0} & \tabnode{$0$} & \tabnode{$0$} & \tabnode{$0$} & \tabnode{$0$} & \tabnode{$0$} \\
 \hline
 \tabnode{1} & \tabnode{$0$} & \tabnode{$0$} & \tabnode{$0$} & \mytabnode{nodeba}{$1$} & \tabnode{$1$} \\ 
 \hline
 \tabnode{2} & \tabnode{$0$} & \tabnode{$1$} & \tabnode{$1$} & \tabnode{$2$} & \tabnode{$2$} \\ 
 \hline
 \tabnode{3} & \tabnode{$0$} & \tabnode{$1$} & \tabnode{$2$} & \mytabnode{nodebb}{$3$} & \tabnode{$3$} \\ 
 \hline
 \tabnode{4} & \tabnode{$0$} & \tabnode{$1$} & \tabnode{$2$} & \tabnode{$3$} & \tabnode{$4$} \\
 \hline
\end{tabular}
\end{tabular}
\begin{tikzpicture}[overlay]
\draw [blue](nodeba.north west) -- (nodeba.north east) -| (nodebb.south east) |- (nodebb.south west) |- (nodeba.north west);
\end{tikzpicture}
~~~
\begin{tabular}{l}
$\tau(i,\eta):$\\
\begin{tabular}{|c|| c c c c c |} 
 \hline
 \backslashbox{$i$}{$\eta$} & $-\infty$ & $-2$ & $-1$ & $1$ & $2$ \\
 \hline\hline
 1 & $0$ & $0$  & $0$  & $1$  & $0$ \\ 
 \hline
 2 & $0$ & $1$  & $0$  & $0$  & $0$ \\ 
 \hline
 3 & $0,$ & $0$  & $1$  & $0$  & $0$ \\ 
 \hline
 4 & $0$ & $-1$ & $-1$ & $-1$ & $0$ \\
 \hline
\end{tabular}
\end{tabular}
\end{center}
Hence, $i_0 = 1$, $i_2 = 4$, $s_2 = 4$, $\eta_{s_2} = 2$, $i_1 = 1$, $s_1 = 3, \eta_{s_1} = 1$.
Consequently, we obtain the elementary flip
$E_\bullet^{(1)} \et E_\bullet^{(2)}$, where 
\[E_\bullet^{(2)} = 
\{\emptyset \subset \{2\}\subset \{-2, 2\}, \{-2, -1, 2\}, \Eall\}.\]
\begin{center}
\begin{tabular}{l}
$\sigma^{(2)}(i,\eta):$\\
\begin{tabular}{|c|| c c c c c |} 
 \hline
 \backslashbox{$i$}{$\eta$} & $-\infty$ & $-2$ & $-1$ & $1$ & $2$ \\
 \hline\hline
 \tabnode{0} & \tabnode{$0$} & \tabnode{$0$} & \tabnode{$0$} & \tabnode{$0$} & \tabnode{$0$} \\
 \hline
 \tabnode{1} & \tabnode{$0$} & \tabnode{$0$} & \tabnode{$0$} & \tabnode{$0$} & \tabnode{$1$} \\ 
 \hline
 \tabnode{2} & \tabnode{$0$} & \tabnode{$1$} & \tabnode{$1$} & \tabnode{$1$} & \tabnode{$2$} \\ 
 \hline
 \tabnode{3} & \tabnode{$0$} & \tabnode{$1$} & \mytabnode{nodec}{$2$} & \tabnode{$2$} & \tabnode{$3$} \\ 
 \hline
 \tabnode{4} & \tabnode{$0$} & \tabnode{$1$} & \tabnode{$2$} & \tabnode{$3$} & \tabnode{$4$} \\
 \hline
\end{tabular}
\end{tabular}
\begin{tikzpicture}[overlay]
\draw [blue](nodec.north west) -- (nodec.north east) -| (nodec.south east) |- (nodec.south west) |- (nodec.north west);
\end{tikzpicture}
~~~
\begin{tabular}{l}
$\tau(i,\eta):$\\
\begin{tabular}{|c|| c c c c c |} 
 \hline
 \backslashbox{$i$}{$\eta$} & $-\infty$ & $-2$ & $-1$ & $1$ & $2$ \\
 \hline\hline
 1 & $0$ & $0$  & $0$  & $0$  & $0$ \\ 
 \hline
 2 & $0$ & $1$  & $0$  & $0$  & $0$ \\ 
 \hline
 3 & $0$ & $0$  & $1$  & $0$  & $0$ \\ 
 \hline
 4 & $0$ & $-1$ & $-1$ & $0$ & $0$ \\
 \hline
\end{tabular}
\end{tabular}
\end{center}
Hence, $i_0 = 2$, $i_2 = 4$, $s_2 = 3$, $\eta_{s_2} = 1$, $i_1 = 3$, $s_1 = 2, \eta_{s_1} = -1$.
Consequently, we obtain the elementary flip
$E_\bullet^{(2)} \et E_\bullet^{(3)}$, where 
\[E_\bullet^{(3)} = 
\{\emptyset \subset \{2\}\subset \{-2, 2\}, \{-2, 1, 2\}, \Eall\}.\]
We can continue the algorithm for one last step, or alternatively note that $E_\bullet^{(3)}\et E_\bullet'$ is an elementary flip.
\end{ex}

\begin{proof}[Proof of Lemma \emph{\ref{lem:single advance}}]
\ref{claim:transitive} and Theorem \ref{thm: decomposition to elementary} shows that it is enough to prove Lemma \ref{lem:single advance} whenever $E_\bullet\et E_\bullet'$ is an elementary flip. 
By Claim \ref{claim: enough when source is inv}, we assume $V_\bullet$ is a flag with $V_l\in \gr_{E_l}^g$. Hence, $V_\bullet = uV_\bullet^{(0)}$, where $u\in O_g$, $V_l^{(0)} = \spa \{e_j:j\in F_l\}$, $F_l \subseteq \{1,\dots,n\}$ and $E_l=\{\eta_j:j\in F_l\}$. 
Since $E_\bullet\et E_\bullet'$ is an $(\eta, \tilde \eta,a,b)$-elementary flip, there exist $j_0$ with $\eta_{j_0}=\eta$ and $j_0\in F_a\setminus F_{a-1}$, and $j_1$ with $\eta_{j_1} = \tilde\eta$ and $j_1\in F_b\setminus F_{b-1}$. 
The matrix $h = \Id+\xi e_{j_1,j_0}$ lies in $B_{d_\varphi}(\Id;1)$ provided $\xi=\Theta(1)>0$ is small enough. It preserves the spaces $V_{l_i}^{(0)}$ for $i<a$ or $i\ge b$. If $a\le i < b$, then the blade corresponding to $hV_{l_i}^{(0)}$ is 
\begin{align}
\label{eq:sigle advance}
(e_{j_0}+\xi e_{j_1})\wedge \bigwedge_{j\in F_{l_i}\setminus \{j_0\}}e_j = \pm\bigwedge_{j\in F_{l_i}}e_j \pm\xi \bigwedge_{j\in F_{l_i}\cup \{j_1\}\setminus \{j_0\}}e_j.
\end{align}

The expression in Eq. \eqref{eq:sigle advance} is the sum of two $g_t$-eigenvectors, with respective eigenvalues $\exp\left(t\sum_{\eta\in E_i}\eta\right)$ and 
$\exp\left(t\sum_{\eta\in E_i'}\eta\right)$. The second eigenvalue is exponentially larger than the first, and so
\[g_thV_{l_i}^{(0)} \xrightarrow{t\to \infty}\spa\{e_j:j\in F_{l_i}\cup \{j_1\}\setminus \{j_0\}\}.\] 
It follows that $(hV_{l}^{(0)})^{\to}\in \gr_{\to E_{l}}$ for every $l\in L$, and the first claim of Lemma \ref{lem:single advance} holds for $V_\bullet^{(0)}$. 
As for $V_\bullet$, we have that $h' := uhu^{-1}\in B_{d_\varphi}(\Id; 1)$ satisfies $h'V_\bullet = uhV_\bullet^{(0)}$ and $h'V_{l_i}\in \gr_{\to E_l'}$ for every $l\in L$. Moreover, $h'V_{l}$ lies in the compact set $O_ghV_{l}^{(0)} \subseteq V_{l}\in \gr_{\to E_l'}$ for every $l\in L$. Claim \ref{claim: uniform convergence meaning} implies the desired for $V_\bullet$. 

\end{proof}
\subsection{Adding Spaces to the Flag} 
\label{sub:adding_spaces_to_the_flag}
In view of Observation \ref{obs: The behavior of templates}, the following lemma will construct Alice's step whenever $f^{h_\infty\Lambda}_{E,\bullet}$ changes from time $T_m$ to time $T_{m+1}$ via a null vertex $*\to E$ at $l$. In these steps $\#A_{m+1} = 1$, that is, Alice's choice determines $h_{m+1}$.
As in Subsection \ref{sub:single_dimension_analysis}, Lemma \ref{lem: enlarging the flag} below will prove the existence of an element $h\in B_{d_\varphi}(\Id;1)$ with a certain behavior.

Recall that a sublattice $\Gamma\subseteq \Lambda$ is called \emph{primitive} if there is no larger sublattice $\Gamma\subset \Gamma'\subseteq \Lambda$ with $\spa(\Gamma')=\spa(\Gamma)$.
\begin{claim}\label{claim: lattice approximation}
Let $\Lambda$ be a lattice, and denote by $\lambda_1(\Lambda)$ the size of the shortest nonzero vector in $\Lambda$. 
For every $l=0,\dots,n, V\in \gr$ and $\varepsilon>0$ there exists a primitive sublattice $\Gamma\subseteq \Lambda$ of rank $l$ and covolume $\cov\Gamma$ such that
\begin{align}\label{eq: log cov sublattice}
\log \cov\Gamma = \frac{l}{n}\log \cov \Lambda + O(\log \lambda_1(\Lambda) - \log \varepsilon),
\end{align}
 and $\spa \Gamma\in U_\varepsilon(V)$. 
\end{claim}
\begin{remark}\label{rem: implicit constant}
As always, the implicit constant in the big $O$ notation in Eq. \eqref{eq: log cov sublattice} depends on $n$, see Subsection \ref{sub:notations}. 
\end{remark}
\begin{proof}
It is sufficient to construct a non-primitive sublattice $\Gamma\subseteq \Lambda$. Indeed, given a non-primitive sublattice $\Gamma\subseteq \Lambda$ that satisfies the  other condition of the lemma, the lattice $\Gamma'=\spa \Gamma\cap \Lambda$ is primitive with the same span and smaller covolume. Since $\lambda_1(\Gamma') \ge \lambda_1(\Lambda)$, Minkowski's Theorem (\cite[Chapter III, Theorem II]{C}) implies that 
$\cov(\Gamma') =\Omega(\lambda_1(\Lambda)^{\rk \Gamma})$, hence $\Gamma'$ satisfies the desired condition. 

Up to multiplication by a scalar, assume that $\cov \Lambda = 1$. 
It is enough to prove the claim for $l=1$. 
Indeed, fixing an orthogonal basis $(v_i)_{i=1}^l$ and applying the theorem to each of the linear spaces $\RR v_i$ for $i=1,...,l$ yields $l$ rank-$1$ sublattices of $\Lambda$ whose sum is the sought-for sublattice.
After rotating $\Lambda$ and $V$ we may assume $V = \{x\in \RR^n:x_1=\dots=x_{n-1} = 0\}$.
By Minkowski's Theorem, there exists a vector $v=(v_i)_{i=1}^n\in \Lambda\setminus \{0\}$ satisfying $|v_1|,\dots,|v_{n-1}| \le \varepsilon_1$ and $|v_n|\le \varepsilon_1^{1-n}$ for every $\varepsilon_1>0$. Note that if $\varepsilon_1 \le \lambda_1(\Lambda) / 2\sqrt n$, then 
\[\|v\|^2 \le v_1^2 + \dots+ v_{n-1}^2 + v_n^2\le (n-1)\varepsilon_1^2 + v_n^2 \le \lambda_1(\Lambda)^2/2 + v_n^2.\]
Consequently, $|v_n| = \Omega(\lambda_1(\Lambda))$, and hence 
\[d_{\grl 1}(\RR v, V) O(d_{\RR^n}(v, V)/\|v\|) = O(\varepsilon/v_n) = O(\varepsilon_1/\lambda_1(\Lambda)).\] 
Therefore, there is a choice of $\varepsilon_1 = \Theta(\varepsilon\lambda_1(\Lambda))$ for which $d_{\grl1}(\RR v,V) \le \varepsilon$. 
Consequently, for $\Gamma = \ZZ v$ we have $\cov \Gamma = O((\varepsilon\lambda_1(\Lambda))^{1-n})$, as desired.
\end{proof}
\begin{remark}\label{rem:Findind a lattice}
The bound in Claim \ref{claim: lattice approximation} is not optimal. Though deriving the optimal bound may be of interest, it will not improve qualitatively the results of this paper.
\end{remark}
\begin{de}[Relative Covolume]\label{de: relatively covolume}
\index{Relative covolume}\hypertarget{bcf}{}
\index{cov@$\cov(\Gamma_2/\Gamma_1/\Gamma_0)$}\hypertarget{bcg}{}
Let $\Gamma_{0} \subseteq \Gamma_1\subseteq \Gamma_2$ be lattices in $\RR^n$ of ranks $l_{0}\le l_1\le l_2$, respectively.
Define the \emph{relative covolume} of $\Gamma_0, \Gamma_1, \Gamma_2$ by
\[\log\cov(\Gamma_2/\Gamma_1/\Gamma_0) :=\log \cov \Gamma_1 - \frac{l_2-l_1}{l_2-l_0}\log \cov \Gamma_0 - \frac{l_1-l_0}{l_2-l_0}\log \cov \Gamma_{2}.\]
The relative covolume is the signed height of $p_{\Gamma_1}=(\rk \Gamma_1, \log \cov\Gamma_1)$ above the interval $[p_{\Gamma_0}, p_{\Gamma_2}]$. It is also the normalized covolume of $\pi(\Gamma_1)$ in $\pi(\Gamma_2)$ where $\pi:\RR^n\to \spa(\Gamma_0)^\perp$, and we normalize so that $\cov(\pi(\Gamma_2)) = 1$.
\end{de}
\begin{lem}[Adding a lattice to a flag]\label{lem: enlarging the flag}
Let $1>\varepsilon>0$, let $\Lambda$ be a lattice, let $E_\bullet$ be a direction filtration and let $\Gamma_\bullet$ be a filtration such that $L(E_\bullet) = L(\Gamma_\bullet)$ and
\[\spa \Gamma_l\in\gr_{\to E_l}^{\varepsilon, g}.\]
Let
$l_0<l_1\in L(\Gamma_\bullet)$ be consecutive elements with $l_0+2\le l_1$, let $E_{l_0}\subset E\subset E_{l_1}$,
and assume that $\partial^2\HN(\Gamma_{l_1}/\Gamma_{l_0})_{H,l} \le C$ for every $l=1,\dots,l_1-l_0$.
Then there exists $h\in H_{\spa \Gamma_\bullet\to E_\bullet}$ with $d_\varphi(h,\Id) \le 1$ and $h\Gamma_{l_0}\subseteq\Gamma\subseteq h\Gamma_{l_1}$ such that 
\[\spa\Gamma\in \gr_{\to E}^{O(\varepsilon), g}\]
and 

\begin{align}\label{eq: Gamma cov}
\log \cov (h\Gamma_{l_1}/\Gamma/ h\Gamma_{l_0}) = O(C-\log \varepsilon).
\end{align}

\end{lem}
\begin{proof}
Denote $V_l = \spa \Gamma_l$ for every $l\in L(\Gamma_\bullet)$. 
We will first find $V_{l_0}\subset V\subset V_{l_1}$ such that $V\in \gr_{\to E}$. 
By Lemma \ref{lem: making q0}, there exists $q_0\in H^-$ such that $V_l = q_0V_l^\to $ for all $l\in L(\Gamma_\bullet)$. In addition, $d_{H^-}(q_0, \Id) = O(\varepsilon)$.
By decomposing $V_{l_0}^\to$ and $V_{l_1}^\to$ into $g_t$-eigenspaces, we deduce that there exists $V'\in \gr_E^g$ such that $V_{l_0}^\to\subset V'\subset V^\to_{l_1}$. Set $V = q_0V'$ and note that $V^\to = V'$. 

We will apply Claim \ref{claim: lattice approximation} to $\Gamma_{l_1}/\Gamma_{l_0}$ and $V/V_{l_0}$ in place $\Gamma$ and $V$. Note that $\log \lambda_1(\Gamma_{l_1}/\Gamma_{l_0}) = O(C)$ since the height $\HN(\Gamma_{l_1}/\Gamma_{l_0})_{H,1}$ depends linearly on $\partial^2\HN(\Gamma_{l_1}/\Gamma_{l_0})_{H,l}$ for $l=1,...,l_1-l_0-1$, which are bounded by $C$. Then there exists $\Gamma'\subseteq \Gamma_{l_1}/\Gamma_{l_0}$ with $d(\spa \Gamma', V/V_{l_0})\le \varepsilon_1$ and 
\[\log \cov\Gamma' = O(C-\log \varepsilon_1) + \frac{l_{l_1}-l}{l_{l_1}-l_{l_0}}\left(\log\cov \Gamma_{l_1}- \log\cov \Gamma_{l_0}\right),\] 
where $\varepsilon_1$ will be chosen later.
Let $\Gamma$ denote the inverse image of $\Gamma'$ in $\Lambda$; note that $\Gamma_{l_0}\subset\Gamma\subset\Gamma_{l_1}$. 
Then the covolume of $\Gamma$ satisfies Eq. \eqref{eq: Gamma cov}, provided that $\varepsilon_1 = \Theta(\varepsilon)$.
Since $\Gamma$ is $\varepsilon_1$ close to $V$ and $\varepsilon_1$ can be chosen sufficiently small, there exists $m\in \SL_n(\RR)$ with $m\spa\Gamma = V$ and $mV_l=V_l$ with $d_{\SL_n(\RR)}(m,\Id) = O(\varepsilon_1)$.
If $\varepsilon_1>0$ is sufficiently small, by Lemma \ref{lem: mult map} we can write $m = q_1h$ for $q_1\in H^{-0}$ and $h\in H$ with $d_{\SL_n(\RR)}(q_1,\Id), d_{\SL_n(\RR)}(h,\Id) = O(\varepsilon_1)$. 
It follows that $h\in B_{d_\varphi}(\Id;1)$, 
$hV_l = q_1^{-1}V_l \in \gr_{\to E_l}$, and $h\Gamma = q_1^{-1}V\in \gr_{\to E_l}$. 
Since $d_{\SL_n(\RR)}(q_0,\Id) = O(\varepsilon)$, we see that 
\begin{align*}
d_\gr(&\spa\Gamma, \gr_{E}^g) \le 
d_\gr(q_1^{-1}V,V) + d_\gr(V, \gr_{E}^g)  \\&
\le 
d_\gr(q_1^{-1}V,V) + d_\gr(q_0V^\to, V^\to) \\ &\le 
O(d_{\SL_n(\RR)}(q_1, \Id)+d_{\SL_n(\RR)}(q_0,\Id)) = O(\varepsilon+\varepsilon_1).
\end{align*}
Note that we can choose $\varepsilon_1 = \Theta(\varepsilon)$. The result follows.
\end{proof}
\begin{remark}\label{rem: where h}
In Lemma \ref{lem: enlarging the flag} we can require $h$ to be in any neighborhood of $\Id$. Such a choice will only change the implicit constant.
\end{remark}

\section{Alice's Strategy} 
\label{sec:alice_s_strategy}
Fix a $g$-template $f$ and a lattice $\Lambda\in X_n$ of covolume $1$ for the rest of the section. Let $T>0$ be a real number. Its choice will be formalized in Subsection \ref{sub:preliminaries}.
Recall that \[Y_{\Lambda, f}:=Y_{\Lambda, [f]}=\{h\in H:f^{h\Lambda}\sim f\}\subseteq H,\]
where $[f]$ is the set of $g$-template which are equivalent to $f$.

In this section we will construct Alice's strategy in the $(T, g)$-game of $Y_{\Lambda,f}$, see Subsection \ref{sub:hausdorff_games}, and study some of its properties as explained in Subsection \ref{sub:hausdorff_game_applied_to_traje}.
\begin{de}
A vertex $t_0$ of $f$ at $l$ is called a vertex of type $E_-\to E_+$ if $f_{E,l}(t)$ changes from $E_-$ to $E_+$ near $t_0$, i.e., 
\[E_\pm = \lim_{\varepsilon\searrow 0} f(t_0\pm \varepsilon).\]
\end{de}
Informally, and somewhat imprecisely, Alice's strategy goes as follows: 	
The initialization step is $h_0=\Id, T_0 = 0$. This defines $T_m= mT$.
In the $m$-th step Alice chooses a set $A_{m+1}$ of possible choices of $h_{m+1}$ to ensure that the following three conditions are satisfied:
\begin{enumerate}
	\item \label{cond: L=L}$L_f(T_{m+1}) = L(\HN(g_{T_{m+1}}h_{m+1}\Lambda))$.
	\item $\forall l\in L_f(T_{m+1}),~ \HN(g_{T_{m+1}}h_{m+1}\Lambda)_{V, l}\in \gr_{\to f_{E, l}({T_{m+1}})}$.
	\item \label{cond: error=O(1)} $\forall l\in L_f(T_{m+1}), ~ \HN(g_{T_{m+1}}h_{m+1}\Lambda)_{H, l}= f_{H, l}(T_{m+1}) + O(1)$.
\end{enumerate}
There will be three different types of steps.
In every step in which $f_{E,\bullet}(T_m) = f_{E,\bullet}(T_{m+1})$, we use Lemma \ref{lem:the right counting theorem} to construct a set $A_{m+1}$ of size $\Theta(\exp(T\delta(f_E(Tm))))$.
These steps are called \textbf{Standard Steps}.
If $f$ has a null vertex of type $E\to *$ in $[T_m,T_{m+1}]$, we will use a Standard Step as well.
If $f$ has a null vertex of type $*\to E$ in $[T_m,T_{m+1}]$, we will use Lemma \ref{lem: enlarging the flag} to construct a singleton $A_{m+1} = \{h_{m+1}\}$. These steps are called \textbf{Adding Steps}.
If $f$ has a non-null vertex in $[T_m,T_{m+1}]$, we will use Lemma \ref{lem:single advance} to construct a singleton $A_{m+1} = \{h_{m+1}\}$. These steps are called \textbf{Changing Steps}.

There are several problems with this approach:
\begin{enumerate}
	\item\label{prob: lattice enter time} We are not able to make a lattice enter the filtration $\HN(g_{T_m}h_m\Lambda)$ at the exact step we need.
	\item\label{prob: noise lattice} There may be undesired elements in $\HN(g_th_m\Lambda)_{V, l}$. 
	\item\label{prob: drift} The implicit constant in Condition \ref{cond: error=O(1)} might increase with $m$. 
\end{enumerate}

\subsubsection{Analogy to Error Correction} 
\label{ssub:analogy}
To illustrate the idea behind the solution of Problems \eqref{prob: lattice enter time} and \eqref{prob: drift}, we consider an analogy.
Suppose Alice has to drive an infinite route in a gridlike city. A navigation app provides to Alice a list of directions, of the form ``After 3 junctions turn left''.
Unfortunately, Alice always makes mistakes and never turns at the right junction.
If the app does not update its instructions, Alice will eventually drift far from her original route. 
Therefore, after every turn, the app moderately updates the instructions. 

The original route of Alice is analogous to the $g$-template $f$. The route she finally takes is the $g$-template $f^{h_\infty \Lambda}$. The turns Alice makes are analogous to the vertices of the $g$-template. The updates of the app are analogous to perturbing the $g$-template as described below, see Subsubsections \ref{ssub:the_adjustable f} and \ref{ssub:updating pi m, Gamma m}.

The analogous concept to $C$-separated $g$-template is a list of instructions with bounded amount of turns every $C$ kilometers.
To add Bob to the analogy, we can think of Bob as a random noise, which is weaker the more Alice focuses, and
Alice wants to focus the least, while still reaching her target.
\subsubsection{Overview of the Strategy} 
\label{ssub:formal strategy}


Formally, a strategy is a choice of a move $(h_0,T_0)$ in the initialization step and a map 
\[\sigma:(h_0, T_0; A_1, h_1,\dots,A_m, h_m)\mapsto A_{m+1},\] which assigns for every $m\ge 0$ and history of moves up to Alice's $m$-th step (excluded), the move of Alice in the current step $m$.
To compute the value of the game it is enough to define $\sigma_{\rm Alice}$ for histories that are consistent with $\sigma_{\rm Alice}$ up to step $m$. 

We will use auxiliary information, that is, Alice will maintain a filtration $\Gamma_\bullet^m$ of $\Lambda$ and a locally constant error function $\bpi^m = (\pi_l^m)_{l=0}^n$ with $\pi_l^m:U_l(f)\to \RR$. 
The strategy we will define will be a map 
\[\sigma_{\rm Alice}:(h_0, T_0; A_1, h_1,\dots,A_m, h_m, \Gamma^m_\bullet, \bpi^m)\mapsto (A_{m+1}, \Gamma^{m+1}_\bullet, \bpi^{m+1}).\]

For each history, the computation of $A_{m+1}, \Gamma_{\bullet}^{m+1}$, and $\bpi^{m+1}$ will have one of several forms.
For each step we will decide whether it is a \textbf{Standard Step}, an \textbf{Adding Step}, or a \textbf{Changing Step}.
In addition, we have three ways to update $\Gamma^m_\bullet$ to $\Gamma^{m+1}_\bullet$: we can \textbf{preserve} it, and then $\Gamma^{m+1}_\bullet = \Gamma^{m}_\bullet$; we can \textbf{add} a lattice to it, and then $\Gamma^{m+1}_\bullet \supset \Gamma^{m}_\bullet$; or we can remove a lattice from it, and then $\Gamma^{m+1}_\bullet \subset \Gamma^{m}_\bullet$. 
Further, there are two ways to update $\bpi^m$: we can either \textbf{anchor} it (see Subsubsection \ref{ssub:updating pi m, Gamma m}) or not anchor it, and then $\bpi^{m+1} = \bpi^{m}$. 





\subsection{Preliminaries} 
\label{sub:preliminaries}
Fix $C_2\ggg C_1\ggg T$ (see Subsection \ref{sub:notations} for the definition of $\ggg$). The strategy will depend on $T, C_1, C_2$.
Recall that $f$ is a $g$-template fixed for the entire section. Let $f'$ be the $g$-template $f$ shifted so that $f'_{H,l}(0) = 0$ for every $l$ with $0\le l\le n$. To shift $f$ to $f'$ we use the independent shift sequence with parameters $C$ and $(0)_{J\in \cG_f}$, for a large enough $C$. By Lemma \ref{lem: shift is equiv}, the $g$-templates $f, f'$ are equivalent. 
By Lemma \ref{lem:separated}, there is a $C_2$-significant and $C_2$-separated $g$-template $f''$ that is equivalent to $f'$. By the proof of Lemma \ref{lem:separated}, we get that $f''$ is a shift of $f'$, hence $f''_{H,l}(0) = 0$ for every $l$ with $0\le l\le n$. Since we care about $f$ only up to equivalence, we may assume that $f$ is $C_2$-significant, $C_2$-separated, and $f_{H,l}(0)=0$ for every $l$ with $0\le l\le n$.
\subsubsection{Adjusting $f$} 
\label{ssub:the_adjustable f}
Define the shift sequence 
\begin{align}\label{eq: what is brho}
\brho := (2l(n-l)C_1+ C_1)_{l=1}^{n-1}.
\end{align}
For every sequence of locally constant functions $\bpi = (\pi_l)_{l=1}^{n-1}$, $\pi_l:U_l(f)\to [-C_1, C_1]$, we get that $\brho+\bpi$ is an independent shift sequence with parameters $2C_1, C_1+\bpi$. 

By Lemma \ref{lem: stability or sep+sig}, since $C_2$ is large enough with respect to
$C_1$, for every nontriviality interval $J$ of $f$ at $l$ there exists a nontriviality interval $J'$ of $f^\brho$ at $l$ with $J'\subseteq J$. 
In addition, $f^\brho$ is also $C_2/2$-separated and $C_2/2$-significant.
\subsubsection{Preliminaries on Directions} 
\label{ssub:preliminaries_on_directions}
\begin{de}[Final direction]
\index{Final multiset}\hypertarget{bch}{}
Let $E_{-1} \subset E_0\subset E_1$ be submultisets of $\Eall$. We say that $E_0$ is \emph{final} in $E_{-1}\subset E_1$ if it is the maximum with respect to the partial order $\eteq$ of all $E$ of size $\#E_0$ which satisfy $E_{-1}\subset E \subset E_{1}$. 
\end{de}
\begin{obs}\label{obs: final implications}
Let $E_\bullet$ be a direction filtration and $E_{l}\subset E_{l'}$ be two consecutive elements in $E_\bullet$. Let $E$ be a multiset between them, i.e., $E_{l}\subset E\subset E_{l'}$.
Then $\delta(E_\bullet) \ge \delta(E_\bullet \cup \{E\})$, with equality if and only if $E$ is final in $E_{l}\subset E_{l'}$.
\end{obs}
\begin{de}[Scalar pairs]\label{de: scalar pair}
\index{Scalar pair}\hypertarget{bci}{}
We say that a pair of multisets $E\subset E'$ is a \emph{scalar pair} if $E' - E$ has no two distinct elements (i.e., has one element, possibly with multiplicity). 
\end{de}
\begin{obs}\label{obs: all between scalar}
If $E\subset E'$ is a scalar pair, then for every $l$ with $\#E<l<\#E'$ there is a unique $E\subset E'' \subset E'$ of size $l$, and it is final in $E\subset E'$.
\end{obs}
We will treat scalar pairs differently because of the following property they enjoy (which is the reason for their name):
\begin{obs}[Scalar action]\label{obs: scalar action}
Let $E\subset E'$ be a scalar pair, with $l:=\#E, l':=\#E'$.
Let $V\in \gr_{E}^g$ and $V'\in \grl{l'}_{E'}^g$ be two linear spaces with $V\subset V'$. 
Then the induced action $g_t\acts V'/V$ is multiplication by a scalar. 
\end{obs}
\begin{lem}\label{lem: final nonscalar has partial2}
Let $E_{-1}\subset E_0\subset E_1$ be multisets of sizes $l_{-1}<l_0<l_1$ respectively, such that $E_0$ is final in $E\subset E_1$. 
Then
\begin{align*}
\eta_{E_0} \ge \frac{l_0-l_{-1}}{l_1-l_{-1}}\eta_{E_1} + \frac{l_1-l_0}{l_1-l_{-1}}\eta_{E_{-1}},
\end{align*}
and equality holds if and only if $E_{-1}\subset E_1$ is a scalar pair.
\end{lem}
\begin{proof}
Denote the values with multiplicity of $E_1 - E_{-1}$ by 
\[\eta'_1\le \eta'_2\le \dots \le \eta'_{l_1-l_{-1}}.\] 
This sequence is constant if and only if $E_{-1}\subset E_1$ is a scalar pair.
It follows that the sequence $l_0'\mapsto \sum_{i=1}^{l_0'-l_{-1}}\eta'_i$ for $l_{-1}\le l_0'\le l_1$ is convex, and is linear if and only if $E_{-1}\subset E_1$ is a scalar pair.
The result follows since $E_0 = E_{-1} + \{\eta_i':i=l_0-l_{-1}+1,\dots,l_1-l_{-1}\}$.
\end{proof}
\begin{lem}\label{lem: not scalar at edge}
Let $f$ be a $g$-template, let $J$ be a nontriviality interval of $f$ at $l_0$ and let $t_0$ one of its endpoints. 
Let $E_\bullet = \lim\limits_{J\ni t\to t_0}f_{E, \bullet}(t)$, and 
let $l_{-1}, l_1$ be the adjacent elements in $L(E_\bullet)$ to $l_0$ satisfying $l_{-1}<l_0<l_1$. 
Then $E_{l_{-1}}\subset E_{l_{1}}$ is not a scalar pair.
\end{lem}
\begin{proof}
Let $J_0\subseteq J$ be an open interval in which $f$ has no vertices and such that $t_0$ is an endpoint of $J_0$. 
By the definition of $E_\bullet$, for every $t\in J_0$ we have $f_{E,\bullet}(t) = E_\bullet$. 
Consequently, $t\mapsto f_{H,l}(t)$ is linear in $J_0$ for every $l$ with $0\le l\le n$, and so is $t\mapsto \partial^2f_{H,l_0}(t)$.
Since $t_0\nin U_{l_0}(f)$, it follows that $\partial^2f_{H,l_0}(t_0) = 0$, while for every $t\in J_0$ we have $\partial^2f_{H,l_0}(t_0) \neq 0$. Consequently, $\frac{d}{dt}\partial^2f_{H,l_0}(t)$ does not vanish for $t\in J_0$.

From Observation \ref{obs: partial2 comutation} we get that 
\begin{align}\label{eq:local partial2}
\partial^2&f_{H,l_0}(t)
\\\nonumber &
= \frac{1}{l_0-l_{-1}} f_{H, l_{-1}} + \frac{1}{l_1-l_0} f_{H, l_{1}} - \frac{l_1-l_{-1}}{(l_1-l_0)(l_0-l_{-1})} f_{H, l_{0}}.
\end{align}
Differentiating Eq. \eqref{eq:local partial2} we obtain
\begin{align*}
0\neq&\frac{d}{dt}\partial^2f_{H,l_0}(t)
= \frac{1}{l_0-l_{-1}} \eta_{E_{l_{-1}}} + \frac{1}{l_1-l_0} \eta_{E_{l_{1}}} - \frac{l_1-l_{-1}}{(l_1-l_0)(l_0-l_{-1})} \eta_{E_{l_{0}}},
\end{align*}
which together with Lemma \ref{lem: final nonscalar has partial2} implies that $E_{l_{-1}}\subset E_{l_{1}}$ is not a scalar pair.
\end{proof}

\subsubsection{Constants Order of Magnitude} 
\label{ssub:constants_order_of_magnitude}
Our computations will use several universal constants, which depend on $\Eall$ but not on $T$, $f$, and $\Lambda$, and should obey certain relations. Specifically, we require that 
\index{gammastar@$\gamma_*$}\hypertarget{bcj}{} 
\begin{align}\label{eq: gamma dependency}
\gamma_{10} \ggg \gamma_9  \ggg \gamma_8 \ggg \gamma_7 \ggg \gamma_6 \ggg \gamma_5 \ggg \gamma_4 \ggg \gamma_3, \gamma_3' \ggg \gamma_2 \ggg \gamma_1, \gamma_1' \ggg 0.
\end{align}
Note that Eq. \eqref{eq: gamma dependency} states that both $\gamma_4 \ggg \gamma_3\ggg \gamma_2$ and $\gamma_4 \ggg \gamma_3'\ggg \gamma_2$ hold, and a similar interpretation applies to $\gamma_1, \gamma_1'$.
We will choose $\varepsilon>0$ small enough to satisfy several conditions that will be detailed throughout the proof. In particular, $\varepsilon$ is small enough so that the conclusions of Lemma \ref{lem:the right counting theorem}, Theorem \ref{thm:lin path}, and Lemma \ref{lem: E_i to V_i is functorial} are satisfied.

\subsubsection{Convenient Assumption} 
\label{ssub:convenient_assumption}
Throughout the section we evaluate the category flow of various $g$-templates on $\NN_0T=\{T_m:m\ge 0\}$. 
The category flows might be undefined at $T_m$ for some $m\ge 0$. 
To avoid these technicalities we will assume that all $g$-templates have no vertices on $\NN_0T$, i.e., they define direction filtrations on $\NN_0T$. 
To justify this assumption we describe two operations that guarantee that there will be no such vertices in all templates we ever use. 

First, translate $f$ by a random small number to obtain a $g$-template $f'$ with no vertices on $\NN_0T$. 
We can therefore assume w.l.o.g. that $f$ has no vertices on $\NN_0T$. The other $g$-templates we will encounter are shifts of $f$, namely $f^\brho$ for some values of $\brho$. Since $f$ has no non-null vertices on $\NN_0T$, it follows that $f^\brho$ has no such vertices as well. To ensure $f^\brho$ has no null vertices on $\NN_0T$, we perturb $\brho$. The perturbations affect Subsubsection \ref{ssub:the_adjustable f} below and the anchoring operation that we will define in Subsubsection \ref{ssub:updating pi m, Gamma m}.
This procedure will not affect the arguments.
\subsection{Setup for the Strategy \texorpdfstring{$\sigma_{\rm Alice}$}{of Alice}} 
\label{sub:the_strategy}
\subsubsection{The State} 
\label{ssub:the_state}
In each step we keep a state that consists of 
\begin{enumerate}[label=(S.\arabic*), ref=(S.\arabic*)]
  \item 
$h_{m}\in H$ and $T_m = Tm$. 
  \item 
  An error function $
  \index{pim@$\widetilde \bpi^m$}\hypertarget{bda}{}
  \bpi^m=(\pi_l^m)_{l=1}^{n-1}$, where $\pi_l^m: U_l(f)\to [-C_1, C_1]$ is locally constant.
  \item
  \index{Gammam@$\Gamma_\bullet^m$}\hypertarget{bdb}{}
  A filtration $\Gamma_\bullet^m$ of $\Lambda$.
\end{enumerate}
\index{Lambdamtilde@$\widetilde \Lambda^m$}\hypertarget{bdc}{}
\index{Gammamtilde@$\widetilde \Gamma^m_\bullet$}\hypertarget{bdd}{}
Denote $\widetilde \Lambda^m := g_{T_m}h_m\Lambda$ and $\widetilde \Gamma_l^m := g_{T_m}h_m\Gamma_l^m$ for every $l\in L(\Gamma^m_\bullet)$. 
We require the data in the state to satisfy the following conditions:
\begin{enumerate}[resume,label=(S.\arabic*), ref=(S.\arabic*)]
\item \label{asser: existence of lattices}
$L(\Gamma^m_\bullet)\subseteq L_f(T_m)$. We will see in Remark \ref{rem: where L = L} that a stronger condition on $L(\Gamma^m_\bullet)$ actually holds.
\item \label{asser: approximation of flag}
For every $l\in L(\Gamma^m_\bullet)$ we have 
\begin{align}\label{eq: aligning Gamma}
\spa\widetilde \Gamma^m_l\in \gr_{\to f_{E,l}(T_m)}^{\varepsilon, g}
\end{align}
and 
\begin{align}\label{eq: bpi = det Gamma - f}
\log \cov(\widetilde\Gamma^m_l) = f_{H,l}^{\brho + \bpi^m}(T_m) + O(1),
\end{align}
where $\brho$ is defined as in Eq. \eqref{eq: what is brho}.
\item \label{asser: bound on noise}
\index{Noise lattice}\hypertarget{bde}{}
For every two consecutive $l<l'\in L(\Gamma^m_\bullet)$, consider lattices $\widetilde\Gamma^m$ with $\widetilde\Gamma^m_l\subset \widetilde\Gamma^m \subset \widetilde\Gamma^m_{l'}$ that project to elements in $\HN(\widetilde\Gamma^m_{l'}/\widetilde\Gamma^m_l)$.
We call such $\widetilde \Gamma^m$ a \emph{noise lattice}. 
We require that the noise lattices are insignificant, in the sense that 
\begin{align}\label{eq: O(T) bound on noise}
\log\cov(\widetilde\Gamma^m_{l'}/\widetilde\Gamma^m/\widetilde \Gamma^m_l)=O(T).
\end{align}
The choice of $\widetilde \Gamma^m$ is equivalent to the choice of $\Gamma$ such that $\Gamma^m_l\subset \Gamma \subset \Gamma^m_{l'}$, where $\widetilde \Gamma^m = g_{T_m}h_m\Gamma$.
\end{enumerate}

The strategy selects a finite set $A_{m+1}$ of possible matrices $h_{m+1}\in B_{d_\varphi}(h_m; (1-\exp(-\alpha_\varphi T))^{1/\alpha_\varphi}\exp(-T_m))$, where $\alpha_\varphi$ is a constant for which $d_\varphi^{\alpha_\varphi}$ is a metric, see Subsection \ref{sub:Case of interest}.
The set $A_{m+1}$ is such that $d_\varphi(h_{m+1}, h_{m+1}')\ge 3^{1/\alpha_\varphi}\exp(-T_{m+1})$ for every pair of distinct elements $h_{m+1}, h_{m+1}'\in A_{m+1}$.

\index{Amp1tilde@$\widetilde A_{m+1}$}\hypertarget{bdf}{}
Instead of $A_{m+1}$, we can consider the set 
\begin{align}\label{eq: tilde A m}
\widetilde A_{m+1} := g_{T_m}A_{m+1}h_m^{-1}g_{-T_m},
\end{align}
which satisfies 
\index{BT@$B_T$}\hypertarget{bdg}{}
\begin{itemize}
  \item $\widetilde A_m \subseteq B_T:= B_{d_\varphi}(\Id;(1-\exp(-\alpha_\varphi T))^{1/\alpha_\varphi})$. 
  \item $d_\varphi(h_{m+1}, h_{m+1}')\ge 3^{1/\alpha_\varphi}\exp(-T)$ for every pair of distinct  $\widetilde h_{m+1}, \widetilde h_{m+1}'\in \widetilde A_{m+1}$.
\end{itemize}
This translation is possible due to the right-invariance of the semi-metric $d_\varphi$. The set $B_T$ is chosen in view of Eq. \eqref{eq: A_m location} in the definition of the dimension game. It has the properties that, for every $h_{m+1}'\in A_m$,
\begin{align}\label{eq: ball containment}
B_{d_\varphi}(h_{m+1};\exp(-T_{m+1})) \subseteq B_{d_\varphi}(h_m;\exp(-T_{m})).
\end{align}
Sometimes the set $A_{m+1}$ will consist of a single element. In that case, choosing $A_{m+1}$ is equivalent to choosing the unique element $h_{m+1}$ in it.

\begin{de}[Alignment of noise lattices]
\index{Final noise lattice}\hypertarget{bdh}{}
\index{Aligned noise lattice}\hypertarget{bdi}{}
Let $l_0<l_1$ be two consecutive elements in $L(\Gamma^m_\bullet)$.
We say that a noise lattice $\Gamma$ with $\Gamma^m_{l_0}\subset\Gamma\subset\Gamma_{l_1}^m$ is \emph{$E$-aligned} at step $m$ if $\widetilde \Gamma^m = g_{T_m}h_m\Gamma$ satisfies $\spa \widetilde \Gamma^m\in U_\varepsilon(\gr_E)$, for some $f_{E,l_0}(T_m)\subset E \subset f_{E,l_1}(T_m)$, where $\rk \Gamma = l$. We say that $\Gamma$ is \emph{final} at step $m$ if it is $E$ aligned and $E$ is
final in $f_{E,l_0}(T_m)\subset f_{E,l_1}(T_m)$.
\end{de}

\subsubsection{Initialization} 
\label{subs:initialization}
Sets $h_0 = \Id$, $T_0 = 0$, and $\bpi^0 = 0$ everywhere. Since $f_{H,l}(0)=0$ for every $l$ with $0\le l\le n$, it follows that $L_{f^{\brho+\bpi^0}}(t) = \{0,n\}$, and so we choose $\Gamma_\bullet^0 := (\{0\}\subseteq \Lambda)$.

\subsubsection{Updating $\bpi^m$ and $\Gamma^m_\bullet$.} 
\label{ssub:updating pi m, Gamma m}
In this subsubsection we describe how $\bpi^m, \Gamma^m_\bullet$ are updated at each step.

There are three different ways to update $\Gamma^m_\bullet$ to $\Gamma^{m+1}_\bullet$: \textbf{Preserve}, \textbf{Add}, and \textbf{Remove}, which respectively sets $\Gamma^{m+1}_\bullet = \Gamma^{m}_\bullet$, adds a single element to $\Gamma^m_\bullet$ to get $\Gamma^{m+1}_\bullet$, and removes a single element from $\Gamma^m_\bullet$ to get $\Gamma^{m+1}_\bullet$.
We will use preserve and remove at Standard Steps, preserve at Changing Steps, and add at Adding Steps.

\begin{de}[Anchoring]\label{de: anchor}
\index{Anchor}\hypertarget{bdj}{}
By default, $\bpi^{m+1}=\bpi^m$.
After every adding or changing step, we will use the \textbf{Anchor} operation. For each $1\le l\le n-1$, to anchor $\pi^{m+1}_l$ is to fix $\pi^{m+1}_l(T_{m+1})$ so that equality is attained in the formula
\[\log \cov(\widetilde \Gamma_l^{m+1}) = f_{E,l}(T_{m+1}) + \rho_l(T_{m+1}) + \pi^{m+1}_l(T_{m+1}),\]
i.e.,
\[\pi^{m+1}_l(T_{m+1}) := \log \cov(\widetilde \Gamma_l^{m+1}) - f_{E,l}(T_{m+1}) - \rho_l(T_{m+1}).\]
Note that this operation changes $\pi^m_l$ on the entire connected component of $T_m$ on $U_l(f)$.
This is the only operation that changes $\bpi^m$. If we do not anchor $\pi^{m+1}_l$, then $\pi^{m+1}_l = \pi^{m}_l$. 
\end{de}


\subsection{Geometric Picture} 
\label{sub:geometric_picture}
We will describe here geometric features of the behavior of the $(T, g)$-game. 

\index{hm0m1@$\widetilde h^{m_0\to m_1}$}\hypertarget{bea}{}
Assume the game has been played, and has generated a sequence $h_0,A_1, h_1,\ldots$. 
For every $0\le m_0 \le m_1< \infty$ such that $m_0 <\infty$, denote 
\[\widetilde h^{m_0\to m_1}:= g_{T_{m_0}}h_{m_1}h_{m_0}^{-1}g_{-T_{m_0}}.\]
Eq. \eqref{eq: ball containment} implies that $d_\varphi(h_{m_0}, h_{m_1}) \le \exp(-T_{m_0})$ and for every $m_0\le m_1\le m_2$ we have $d_\varphi(\widetilde h^{m_0\to m_1}, \widetilde h^{m_0\to m_2}) \le \exp(-(T_{m_1}-T_{m_0}))$. 
Hence, the limit $\lim_{m\to\infty}h_m$ exists. Denote $h_\infty:=\lim_{m\to\infty}h_m$, and for every $0\le m_0<\infty$ denote 
\[\widetilde h^{m_0\to \infty}:= g_{T_{m_0}}h_{\infty}h_{m_0}^{-1}g_{-T_{m_0}}.\]

For every $m\ge 0$ denote $\widetilde h^{m+1} = \widetilde h^{m\to m+1}$. 
Note that 
\begin{align}
h_m = g_{-T_m}\widetilde h^m g_T\widetilde h^{m-1}g_T\cdots g_T\widetilde h_1 g_{T_0}.
\end{align}

\begin{remark}[Relative notation]\label{rem:relative notation}
Some objects are defined from the absolute perspective of time $0$, like $\Lambda$, $\Gamma^m_\bullet$, $h_m$, and $A_m$. 
Sometimes it is more convenient to consider these objects from the perspective of time $T_m$. Whenever we do so, we denote the relative objects with $\widetilde\ $ (tilde). 
For example, $\widetilde \Gamma_\bullet^{m} = g_{T_m}h_m\Gamma_\bullet^{m}$, $\widetilde h_{m+1} = g_{T_m} h_{m+1}h^{-1}_mg_{-T_m}$, and $\widetilde A_{m+1} = g_{T_m} A_{m+1}h^{-1}_mg_{-T_m}$ are the relative versions of $\Gamma^m_\bullet$, $h_m$, and $A_m$, respectively. The advantage of the relative perspective (with tilde) is that the formulas do not involve $T_m$, and are time invariant. The advantage of the absolute perspective is the locally constant nature of various definitions, such as $\Gamma_\bullet^{m}$. 
\end{remark}

The following is a corollary of Lemma \ref{lem: making q0}. 
\begin{cor}\label{cor: blade behavior over time}
There exists an $\varepsilon_0>0$ small enough such that the following holds. 
For every $\varepsilon$ with $0<\varepsilon<\varepsilon_0$, $l$ with $0\le l\le n$, two integers $0\le m_0\le m_1$, and $V\in \gr$ such that $\widetilde V^{m}:=g_{T_m}h_mV\in \gr_{\to E}$ 
for all $m$ with $m_0\le m\le m_1$ and $\widetilde V^{m_0} \in \gr_{\to E}^{\varepsilon, g}$, we have
\[\log d_H(\widetilde V^{m_1}, \gr_{E}^g) = -\Omega(-\log \varepsilon + (T_{m_1}-T_{m_0})),\]
i.e., $d_H(\widetilde V^{m_1}, \gr_{E}^g)$ decays exponentially in $m_1-m_0$.
In addition, for all $l$-blades $w\in \bigwedge^l V$
we have 

\begin{align}\label{eq: change norm}
\log \|g_{T_{m_1}}h_{m_1}w\| = \log \|g_{T_{m_0}}h_{m_0}w\| + (m_1-m_0)T\eta_E + O(1).
\end{align}

\end{cor}
\begin{proof}
Since $\widetilde V^{m_0}\in \gr_{\to E}^{\varepsilon, g}$, by Lemma \ref{lem: making q0} there exists $q_0\in H^-$ such that $d_{H^-}(q_0, \Id) = O(\varepsilon)$ and $\widetilde V^{m_0} = q_0(\widetilde V^{m_0})^\to$. Apply Lemma \ref{lem: understanding V to E} to $\widetilde V^{m_0}, q_0$, and $E$. We obtain two open sets $U_1\subseteq H_{\widetilde V^{m_0}\to E}, U_2\subseteq H\big((\widetilde V^{m_0})^\to\big)$ and two maps 
\[\check \funh:U_1\xrightarrow\sim U_2,\text{ and }\check \funq:U_2 \to H^{-0},\]
where $\check \funh$ is a diffeomorphism such that for every $h\in U_1$ we have 
\begin{align}\label{eq: hq behanior}
hq_0 = \check \funq(h)\check \funh(h). 
\end{align}
Moreover, since $d_{H^-}(q_0,\Id)$ is arbitrarily small, $U_1$ and $U_2$ are arbitrarily large. Specifically, for some $\varepsilon_0>0$ small enough which depends only on $n$, we get that $B_{d_\varphi}(\Id;1)\subseteq U_1$.
We will apply these functions to $\widetilde h^{m_0\to m_1}$:
\begin{align*}
\widetilde V^{m_1} &= g_{(m_1-m_0)T} \widetilde h^{m_0\to m_1} \widetilde V^{m_1} =  g_{(m_1-m_0)T} \widetilde h^{m_0\to m_1} q_0 \cdot (\widetilde V^{m_1})^\to\\&=
g_{(m_1-m_0)T}\check \funq\big(\widetilde h^{m_0\to m_1}\big) \check \funh\big(\widetilde h^{m_0\to m_1}\big) \cdot (\widetilde V^{m_1})^\to 
\\&= 
g_{(m_1-m_0)T}\check \funq\big(\widetilde h^{m_0\to m_1}\big) \cdot (\widetilde V^{m_1})^\to 
\\&= 
g_{(m_1-m_0)T}\check \funq\big(\widetilde h^{m_0\to m_1}\big) g_{-(m_1-m_0)T}\cdot (\widetilde V^{m_1})^\to.
\end{align*}
Since 
\begin{itemize}
  \item 
$\check \funq$ depends differentiably on $q_0$,
\item if $q_0=\Id$ then $\check \funq\equiv \Id$,
\item $d_{H^-}(q_0, \Id) = O(\varepsilon)$,
\end{itemize}
it follows that for all $h\in B_{d_\varphi}(\Id;1)$ we have 
$d_{H^{-0}}(\check \funq(h),\Id) = O(\varepsilon)$ for every $h\in B_{d_\varphi}(\Id;1)$. In particular, $d_{H^{-0}}(\check \funq(\widetilde h^{m_0\to m_1}),\Id)=O(\varepsilon)$. 
Write $\check \funq\big(\widetilde h^{m_0\to m_1}\big) = q^-_1q_1^0$ where $q_1^-\in H^-, q_1^0\in H^0$. We deduce that $d_{H^{-1}}(q^-_1, \Id) = O(\varepsilon)$. 
Altogether, we get
\begin{align*}
d_{\gr}&(\widetilde V^{m_1}, \gr_E^g) \\&\le 
d_{\gr}(g_{(m_1-m_0)T}\check \funq\big(\widetilde h^{m_0\to m_1}\big) g_{-(m_1-m_0)T}\cdot (\widetilde V^{m_1})^\to, \gr_E^g)
\\&=d_{\gr}(g_{(m_1-m_0)T}q_1^-q_1^0 g_{-(m_1-m_0)T}\cdot (\widetilde V^{m_1})^\to, \gr_E^g)
\\&=d_{\gr}(g_{(m_1-m_0)T}q_1^- g_{-(m_1-m_0)T}q_1^0\cdot (\widetilde V^{m_1})^\to, \gr_E^g)
\\&\le d_{\gr}(g_{(m_1-m_0)T}q_1^- g_{-(m_1-m_0)T}q_1^0\cdot (\widetilde V^{m_1})^\to, q_1^0\cdot (\widetilde V^{m_1})^\to)
\\&= O\left(d_{H^-}(g_{(m_1-m_0)T}q_1^- g_{-(m_1-m_0)T}, \Id)\right).
\end{align*}
The distance result follows from the uniform contraction rate of conjugation of $g_t$ on $H^-$.

To show Eq. \eqref{eq: change norm}, note that 
\[g_{T_{m_1}}h_{m_1}w = g_{(m_1-m_0)T} \widetilde h^{m_0\to m_1}g_{T_{m_0}}h_{m_0}w.\]
By Eq. \eqref{eq: hq behanior}, we can factor $\widetilde h^{m_0\to m_1}$ as 
\[\widetilde h^{m_0\to m_1} = \check \funq(\widetilde h^{m_0\to m_1})\check \funh(\widetilde h^{m_0\to m_1})q_0^{-1}.\]
Further by the definition of $q_0$, \[q_0^{-1}\widetilde V^{m_0} = q_0^{-1}g_{T_{m_0}}h_{m_0}V \in \gr_{E}^g,\]
and hence $q_0^{-1}g_{T_{m_0}}h_{m_0}w\in \bigwedge^l(q_0^{-1}g_{T_{m_0}}h_{m_0}V)$ is a $g_t$-eigenvector of $\bigwedge^l \RR^n$, with eigenvalue $\exp(\eta_E t)$.
Consequently, 
\begin{align*}
\log \|g_{T_{m_1}}h_{m_1}w\| &= 
\log \|g_{(m_1-m_0)T} \widetilde h^{m_0\to m_1}g_{T_{m_0}}h_{m_0}w\|
\\&=
\log \|g_{(m_1-m_0)T} \check \funq(\widetilde h^{m_0\to m_1})\check \funh(\widetilde h^{m_0\to m_1})q_0^{-1}g_{T_{m_0}}h_{m_0}w\|
\\&=
\log \|g_{(m_1-m_0)T} \check \funh(\widetilde h^{m_0\to m_1})q_0^{-1}g_{T_{m_0}}h_{m_0}w\| + O(1).
\end{align*}
The last equality holds because \[d_{H^{-0}}(g_{(m_1-m_0)T} \check \funq(\widetilde h^{m_0\to m_1})g_{-(m_1-m_0)T}, \Id) = O(1).\]
Since $\check \funh(\widetilde h^{m_0\to m_1})$ preserves $(\widetilde V^{m_0})^\to = q_0^{-1}g_{T_{m_0}}h_{m_0}V$, it stabilizes the vector $q_0^{-1}g_{T_{m_0}}h_{m_0}w$ and hence
\begin{align*}
\log \|g_{T_{m_1}}h_{m_1}w\| &= 
\log \|g_{(m_1-m_0)T} q_0^{-1}g_{T_{m_0}}h_{m_0}w\| + O(1)
\\&=
(m_1-m_0)T\eta_E +
\log \|q_0^{-1}g_{T_{m_0}}h_{m_0}w\| + O(1)
\\&=
(m_1-m_0)T\eta_E +
\log \|g_{T_{m_0}}h_{m_0}w\| + O(1),
\end{align*}
as desired.
\end{proof}
\subsection{The Different Steps} 
\label{sub:the_different_steps}
Here we describe the three types of steps of $\sigma_{\rm Alice}$: adding, changing and standard. For each type we provide the condition that should be satisfied so that step $m$ has this type, the way $A_{m+1}$ is chosen, and the way $\bpi^m$ and $\Gamma^m_\bullet$ are chosen.

\subsubsection{Adding Step} 
\label{ssub:adding_step}
 \index{Steps of Alice!adding}\hypertarget{beb}{}
To determine whether step $m$ is an Adding Step we do as follows.
Consider a null vertex $t_0$ of $f$ at $l$ of the form $*\to E$. By Lemma \ref{lem: stability or sep+sig}, for every independent shift sequence $\brho'$ with parameters $2C_1, (\nu_J)_{J\in \cG_f}\in [0,2C_1]^{\cG_f}$ there is a unique vertex of $f^{\brho'}$ that corresponds to $t_0$, denoted $t_0^{\brho'}$. Since $t_0^{\brho'}$ depends continuously on $\brho$, the following set is an interval:
\[[a_{t_0}, b_{t_0}] := \left\{t_0^{\brho'}:\begin{matrix*}[l] \brho'
\text{ is an independent shift sequence}
\\
\text{with parameters $2C_1, (\nu_J)_{J\in \cG_f}\in [0,2C_1]^{\cG_f}$}
\end{matrix*} \right\}.\]
Since $f$ is $C_1$-separated, $b_{t_0}-a_{t_0} = \Theta(C_1)$.
Denote $m_{t_0}^a := \lfloor a_{t_0}/T\rfloor$ and $m_{t_0}^b:=\lfloor b_{t_0}/T\rfloor$. 
By Lemma \ref{lem: stability or sep+sig}, for every such $\brho'$ the $g$-template $f^{\brho'}$ has exactly one vertex in $[a_{t_0}, b_{t_0}]$, namely $t_0^{\brho'}$. 
Denote by $m_{t_0}$ the first $m'\in [m_{t_0}^{a},m_{t_0}^{b}]$ for which \[L_f^{\brho+\bpi^{m_{t_0}^a}}(T_{m'+1})= L_f^{\brho+\bpi^{m_{t_0}^a}}(T_{m'}) \cup \{l\}.\]
As we will see below, there will be only Standard Steps at $m \in [m_{t_0}^a,m_{t_0}^b] \setminus \{m_{t_0}\}$,
and Standard Steps do not change $\bpi^m$. Hence, $m_{t_0}$ is the unique $m'\in [m_{t_0}^{a},m_{t_0}^{b}]$ for which 
\[L_f^{\brho+\bpi^{m_{t_0}^a}}(T_{m'}) = L_f^{\brho+\bpi^{m_{t_0}^a}}(T_{m_{t_0}^a}),\] and 
\[L_f^{\brho+\bpi^{m_{t_0}^a}}(T_{m'+1}) = L_f^{\brho+\bpi^{m_{t_0}^a}}(T_{m_{t_0}^a})\sqcup \{l\} = L_f^{\brho+\bpi^{m_{t_0}^a}}(T_{m_{t_0}^b+1}).\]

\noindent\hypertarget{cond:add}{\textbf{Condition-Adding:}}
Step $m$ is an \textbf{Adding Step} if $m \in [m_{t_0}^a, m_{t_0}^b]$ and $m=m_{t_0}$ for some null vertex $t_0$ of $f$ of type $*\to E$. Note that there may exist a unique such null vertex $t_0$ with $m \in [m_{t_0}^a, m_{t_0}^b]$. 

The following condition is equivalent to \hyperlink{cond:add}{Condition-Adding}:

\noindent\textbf{Condition-Adding$'$:} 
Step $m$ is an \textbf{Adding Step} if \[L(\Gamma^m_\bullet) = L_{f^{\brho+\bpi^{m}}}(T_{m}) \subsetneq L_{f^{\brho+\bpi^{m}}}(T_{m+1}).\]

\noindent\textbf{Construction of $A_{m+1}$:}
When $m$ is an Adding Step, $A_{m+1}$ consists of a single element $h_{m+1}$. Let us explain how $\widetilde h_{m+1}$ is calculated.
Denote by $l^-<l<l^+$ the consecutive elements of $l$ in $L_{T_{m+1}}(f^{\brho+\bpi^m})$. 
We will use Lemma \ref{lem: enlarging the flag} to obtain a lattice $\Gamma$ with $\widetilde \Gamma^m_{l^-} \subset \Gamma \subset \widetilde\Gamma_{l^+}^m$ and an element $\widetilde h_{m+1}\in B_T\cap H_{\spa\Gamma_\bullet^m\to f_{E,l}^{\brho+\bpi^m}}$ such that
\begin{align}\label{eq:adding step - new space}
\widetilde h_{m+1}\spa\Gamma\in \gr^{O(\varepsilon), g}_{\to f^{\brho+\bpi^m}_{E,l}(T_{m+1})}. 
\end{align}


To use Lemma \ref{lem: enlarging the flag} we have to obtain some bound on the noise lattices lying between $\Gamma_l^m$ and $\Gamma_{l'}^m$. 
In Corollary \ref{cor: bound on noise} below we will find $\gamma_4 > 0$ depending only on $\Eall$
such that every noise lattice $\Gamma'$ at an Adding Step with
$\widetilde \Gamma_l^m\subset \Gamma'\subset \Gamma_{l'}^m$ 
has $\log \cov (\widetilde\Gamma_{l'}^m/\Gamma'/\widetilde\Gamma_l^m) \le \gamma_4 T$.



By the bound on the noise lattices and Lemma \ref{lem: enlarging the flag} we have
\[\left|\log \cov (\widetilde\Gamma_{l'}^m/\Gamma'/\widetilde\Gamma_l^m)\right| \le \gamma_5T,\]
for some $\gamma_5>0$ depending only on $\Eall$. 

\noindent\textbf{Updating the auxiliary parameters:}
We add $g_{-T_m}h_{m}^{-1}\Gamma$ to the filtration $\Gamma^m_\bullet$,
anchor $\bpi_l^{m+1}$, and get $|\bpi_l^{m+1}(T_{m+1})| \le \gamma_5T$.
\begin{remark}\label{rem: where L = L}
We will compare the two sets $L(\Gamma^m_\bullet)$ and $L_{f^{\brho+\bpi^m}}(T_m)$ for every step $m$. 
By \hyperlink{cond:add}{Condition-Adding}, an index is added to $L(\Gamma^m_\bullet)$ exactly at Adding Steps, at $m=m_{t_0}$ for some null vertex $t_0$ of $f$ of type $*\to E$.

Although the behavior of $L_{f^{\brho+\bpi^m}}(T_m)$ is different from that of $L(\Gamma^m_\bullet)$, we know that when going from $m=m_{t_0}^a$ to $m=m_{t_0}^b$, a single index is added to $L_{f^{\brho+\bpi^m}}(T_m)$ (this index might be added to $L_{f^{\brho+\bpi^m}}(T_m)$ at a step different from $m_{t_0}$).

We will see in the Standard Step that near a null vertex of the form $E\to *$ we do have $L(\Gamma^m_\bullet) = L_{f^{\brho+\bpi^m}}(T_m)$, and hence for every 
$m\not\in \bigsqcup_{t_0}[m_{t_0}^a, m_{t_0}^b]$ we have $L(\Gamma^m_\bullet) = L_{f^{\brho+\bpi^m}}(T_m)$, where the union is taken over the null vertices of $f$ of type $*\to E$.
Note that $\bigsqcup_{t_0}[T_{m_{t_0}^a}, T_{m_{t_0}^b}]$ is contained in a union of $\Theta(C_1)$-neighborhoods of the null vertices of $f$. 
\end{remark}
\subsubsection{Changing Step} 
 \label{ssub:change_step}
 
\noindent\hypertarget{cond:change}{\textbf{Condition-Changing:}}
\index{Steps of Alice!changing}\hypertarget{bec}{}
Step $m$ is a \textbf{Changing Step} 
if $f$ has a non-null vertex $t_0$ in $[T_m, T_{m+1})$. 

In that case $f^{\brho + \bpi^m}$ also has a non-null vetex in $[T_m, T_{m+1})$. Since $f^{\brho + \bpi^m}$ is $C_2$-separated, it follows that $L_{f^{\brho+\bpi^m}}(T_m) = L_{f^{\brho+\bpi^m}}(T_{m+1})$, and this condition cannot hold simultaneously with \hyperlink{cond:add}{Condition-Adding}.

\noindent\textbf{Construction of $A_{m+1}$:}
When $m$ is a Changing Step, $A_{m+1}$ consists of a single element $h_{m+1}$. Let use explain how $\widetilde h_{m+1}$ is calculated.
By the $C_2$-separatedness of $f$ we get that $m$ is at a distance of at least $C_2/(2T)$ from an Adding Step, and hence, by Remark \ref{rem: where L = L}, we get that $L(\Gamma_\bullet) = L_{T_m}(f^{\brho+\bpi^m})$.
Denote $E_\bullet= f_{E,\bullet}^{\brho+\bpi^m}(T_{m})$ and $E_\bullet'=f_{E,\bullet}^{\brho+\bpi^m}(T_{m+1})$.

By Lemma \ref{lem:single advance}, there is $\widetilde h_{m+1} \in B_T$ such that 
$\widetilde h_{m+1}\spa \widetilde \Gamma^m_l \in \gr_{\to E_l'}$ for every $l\in L_{T_{m+1}}(f^{\brho+\bpi^m})$. 
The uniformity in Lemma $\ref{lem:single advance}$ shows that for $T$ large enough (depending only on $\Eall$) we have 
\begin{align}\label{eq: changeing step-space locantion}
g_{T}\widetilde h_{m+1}\spa \widetilde \Gamma^m_l \in U_{\varepsilon}\left(\gr_{E_l'}^g\right).
\end{align}

In addition,
\begin{align}\label{eq: change in Gamma size}
\cov(g_{T}\widetilde h_{m+1}\widetilde \Gamma^m_l)
&= \Theta(\cov(\widetilde h_{m+1}\widetilde \Gamma^m_l))\cdot \exp(O(T))\\ 
\nonumber&= \Theta(\cov(\widetilde \Gamma^m_l))\cdot \exp(O(T))
\end{align}

\noindent\textbf{Updating the auxiliary parameters:}
We anchor $\pi^{m+1}_l$ for every $l$ at which $f$ has a non-null vertex in $[T_m, T_{m+1})$. By Eq. \eqref{eq: change in Gamma size} and Definition \ref{de: anchor}, $\pi^{m+1}_l(T_{m+1}) - \pi^{m}_l(T_{m+1}) = O(T)$. 
In this step we preserve $\Gamma_\bullet^m$. 

\subsubsection{Standard Step} 
\label{ssub:standard_step}
\index{Steps of Alice!standard}\hypertarget{bed}{}
\noindent\textbf{Condition-Standard:}
Step $m$ is a \textbf{Standard Step} if it is neither an Adding Step nor a Changing Step.

\textbf{Construction of $A_{m+1}$:}
Denote $L = L(\Gamma^m_\bullet)$ and by $E_\bullet$ the filtration $f_{E,\bullet}(T_m)$ restricted to $L$. 
Let $\widetilde A_{m+1}'$ be the set of $\widetilde h\in H_{g_{T_m}h_m\spa \Gamma^m_\bullet\to E_\bullet}$ that is given by Lemma \ref{lem:the right counting theorem} and satisfies:
  \begin{itemize}
    \item For every two distinct elements $\widetilde h_{m+1}, \widetilde h_{m+1}'\in \widetilde A_{m+1}'$ we have $d_\varphi(\widetilde h_{m+1}, \widetilde h_{m+1}')>3^{\alpha_\varphi}\exp(-T)$.
    \item $\#A_{m+1}' = \Theta(\exp(T\delta(E_\bullet)))$.
  \end{itemize}
For each noise lattice $\widetilde \Gamma^m$ with $\widetilde \Gamma^m_l\subset \widetilde \Gamma^m\subset \widetilde \Gamma^m_{l'}$,
let $E$ with $E_l\subset E \subset E_{l'}$ be a non-final multiset of size $\rk \Gamma^m$. 

We will mark $\widetilde h_{m+1}\in \widetilde A_{m+1}'$ as $E$-\emph{forbidden} if 
\[\widetilde h_{m+1}\widetilde \Gamma^m \in U_{\varepsilon}(\grl{\rk \Gamma}^{g}_{E})\]
and there is an
\[\widetilde h_{m+1}'\in B_{d_\varphi}(h_{m+1}; 3^{\alpha_\varphi} \exp(-T))\cap H_{g_{T_m}h_m\spa \Gamma^m_\bullet\to f_E^{\brho+\bpi^m}(T_{m})}\]
such that $\widetilde h_{m+1}'\widetilde \Gamma^m \in \grl{\rk \Gamma}^{\varepsilon, g}_{\to E}$.

Let $\widetilde A_{m+1}\subseteq \widetilde A_{m+1}'$ be the set of non-forbidden elements.
By Lemma \ref{lem:the right counting theorem}, the number of $E$ forbidden elements is 
\[O(\exp(T\delta(E_\bullet \cup \{E\}))).\]
By Observation \ref{obs: final implications}, we have
$
\delta(E_\bullet \cup \{E\})
< \delta(E_\bullet)
$.
Consequently, the number of forbidden elements is negligible relative to $\#\widetilde A_{m+1}'$, and 
$\#\widetilde A_{m+1} = \Theta(\exp(T\delta(E_\bullet)))$ for every $T$ large enough.
Following Subsubsection \ref{ssub:the_state} and Eq. \eqref{eq: tilde A m}, define $A_{m+1} := g_{-T_m} \widetilde A_{m+1} g_{T_m}h_m$.

\noindent\textbf{Updating the auxiliary parameters:}
In a Standard Step we do not anchor $\bpi$. 
If there is a vertex of type $f_{E,l}^{\brho+\bpi^m}(T_m) \to *$ in $[T_m, T_{m+1})$, then we remove $\Gamma_l^m$ from the sequence $\Gamma_\bullet^m$. Otherwise we preserve $\Gamma_\bullet^m$.

\subsection{Proof that the Assertions Hold} 
\label{sub:proof_that_the_strategy_works}
We have to prove four properties: 
First, we need to show that with all the anchoring we did, we preserved the property that $\bpi^m_l(t)\in [-C_1, C_1]$ for every $t\ge 0$ and $m\ge 0$.
We also need to prove the three Assertions \ref{asser: existence of lattices}, \ref{asser: approximation of flag} and \ref{asser: bound on noise}. 
Only the proof of Assertion \ref{asser: bound on noise} does not follow from the above discussion.

To define the behavior of Alice in step $m$, we need to show that these assertions hold up to step $m$.
To establish that the assertions hold at step $m$, it will be convenient to assume that the whole game has been played while Alice followed the strategy $\sigma_{\rm Alice}$. The proof will not depend on the moves after step $m$, hence this assumption is harmless.

\subsubsection{Bound on $\bpi$} 
\label{ssub:bound_on_bpi}
The bound on $\bpi^m_l$ follows from the discussion. The first time we encounter a nontriviality interval $J$ we anchored $\pi_l^m$ to a value of size $O(\gamma_5T)$, and afterwards changed by at most $O(T)$ every time there was a non-null vertex $E\to E'$ at $l$. Since there are at most $l(n-l)$ such vertices in $J$,  the result follows provided $C_1$ is large enough.

\subsubsection{Existence of Lattices - Assertion \emph{\ref{asser: existence of lattices}}} 
\label{ssub:existence_of_lattices_assertion_asser: existence of lattices}
By the conditions of the Adding Step and the removing of a lattice from $\Gamma_m$,
we get that if $l\in L(\Gamma_\bullet^m)$, then $l\in L_{f^{\brho+\bpi^{m'}}}(T_m)$ for some $m'$, and hence $l\in L_f(T_m)$, as desired.

\subsubsection{Bounds on the Flag Elements - Assertion \emph{\ref{asser: approximation of flag}}} 
\label{ssub:bounds_on_the_flag_elements}
Eq. \eqref{eq: aligning Gamma} holds at Standard Steps by Corollary \ref{cor: blade behavior over time}, at Adding Steps by Eq. \eqref{eq:adding step - new space} and Corollary \ref{cor: blade behavior over time}, and at Change Steps by Eq. \eqref{eq: changeing step-space locantion}.

As for Eq. \eqref{eq: bpi = det Gamma - f}, 
for every $m$ in which we anchored $\pi^{m+1}_{l}$, step $m$ was either a Changing Step which changed ${l}$ or an Adding Step which added $\Gamma_l^{m+1}$ to the filtration $\Gamma_\bullet^{m+1}$. 
In these step we constructed $\Gamma_l^{m+1}$ and $h_{m+1}$ to satisfy this very assertion.
For any step $m' > m$ in which $\Gamma_l^{m''} = \Gamma_l^{m}$ for every $m''$ with $m\le m''\le m'$, the assertion follows from Corollary \ref{cor: blade behavior over time}.





\subsubsection{Bounds on Noise - Assertion \emph{\ref{asser: bound on noise}}} 
\label{ssub:bonds_on_noise}
The rest of the subsection is devoted to the proof of Assertion \ref{asser: bound on noise}. 
We will bound the noise at step $m$ using the bound on the noise at previous steps. To make sure the noise does not amplify, we will identify regions where the noise decays.

We need to prove that the noise is bounded by $O(T)$. To this end we will show that for every lattice $\Gamma$ such that $\widetilde\Gamma_l^m\subset \Gamma\subset \widetilde\Gamma_{l'}^m$, with $l,l'$ consecutive elements in $L(\Gamma^m_\bullet)$ we have
\begin{align}\label{eq: desired noise bound}
\log \cov (\widetilde\Gamma_{l}^m/\Gamma/\widetilde\Gamma_{l'}^m) \le \gamma_{10} T.
\end{align}
The constant $\gamma_{10}$ is larger than the bound on the noise we used in the Adding Step \ref{ssub:adding_step}.
When $E = f_{E, {l}}(T_m) \subset E' = f_{E, {l'}}(T_m)$ is a non-scalar pair for $m=m_0,\dots,m_1$ and if $m-m_0$ is large enough, we will have a better bound of $\gamma_4 T$ in Eq. \eqref{eq: desired noise bound} for $m=m_1$, see Eq. \eqref{eq: gamma dependency}.

\begin{de}[Rectangle]
\index{Rectangle}\hypertarget{bee}{}
Let $[m_0,m_1]$ with $0\le m_0<m_1$ be a maximal interval such that $l_0<l_1$ are consecutive elements in $L(\Gamma^m_\bullet)$ for all $m$ with $m_0\le m\le m_1$. Note that we may have $m_1 = \infty$.
We call such a configurations of $(m_0<m_1, l_0<l_1)$ a \emph{rectangle}. Every rectangle satisfies exactly one of the following conditions:
\begin{enumerate}[label=(R.\arabic*),ref=R.\arabic*]
	\item 
	$m_0=0$. In this case we say that $(m_0<m_1, l_0<l_1)$ is an \emph{initial rectangle}.
	\item\label{de:merge definition}
	It was preceded by two rectangles, $(m_{-1}<m_0-1, l_0<l_{1/2})$ and $(m_{-1}'<m_0-1, l_{1/2}<l_{1})$.
	In this case we say that these two rectangles \emph{merge} into $(m_0<m_1, l_0<l_1)$. 
	\item
	It is split from another rectangle, see Case \eqref{de: split definition} below.
\end{enumerate}
In addition, every rectangle satisfies exactly one of the following conditions:
\begin{enumerate}[resume,label=(R.\arabic*),ref=R.\arabic*]
	\item 
	$m_1=\infty$. In this case we say that $(m_0<m_1, l_0<l_1)$ is a \emph{final rectangle}.
	\item
	It is merged with another rectangle, see Case \eqref{de:merge definition} above.
	\item\label{de: split definition}
	It is succeeded by two rectangles, $(m_{1}+1<m_2, l_0<l_{1/2})$ and $(m_{1}+1<m_2', l_{1/2}<l_{1})$.
	In this case we say that $(m_0<m_1, l_0<l_1)$ \emph{splits} into these two rectangles. 
\end{enumerate}
If $(m_0<m_1, l_0<l_1)$ splits, then there is an Adding Step at $m_1$. 
If the rectangle is the merger of two other rectangles, then there was a Standard Step at $m_0-1$, and hence $f^{\brho+\bpi^{m_0-1}}$ has a null vertex at some $l_0<l<l_1$ in $[T_{m_0-1}, T_{m_0}]$.
\end{de}


Fix a rectangle $(m_0<m_1, l_0<l_1)$.
We will consider the following invariant, which depends on the rectangle.
\begin{de}[The potential invariant]
\index{Pm@$P_m$}\hypertarget{bef}{}
For every $m\in [m_0, m_1]$ denote 
\begin{align}\label{eq: P_m definition}
P_m:= -\min_{\widetilde \Gamma^m_{l_0}\subset \Gamma\subset\widetilde \Gamma^{m}_{l_1}}\xi(\rk\Gamma)\log \cov(\widetilde \Gamma^{m}_{l_1}/\Gamma/\widetilde \Gamma^{m}_{l_0}),
\end{align}
where $\xi:[l,l'] \to \RR$ is given by $\xi(x) = \tfrac{1}{(x-l-1)(l'+1-x)}$.
\end{de}
Note that the function $x\mapsto \frac{1}{\xi(x)}$ is positive and strictly concave.
\begin{lem}\label{lem: only harder elems}
The minimum in Eq. \eqref{eq: P_m definition} can be taken over lattices which project to elements of the Harder-Narasimhan filtration $\HN(\Gamma^{m}_{l_1}/\Gamma^{m}_{l_0})$. 
Moreover, if $\Psi$ with $\widetilde \Gamma^m_{l_0}\subset \Psi\subset\widetilde \Gamma^{m}_{l_1}$ is a primitive lattice not in the filtration, then \[\xi(\rk\Gamma)\log \cov(\widetilde \Gamma^{m}_{l_1}/\Gamma/\widetilde \Gamma^{m}_{l_0})\le (1-\gamma_1^{-1})P_m,\]
for some $\gamma_1>0$ which depends only on $n$. 
\end{lem}
Before the proof we will quote a quantitative version of the claim of uniqueness of lattices representing vertices of $\HN(\Lambda)$. It is equivalent to \cite[Corollary 1.31]{G}. 
\begin{lem}\label{lem: HN vertex uiniqueness}
Let $\Gamma\subseteq \Lambda$ be a primitive lattice, such that $\Gamma$ is not in the Harder-Narasimhan filtration. There exist primitive lattices $\Gamma_0, \Gamma_2$ with $\Gamma_0\subset \Gamma\subset \Gamma_2\subseteq \Lambda$ such that $\log \cov(\Gamma_2/\Gamma/\Gamma_0)\ge 0$. 
\end{lem}
\begin{proof}[Proof of Lemma \emph{\ref{lem: only harder elems}}]
Let $\Psi$ with $\widetilde\Gamma_{l_0}^m\subset\Psi\subset\widetilde\Gamma_{l_1}^m$ be a primitive lattice which is not in the filtration. 
By Lemma \ref{lem: HN vertex uiniqueness} applied to $\Lambda = \widetilde \Gamma_{l_1}^m/\widetilde \Gamma_{l_0}^m$, the are $\Psi',\Psi''$ with
\[\widetilde\Gamma_{l_0}^m\subseteq\Psi'\subset \Psi\subset \Psi''\subseteq \widetilde\Gamma_{l_1}^m\]
such that 
$\log \cov(\Psi''/ \Psi/ \Psi' )\ge 0$.

Denote $l_{1/2} := \rk \Psi,\, l_{1/4} := \rk\Psi',\, l_{3/4} := \rk\Psi''$.
Then
\begin{align*}
-\xi(l_{1/2})& \log \cov(\Gamma_{l_1}^m/\Psi/\Gamma_{l_0}^m) 
\\&\le 
-\xi(l_{1/2})\left(\frac{l_{3/4}-l_{1/2}}{l_{3/4}-l_{1/4}}\log \cov(\Gamma_{l_1}^m/\Psi'/\Gamma_{l_0}^m)\right. \\&\quad\quad\quad\left.+ \frac{l_{1/2}-l_{1/4}}{l_{3/4}-l_{1/4}}\log \cov(\Gamma_{l_1}^m/\Psi''/\Gamma_{l_0}^m)\right)\\
\\&\le 
\xi(l_{1/2})\left(\frac{l_{3/4}-l_{1/2}}{l_{3/4}-l_{1/4}}\frac{P_m}{\xi(l_{1/4})} + \frac{l_{1/2}-l_{1/4}}{l_{3/4}-l_{1/4}}\frac{P_m}{\xi(l_{3/4})}\right)
\\&= 
P_m\xi(l_{1/2})\left(\frac{l_{3/4}-l_{1/2}}{l_{3/4}-l_{1/4}}\frac{1}{\xi(l_{1/4})} + \frac{l_{1/2}-l_{1/4}}{l_{3/4}-l_{1/4}}\frac{1}{\xi(l_{3/4})}\right).
\end{align*}
Since $1/\xi$ is a strictly concave function, we can choose $\gamma_1>0$ so that
\begin{align*}
1-\gamma_1^{-1}>\max_{1\le l_{1/2}\le n-1}\xi(l_{1/2})\max_{0\le l_{1/4}<l_{1/2}<l_{3/4}\le n}\frac{l_{3/4}-l_{1/2}}{l_{3/4}-l_{1/4}}\frac{1}{\xi(l_{1/4})} + \frac{l_{1/2}-l_{1/4}}{l_{3/4}-l_{1/4}}\frac{1}{\xi(l_{3/4})} < 1,
\end{align*}
as desired. Note that $\gamma_1$ depends only on $\Eall$.
\end{proof}
\begin{obs}[$P_m$ is Lipschitz]\label{obs: tameness of P_m}
For every $m\in [m_0, m_1]$ and $\widetilde \Gamma^m_{l_0}\subset \Gamma\subset\widetilde \Gamma^{m}_{l_1}$
we have \[\left|\log \cov(\widetilde \Gamma^{m+1}_{l_1}/g_T\widetilde h_{m+1}\Gamma/\widetilde \Gamma^{m+1}_{l_0}) - \log \cov(\widetilde \Gamma^{m}_{l_1}/g_T\widetilde h_{m}\Gamma/\widetilde \Gamma^{m}_{l_0})\right| \le \gamma_1'T,\]
for some $\gamma_1'$, which depends only on $\Eall$. 
This implies that $|P_{m+1} - P_m| \le \gamma_1'T$. 
\end{obs}

We next prove the following lemma, which relates the functions $P_m$ for $m$ in certain intervals. 
\begin{lem}\label{lem: diffrential P_m behaviour}
Let $[m_0', m_1']$ with $m_0\le m_0'<m_1'\le m_1$ be an interval in which 
$f_{E, l_0}(T_m)=E, f_{E, l_1}(T_m)=E'$ for every $m\in [m_0', m_1']$. Then:
\begin{enumerate}[label=\emph{(\arabic*)}, ref=\arabic*]
  \item \label{item: diff scalar} If $E\subseteq E'$ is a scalar pair, then $P_{m_1'} = P_{m_0'} + O(1)$. 
  \item \label{item: diff non-scalar} There is $d_0$ \emph(depending only in $\Eall$\emph) such that
  if $E\subseteq E'$ is not a scalar pair $m_1'-m_0'\ge d_0$, then 
  \[P_{m_1'} \le \max (\gamma_3 T, P_{m_0'} - (m_1'-m_0')/\gamma_3'),\]
  where $\gamma_3, \gamma_3'\ggg \gamma_1$.
\end{enumerate}
\end{lem}
\begin{proof}
We start by proving Item \eqref{item: diff scalar}.
Let $l_{1/2}$ with $l_0<l_{1/2}<l_1$ and $\Gamma$ with $\Gamma_{l_0}^{m_0}\subseteq \Gamma\subseteq \Gamma_{l_1}^{m_0}$ 
be a rank-$l_{1/2}$ lattice. 
Since $\spa\widetilde \Gamma_{l_0}^{m_0}\in \grl{l_0}^{\varepsilon, g}_{E}$ and $\spa\widetilde \Gamma_{l_1}^{m_0}\in \grl{l_1}^{\varepsilon, g}_{E'}$, Lemma \ref{lem: making q0} provides an element $q_0$ for which $q_0^{-1}\spa\Gamma_{l_0}^{m_0} = \spa(\Gamma_{l_0}^{m_0})^\to$ and $q_0^{-1}\spa\Gamma_{l_1}^{m_0} = \spa(\Gamma_{l_1}^{m_0})^\to$, and in addition, $d_{H^-}(q_0, \Id) = O(\varepsilon)$.

Decomposition into eigenvalues shows that 
$q_0^{-1}\spa\Gamma^{m_0}\in \grl{l_{1/2}}^g_{E''}$,
where $E''$ with $E\subset E''\subset E'$ is the unique such multiset with $l_{1/2}$ elements, we get that
$\spa\Gamma^{m_0}\in \grl{l_{1/2}}_{E''}\in \grl{l_{1/2}}_{\to E''}$. 
Since $d_{H^-}(q_0, \Id) = O(\varepsilon)$,
it follows that 
$\spa\Gamma^{m_0}\in \grl{l_{1/2}}_{E''}\in \grl{l_{1/2}}^{O(\varepsilon), g}_{\to E''}$.
Assume that $\varepsilon$ is chosen to be small enough so that Corollary 
\ref{cor: blade behavior over time} is applicable to $\spa\Gamma^{m_0}$, which completes the proof of \eqref{item: diff scalar}.

We turn to prove Item \eqref{item: diff non-scalar}.
Assume $E\subset E'$ is not a scalar pair. 
\begin{claim}\label{claim:only current noise mattes}
There are $\gamma_2\ggg \gamma_1, \gamma_1'$ such that if $P_{m_0'} \ge (m_1' - m_0')\gamma_2 T$,
then only lattices in $\HN(\widetilde\Gamma_{l_1}^{m_0'}/\widetilde\Gamma_{l_0}^{m_0'})$ should be taken into account when computing $P_{m_1'}$. 
That is, if 
$\widetilde\Gamma_{l_0}^{m_0'}=\Psi_0\subset \Psi_1\subset\dots\subset \Psi_k = \widetilde \Gamma_{l_1}^{m_0'}$ 
is the filtration projecting to $\HN(\widetilde\Gamma_{l_1}^{m_0'}/\widetilde\Gamma_{l_0}^{m_0'})$, then to compute 
$P_{m_1'}$ it is enough to use the lattices $\widetilde \Psi_j^{m_1'}$ for $j=0,\dots,k$, where
\[
\widetilde \Psi_j^{m}
:=g_{T_{m}}h_{m}h_{m_0'}^{-1}g_{-T_{m_0'}}\Psi_j.\] 
\end{claim}
\begin{proof}
Let \[\widetilde \Psi^{m_1'}=g_{T_{m_1'}}h_{m_1'}h_{m_0'}^{-1}g_{-T_{m_0'}}\Psi; ~~ \widetilde\Gamma_{l_0}^{m_1'}\subseteq \widetilde \Psi^{m_1'} \subseteq \widetilde\Gamma_{l_1}^{m_1'},~ \widetilde\Gamma_{l_0}^{m_0}\subseteq  \Psi \subseteq \widetilde\Gamma_{l_1}^{m_0},\] 
with $P_{m_1'} = -\xi(\rk \widetilde \Psi^{m_1'})\log \cov(\widetilde\Gamma_{l_1}^{m_1'}/\widetilde \Psi^{m_1'}/\widetilde\Gamma_{l_0}^{m_1'})$.

The lattice $\widetilde \Psi^{m_1'}$ is primitive in $\widetilde\Gamma_{l_1}^{m_1'}$, as we can replace it by the primitive closure and decrease its covolume. 
If $\Psi$ is not one of the $\Psi_j$-s, then, by Lemma \ref{lem: only harder elems},  
\[-\xi(\rk \Psi)\log \cov(\widetilde\Gamma_{l_1}^{m}/\Psi/\widetilde\Gamma_{l_0}^{m}) \le (1-\gamma_1^{-1})P_{m_0'}.\]
It follows from Observation \ref{obs: tameness of P_m} that 
\[P_{m_1'} = -\xi(\rk \widetilde \Psi^{m_1'})\log \cov(\widetilde\Gamma_{l_1}^{m_1'}/\widetilde \Psi^{m_1'}/\widetilde\Gamma_{l_0}^{m_1'}) \le (1-\gamma_1^{-1})P_{m_0'} + \gamma_1'dT.\]
On the other hand, again by Observation \ref{obs: tameness of P_m}, $P_{m_1'}\ge P_{m} - \gamma_1'dT$. Consequently, we reach a contradiction as soon as $P_{m_0'} > 2\gamma_1'\gamma_1dT$.
Therefore, we may choose $\gamma_2 = 2\gamma_1'\gamma_1$. This proves Claim \ref{claim:only current noise mattes}.
\end{proof}
Now assume that $P_{m_0'} \ge \gamma_2(m_0'-m_1') T$.
Let $\Psi = \Psi_j$ be the lattice for which $\widetilde \Psi^{m_1'} = \widetilde \Psi_j^{m_1'}$ 
is the lattice that attains the minimum in the definition of $P_{m_1'}$. 
Denote $l:=\rk \Psi$.
By the discussion above, for every $m_0'\le m\le m_1'$ the lattice $\widetilde\Psi^m:= g_{T_m}h_m^{-1}h_{m}^{-1}g_{-T_{m_0'}}\Psi$ projects to an element in the Harder-Narasimhan filtration $\HN(\widetilde\Gamma_{l_1}^m/\widetilde\Gamma_{l_0}^m)$, and hence is a noise lattice. 
Denote $\bar\Psi^m = g_{-(m_1'-m)T}\widetilde\Psi^{m_1'}$, and note that 
$\bar\Psi^m = h^{m\to m_1'}\widetilde \Psi^d$
for $h^{m\to m_1'}\in B_{d_\varphi}(\Id;1)$; that is, $\bar\Psi^m$ and $\widetilde \Psi^m$ are somewhat close.

Note that if $C_2 > (m_1'-m_0')T$, then by Lemma \ref{lem: not many vertices} only $O(1)$ steps between $m_0', m_1'$ may not be Standard Steps. 
Moreover, every Standard Step $m$, is designed to push $\widetilde \Psi^d$ away from $\gr_{\to E''}$, for any non-final $E''$ with $E\subset E''\subset E'$. 
Specifically, for every $m\in [m_0', m_1']$ denote by $E^m\in \cI_l$ the unique multiset for which $\bar\Psi^m \in U_{\varepsilon/2}(\gr^g_{E^m})$, if such a multiset exists, and otherwise it remains undefined. By Theorem \ref{thm:lin path}, $E^m$ is not defined for at most $l(n-l)$ values of $m$. 

We claim that if $m$ is a Standard Step and $E^m$ is defined and non-final, then $E^{m+1}$ does not equal to $E^m$. 
Indeed, since \[\bar\Psi^m =h^{m\to m_1'}\widetilde \Psi^m\in U_{\varepsilon/2}(\gr^g_{E^m})\] 
and $d_\varphi(h^{m\to m_1'}, \widetilde h_{m+1}) \le \exp (-T)$, it follows that 
\[\widetilde \Psi^m\in U_{\varepsilon}(\gr^g_{E^m}).\]
Hence, since Alice did not forbid $h_{m+1}$, there is no 
\[
\widetilde h'\in B_{d_\varphi}(\widetilde h_{m+1}; 3^{\alpha_\varphi} \exp(-T))\cap H_{\spa\widetilde \Gamma_\bullet^m \to f_{E,\bullet}(T_{m})} \]
such that $\widetilde h'\spa\widetilde \Psi^m\in \gr^{\varepsilon, g}_{\to E^m}$.
Therefore, there is no 
\begin{align}\label{eq:h''}
\widetilde h''&\in 
B_{d_\varphi}(\Id; 2^{\alpha_\varphi} \exp(-T))\cap H_{h^{m\to m_1'}\spa\widetilde \Gamma_\bullet^m \to f_{E,\bullet}(T_{m})}
\end{align}
such that $\widetilde h''\spa\bar \Psi^m\in \gr^{\varepsilon, g}_{\to E^m}$. Since $\spa \bar\Psi^m\in U_{\varepsilon/2}(\gr^g_{E^m})$, the $\varepsilon, g$ part is redundant, and there is no $\widetilde h''$ such that Eq. \eqref{eq:h''} holds and $\widetilde h''\spa\bar \Psi^m\in \gr_{\to E^m}$.

Conjugating Eq. \eqref{eq:h''} by $g_T$ we deduce that 
there is no matrix 
\[
\widetilde h'''\in 
B_{d_\varphi}(\Id; 2^{\alpha_\varphi} )\cap H_{h^{m+1\to m_1'}\spa\widetilde \Gamma_\bullet^{m+1} \to f_{E,\bullet}(T_{m+1})}
\]
such that $\widetilde h'''\bar \Psi^{m+1}\in \gr_{\to E^m}$. 
If $\bar\Psi^{m+1}\in U_{\varepsilon/2}(\gr^{g}_{E^m})$, that is, $E^{m+1} = E^m$,
then there exists $s\in U_{O(\varepsilon)}(\Id; d_{\SL_n(\RR)})$ such that $s^{-1}\bar\Psi^{m+1} \in \gr^{g}_{E^m}$ and $s^{-1}h^{m+1\to m_1'}\spa\widetilde \Gamma^{m+1}_\bullet = h^{m+1\to m_1'}\spa\widetilde \Gamma^{m+1}_\bullet$. 
By Lemma \ref{lem: mult map}, if $\varepsilon$ is small enough, then there are $q\in H^{-0}$ and $\widetilde h'''\in H$ such that $d_{H^{-0}}(q,\Id), d_{H}(\widetilde h''',\Id) = O(\varepsilon)$ and $s=\widetilde h'''q$. 
This $\widetilde h'''$
satisfies the above condition for $\widetilde h$, which contradicts the assumption that $E^{m+1}=E^m$.



It follows that for $m\ge m_{1/2} = m_0' + O(1)$ we have 
\[\spa\widetilde \Psi^{m}, \spa \bar \Psi^{m} \in U_{\varepsilon}\left(\grl{l}^g_{E''_{\rm final}}\right),\]
where $E_{\rm final}$ is the unique final multiset of size $l$ in $E\subseteq E'$.
By Observation \ref{obs: tameness of P_m}, 
\[\xi(\rk \widetilde \Psi^{m_{1/2}}) \log \cov(\widetilde\Gamma_{l_1}^{m_{1/2}}/\widetilde \Psi^{m_{1/2}}/\widetilde\Gamma_{l_0}^{m_{1/2}})\le P_m + O(\gamma_1').\]
By Corollary \ref{cor:blade behavior}, for every $m\ge m_0'$ we have 
\begin{align*}
\log& \cov(\widetilde\Gamma_{l_1}^{m}/\widetilde \Psi^{m}/\widetilde\Gamma_{l_1}^{m}) - \cov(\widetilde\Gamma_{l_1}^{m_{1/2}}/\widetilde \Psi^{m_{1/2}}/\widetilde\Gamma_{l_1}^{m_{1/2}}) \\&= 
(m-m_{1/2})\left(\frac{\rk\Psi - l}{l'-l}\eta_{E}+\frac{l' - \rk\Psi}{l'-l}\eta_{E'} - \eta_{E_{\rm final}}\right)+ O(1) 
\end{align*}
Define 
\[\Phi:= \frac{\rk\Psi - l}{l'-l}\eta_{E}+\frac{l' - \rk\Psi}{l'-l}\eta_{E'} - \eta_{E_{\rm final}}.\]
By Lemma \ref{lem: final nonscalar has partial2}, $\Phi < 0$. 
Consequently, if $P_{m_0'}\ge (m_1' - m_0')\gamma_2 T$, then
\begin{align*}
P_{m_1'} &= \xi(l)\log \cov(\widetilde\Gamma_{l_1}^{m_1'}/\widetilde \Psi^{m_1'}/\widetilde\Gamma_{l_0}^{m_1'})
\\&\le 
\xi(l)\log \cov(\widetilde\Gamma_{l_1}^{m_{1/2}}/\widetilde \Psi^{m_{1/2}}/\widetilde\Gamma_{l_0}^{m_{1/2}}) + (m_1'-m_{1/2}) \Phi+O(1)
\\&=P_{m_0'} + (m_{1/2}-m_0')\gamma_1' + (m_1'-m_{1/2}) \Phi + O(1).
\end{align*}
Hence, if $m_1-m_0 \ggg -\gamma_1/\Phi$, then \[P_{m_1'} - P_{m_0} < (m_1'-m_0)\frac{\Phi}{2}.\]
The result follows, with $d_0$ being the implicit lower bound on $m_1'-m_0$.
\end{proof}

We next show that $P_m$ is bounded by $\gamma_7 T$. 
We will do so by induction on all rectangles $(m_0<m_1, l_0<l_1)$, sorted by $m_0$.
That is, when proving that $P_m$ is bounded by $\gamma_7 T$ for $(m_0<m_1, l_0<l_1)$, we assume this claim holds for all rectangles $(m_0'<m_1', l_0'<l_1')$ with $m_0'<m_0$.
The tricky part is to show that the implicit constant does not increase at every step of the induction. 
What makes this possible is the decaying behavior of $P_m$ in case $2$ of Lemma \ref{lem: diffrential P_m behaviour}, and the fact that this case occurs near every null point.

\begin{lem}\label{lem: circular inductive lemma}
There exist $\gamma_9>0$ and bounds $\gamma_8>\gamma_7 >\gamma_3$ that depend only on $\Eall$, such that 
for every rectangle $(m_0 < m_1, l_0<l_1)$ we have
\[
P_m \le 
\begin{cases}
\gamma_8 T, & \text{if }m_0\le m< m_0 + \gamma_9,\\
\gamma_7 T, & \text{if } m_0 + \gamma_9 \le m \le m_1.
\end{cases}
\]
Moreover, if the rectangle splits \emph{(}see Definition \emph{\ref{de: split definition}}\emph{)} then $P_{m_1}<\gamma_3T$. 
\end{lem}
\begin{cor}\label{cor: bound on noise}
At every Adding Step the noise is bounded by $\gamma_4$. Formally, if Alice attempts to add a lattice $\Gamma_l^{m+1}$ at step $m$, and if $l^-<l<l^+$ are the consecutive elements of $l$ in $L_{T_{m+1}}(f^{\brho+\bpi^m})$, then for every $\widetilde \Gamma_{l^-}^m\subset\Gamma\subset \widetilde \Gamma_{l^+}^m$ we have 
\[\log \cov(\widetilde \Gamma_{l^+}^m/\Gamma/ \widetilde \Gamma_{l^-}^m) \ge -\gamma_4T.\]
\end{cor}
\begin{proof}[Proof of Corollary \emph{\ref{cor: bound on noise}}]
The claim follows from the fact that in this case there is a rectangle $(m_0<m, l^-<l^+)$ that splits, and hence $P_m<\gamma_3T$. 
\end{proof}
Recall that Corollary \ref{cor: bound on noise} was used in the definition of the Adding Step. We turn to the proof of Lemma \ref{lem: circular inductive lemma}.
\begin{proof}[Proof of Lemma \emph{\ref{lem: circular inductive lemma}}]
We distinguish three cases, depending on whether the rectangle is a merger of two rectangles, is split from another rectangle or is an initial rectangle. 

\textbf{Rectangle is merged:}
Assume the rectangle is merged from, say, 
$(m_{-1} < m_0-1, l_0<l_{1/2})$ and
$(m_{-1}' < m_0-1, l_{1/2}<l_1)$. The induction hypothesis on the two rectangles bounds the corresponding potential invariants: $P'_{m_0-1}, P''_{m_0-1}\le \gamma_7 T'$. 
Since $\widetilde \Gamma_{l_{1/2}}^{m_0-1}$ was removed at step $m_0-1$,
and
by Eq. \eqref{eq: bpi = det Gamma - f} we get that $\cov(\widetilde\Gamma^{m_0-1}_{l_{1}}/\widetilde\Gamma^{m_0-1}_{l_{1/2}}/\widetilde\Gamma^{m_0-1}_{l_{0}}) = O(T)$. 
It follows that the bound on $P'_{m_0-1},P''_{m_0-1}$ gives a bound on $P_m = O(T+ P''_{m_0-1}+P''_{m_0-1}) = O(\gamma_7T)$. 

Since $f^{\brho+\bpi^{m_0-1}}$ has a null vertex at $[T_{m_0-1}, T_{m_0})$
and $f^{\brho+\bpi^{m_0-1}}$ is $C_2/2$-separated it follows that for all $m_0\le m \le m_0+C_2/(2T)$ the multisets $f_{E,l_0}(T_m)$ and $f_{E,l_1}(T_m)$ are constant.
By Lemma \ref{lem: not scalar at edge}, these $f_{E,l_0}(T_m)\subset f_{E,l_1}(T_m)$ do not form a scalar pair.

Let $a_0,\dots,a_1$ be the shortest sequence of numbers with $m_0=a_0<a_1<\dots<a_k =m_1$ such that $m\mapsto (f_{E,l_0}(T_m), f_{E,l_1}(T_m))$ is constant for $a_i\le m\le a_{i+1}-1$ for each $0\le i \le k-1$.
Since there are at most $O(n^2)$ vertices in every nontriviality interval, it follows that $k=O(n^2)$.
Lemma \ref{lem: diffrential P_m behaviour} implies that $P_m$ decays on $[a_0, a_{1}-1]$, and Observation \ref{obs: tameness of P_m} and Lemma \ref{lem: diffrential P_m behaviour} imply that $P_m$ does not expand more than $O(\gamma_1d_0T)$ over the succeeding intervals.

\textbf{Rectangle is split:} 
Suppose that the rectangle did split from another rectangle, w.l.o.g., $(m_{-1}<m_0-1, l_0 < l_2)$ split into $(m_0<m_1, l_0<l_1)$ and $(m_0<m_1', l_1<l_2)$. 
By the construction of the Adding Step,  $\cov (\Gamma_{l_2}^{m_0}/\Gamma_{l_1}^{m_0}/\Gamma_{l_0}^{m_0}) = O(\gamma_5T)$. In addition, if $P_m'$ is the potential invariant of the rectangle $(m_{-1}<m_0-1, l_0 < l_2)$, then by induction $P_{m_0-1}'\le \gamma_3T$.
This yields the bound  $P_{m_0} = O(\gamma_5T)$. As in the previous case, we obtain a bound on $P_{m}$, for all $m_0\le m\le m_1$. 

\textbf{Rectangle is initial:} At step $0$, the only noise lattices are the lattices in $\HN(\Lambda)$. It follows that $P_0 = O(T)$, provided $T$ is large enough. As in the previous cases, we obtain a bound on $P_{m}$ for all $m_0\le m\le m_1$.

To bound $P_{m_1}$ in the case that the rectangle splits, note that since $f^{\brho+\bpi^{m_1}}$ is separated we get that $f_{E,l_1}(T_m), f_{E,l_2}(T_m)$ are constant for $m_1-C_2/(2T)\le m \le m_1$, and hence the claim follows from Lemmas \ref{lem: diffrential P_m behaviour} and \ref{lem: not scalar at edge}. 
\end{proof}

\subsection{Existence of a Lattice with a Given Template - Proof of Corollary \ref{cor:template existence}} 
\label{sub:existence_of_lattice_given_template}
In this subsection we will prove Corollary \ref{cor:template existence} which states that every template describes at most one lattice. 
Let $\sigma_{\rm Alice}$ be the strategy constructed above for the $(T, g)$-game. 
Let $h_0,h_1,...$ be the output of a $(T, g)$-game where Alice plays via $\sigma_{\rm Alice}$ and Bob plays arbitrarily. It is sufficient to prove that the lattice $h_\infty\Lambda$ satisfies $f^{h_\infty\Lambda}\sim f$.
\begin{lem}\label{lem: alice strategy good}
Let $h_0,h_1,...$ be the output of a $(T, g)$-game where Alice plays via $\sigma_{\rm Alice}$ and Bob plays arbitrarily. Then $f^{h_\infty\Lambda}\sim f$.
\end{lem}
Before the proof we will prove the following corollary of Lemma \ref{lem: HN vertex uiniqueness}.
\begin{cor}\label{cor: convex is harder}
Let $\Lambda$ be a lattice and let $\Gamma_\bullet$ be a filtration of $\Lambda$ composed of primitive sublattices. 
Let $\Gamma'_\bullet$ be the filtration $\Gamma_\bullet$ obtained by adding any two consequtive $\Gamma_{l}\subset\Gamma_{l'}$ the inverse image of $\HN(\Gamma_{l'}/\Gamma_l)$. Consider the piecewise linear function $f:[0,n]\to \RR$ defined by its values on $L(\Gamma'_\bullet)$ as $f(l)= \log \cov\Gamma'_l$. The following statements are equivalent:
\begin{enumerate}[label={\emph{(\arabic*)}}, ref={\arabic*}]
  \item \label{cond: convex func} $f$ is convex and strictly convex at $L(\Gamma_\bullet)\setminus \{0,n\}$. 
  \item \label{cond: filtration contained} $\Gamma_\bullet$ is contained in $\HN(\Lambda)$. 
  \item \label{cond: filtration equal} $\Gamma_\bullet'=\HN(\Lambda)$. 
 \end{enumerate}
\end{cor}
\begin{proof}[Proof of Lemma \emph{\ref{lem: alice strategy good}}]
The verification that $\eqref{cond: filtration equal} \iff \eqref{cond: filtration contained} \implies \eqref{cond: convex func}$ follows from the definition of the Harder-Narasimhan filtration and the definition of $f$. 
To show that $\eqref{cond: convex func}\implies \eqref{cond: filtration contained}$, we use induction on $n$. For $n=1$ there is nothing to prove. Assume \eqref{cond: convex func} holds for $\Lambda$, but not $\eqref{cond: filtration contained}$, i.e., for some $l\in L(\Gamma_\bullet)$ with $l\neq 0,n$ we have $\Gamma_l\nin \HN(\Lambda)$.  
By the induction hypothesis applied to $\Gamma_l, \Lambda/\Gamma_l$, we conclude that $\Gamma'_\bullet$ is composed of $\HN(\Gamma_l)$ and the inverse image of $\HN(\Lambda/\Gamma_l)$ in $\Lambda$. 
Applying Lemma \ref{lem: HN vertex uiniqueness} one contradicts the strict convexity of $f$ at $l$, as desired.
\end{proof}

\begin{proof}
By Observation \ref{obs: The behavior of templates}, $f^{h_\infty\Lambda}|_{[0,T_m]}\sim f^{h_m\Lambda}|_{[0,T_m]}$, 
where the implicit constants are independent of $m,T$. 
Hence, $|f^{h_m\Lambda}_{H,i}(t) - f^{h_\infty\Lambda}_{H,i}(t)| = O(1)$ for every $0\le i\le n$ and $t\in [0,T_m]$. 
By Remark \ref{rem: where L = L}, for every $m$ for which $T_m$ lies outside a $\Theta(C_1)$-neighborhood of the vertices of $f$ we have $L(\Gamma_\bullet^m) = L_f(T_m)$. Since $f$ is $C_2$-significant, if $T_m$ is outside a (maybe larger) $O(C_1)$-neighborhood of the vertices of $f$, $\partial^2f(T_m)\in \{0\}\cup [\Theta(C_1), \infty)$. 
By Assertion \ref{asser: approximation of flag} Eq. \eqref{eq: bpi = det Gamma - f} and Assertion \ref{asser: bound on noise} Eq. \eqref{eq: O(T) bound on noise} we can apply Corollary \ref{cor: convex is harder} to the filtration $\widetilde \Gamma_\bullet^m$ of $\widetilde \Lambda^m$ and deduce that 
$\HN(\widetilde \Lambda^m)$ contains $\tilde\Gamma^m_\bullet$ and that 
$\HN(\widetilde \Lambda^m)_{H,l} = f_{H,l}(T_m)+O(C_1)$ for every $0\le l\le n$.
Since $h^{m\to \infty}\widetilde \Lambda^m = g_{T_m}h_\infty\Lambda$ and $d_H(h^{m\to \infty}, \Id) = O(1)$ (see Subsection \ref{sub:geometric_picture}) it follows that for every $m$ such that $T_m$ is outside an $O(C_1)$-neighborhood of the vertices of $f$ we have $\HN(g_{T_m}h_\infty\Lambda)_{H,l} = f_{H,l}(T_m)+O(C_1)$. 
Since for every $t\ge 0$ there is such an $m$ with $t-T_m = O(C_1)$, we get that for every $t\ge 0$ 
\begin{align}\label{eq: height is right}
\HN(g_{t}h_\infty\Lambda)_{H,l} = f_{H,l}(t)+O(C_1).
\end{align}
This proves Condition \ref{cond: template height} of Definition \ref{de:f Lambda matching template} of matching templates. As for condition \ref{cond: template direction}, 
let $[a,b]$ be an interval such that $f_{E,l}(t) = E\in \cI_l$
and $\partial^2f_{H, l}(t)\ge C$ for every $t\in [a,b]$, 
for some $C>0$ to be determined later. 
By Eq. \eqref{eq: height is right}, if $C = \Theta(C_1)$,
then $\HN(g_{t}h_\infty\Lambda)_{H,l} = \Theta(C)$ for every $t\in [a,b]$. 
It follows from Assertion \ref{asser: bound on noise} that $h_\infty^{-1}g_{-T_m}\HN(g_{T_{m}}h_\infty\Lambda)_{\Gamma,l}$ is not a noise lattice at step $m$, and hence $\Gamma_l^m = h_\infty^{-1}g_{-T_m}\HN(g_{T_{m}}h_\infty\Lambda)_{\Gamma,l}$. 
Assertion \ref{asser: approximation of flag} and Eq. \eqref{eq: aligning Gamma} 
now implies that $\spa(\tilde\Gamma^m_l)\in U_\varepsilon(\gr^g_E)$. 
Using arguments similar to those used to prove Corollary \ref{cor: blade behavior over time}, we conclude that for every 
$t\in [a+O(|\log\varepsilon|),b-O(|\log\varepsilon|)]$
we have $\HN(g_{T_{m}}h_\infty\Lambda)_{\Gamma,l} \in U_\varepsilon(\gr_E^g)$.
This shows that $h_\infty\Lambda$ $(\varepsilon, C)$-matches $f$. 
\end{proof}



\section{Bob's Strategy} 
\label{sec:Bob_s_strategy}
Bob's strategy is a function $\sigma_{\rm Bob}^T:(T_0, h_0, A_1, h_1,\dots, A_{m+1})\mapsto h_{m+1}$, where $h_{m+1}\in A_{m+1}$. As it turns out, the construction of Bob's strategy is simpler than the that of Alice's strategy. While Alice tries to create a good template, Bob's goal is to interfere with her attempts. 

Let $\Lambda$ be a lattice.
At each step $m$, the set $A_{m+1}$ satisfies 
\begin{itemize}
  \item $A_{m+1}\subseteq g_{-T_m}B_Tg_{T_m}h_m$.
  \item $d_\varphi(h_{m+1}, h_{m+1}') \ge \exp(-3^{\alpha_\varphi} T_{m+1})$ for every $h_{m+1}, h_{m+1}'\in A_{m+1}$.
\end{itemize}
Equivalently, denoting $\widetilde A_{m+1}:=g_{T_m}A_{m+1}h_{m}^{-1}g_{-T_m}$, we have
\begin{itemize}
  \item $\widetilde A_{m+1}\subseteq B_T$.
  \item $d_\varphi(h_{m+1}, h_{m+1}') \ge \exp(-3^{\alpha_\varphi} T))$ for every $\widetilde h_{m+1}, \widetilde h_{m+1}'\in A_{m+1}$.
\end{itemize}
Bob needs to choose $h_{m+1}\in A_{m+1}$, or, alternatively, $\widetilde h_{m+1}=g_{T_m}h_{m+1}h_{m}^{-1}g_{-T_m}\in \widetilde A_{m+1}$. 

In the end Alice and Bob generate $h_\infty := \lim_{m\to \infty} h_m \in H$. Bob's goal is to minimize $\Delta(f^{h_\infty\Lambda})$. 
\subsection{Point Sorting in a Height Sequences} 
\label{sub:point_sorting_in_a_height_sequence}
\index{Vanishing number}\hypertarget{beg}{}
To construct Bob's strategy, we need to introduce an order on the indices of a height sequence. 
Define $\brho(\zeta) = (\zeta l(n-l))_{l=0}^l$ for every $\zeta>0$. 
For every height sequence $a_\bullet$, define the \emph{vanishing number} $\zeta_l(a_\bullet)$ of any index $1\le l\le n-1$ by 
\[\zeta_l(a_\bullet) = 
\min\{\zeta>0:l\nin L(a_\bullet^{\brho(\zeta)})\}.\]
The number $\zeta_l(a_\bullet)$ behaves similarly to $\partial^2a_l$: they vanish simultaneously in height sequences and one can show that $\zeta_l(a_\bullet) = \Theta(\partial^2a_l)$.

\subsection{The Strategy} 
\label{sub:the_strategy2}
For every $m$, consider the lattice $\widetilde \Lambda^m := g_{T_m}h_m\Lambda$ and its Harder-Narasimhan filtration $\HN(\widetilde \Lambda^m)$.
Let $\varepsilon>0$ be sufficiently small to satisfy Theorem \ref{thm:lin path} and Lemma \ref{lem: E_i to V_i is functorial}.
Consider the subset of indices
\[L_0^m := \{l\in L(\HN(\widetilde \Lambda^m))\setminus\{0,n\}: ~\exists E(m,l)\in \cI_{l}, ~\spa\widetilde \Gamma_l^m\in U_\varepsilon(\gr^g_{E(m,l)}\}.\]
By Lemma \ref{lem: E_i to V_i is functorial}, the sequence $(E(m,l))_{l\in L_0^m}$ forms a direction filtration without the extremal multisets $\emptyset, \Eall$. 
Sort the set $L_0^m$ by the value of $\zeta_l(\HN(\widetilde \Lambda^m)_{H,\bullet})$ for $l\in L_0^m$, that is, $L_0^m = \{l_1^m, l_2^m, \dots,l_{k^m}^m\}$, where 
\[\zeta_{l_1^m}(\HN(\widetilde \Lambda^m)_{H,\bullet}) \ge \zeta_{l_2^m}(\HN(\widetilde \Lambda^m)_{H,\bullet})\ge\dots\ge \zeta_{l_{k^m}^m}(\HN(\widetilde \Lambda^m)_{H,\bullet}).\]
Let $i^m$ denote the first index of $l^m_\bullet$ with the following property: there is $\widetilde h_{m+1}\in \widetilde A_{m+1}$ such that for every $h'\in B_{d_\varphi}(h_{m+1}; 3^{\alpha_\varphi}\exp(-T))$ one has  
\[\exists~ 1\le i\le i^m,~h'\Gamma_{i_k}\nin H_{\spa\widetilde \Gamma_{l_i^m}^m\to E(m,l_i^m)}.\]
\index{Steps of Bob!idle}\hypertarget{beh}{}
\index{Steps of Bob!interruption}\hypertarget{bei}{}
If such $i^m$ exists, the strategy selects this $\widetilde h_{m+1}$. In this case we call the step an \textbf{Interruption Step}, and say that Bob \emph{interrupted $i^m$ indices}. 
If no such $i^m$ exists, let $\widetilde h_{m+1}\in \widetilde A_{m+1}$ be arbitrary, and call step $m$ an \textbf{Idle Step}.


\section{Computation of the Dimension --- Proof of Theorem \ref{thm: general scewed dimention formula}} 
\label{sec:computation_of_the_dimension}
\subsection{Computation of the Payoff} 
\label{sub:computation_of_alice_s_payoff}
We fix a family $\cF$ of $g$-templates that is invariant to equivalence of $g$-templates.
For every lattice $\Lambda$, we wish to prove that 
\[\dim_\funH(Y_{\Lambda, \cF}; d_\varphi) = \sup_{f\in \cF}\Delta_0(f).\]
We will use the Dimension Game introduced in Subsection \ref{sub:hausdorff_games}. 

To prove the lower bound, it is enough to fix $f_0\in \cF$ with $\Delta_0(f_0)\ge \sup_{f\in \cF}\Delta_0(f) - \varepsilon$, and to prove that Alice has a strategy $\sigma_{\rm Alice}^T$ that guarantees that the payoff is at least $D_{\rm Alice}^T$, for some $D_{\rm Alice}^T$ such that $\lim_{T\to \infty}D_{\rm Alice}^T \ge \Delta_0(f)$. Any $h_\infty$ attained with Alice's strategy will have $f^{h_\infty\Lambda}\sim f_0$ by \ref{lem: alice strategy good}.

To prove the upper bound, we have to prove that Bob has a strategy $\sigma_{\rm Bob}^T$ that guarantees that the payoff is at most $D_{\rm Bob}^T$, for some $D_{\rm Bob}^T$ such that \[\lim_{T\to \infty}D_{\rm Bob}^T \le \sup_{f\in \cF}\Delta_0(f).\]

\subsection{The Payoff Guaranteed by Alice's Strategy} 
\label{sub:alice_s_strategy_value}
Fix $\varepsilon>0$ as in Section \ref{sec:alice_s_strategy} and let
$f\in \cF$ with $\Delta_0(f)\ge \sup_{f\in \cF}\Delta_0(f) - \varepsilon$. To apply Alice's strategy we have to choose $C_2\ggg C_1\ggg T$ and approximate $f$ by a $C_2$-separated and $C_2$-significant $g$-template. 
By Lemma \ref{lem:separated}, there is an equivalent $g$-template $f'\sim f$ that is $C_2$-separated and $C_2$-significant. The construction of $f'$ was carried out by constructing a proper shift sequence $\brho$ and getting $f'=f^\brho$. The following claim will allow us to replace $f$ by $f'$. 
\begin{claim}
For every $g$-template $f$ and every shift sequence $\brho$ we have $\Delta_0(f^\brho)\ge \Delta_0(f)$. 
\end{claim}
\begin{proof}
To prove the result it is enough to show that $\delta(f_{E,\bullet}(t))\le \delta(f_{E,\bullet}^\brho(t))$ (see Definition \ref{de: local entropy} of shifting a height sequence). This inequality holds since $\delta$ is monotone nonincreasing with respect to adding multisets and the fact that each multiset in $f_{E,\bullet}^\brho(t)$ appears also in $f_{E,\bullet}(t)$.
\end{proof}

We can thus assume w.l.o.g. that $f$ is $C_2$-significant and use Alice's strategy as defined in Section \ref{sec:alice_s_strategy}.
Let $A_0, h_0, A_1, h_1, \ldots$ be the choices made by Alice and Bob along the game. Note that Bob may play arbitrarily, not necessarily according to the strategy $\sigma_{\rm Bob}^T$ described in Section \ref{sec:Bob_s_strategy}.
By Lemma \ref{lem: alice strategy good}, the resulting matrix $h_\infty = \lim_{m\to \infty}$ satisfies $f^{h_\infty\Lambda}\sim f$. It remains to compute the first option of Eq. \eqref{eq: value definition}, namely $\frac{1}{T}
\liminf_{m\to \infty}\frac1m\sum_{k=1}^m {\log \#A_k}$.

The number of elements in $A_{m+1}$ is larger than $1$ only if $m$ is a Standard Step. 
In every Standard Step we have $\#A_{m+1} = \Theta(\exp(T\delta(f_E^{\brho+\bpi^m}(T_m)))$, and hence $\log \#A_{m+1} = T\delta(f_E^{\brho+\bpi^m}(T_m)) + O(1)$. 
Thus,
\[\liminf_{m\to \infty}\frac{1}{Tm}\sum_{m'=0}^m \log \#A_{m'} = \liminf_{m\to \infty}\frac{1}{m}\sum_{\substack{m'=0\\m'\text{ is a Standard Step}}}^m \delta(f_E^{\brho+\bpi^{m'}}(T_{m'}))+O(T^{-1}).\]
We will restrict the set of $m'$-s we sum over even further to obtain $f_E^{\brho+\bpi^{m'}}(T_{m'})=f_E(T_{m'})$. 

For every $t>0$ such that $\partial^2f_{H, l}(t) \nin (0,4C_1]$ for every $1\le l\le n-1$ 
we have $f_E^{\brho+\bpi^{m'}}(t) = f_E^{\brho+\bpi^{m'}}(t)$.
Since $f$ is $C_2$-significant, we have $\partial^2f_{H, l}(t)\le (0,4C_1]$ only if $t$ is in an $O(C_1)$-neighborhood of a null vertex of $f$. 

In addition, if $T_m$ is not in an $O(C_1)$-neighborhood of any vertex, $m$ is a Standard Step. 

Therefore, 
\begin{align}\label{eq: delta bound}
&\sum_{\substack{m'=0\\m'\text{ is a Standard Step}}}^m \delta(f_E^{\brho+\bpi^{m'}}(T_{m'})) 
\ge 
\frac{1}{T}\int_{J_m}\delta(f_E(t))dt
\\\nonumber&~~~\ge
\frac{1}{T}\int_{0}^{T_m}\delta(f_E(t))dt - O(C_1\cdot\#\{\text{vertices of $f$ until $T_m+O(C_1)$}\})
\\\nonumber&~~~\ge
\frac{1}{T}\int_{0}^{T_m}\delta(f_E(t))dt - O(C_1\cdot(2n^3+O(T_m/C_2))) = O(C_1 \cdot (1+T_m/C_2)),
\end{align}
where $J_m$ is the set of points $0\le t\le T_m$ such that $t$ is not in an $O(C_1)$-neighborhood of a vertex of $J$. 
Note that inequality \eqref{eq: delta bound} follows from Lemma \ref{lem: not many vertices}. 
Consequently, 
\[\liminf_{m\to \infty}\frac{1}{Tm}\sum_{m'=0}^m \log \#A_{m'} \ge \Delta_0(f) - O(TC_1/C_2).\]
If we choose a sequence of $C_2(T)\ggg C_1(T)\ggg T\to \infty$ such that $TC_1(T)/C_2(T)\xrightarrow{T\to \infty} 0$, we obtain the upper bound.
\subsection{The Payoff Guaranteed by Bob's Strategy} 
\label{sub:Bob_s_strategy_value}
In this subsection we prove the lower bound. To this end we fix $\varepsilon>0$ and consider Bob's strategy $\sigma_{\rm Bob}^T$, both described in Section \ref{sec:Bob_s_strategy}. 
We then show that when Bob follows this strategy, no matter how Alice plays (not necessarily according to the strategy described in Section \ref{sec:alice_s_strategy}), the payoff is at most $D_{\rm Bob}^T$, where $\lim_{T\to\infty} D_{\rm Bob}^T = \sup_{f\in \cF} \Delta_0(f)$.

As before, denote by $(T_0, h_0,A_1, h_1, \cdots)$ the choices of Alice and Bob along the game.
Denote $\bar\Lambda^m := g_{T_m}h_\infty\Lambda$, where $h_\infty:= \lim_{\to \infty} h_m$, and consider $f:=f^{h_\infty\Lambda}$. 
We will bound $\Delta(f)$ using $\#A_{m+1}$. Let $C$ be the constant given by Theorem \ref{thm:Lattice approx} for $\varepsilon$.
Assume that $h_\infty\Lambda$ $(\varepsilon, C)$-matches $f$.
Fix $\zeta_1 > C$ and $\zeta_2 > 4C$ to be determined later. Let $C_1>0$, to be defined later as well.

\begin{lem}\label{lem: bobs strategy interrupts}
Let $\zeta\ggg T$ and let $m\ge 0$ be a step for which $f_{E,\bullet}^{\brho(\zeta)}(t) = f_{E,\bullet}^{\brho(\zeta + \zeta_2)}(t)$ is constant for every $t\in [T_m - C - C_1, T_{m+1} + C + C_1]$. 
Then $\#A_{m+1} = O(\exp (T\delta(f_{E,\bullet}^{\brho(\zeta)}(T_m))))$.
\end{lem}
\begin{proof}
Denote the distinct elements of $L(f_{E,\bullet}^{\brho(\zeta)}(T_m))$ by $\{0,l_1,\dots,l_r,n\}$ and fix $t\in [T_m - C - C_1, T_{m+1} + C + C_1]$. 
The elements $\left(\zeta_{l_i} (f_{H,\bullet}(t))\right)_{i=1}^r$ are the $r$ largest elements in $\{\zeta_l (f_{H,\bullet}(t)):l=1,\dots,n-1\}$, and are larger than their successors by at least $\zeta_2$.
Since $h\Lambda$ $(\varepsilon, C)$-matches $f$, it follows that $\left(\zeta_{l_i}(\HN(g_{T_m}h\Lambda)_{H,\bullet}(t))\right)_{i=1}^{r}$ are the $r$ largest elements in $\{\zeta_l\HN(g_{t}h\Lambda)_{H,\bullet}:l=1,\dots,n-1\}$and are larger than their successors by at least $\zeta_2-4C$.
Recall from Subsection \ref{sub:geometric_picture} that $d_{H}(g_{T_m}h_m h_{\infty}^{-1}g_{-T_m}) \le 1$. 
Hence, if $\zeta_2-4C$ is large enough, 
$\left(\zeta_{l_i}(\HN(g_{T_m}h_m\Lambda)_{H,\bullet})\right)_{i=1}^r$ are the largest among all $\{\zeta_l(\HN(g_{T_m}h_m\Lambda)):1\le l\le n-1\}$.

Denote by $\Gamma_{l_i}\subseteq \Lambda$ the lattices that satisfy $g_{t}h\Gamma_{l_i} = \HN(g_{t}h\Lambda)_{\Gamma,l_i}$ for $t=T_{m}, T_{m+1}$.
Then $g_{t}h_{m+1}\Gamma_{l_i} = \HN(g_{t}h_m\Lambda)_{\Gamma,l_i}$.

Fix $\varepsilon_1>0$, and choose $C_1$ such that if for some $1\le l\le n-1$ and $V\in \gr$ we have $g_tV\in U_\varepsilon(\gr^g_E)$ for every $t \in [-C_1, C_1]$, then $g_tV\in U_{\varepsilon_1}(\gr^g_E)$. Such a constant exists by Theorem \ref{thm:lin path}.
Since $h\Lambda$ $(\varepsilon, C)$-matches $f$, it follows that
\[\forall 1\le i\le r,~\spa g_{t}h\Gamma_{l_i} \in U_\varepsilon\hspace{-3pt}\left(\grl{l_i}^g_{f_{E, l_i}}\right)\text{ for }T_m - C_1\le t \le T_{m+1}+C_1.\]
Consequently
\[\spa g_{t}h\Gamma_{l_i} \in U_{\varepsilon_1}\hspace{-4pt}\left(\grl{l_i}^g_{f_{E, l_i}}\right)\text{ for }t=T_m, T_{m+1}.\]

Let $M\in \SL_n(\RR)$ be a matrix with $d_{\SL_n(\RR)}(M, I) = O(\varepsilon_1)$ such that \[\forall 1\le i\le r,~M\spa g_{T_{m+1}}h\Gamma_{l_i} \in \grl{l_i}^g_{f_{E, l_i}(T_{m+1})}.\]
Using Lemma \ref{lem: mult map} we get $h'\in B_{d_{H}}(\Id;  O(\varepsilon_1))\subseteq B_{d_{H}}(\Id; 1)$ such that
\[\forall 1\le i\le r,~ h'\spa g_{T_{m+1}}h\Gamma_{l_i} \in \grl{l_i}^{O(\varepsilon_1), g}_{\to f_{E, l_i}}.\]
Then, $h'':=g_{-T}h'g_T\in B_{d_\varphi}(\Id;\exp(-T))$ satisfies 
$h''g_{T_m}h_\infty\Gamma \in \grl{l_i}^g_{\to f_{E, l_i}}$.
Since $d_\varphi(h_\infty, h_{m+1}) \le \exp(-T_{m+1})$, we have $g_{T_m}h_\infty h_{m+1}^{-1}g_{-T_m}\in B_{d_\varphi}(\Id;\exp(-T))$. Observe that
\[d_\varphi(h''\cdot g_{T_m}hh_{m+1}^{-1}g_{-T_m}, \Id)\le 2^{\alpha_\varphi}.\]
Since $\spa g_{T_{m}}h_\infty\Gamma_{l_i}\in U_{\varepsilon_1}\hspace{-4pt}\left(\grl{l_i}^g_{f_{E, l_i}}\right)$
we have \[\spa g_{T_{m}}h_{m+1}\Gamma_{l_i} = (g_{T_m}h_\infty h_{m+1}^{-1}g_{-T_m})^{-1}g_{T_{m}}h\Gamma_{l_i} \in U_\varepsilon(\grl{l_i}^g_{f_{E, l_i}}),\]
provided that $\varepsilon_1$ is small enough as a function of $\varepsilon$, and $T$ is large enough.




As a result, the $m$-th step is either an Idle Step 
or an Interruption Step, in which Bob interrupted more than $r$ indices. 
Otherwise, Bob's choice for $\widetilde h_{m+1}$ contradicts the existence of $h''$. 
By Lemma \ref{lem:the right counting theorem}, we conclude that $\#A_{m+1} = O(\exp (T\delta(f_{E,\bullet}^{\brho(\zeta)}(t))))$.
\end{proof}
\begin{cor}\label{cor: bobs strategy rulse}
There exists a $\zeta_1 \ggg T$ such that the following is true. 
Let $\zeta_3>0$ and let $m\ge 0$ be a step for which $f^{\brho(\zeta_3)}$ has no non-null vertices $t\in [T_m-C-C_1, T_{m+1}+C+C_1]$.
Then $\#A_{m+1} = O(\exp (T\delta(f_{E,\bullet}^{\brho(\zeta_1+\zeta_3)}(T_m)))$.
\end{cor}
\begin{proof}
For every $\zeta>0$ consider the map \[\psi_\zeta:[0,\zeta_2]\times [T_m-C-C_1, T_{m+1}+C+C_1]\to \{\text{height sequences}\},\]
which assigns 
\[(\zeta_0,t)\mapsto f_{E, \bullet}^{\brho(\zeta + \zeta_0)}(t).\]
Note that if $f$ has no non-null vertex in $[T_m-C-C_1, T_{m+1}+C+C_1]$, then the total length of 
$\left\{\zeta>0:\psi_\zeta\text{ is non-constant}
\right\}
$ is $O(T)$.
We get that $\psi_\zeta$ is constant
for some $\zeta_3\le \zeta\le \zeta_1+\zeta_3$ provided that $\zeta_3$ is large enough. 
Hence, the conditions of Lemma \ref{lem: bobs strategy interrupts} are met for $f^{\brho(\zeta_1)}$ and $\zeta$. 
The result follows since $\delta(f_{E,\bullet}^{\brho(\zeta)}(T_m))$ is nondecreasing as we increase $\zeta$ to $\zeta_1+\zeta_3$. 
\end{proof}

By Lemma \ref{lem:the right counting theorem}, for the empty direction filtration $0\subseteq \Eall$ one has that $\#A_{m+1} = O(\exp(DT))$.

By Corollary \ref{cor: bobs strategy rulse}, for every $\zeta_3>0$ we have
\begin{align*}
\sum_{m'=1}^m& \#A_{m'} \le \int_{0}^{T_{m+1}} \log \delta(f_{E,\bullet}^{\brho(\zeta_1+\zeta_3)}(t))dt + O(m) \\&~~+ D(2C+T+1)\#\{\text{non-null vertices of $f^{\brho(\zeta_3)}$ until $T_{m+1}$}\}.
\end{align*}

Since $t\mapsto \zeta_l(f_{H,\bullet}(t))$ is a Lipschitz function, $\max_{t\in J}\zeta_l(f_{H,\bullet}(t))=O(\vol(J))$ for every nontriviality interval $J$ of $f$ at $l$. 
Consequently, only connected components $J$ of $U_l(f)$ with $\vol(J) = \Omega(\zeta_3)$ are such that the function
$\zeta_l(f_{H,\bullet}^{\brho(\zeta_3)}(t))|_J$ does not vanish. 
There are at most $O(Tm/\zeta_3 + 1)$ connected components of $U_l(f)$ that are larger than $\zeta_3$ and intersect $[0,T_{m+1}]$. Each of them contains at most $n^2$ non-null vertices.
It follows that the number of non-null vertices of $f^{\brho(\zeta_3)}$ until $T_{m+1}$ is $O(Tm/\zeta_3+1)$. 
Consequently, 
\[\sum_{m'=1}^m \#A_{m'} \le \int_{0}^{T_{m+1}} \log \delta(f_{E,\bullet}^{\brho(\zeta_1+\zeta_3)}(t))dt + O(m+T^2m/\zeta_3),
\]
and hence for every $\zeta_3>0$ the payoff is 
\begin{align*}
D_{\rm Bob}^T &:= \liminf_{m\to \infty} \frac{1}{Tm}\sum_{m'=1}^m \#A_{m'} \\&\le 
\liminf_{m\to \infty}\int_{0}^{T_{m+1}} \log \delta(f_{E,\bullet}^{\brho(\zeta_1+\zeta_3)}(t))dt + O(T^{-1}+T/\zeta_3)\\ &= \Delta_0(f^{\brho(\zeta_1 + \zeta_3)}) + O(T^{-1} + T/\zeta_3).
\end{align*}
Therefore, $D_{\rm Bob}^T\le \Delta(f) + O(T^{-1})$, as desired.




\section{The dimension of the set of Divergent Trajectories} 
\label{sec:combinatorial_computations}
In this section we will prove Corollary \ref{cor: general scewed dimention of divergent}. In view of Theorem \ref{thm: general scewed dimention formula} we will be interested in computing $\sup_{f\in \cF}\Delta(f)$ for the set $\cF$ of $g$-templates that correspond to divergent trajectories.
Recall that  
\[D := \dim_\funH (H; d_\varphi) = \delta(\emptyset\subseteq \Eall) = \sum_{\eta, \eta'\in \Eall}(\eta-\eta')^+.\]
Denote \[\Xi = \sum_{\eta\in \Eall}\eta^+.\]
The following two lemmas together with Theorem \ref{thm: general scewed dimention formula} imply Corollary \ref{cor: general scewed dimention of divergent}. The first implies the upper bound and the second implies the lower bound.
Let $\cF_{\rm sing}:=\{g\text{-template }f:f_{H,1}\to \infty\}$. It follows that ${\rm Sing}(H, \Lambda; g_t) = Y_{\Lambda, \cF_{\rm sing}}$.
\begin{lem}\label{lem:upper bound diver}
$\Delta_0(f) \le D-\Xi$ for every $f\in \cF_{\rm sing}$.
\end{lem}
\begin{lem}\label{lem: there is a good template}
There is a $g$-template $f\in \cF_{\rm sing}$ with $\Delta(f) = D-\Xi$, provided that $n\ge 3$. 
\end{lem}
\subsection{Single Dimension Analysis} 
\label{sub:single_dimension_analysis}
We will study the behavior of the $g$-template $f$ only on a single $l$. 
Fix $n_+, n_-\ge 1$ such that there are at most $n_+$ positive $\eta_i$-s, at most $n_-$ negative $\eta_i$-s, and $n_+ + n_- = n$. (We use this ambiguous notation to preserve symmetry).
Fix $1\le l\le n-1$.
Denote \[\beta_l := \begin{cases}
\frac{n_+}{l},&\text{if }l\le n_+,\\
\frac{n_-}{n-l},&\text{if }l> n_+.
\end{cases}\]
For $E\in \cI_l$ denote $\delta(E):=\delta(\emptyset\subseteq E\subset \Eall)$.
\begin{lem} \label{lem:pre H inequality}
For every $E\in \cI_l$ we have
\[\delta(E) - \beta_l\eta_{E} \le D - \Xi\]
\end{lem}
Lemma \ref{lem:pre H inequality} is analogous to \cite[Lemma 8.1]{DFSU} and contains the combinatorial essence of \cite[Proposition 3.1]{KKLM}.
\begin{remark}\label{rem:Why inequality good}
The inequality is important because we need to integrate numbers of the form $\delta(f_{E,l}(t))$ in order to compute $\Delta_0(f)$ for a $g$-template $f$. Lemma \ref{lem:pre H inequality} will enable us to transform the problem into bounding an integral of $\frac{d}{dt}f_{H,l}(t)$, which is more accessible.
\end{remark}
\begin{remark}
In general equality occurs in Lemma \ref{lem:pre H inequality} only if $l=1, n-1$. In the case $l=1$, if $\Eall\cap \{0\} = \emptyset$, then equality occurs at $E = \{\eta\}$, whenever $\eta$ is the maximal negative or the minimal positive element in $\Eall$. 
If $\Eall\cap \{0\} \neq \emptyset$, then equality occurs for $E = \{0\}$, and perhaps at the maximal negative or the minimal positive elements in $\Eall$, depending on the choice of $n_+, n_-$. 
\end{remark}
\begin{proof}
We reduce the proof to a polynomial inequality on integers, which will be proved in the appendix.

Note that \[D-\delta(E) = \sum_{\substack{\eta\in \Eall-E\\\eta'\in E}}(\eta-\eta')^+.\]
Consequently, we need to show that 
\begin{align}\label{eq: reduction of ineq 1}
\sum_{\substack{\eta\in \Eall-E\\\eta'\in E}}(\eta-\eta')^+ + \beta_l\eta_E \ge \Xi.
\end{align}
Let $E'\subseteq \{1,\dots,n\}$ be a set for which $E = \{\eta_j:j\in E'\}$. 
Inequality \eqref{eq: reduction of ineq 1} becomes
\begin{align}\label{eq: reduction of ineq 2}
V(\eta_1,\dots,\eta_n):=\sum_{\substack{j\in \{1,\dots,n\}\setminus E'\\j'\in E}}(\eta_{j}-\eta_{j'})^+ + \beta_l\sum_{j\in E'}\eta_j - \sum_{j=n_- + 1}^n\eta_j \ge 0.
\end{align}
In this form the inequality is symmetric with respect to the mapping
\[\eta_j \mapsto -\eta_{n-j+1},\quad E \mapsto \{1\le j\le n:n-j+1 \notin E\},\quad n_+\leftrightarrow n_-,\]
and hence we may assume that $l\le n_+$.

Inequality \eqref{eq: reduction of ineq 2} is a linear homogeneous inequality in $(\eta_i)_{i=1}^n$, which satisfy the linear inequalities
\begin{align}\label{eq: eta cond 1}
&\eta_1\le \dots\le\eta_{n_-}\le 0\le \eta_{n_-+1}\le \dots\le \eta_n,\\\label{eq: eta cond 2}
&\sum_{i=1}^n\eta_i = 0.
\end{align}
The set of $\eta_i$ satisfying \eqref{eq: eta cond 1} and \eqref{eq: eta cond 2} is a cone defined by the rays
\[P_{a,b,c}:=(\underbrace{-c,\dots,-c}_{a\text{ times}},\underbrace{0,\dots,0}_{b\text{ times}},\underbrace{a,\dots,a}_{c\text{ times}});~~1\le a\le n_-, 1\le c\le n_+, a+b+c=n.\]
It is left to prove that $V(P_{a,b,c}) \ge 0$ for every such $a,b,c$.
Denote 
\[a_0 = \#E'\cap \{1,\dots,a\}, b_0 = \#E'\cap \{a+1,\dots,a+b\}, c_0 = \#E'\cap \{a+b+1,\dots,n\}.\]

Since $\beta_l = \tfrac{n_+}{l} = \tfrac{n_+}{a_0+b_0+c_0}$, it remains to show that
\begin{align}
V(P_{a,b,c}) = &(c-c_0)a_0(a+c) + (b-b_0)a_0c + (c-c_0)b_0a \\&\nonumber
+ \frac{n_+}{a_0+b_0+c_0}(-a_0c + c_0 a) - ac \ge 0.
\end{align}
This inequality is proved in Appendix \ref{sub:dimension_inequality}.
\end{proof}

\subsection{Upper Bound on Divergent Trajectories --- Lemma \ref{lem:upper bound diver}} 
\label{sub:nowhere_trivial_g}
Let $f$ be a $g$-template on $I$. 
For every $t\in \bigcup_{l=1}^{n-1}U_l(f)$ denote by \[\alpha_f(t) := \min_{l=1,\dots,n-1} \beta_lf_{H,l}(t).\]
This is a nonpositive continuous piecewise linear function, being the minimum of a family of such functions.
Note that $\alpha_f(t)<0$ for every $t\in \bigcup_{l=1}^{n-1}U_l(f)$. 
Since $\beta_l$ is convex in $l$, the minimum must be attained at a point $l_f(t)\in L_f(t)\setminus \{0,n\}$. 

Consequently, for every differentiable point $t\in \bigcup_{l=1}^{n-1}U_l(f)$ of $\alpha_f$ we have 

\[\frac{d}{dt}\alpha_f(t) = \frac{d}{dt}\beta_{l_f(t)}f_{H,l_f(t)}(t) = \beta_{l_f(t)}\eta_{f_{E,l_f(t)}(t)}.\]

The following lemma follows, and in turn implies Lemma \ref{lem:upper bound diver}.
\begin{lem}\label{lem: H inequality}
For every interval $(t_0,t_1)\subseteq \bigcup_{l=1}^{n-1}U_l(f)$,
\[
\int_{t_0}^{t_1}\delta(f_{E,*}(t))dt \le (t_0-t_1)(D-\Xi) + \alpha_f(t)|_{t=t_0}^{t_1}.
\]
\end{lem}
\begin{proof}
By Lemma \ref{lem:pre H inequality},
\begin{align}\label{eq: diffrential claim}
\delta(f_{E, *}(t)) \le \delta(f_{E,l_f(t)}(t))\le D-\Xi + \beta_{l_f(t)}\eta_{f_{E,l_f(t)}} = D-\Xi + \frac{d}{dt}\alpha_f(t).
\end{align}
Integrating Eq. \eqref{eq: diffrential claim} from $t_0$ to $t_1$ yields the result.
\end{proof}



\subsection{Construction of a Special \texorpdfstring{$g$}{g}-Template --- Proof of Lemma \ref{lem: there is a good template}} 
\label{sub:construction_of_a_special_g_template_proof_of_lemma_lem there is a good template}
For the rest of the section assume $n\ge 3$.
\begin{de}\label{de: standard E all}
\index{Standard weights}\hypertarget{bej}{}
We say that $\Eall$ is \emph{standard} if its elements are proportional to $-(n-1), 1,\dots,1$. This is the flow corresponding to the standard vector approximation. 
\end{de}
The standard case was dealt with in \cite[Section 9]{DFSU}. The arguments that we present below deals with the non-standard case, and can be adapted to handle the standard case as well. For the rest of the section assume $\Eall$ is not standard.
\begin{lem}[The connecting template]\label{lem:the connecting template}
Let $\eta_-$ be the maximal negative weight, and $\eta_+$ be the minimal positive weight. 
There exist a $g$-template $f^\eta$ and $C>0$ such that 
\begin{itemize}
	\item $f^\eta_{H, 1}(t) = \begin{cases}
	\eta_+ t - 1, & \text{if }t\le 0,\\
	\eta_- t - 1, & \text{if }t\ge 0.
	\end{cases}$
	\item For every $t\in \RR\setminus [-C, C]$,
	\[f^\eta_{E, l}(t) = \begin{cases}
	\{\eta_+\}, & \text{if }l = 1, t<-C,\\
	\{\eta_-\}, & \text{if }l = 1, t>C,\\
	*, & \text{otherwise.}
	\end{cases}\]
\end{itemize}
\end{lem}
\begin{proof}
Define
\[f^\eta_{E, 1}(t) := \begin{cases}
\{\eta_+\}, & \text{if }l = 1, t<-C,\\
\{\eta_-\}, & \text{if }l = 1, t>C,\\
*, & \text{otherwise.}
\end{cases}\]
By definition, $t=0$ is a null point of $f^\eta$ at $l=1$, hence must have $f^\eta_{E,l}(0) \neq *$ for some $1<l<n$. We will use $l=2$. 
Define $\tilde f^\eta_+(t) := (\eta_{-} + \eta_n)t - 2$ and let $t_+ := \tfrac{n}{(n-1)\eta_n + \eta_-} > 0$. Then $t_+$ is a solution of
\[\tilde f^\eta_{+,2}(t_+) = \frac{n-2}{n-1}f^\eta_{E, 1}(t_+) = \frac{n-2}{n-1}(\eta_-t_+ - 1).\] 
Since $\Eall$ is not standard, $(n-1)\eta_n + \eta_- >0$, since instead of considering the sum of all $\eta_i$, we take $n-1$ times the largest weight together with another weight. 
Similarly, set $\tilde f^\eta_-(t) := (\eta_{+} + \eta_1)t - 2$ and let $t_- := \tfrac{n}{(n-1)\eta_1 + \eta_+} < 0$. Then $t_-$ is the solution of
\[\tilde f^\eta_{+,2}(t_-) = \frac{n-2}{n-1}f^\eta_{E, 1}(t_-) = \frac{n-2}{n-1}(\eta_+t_- - 1).\] 
Define 
\[f^\eta_{E,2} := \begin{cases}
*, & \text{if }t < t_-\text{ or } t_+ <t,\\
\{\eta_+, \eta_1\}, & \text{if } t_- < t < 0,\\
\{\eta_-, \eta_n\}, & \text{if } 0 < t < t_+,\\
\end{cases}\]
with the nontrivial connecting morphism $\{\eta_+, \eta_1\}\to \{\eta_-, \eta_n\}$ at $0$. 
Define
\[f^\eta_{H,2}(t) = \begin{cases}
\tilde f^\eta_{-,2},& \text{if } t_- \le  t \le 0,\\
\tilde f^\eta_{+,2},& \text{if } 0 \le  t \le t_+,\\
\frac{n-2}{n-1}f^\eta_{H,1}(t), & \text{if } t\in (-\infty, t_-]\cup [t_+,\infty).
\end{cases}\]
Extrapolate $(f^\eta_{H,l})_{l=1,2}$ to $f^\eta_{H,\bullet}$ by convexity and the fact that no other $f^\eta_{E, l}$ is nontrivial.
\end{proof}
\begin{figure}[ht]
\label{fig:standard_template}
\begin{tikzpicture}[line cap=round,line join=round,>=triangle 45,x=1.0cm,y=1.0cm]
\clip(-2.7,-6.3) rectangle (2.7,.6);

[y = 0x + 0; x in [-7.583333333333334, 7.583333333333334]]
[y = -1.2x + -7.9; x in [-6.583333333333334, -5], y = 0.3x + -0.40000000000000036; x in [-5, 0], y = -1.2x + -0.40000000000000036; x in [0, 7.583333333333334]]
[y = -0.8999999999999999x + -0.8000000000000004; x in [-0.5714285714285716, 0], y = -0.29999999999999993x + -0.8000000000000004; x in [0, 2.0000000000000004]]
[y = 0x + 0; x in [-7.583333333333334, 7.583333333333334]]

\draw [line width=1.pt, -to] (-7,0) -- (2.6,0) node[above] {$t$};

\draw [line width=2.pt, ] (-6.583333333333334,0.0) -- (-5,-1.9000000000000004);
\draw [line width=2.pt, ] (-5,-1.9000000000000004) -- (0,-0.40000000000000036);
\draw [line width=2.pt, ] (0,-0.40000000000000036) -- (7.583333333333334,-9.5);
\draw [line width=2.pt, color=redcolor,] (-0.5714285714285716,-0.2857142857142859) -- (0,-0.8000000000000004);
\draw [line width=2.pt, color=redcolor,] (0,-0.8000000000000004) -- (2.0000000000000004,-1.4000000000000004);
\draw [line width=2.pt, dash pattern=on 5pt off 5pt,color=redcolor,] (-6.583333333333334,0.0) -- (-5.0,-0.9500000000000002);
\draw [line width=2.pt, dash pattern=on 5pt off 5pt,color=redcolor,] (-5.0,-0.9500000000000002) -- (-0.5714285714285716,-0.2857142857142859);
\draw [line width=2.pt, dash pattern=on 5pt off 5pt,color=redcolor,] (2.0000000000000004,-1.4000000000000004) -- (7.583333333333334,-4.749999999999999);
\begin{scriptsize}

\draw [fill=bluecolor] (-6.583333333333334,0.0) circle (2.5pt); 
\draw [fill=greencolor] (-5,-1.9000000000000004) circle (2.5pt); 
\draw [fill=bluecolor] (-0.5714285714285716,-0.2857142857142859) circle (2.5pt); 
\draw [fill=bluecolor] (0,-0.40000000000000036) circle (2.5pt); 
\draw [fill=greencolor] (0,-0.8000000000000004) circle (2.5pt); 
\draw [fill=bluecolor] (2.0000000000000004,-1.4000000000000004) circle (2.5pt); 
\draw [fill=bluecolor] (7.583333333333334,-9.5) circle (2.5pt); 
\end{scriptsize}

\draw [line width=2.pt, dash pattern=on 5pt off 5pt,] (-7.583333333333334,-4.1) -- (-6.583333333333334,-4.1);
\draw [color=black] (-7.083333333333334, -3.8) node {$*$};
\draw [line width=2.pt, ] (-6.583333333333334,-4.1) -- (-5,-4.1);
\draw [color=black] (-5.791666666666667, -3.8) node {$\{-1.2\}$};
\draw [line width=2.pt, ] (-5,-4.1) -- (0,-4.1);
\draw [color=black] (-1.5, -3.8) node {$\{0.3\}$};
\draw [line width=2.pt, ] (0,-4.1) -- (7.583333333333334,-4.1);
\draw [color=black] (1, -3.8) node {$\{-1.2\}$};
\draw [fill=bluecolor] (-7.583333333333334,-4.1) circle (2.5pt); 
\draw [fill=bluecolor] (-6.583333333333334,-4.1) circle (2.5pt); 
\draw [fill=greencolor] (-5,-4.1) circle (2.5pt); 
\draw [fill=bluecolor] (0,-4.1) circle (2.5pt); 
\draw [fill=greencolor] (7.583333333333334,-4.1) circle (2.5pt); 
\draw [line width=2.pt, dash pattern=on 5pt off 5pt,color=redcolor,] (-7.583333333333334,-5.1) -- (-0.5714285714285716,-5.1);
\draw [color=redcolor] (-1.5, -4.8) node {$*$};
\draw [line width=2.pt, color=redcolor,] (-0.5714285714285716,-5.1) -- (0,-5.1);
\draw [color=redcolor] (-0.2857142857142858, -5.4) node {$\{-1.2, 0.3\}$};
\draw [line width=2.pt, color=redcolor,] (0,-5.1) -- (2.0000000000000004,-5.1);
\draw [color=redcolor] (1.0000000000000002, -4.8) node {$\{-1.2, 0.9\}$};
\draw [line width=2.pt, dash pattern=on 5pt off 5pt,color=redcolor,] (2.0000000000000004,-5.1) -- (7.583333333333334,-5.1);
\draw [color=redcolor] (2.5, -4.8) node {$*$};
\draw [fill=bluecolor] (-7.583333333333334,-5.1) circle (2.5pt); 
\draw [fill=bluecolor] (-0.5714285714285716,-5.1) circle (2.5pt); 
\draw [fill=greencolor] (0,-5.1) circle (2.5pt); 
\draw [fill=bluecolor] (2.0000000000000004,-5.1) circle (2.5pt); 
\draw [fill=bluecolor] (7.583333333333334,-5.1) circle (2.5pt); 

\end{tikzpicture}
\caption{
The connecting $g$-template with $\Eall = \{-1.2, 0.3, 0.9\}$. 
}
\end{figure}
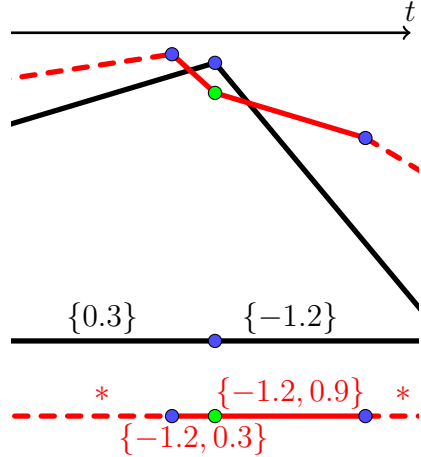

We will use Lemma \ref{lem:the connecting template} to glue simple $g$-templates and thereby generate more complex ones.
\begin{proof}[Proof of Lemma \emph{\ref{lem: there is a good template}}]
We will distinguish between two cases: 
\begin{enumerate}
\item\label{case: eta nonzero}
If $0 \nin \Eall$.
\item\label{case: eta zero}
If $0 \in \Eall$.
\end{enumerate}
\textbf{Case \eqref{case: eta nonzero}: }
Let $\eta_-$ be the largest negative weight and $\eta_+$ be the smallest positive weight. 
Let \[E_{-,1} := \{\eta_-\},\quad E_{+,1} := \{\eta_+\},\quad E_{\pm, 2} := \{\eta_+, \eta_-\}.\]
For every $t_0<t_1$ denote by $f^{t_0,t_1}$ the $g$-template on $\RR$ defined by:
\[f^{t_0, t_1}_{E, l}(t) := \begin{cases}
E_{-,1},& \text{if } t_0 < t < t_{1/2}, l=1,\\
E_{+,1},& \text{if } t_{1/2} < t < t_1, l=1,\\
*,& \text{otherwise},
\end{cases}\]
with the non-null connecting morphism $E_{-,1}\to E_{+,1}$ and with $t_{1/2} = \frac{\eta_+ t_1-\eta_-t_0}{\eta_+ - \eta_-}$ the only value for which the differential equation will define a $f^{t_0, t_1}_{H,1}(t)$ that vanishes at $t=t_1$.
\[f^{t_0, t_1}_{H, 1}(t) := \begin{cases}
\eta_-(t-t_0),& \text{if } t_0 \le t \le t_{1/2},\\
\eta_+(t-t_1),& \text{if } t_{1/2} \le t \le t_1,\\
0, & \text{otherwise},
\end{cases}\]
and extrapolate it by convexity to get $f^{t_0, t_1}_{H\bullet}(t)$ for every $t\in \RR$.

\begin{figure}[ht]
\label{fig:connecting_template}
\begin{tikzpicture}[line cap=round,line join=round,>=triangle 45,x=1.0cm,y=1.0cm]

\draw [line width=1.pt, -to] (-3,0) -- (3.2,0) node[above] {$t$};
\clip(-3,-3) rectangle (3,1);
\draw [color=black] (-2, 0.4) node {$t_0$};
\draw [color=black] (-1.2, 0.4) node {$t_{1/2}$};
\draw [color=black] (2, 0.4) node {$t_1$};
\draw [line width=2.pt, color=black] (-1.2,-0.1) -- (-1.2,0.1); 

\draw [line width=2.pt, dash pattern=on 5pt off 5pt,] (-7.1,0.03) -- (-2,0.03);
\draw [line width=2.pt, dash pattern=on 5pt off 5pt,color=redcolor] (-7,-.03) -- (-2,-.03);

\draw [line width=2.pt, dash pattern=on 5pt off 5pt,] (7,0.03) -- (2.0,0.03);
\draw [line width=2.pt, dash pattern=on 5pt off 5pt,color=redcolor] (7,-.03) -- (2,-.03);

\draw [line width=2.pt, ] (-2,0.0) -- (-1.2,-0.96);
\draw [line width=2.pt, ] (-1.2,-0.96) -- (2.0,0.0);
\draw [line width=2.pt, dash pattern=on 5pt off 5pt,color=redcolor,] (-2.0,0.0) -- (-1.2,-0.48);
\draw [line width=2.pt, dash pattern=on 5pt off 5pt,color=redcolor,] (-1.2,-0.48) -- (2.0,0.0);
\begin{scriptsize}

\end{scriptsize}

\draw [line width=2.pt, dash pattern=on 5pt off 5pt,] (-3,-2.0) -- (-2,-2.0);
\draw[color=black] (-2.5, -1.7) node {$*$};
\draw [line width=2.pt, ] (-2,-2.0) -- (-1.2,-2.0);
\draw[color=black] (-1.6, -1.7) node {$\{-1.2\}$};
\draw [line width=2.pt, ] (-1.2,-2.0) -- (2.0,-2.0);
\draw[color=black] (0.4, -1.7) node {$\{0.3\}$};
\draw [line width=2.pt, dash pattern=on 5pt off 5pt,] (2.0,-2.0) -- (3,-2.0);
\draw[color=black] (2.5, -1.7) node {$*$};
\draw [line width=2.pt, color=black] (-2,-2.1) -- (-2,-1.9); 
\draw [line width=2.pt, color=black] (-1.2,-2.1) -- (-1.2,-1.9); 
\draw [line width=2.pt, color=black] (2.0,-2.1) -- (2.0,-1.9); 
\draw [line width=2.pt, dash pattern=on 5pt off 5pt,color=redcolor,] (-3,-3.0) -- (3,-3.0);
\draw[color=redcolor] (0.0, -2.7) node {$*$};


\end{tikzpicture}
\caption{
The $g$-template $f^{t_0, t_1}$ with $\Eall = \{-1.2, 0.3, 0.9\}$. 
}
\end{figure}
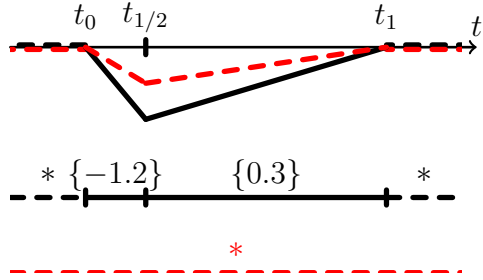

The definition was selected because it satisfies the following equality:
\begin{align*}
\int_{t_0}^{t_1}&\delta\left(f_{E, *}^{t_0, t_1}\right)dt = 
(t_{1/2}-t_0)\delta(\{\eta_-\}) + (t_1-t_{1/2})\delta(\{\eta_+\})\\
&=
(t_1-t_0)D
 - (t_{1/2} - t_0)\sum_{\eta\neq \eta_-}(\eta-\eta_+)^+
  - (t_1 - t_{1/2})\sum_{\eta\neq \eta_+}(\eta-\eta_+)^+ \\
&=
(t_1-t_0)D - (t_{1/2} - t_0)(\Xi - n_+\eta_-) - (t_1 - t_{1/2})(\Xi - n_+\eta_+)\\&=
(t_1-t_0)(D-\Xi).
\end{align*}

For every pair of intervals $t_0 < t_1, t_0' < t_1'$ such that $t_{1/2} < t_0' < t_1 < t_{1/2}'$,
we would like to define the \emph{merger} $f' = f^{t_0, t_1} \vee f^{t_0', t_1'}$ of $f^{t_0, t_1}$ and $f^{t_0', t_1'}$. Unfortunately, the connecting template protrudes beyonds $[t_0', t_1]$, and hence we will need additional constraints on $t_0, t_1, t_0', t_1'$ to define the merger.

\begin{figure}[ht]
\label{fig:a merged template}
\begin{tikzpicture}[line cap=round,line join=round,>=triangle 45,x=1.0cm,y=1.0cm]
\clip(-5.3,-9.3) rectangle (6.3,1);

\draw [line width=2.pt, dash pattern=on 5pt off 5pt,] (-7.1,0.03) -- (-5,0.03);
\draw [line width=2.pt, dash pattern=on 5pt off 5pt,color=redcolor] (-7.2,-.03) -- (-5,-.03);
\draw [line width=1.pt, -to] (-7,0) -- (5.5,0) node[above] {$t$};

\draw [color=black] (-5, 0.4) node {$t_0$};
\draw [color=black] (-3.9, 0.4) node {$t_{1/2}$};
\draw [line width=2.pt, color=black] (-3.9,-0.1) -- (-3.9,0.1); 
\draw [color=black] (.5, 0.4) node {$t_1$};
\draw [line width=2.pt, color=black] (.5,-0.1) -- (.5,0.1); 
\draw [color=black] (-.6, 0.4) node {$t_0'$};
\draw [line width=2.pt, color=black] (-.5,-0.1) -- (-.5,0.1); 

\draw [color=black] (-0.2, 0.4) node {$t_0''$};
\draw [line width=2.pt, color=black] (-.3,-0.1) -- (-.3,0.1); 

\draw [color=black] (2.1, 0.4) node {$t_{1/2}'$};
\draw [line width=2.pt, color=black] (2,-0.1) -- (2,0.1); 

\draw [color=black] (1.4, 0.4) node {$t_{1}''$};
\draw [line width=2.pt, color=black] (1.5,-0.1) -- (1.5,0.1); 

\draw [color=black] (-1.5, 0.4) node {$t_{-1}''$};
\draw [line width=2.pt, color=black] (-1.5,-0.1) -- (-1.5,0.1); 

\draw [line width=1.pt, color=black, dotted] (-0.3,-0.24) -- (.5, 0); 
\draw [line width=1.pt, color=black, dotted] (-0.3,-0.24) -- (-.5, 0); 



\draw [line width=2.pt, ] (-5,0.0) -- (-3.9,-1.3200000000000003);
\draw [line width=2.pt, ] (-3.9,-1.3200000000000003) -- (-0.3,-0.24000000000000035);
\draw [line width=2.pt, ] (-0.3,-0.24000000000000032) -- (1.9999999999999998,-3.0);
\draw [line width=2.pt, ] (1.9999999999999998,-3.0) -- (12.0,0.0);
\draw [line width=2.pt, color=redcolor,] (-0.642857142857143,-0.1714285714285716) -- (-0.3,-0.48000000000000026);
\draw [line width=2.pt, color=redcolor,] (-0.3,-0.4800000000000003) -- (0.9000000000000002,-0.8400000000000003);
\draw [line width=2.pt, dash pattern=on 5pt off 5pt,color=redcolor,] (-5.0,0.0) -- (-3.9,-0.6600000000000001);
\draw [line width=2.pt, dash pattern=on 5pt off 5pt,color=redcolor,] (-3.9,-0.6600000000000001) -- (-0.642857142857143,-0.17142857142857162);
\draw [line width=2.pt, dash pattern=on 5pt off 5pt,color=redcolor,] (0.9000000000000002,-0.8400000000000003) -- (1.9999999999999998,-1.4999999999999998);
\draw [line width=2.pt, dash pattern=on 5pt off 5pt,color=redcolor,] (1.9999999999999998,-1.5) -- (12.0,0.0);
\begin{scriptsize}

\draw [fill=bluecolor] (-5,0.0) circle (2.5pt); 
\draw [fill=greencolor] (-3.9,-1.3200000000000003) circle (2.5pt); 
\draw [fill=bluecolor] (-0.642857142857143,-0.1714285714285716) circle (2.5pt); 
\draw [fill=bluecolor] (-0.3,-0.24000000000000035) circle (2.5pt); 
\draw [fill=greencolor] (-0.3,-0.4800000000000003) circle (2.5pt); 
\draw [fill=bluecolor] (-0.3,-0.24000000000000032) circle (2.5pt); 
\draw [fill=bluecolor] (0.9000000000000002,-0.8400000000000003) circle (2.5pt); 
\draw [fill=greencolor] (1.9999999999999998,-3.0) circle (2.5pt); 
\draw [fill=bluecolor] (12.0,0.0) circle (2.5pt); 
\end{scriptsize}

\draw [line width=2.pt, dash pattern=on 5pt off 5pt,] (-6.1,-4.0) -- (-5,-4.0);
\draw [color=black] (-5.5, -3.7) node {$*$};
\draw [line width=2.pt, ] (-5,-4.0) -- (-3.9,-4.0);
\draw [color=black] (-4.45, -3.7) node {$\{-1.2\}$};
\draw [line width=2.pt, ] (-3.9,-4.0) -- (-0.3,-4.0);
\draw [color=black] (-2.1, -3.7) node {$\{0.3\}$};
\draw [line width=2.pt, ] (-0.3,-4.0) -- (1.9999999999999998,-4.0);
\draw [color=black] (0.8499999999999999, -3.7) node {$\{-1.2\}$};
\draw [line width=2.pt, ] (1.9999999999999998,-4.0) -- (12.0,-4.0);
\draw [color=black] (4.0, -3.7) node {$\{0.3\}$};
\draw [line width=2.pt, dash pattern=on 5pt off 5pt,] (12.0,-4.0) -- (13,-4.0);
\draw [color=black] (12.5, -3.7) node {$*$};
\draw [fill=bluecolor] (-6,-4) circle (2.5pt); 
\draw [fill=bluecolor] (-5,-4) circle (2.5pt); 
\draw [fill=greencolor] (-3.9,-4) circle (2.5pt); 
\draw [fill=bluecolor] (-0.3,-4) circle (2.5pt); 
\draw [fill=greencolor] (1.9999999999999998,-4) circle (2.5pt); 
\draw [fill=bluecolor] (12.0,-4) circle (2.5pt); 
\draw [fill=bluecolor] (13,-4) circle (2.5pt); 
\draw [line width=2.pt, dash pattern=on 5pt off 5pt,color=redcolor,] (-6,-5.0) -- (-0.642857142857143,-5.0);
\draw [color=redcolor] (-3.3214285714285716, -4.7) node {$*$};
\draw [line width=2.pt, color=redcolor,] (-0.642857142857143,-5.0) -- (-0.3,-5.0);
\draw [color=redcolor] (-0.47142857142857153, -5.3) node {$\{-1.2, 0.3\}$};
\draw [line width=2.pt, color=redcolor,] (-0.3,-5.0) -- (0.9000000000000002,-5.0);
\draw [color=redcolor] (0.30000000000000016, -4.7) node {$\{-1.2, 0.9\}$};
\draw [line width=2.pt, dash pattern=on 5pt off 5pt,color=redcolor,] (0.9000000000000002,-5.0) -- (13,-5.0);
\draw [color=redcolor] (4, -4.7) node {$*$};
\draw [fill=bluecolor] (-6,-5) circle (2.5pt); 
\draw [fill=bluecolor] (-0.642857142857143,-5) circle (2.5pt); 
\draw [fill=greencolor] (-0.3,-5) circle (2.5pt); 
\draw [fill=bluecolor] (0.9000000000000002,-5) circle (2.5pt); 
\draw [fill=bluecolor] (13,-5) circle (2.5pt); 

\end{tikzpicture}
\caption{
The merger of $f^{t_0, t_1}$ and $f^{t_0', t_1'}$, where $t_0 = -5, t_1 = 0.5, t_0' = -0.5, t_1' = 12$, and $\Eall = \{-1.2, 0.3, 0.9\}$.
}
\end{figure}
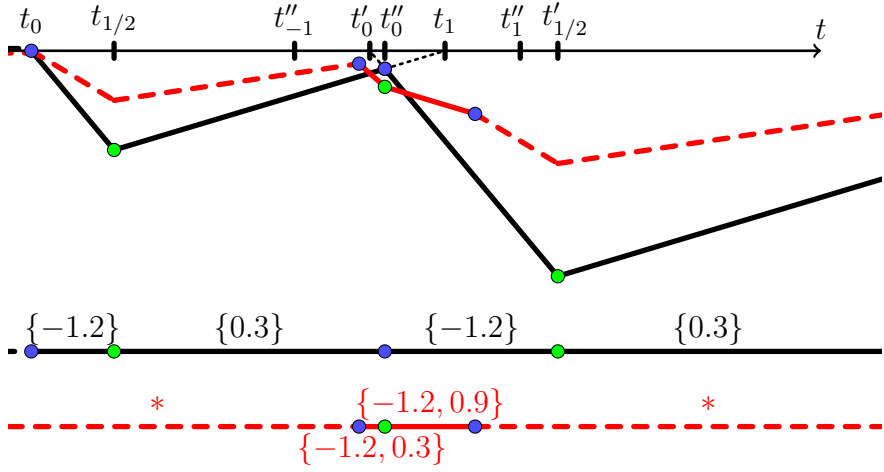

Our goal is to ensure that the merger will satisfy $f'_{H, 1}(t) = \max(f^{t_0,t_1}_{H,1}(t), f^{t_0',t_1'}_{H,1}(t))$.
Let $t''_0 := \frac{\eta_+ t_1-\eta_-t_0'}{\eta_+ - \eta_-}$ denote the unique point $t''_0\in [t_{1/2}, t'_{1/2}]$ for which \[f_{H,1}^{t_0, t_1}(t''_0) = f_{H,1}^{t_0', t_1'}(t''_0) = \frac{t_1-t_0'}{\eta_+^{-1}-\eta_-^{-1}}.\]
Denote \[\kappa := \tfrac{1}{\eta_+^{-1}-\eta_-^{-1}}\quad \text{and}\quad C' := C\kappa,\] 
where $C$ is the constant from Lemma \ref{lem:the connecting template}.
We define the merger $f^{t_0,t_1}\vee f^{t_0', t_1'}$ only if the following inequality holds: 
\[t_{-1}'':=t_0' - (t_1-t_0')C' > t_{1/2}\text{ and }t_{1}'':=t_1 + (t_1-t_0')C < t_{1/2}'.\] 

For $t\in \RR$, define $\tilde f^\eta$ by
\[\tilde f^\eta_{E, l}(t) := f^\eta_{E, l}\left(\frac{t-t''_0}{(t_1-t_0')\kappa}\right),\quad \tilde f^\eta_{H,l}(t):= (t_1-t_0')\kappa f^\eta_{H, l}\left(\frac{t-t''_0}{(t_1-t_0')\kappa}\right).\]
In other words, we translate $f^\eta$ and apply a homothety so that the increasing and decreasing branches of $\tilde f^\eta_{H,1}$ will coincide with the corresponding branches of $f^{t_0,t_1}_{H,1}, f^{t_0',t_1'}_{H,1}$ at $t_{-1}'', t_{1}''$, respectively.
We can then glue the three $g$-templates
$f^{t_0, t_1}|_{(-\infty, t_{-1}'']},\, \tilde f^\eta|_{[t_{-1}'', t_{1}'']},\, f^{t_0', t_1'}|_{[t_{1}'', \infty)}$.

Using this method we can merge arbitrarily many $g$-templates $f^{t_0,t_1}$. 
Let $t_0^{(m)}<t_1^{(m)}, m\ge 0$, be an infinite sequence of pairs with $d^m := t_1^{(m)} - t_0^{(m+1)} > 0 $ and 
\[t_{1/2}^{(m)} \le t_0^{(m+1)} - C'd_m < t_1^{(m)} + C'd_m < t_{1/2}^{m+1}.\]

We can merge the $g$-templates $(f^{t_0^{(m)},t_1^{(m)}})_{m\ge 0}$ into a single $g$-template $f = \bigvee_{m=0}^\infty f^{t_0^{(m)},t_1^{(m)}}$. 
Note that \[f|_{[t_1^{(m)} + C'd_m, t_0^{(m+2)} - C'd_{m+1}]} = f^{t_0^{(m+1)},t_1^{(m+1)}}.\]
Assume that 
\begin{align}
d_m\xrightarrow{m \to \infty} \infty,\label{eq: d_m to infty}
\\
\frac{d_m}{t^{(m)}_1 - t^{(m)}_0} \xrightarrow{m \to \infty} 0,\label{eq: d_m over length to 0}
\\
\frac{t^{(m)}_1 - t^{(m)}_0}{\sum_{m'=0}^m t^{(m')}_1 - t^{(m')}_0}\xrightarrow{m\to \infty} 0.
\label{eq: single over sum to 0}
\end{align}
For instance, the sequence $t_0^{(m)} = m^3 - m, t_1^{(m)} = (m+1)^3 + m$ works for $m$ large enough as a function of $C'$. 
We now wish to compute 
\begin{align*}
\Delta_0(f) = \lim_{T\to \infty}\frac{1}{T}\int_0^T\delta(f_{E, *}(t))dt.
\end{align*}
By Eq. \eqref{eq: single over sum to 0}, it is enough to compute the limit for $T = t_1^{(m)}$.
By Eq. \eqref{eq: d_m over length to 0}, the regions in which $f$ is not one of the $(f^{t_0^{(m)},t_1^{(m)}})_{m\ge 0}$ are negligible. We have
\begin{align*}
\Delta_0(f) &= \lim_{m\to \infty}\frac{1}{t_{1}^{(m)}}\int_0^{t_{1}^{(m)}}\delta(f_{E, *}(t))dt\\
&= \lim_{m\to \infty}\frac{1}{t_{1}^{(m)}}\sum_{m'=0}^m\int_{t_{0}^{(m')}}^{t_{1}^{(m')}}\delta(f_{E, *}^{t_{0}^{(m')},t_{1}^{(m')}}(t))dt\\
&= \lim_{m\to \infty}\frac{1}{t_{1}^{(m)}}\sum_{m'=0}^m(t_{1}^{(m')}-t_{0}^{(m')})(D-\Xi)\\
&= \lim_{m\to \infty}\frac{1}{t_{1}^{(m)}}\sum_{m'=1}^m(t_{1}^{(m')}-t_{1}^{(m'-1)})(D-\Xi)\\
&= D-\Xi.
\end{align*}
Similarly, the local minima of $f_{E,1}(t)$ are $(\kappa d_m)_{m\ge 0}$, which tend to infinity, and hence 
$f_{E,1}(t)\xrightarrow{t\to \infty}-\infty$, as desired.

\textbf{Case \eqref{case: eta zero}:}
The construction is similar. We define $f^{t_0, t_1, \varepsilon}$ by 
\begin{align*}
f^{t_0, t_1, \varepsilon}_{H,1}(t)&:=
\begin{cases}
\max\left((t-t_0)\eta_-, t-t_1\eta_+, -\varepsilon (t_1-t_0)\right),& \text{if }t \in [t_0, t_1],\\
0,&\text{if }t \nin [t_0, t_1],
\end{cases}\\
f^{t_0, t_1, \varepsilon}_{E,l}(t) &:= 
\begin{cases}
\left\{\frac{d}{dt}f^{t_0, t_1, \varepsilon}_{H,1}(t)\right\}, &\text{if }t \in (t_0,t_1) \text{ and } l=1,\\
*, &\text{otherwise,}
\end{cases}
\end{align*}
and extrapolate $f^{t_0, t_1, \varepsilon}_{H,l}$ to $f^{t_0, t_1, \varepsilon}_{H,\bullet}$ by convexity.
\begin{figure}[ht]
\label{fig:simple_template_2}
\begin{tikzpicture}[line cap=round,line join=round,>=triangle 45,x=1.0cm,y=1.0cm]

\draw [line width=1.pt, -to] (-3,0) -- (4.2,0) node[above] {$t$};
\clip(-3,-4.6) rectangle (4,1);

\draw [line width=2.pt, dash pattern=on 5pt off 5pt,] (-7.1,0.06) -- (-2,0.06);
\draw [line width=2.pt, dash pattern=on 5pt off 5pt,color=redcolor] (-7,0) -- (-2,0);
\draw [line width=2.pt, dash pattern=on 5pt off 5pt,color=pinkcolor] (-6.9,-.06) -- (-2,-.06);

\draw [line width=2.pt, dash pattern=on 5pt off 5pt,] (7.2,0.06) -- (3.4,0.06);
\draw [line width=2.pt, dash pattern=on 5pt off 5pt,color=redcolor] (7.1,0) -- (3.4,0);
\draw [line width=2.pt, dash pattern=on 5pt off 5pt,color=pinkcolor] (7,-0.06) -- (3.4,-0.06);

\draw [color=black] (-2.5, -0.3) node {$f_{H, \bullet}:$};

\draw [color=black] (-2.5, -2.2) node {$f_{E, 1}$};
\draw [color=redcolor] (-2.5, -3.2) node {$f_{E, 2}$};
\draw [color=pinkcolor] (-2.5, -4.2) node {$f_{E, 3}$};

\draw [color=black] (-2, 0.4) node {$t_0$};
\draw [color=black] (-1, 0.4) node {$t_{1/3}$};
\draw [color=black] (1, 0.4) node {$t_{2/3}$};
\draw [color=black] (3.4, 0.4) node {$t_1$};
\draw [line width=2.pt, color=black] (-1,-0.1) -- (-1,0.1); 
\draw [line width=2.pt, color=black] (1,-0.1) -- (1,0.1); 

\draw [line width=2.pt, ] (-2,0.0) -- (-1,-1.2);
\draw [line width=2.pt, ] (-1,-1.2) -- (1,-1.2);
\draw [line width=2.pt, ] (1,-1.2) -- (3.4,0.0);
\draw [line width=2.pt, dash pattern=on 5pt off 5pt,color=redcolor,] (-2.0,0.0) -- (-1.0,-0.7999999999999999);
\draw [line width=2.pt, dash pattern=on 5pt off 5pt,color=redcolor,] (-1.0,-0.7999999999999999) -- (1.0,-0.7999999999999999);
\draw [line width=2.pt, dash pattern=on 5pt off 5pt,color=redcolor,] (1.0,-0.8) -- (3.4,0.0);
\draw [line width=2.pt, dash pattern=on 5pt off 5pt,color=pinkcolor,] (-2.0,0.0) -- (-1.0,-0.39999999999999997);
\draw [line width=2.pt, dash pattern=on 5pt off 5pt,color=pinkcolor,] (-1.0,-0.39999999999999997) -- (1.0,-0.39999999999999997);
\draw [line width=2.pt, dash pattern=on 5pt off 5pt,color=pinkcolor,] (1.0,-0.4) -- (3.4,0.0);
\begin{scriptsize}

\draw [fill=bluecolor] (-2,0.0) circle (2.5pt); 
\draw [fill=greencolor] (-1,-1.2) circle (2.5pt); 
\draw [fill=greencolor] (1,-1.2) circle (2.5pt); 
\draw [fill=bluecolor] (3.4,0.0) circle (2.5pt); 
\end{scriptsize}

\draw [line width=2.pt, dash pattern=on 5pt off 5pt,] (-7,-2.5) -- (-2,-2.5);
\draw[color=black] (-4.5, -2.2) node {$*$};
\draw [line width=2.pt, ] (-2,-2.5) -- (-1,-2.5);
\draw[color=black] (-1.5, -2.2) node {$\{-1.2\}$};
\draw [line width=2.pt, ] (-1,-2.5) -- (1,-2.5);
\draw[color=black] (0.0, -2.2) node {$\{0\}$};
\draw [line width=2.pt, ] (1,-2.5) -- (3.4,-2.5);
\draw[color=black] (2.2, -2.2) node {$\{0.5\}$};
\draw [line width=2.pt, dash pattern=on 5pt off 5pt,] (3.4,-2.5) -- (7,-2.5);
\draw[color=black] (5.2, -2.2) node {$*$};
\draw [fill=bluecolor] (-7,-2.5) circle (2.5pt); 
\draw [fill=bluecolor] (-2,-2.5) circle (2.5pt); 
\draw [fill=greencolor] (-1,-2.5) circle (2.5pt); 
\draw [fill=greencolor] (1,-2.5) circle (2.5pt); 
\draw [fill=bluecolor] (3.4,-2.5) circle (2.5pt); 
\draw [fill=bluecolor] (7,-2.5) circle (2.5pt); 
\draw [line width=2.pt, dash pattern=on 5pt off 5pt,color=redcolor,] (-7,-3.5) -- (7,-3.5);
\draw[color=redcolor] (0.0, -3.2) node {$*$};
\draw [fill=bluecolor] (-7,-3.5) circle (2.5pt); 
\draw [fill=bluecolor] (7,-3.5) circle (2.5pt); 
\draw [line width=2.pt, dash pattern=on 5pt off 5pt,color=pinkcolor,] (-7,-4.5) -- (7,-4.5);
\draw[color=pinkcolor] (0.0, -4.2) node {$*$};
\draw [fill=bluecolor] (-7,-4.5) circle (2.5pt); 
\draw [fill=bluecolor] (7,-4.5) circle (2.5pt); 
\end{tikzpicture}
\caption{
The $g$-template $f^{t_0, t_1, \varepsilon}_{H,l}$ with $\Eall = \{-1.2, 0, 0.3, 0.9\}$. 
}
\end{figure}
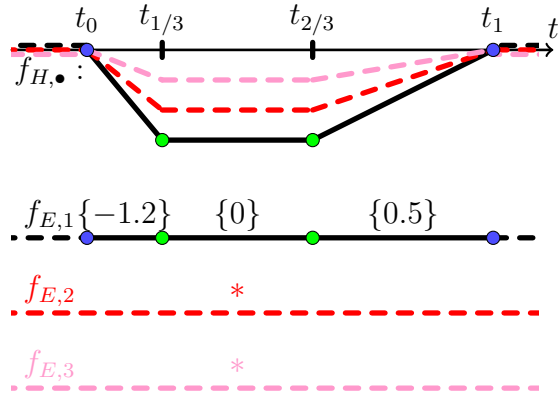
Let $t_{1/3} < t_{2/3}$ denote the two non-null vertices of $f^{t_0, t_1, \varepsilon}$ at $1$.
Then \[\int_{t_{1/3}}^{t_{2/3}} \delta\left(f^{t_0, t_1, \varepsilon}_{E,*}(t)\right)dt = (t_{2/3}-t_{1/3})(D-\Xi).\]

Hence the merger of a sequence of such functions with $\varepsilon\to 0$ provides the desired function as in Case \eqref{case: eta nonzero}.
\end{proof}
\section{Appendix} 
\label{sec:appendix}
\subsection{Dimension Inequality} 
Here we will prove the following inequality that was used in the proof of Lemma \ref{lem:pre H inequality}.
\label{sub:dimension_inequality}
\begin{lem}
Let $a\ge a_0\ge 0, b\ge b_0\ge 0, c\ge c_0\ge 0$ be integers with $a_0+b_0+c_0 \ge 1$, and let $n_+ \in [c,c+b]$. Then
\begin{align*}
V(a,b,c,a_0,b_0,c_0, n_+) := &(c-c_0)a_0(a+c) + (b-b_0)a_0c + (c-c_0)b_0a \\&\nonumber
+ \frac{n_+}{a_0+b_0+c_0}(-a_0c + c_0 a) - ac \ge 0
\end{align*}
provided that $a_0+b_0+c_0\le n_+$. 
\end{lem}
\begin{proof}
Since $V$ is linear in $n_+$ we may assume $n_+$ is one of its boundary values:
\begin{enumerate}
\item $n_+=c+b$,
\item $n_+=c$, or
\item $n_+=a_0+b_0+c_0$. 
\end{enumerate}
First we will show that we may decrease $b$ until $b = b_0$ or $b>b_0$ and $b+c = n_+ = a_0+b_0+c_0$. 
We will show that
\[dV:=V(a,b,c,a_0,b_0,c_0, n_+) - V(a,b-1,c,a_0,b_0,c_0, n_+') \ge 0,\]
where $n_+' = n_+$ in cases $2,3$ or $n_+' = n_+-1$ in case $1$.
In cases $2,3$ we see that $dV = a_0c\ge 0$.
In case $1$ we have 
$dV = \frac{-a_0c + c_0a}{a_0+b_0+c_0} + a_0c \ge 0$ as $a_0+b_0+c_0 \ge 1$. 

The above three cases become four:

\begin{enumerate}
 \item \label{case: appendix 1}$n_+ = c+b$, $b = b_0$.
\item \label{case: appendix 2}$n_+ = c$, $b = b_0$.
\item \label{case: appendix 3}$n_+ = a_0 + b_0 + c_0$, $b = b_0$. 
\item \label{case: appendix 4}$n_+ = a_0 + b_0 + c_0 = c + b$. 
\end{enumerate}
Denote $a_1 = a - a_0$, $b_1 = b-b_0, c_1 = c-c_0$. 

\textbf{Case \eqref{case: appendix 1}:}
Note that since $c+b = n_+ \ge a_0+b_0+c_0 =a_0 + b+c_0$ we get that $c_1\ge a_0$. 
Since $a_0 ^ 2 -a_0, b_0^2 - b_0\ge 0$.
By simple algebraic manipulations,
\begin{align*}
V(a,b,&c,a_0,b_0,c_0, n_+)\cdot (a_0+ b_0 + c_0)\\
=&(a_0^2-a_0)  (a_0c_1 + a_1c_1 + 2c_0c_1 + a_1  c_0 + c_0^2 + c_1^2 + a_0c_0 + b_0  c_0 + 2  b_0  c_1) \\&+ 
(c_1 - a_0)  (a_0  a_1  c_0 + a_0  c_0^2 + a_0c_0c_1 + a_0^2c_0 + a_0b_0c_0) \\&+
(b_0^2 - b_0)  a_1c_1 \\& +
 2  a_0  a_1  b_0  c_1 + a_0  b_0^2  c_1 + a_1b_0c_0c_1 + a_0  b_0  c_1^2 + a_0  b_0  c_0  c_1
\\\ge& 0.
\end{align*}

\textbf{Case \eqref{case: appendix 2}:}
Note that if $b_0=0$ then we were in the previous case, hence assume $b_0 \ge 1$. 
Since $c=n_+ \ge a_0+b_0+c_0$ we get that $c_1\ge a_0+b_0$. In particular, $c_1\ge b_0\ge 1$.
By simple algebraic manipulations,
\begin{align*}
V(a,b,&c,a_0,b_0,c_0, n_+)\cdot (a_0+ b_0 + c_0)\\ =& 
(a_0^2-a_0)(a_0c_1 + b_0c_1 + c_1^2) \\&+
(b_0 - 1)(a_0a_1c_1 + a_1b_0c_1+ 2a_0c_0c_1) \\&+
(c_1 - 1)  (a_0^2c_0 + a_0a_1c_0 + a_1b_0c_0 + a_0c_0^2 + a_0c_0c_1) \\&+ 
(c_1-b_0) a_0c_0 \\&+ 
a_0^2a_1c_1 + a_0^2b_0c_1 + a_0a_1b_0c_1 + a_0b_0^2c_1 + a_0^2c_0c_1 + a_0b_0c_1^2
\\\ge &0.
\end{align*}

\textbf{Case \eqref{case: appendix 3}:}
Note that  $a_0+b_0+c_0= n \in [c,c+b_0]$. If $n$ was one of the endpoints, we would be in the previous cases. Hence, $c<a_0+b_0+c_0 < c+b_0$. Equivalently $a_0<c_1<a_0+b_0$. In particular, $b_0 \ge 2, c_1\ge 1$.

By simple algebraic manipulations,
\begin{align*}
V(a,b,&c,a_0,b_0,c_0, n_+) \\
=&(b_0-2)(a_0  c_1 + a_1  c_1)\\
&+ (c_1-1)a_0c_0\\
&+ a_1c_1 + a_0^2c_1 + a_0a_1c_1 + a_0c_1^2 \\
\ge & 0.
\end{align*}

\textbf{Case \eqref{case: appendix 4}:}
The equation defining \eqref{case: appendix 4} implies $a_0 = c_1 + b_1$. If we had $b_1=0$ then we were in Case \eqref{case: appendix 1}. 
Consequently, $a_0\ge b_1\ge 1$.
By simple algebraic manipulations,
\begin{align*}
V(a,b,&c,a_0,b_0,c_0, n_+) \\
=&(b_1 - 1)(a_0c_0 + a_0c_1)\\
&+(a_0 - 1) (a_0c_1 + a_1c_1) \\&+ 
a_0b_0c_1 + a_1b_0c_1+ a_0c_0c_1 + a_0c_1^2\\\ge & 0
\end{align*}
\end{proof}

\begin{theindex}

  \item $*$, \hyperlink{ee}{9}
  \item $\#$, \hyperlink{dj}{7}
  \item $\eteq$, \hyperlink{ea}{7}
    \subitem $E_\bullet \eteq E'_\bullet$, \hyperlink{bcb}{57}
  \item $\normi\cdot\normi$, \hyperlink{dd}{7}
  \item $\ggg$, \hyperlink{hd}{16}

  \indexspace

  \item $a_\bullet$, \hyperlink{cj}{6}
  \item Aligned noise lattice, \hyperlink{bdi}{73}
  \item $\alpha_\varphi$, \hyperlink{ic}{21}
  \item $A_m$, \hyperlink{ig}{25}
  \item $\widetilde A_{m+1}$, \hyperlink{bdf}{73}
  \item Anchor, \hyperlink{bdj}{74}

  \indexspace

  \item Basic intervals, \hyperlink{baj}{41}
  \item $B_d(x;r)$, \hyperlink{he}{16}
  \item Blades, \hyperlink{de}{7}
  \item $\BL$, \hyperlink{df}{7}
  \item $\bl\Gamma$, \hyperlink{dh}{7}
  \item boundedly expanding action, \hyperlink{hf}{18}
  \item $B_T$, \hyperlink{bdg}{73}

  \indexspace

  \item Category flow, \hyperlink{ej}{10}
  \item Complementary interval, \hyperlink{bba}{41}
  \item Constructible graph, \hyperlink{bai}{33}
  \item $\cov \Gamma$, \hyperlink{bh}{5}
  \item $\cov(\Gamma_2/\Gamma_1/\Gamma_0)$, \hyperlink{bcg}{66}

  \indexspace

  \item $D$, \hyperlink{gh}{15}
  \item $\delta(E_\bullet)$, \hyperlink{gc}{14}
  \item $\Delta_0(f)$, \hyperlink{ge}{14}
  \item $\Delta(f)$, \hyperlink{gf}{14}
  \item $d_{\gr}$, \hyperlink{cb}{6}
  \item $d_{H}$, \hyperlink{bf}{5}
  \item $\dim_\funH$, \hyperlink{hi}{19}
  \item Direction filtration, \hyperlink{cg}{6}
  \item Doubling metric, \hyperlink{ij}{25}
  \item $d_\varphi$, \hyperlink{ib}{21}
  \item $d_{\SL_n}$, \hyperlink{f}{3}
  \item $d_{X_n}$, \hyperlink{e}{3}

  \indexspace

  \item $e_{a,b}$, \hyperlink{bbd}{49}
  \item $\Eall$, \hyperlink{i}{4}
  \item $E_\bullet$, \hyperlink{ch}{6}
  \item $e_{i}$, \hyperlink{bbe}{49}
  \item Elementary flip, \hyperlink{bcd}{59}
  \item Enticement, \hyperlink{bac}{32}
    \subitem compatible with a graph, \hyperlink{bah}{33}
    \subitem finitary collection, \hyperlink{bag}{32}
    \subitem satisfaction, \hyperlink{bae}{32}
    \subitem solution, \hyperlink{baf}{32}
    \subitem support, \hyperlink{bad}{32}
  \item Equivalent trajectories, \hyperlink{j}{4}
  \item $\eta_E$, \hyperlink{ec}{8}
  \item $\eta_i$, \hyperlink{g}{3}
  \item expansion semi-metric, \hyperlink{hg}{18}

  \indexspace

  \item $f$, \hyperlink{ff}{11}
    \subitem $f_{E,\bullet}$, \hyperlink{fg}{11}
    \subitem $f_{H,\bullet}$, \hyperlink{fh}{11}
  \item $\overline F$, \hyperlink{id}{21}
  \item $\HN$, \hyperlink{bj}{6}
  \item Final multiset, \hyperlink{bch}{70}
  \item Final noise lattice, \hyperlink{bdh}{73}
  \item Flag, \hyperlink{ce}{6}
  \item $f^\Lambda$, \hyperlink{gb}{12}
  \item $\underline F$, \hyperlink{ie}{21}

  \indexspace

  \item Game - $(T, g)$, \hyperlink{if}{24}
  \item $\Gamma_\bullet^m$, \hyperlink{bdb}{72}
  \item $\widetilde \Gamma^m_\bullet$, \hyperlink{bdd}{72}
  \item $\gamma_*$, \hyperlink{bcj}{72}
  \item $\cG_f$, \hyperlink{ga}{11}
  \item $\gr$, \hyperlink{ca}{6}
    \subitem $\gr_E^g$, \hyperlink{cd}{6}
    \subitem $\gr_{\to E}$, \hyperlink{bbg}{51}
    \subitem $\gr_{\to E}^{\varepsilon, g}$, \hyperlink{bca}{57}
  \item $g_t$, \hyperlink{h}{4}
  \item $g$-template, \hyperlink{fd}{11}
    \subitem separated, \hyperlink{je}{27}
    \subitem significant, \hyperlink{jg}{28}

  \indexspace

  \item $H$, \hyperlink{ba}{4}
    \subitem $H_V$, \hyperlink{bbb}{48}
    \subitem $H_{V\to E}$, \hyperlink{bbh}{51}
  \item $H^-$, \hyperlink{bb}{4}
  \item $H^0$, \hyperlink{bc}{4}
  \item $H^{-0}$, \hyperlink{bd}{4}
  \item $\check \funh$, \hyperlink{bbi}{55}
  \item Harder-Narasimhan filtration, \hyperlink{bg}{5}
  \item Hausdorff dimension, \hyperlink{hh}{19}
    \subitem with respect to a semi-metric, \hyperlink{hj}{20}
  \item Height sequence, \hyperlink{da}{6}
  \item $h_\infty$, \hyperlink{ii}{25}
  \item $h_m$, \hyperlink{ih}{25}
  \item $\widetilde h^{m_0\to m_1}$, \hyperlink{bea}{74}

  \indexspace

  \item $\cI_l^*$, \hyperlink{ef}{9}
  \item $\cI_l$, \hyperlink{cc}{6}
  \item Independent shift sequence, \hyperlink{bab}{30}
  \item irregular point, \hyperlink{fc}{10}
  \item $\cI_\bullet^*$, \hyperlink{ei}{10}

  \indexspace

  \item $L(-)$, \hyperlink{ci}{6}
  \item $\widetilde \Lambda^m$, \hyperlink{bdc}{72}
  \item $L_f(t)$, \hyperlink{ja}{26}
  \item Lower convex hull, \hyperlink{jh}{28}

  \indexspace

  \item $\mLU$, \hyperlink{be}{4}
  \item $\mult(x;E)$, \hyperlink{bcc}{59}

  \indexspace

  \item $n$, \hyperlink{d}{3}
  \item $\NN_0$, \hyperlink{bce}{60}
  \item Noise lattice, \hyperlink{bde}{73}
  \item Nontrivial place, \hyperlink{db}{6}
  \item Nontrivial point, \hyperlink{fb}{10}
  \item Nontriviality interval, \hyperlink{fj}{11}
  \item Null arrow, \hyperlink{eh}{9}
  \item Null point, \hyperlink{fa}{10}

  \indexspace

  \item $O(-)$, \hyperlink{ha}{16}
  \item $O_g$, \hyperlink{bbc}{49}
  \item $\Omega(-)$, \hyperlink{hc}{16}

  \indexspace

  \item $p_\Gamma$, \hyperlink{bi}{5}
  \item $\partial^2a_\bullet$, \hyperlink{dc}{6}
  \item $\varphi_t$, \hyperlink{ia}{20}
  \item $\widetilde \bpi^m$, \hyperlink{bda}{72}
  \item $(-)^+$, \hyperlink{gd}{14}
  \item $P_m$, \hyperlink{bef}{82}

  \indexspace

  \item $\check \funq$, \hyperlink{bbj}{55}

  \indexspace

  \item Rectangle, \hyperlink{bee}{81}
  \item Relative covolume, \hyperlink{bcf}{66}
  \item $\RR^\times$, \hyperlink{di}{7}

  \indexspace

  \item Scalar pair, \hyperlink{bci}{70}
  \item Shift
    \subitem of $g$-template $f^\brho$, \hyperlink{baa}{30}
    \subitem of height sequence $a^\brho_\bullet$, \hyperlink{ji}{29}
    \subitem sequence, \hyperlink{jj}{29}
  \item $\SL_n(\RR)$, \hyperlink{b}{3}
  \item $\SL_n(\ZZ)$, \hyperlink{c}{3}
  \item Standard weights, \hyperlink{bej}{98}
  \item Steps of Alice
    \subitem adding, \hyperlink{beb}{77}
    \subitem changing, \hyperlink{bec}{79}
    \subitem standard, \hyperlink{bed}{79}
  \item Steps of Bob
    \subitem idle, \hyperlink{beh}{90}
    \subitem interruption, \hyperlink{bei}{90}

  \indexspace

  \item Template, \hyperlink{fe}{11}
  \item $\Theta(-)$, \hyperlink{hb}{16}
  \item Transition arrow, \hyperlink{eg}{9}

  \indexspace

  \item $U_\varepsilon$, \hyperlink{eb}{7}
  \item $U_l(f)$, \hyperlink{fi}{11}

  \indexspace

  \item V:$V_\eta$, \hyperlink{ed}{8}
  \item Vanishing number, \hyperlink{beg}{90}
  \item $V_\bullet$, \hyperlink{cf}{6}
  \item Vertex, \hyperlink{jb}{26}
    \subitem impact, \hyperlink{jf}{28}
    \subitem non-null, \hyperlink{jc}{26}
    \subitem null, \hyperlink{jd}{26}
  \item $V^\to$, \hyperlink{bbf}{51}

  \indexspace

  \item $\bigwedge^l \RR^n$, \hyperlink{dg}{7}

  \indexspace

  \item $\Xi$, \hyperlink{gi}{15}
  \item $X_n$, \hyperlink{a}{3}

  \indexspace

  \item $Y_{\Lambda, \cF}$, \hyperlink{gg}{15}
  \item $Y_{\Lambda, f}$, \hyperlink{gj}{16}

\end{theindex}

\end{document}